\documentclass[twoside,11pt,reqno]{amsart}
\usepackage{amsmath,amssymb,amscd,mathrsfs,epic,empheq}
\usepackage{latexsym,amsthm,amscd,enumerate}
\makeatletter

\@addtoreset{equation}{section}
\@addtoreset{figure}{section}

\newtheorem{Proposition}{Proposition}[section]
\newtheorem{Lemma}[Proposition]{Lemma}
\newtheorem{Theorem}[Proposition]{Theorem}
\newtheorem{Corollary}[Proposition]{Corollary}

\newtheorem{Remark}[Proposition]{Remark}

\newtheorem{Example}[Proposition]{Example}

\def\phantomsubsection#1{\vspace{2mm}\noindent{\bf #1.}}

\newbox\squ  % box character for ends of proofs
\setbox\squ=\hbox{\vrule width.6pt
   \vbox{\hrule height.6pt width.4em\kern1ex
         \hrule height.6pt}%
   \vrule width.6pt}

\def\g{\mathfrak{g}}
\def\b{\mathfrak{b}}
\def\h{\mathfrak{h}}
\def\cD{\mathcal{D}}
\def\cE{\mathcal{E}}
\def\cO{\mathcal{O}}
\def\cF{\mathcal{F}}
\def\cT{\mathcal{T}}
\def\cG{\mathcal{G}}
\def\cL{\mathcal{L}}
\def\cV{\mathcal{V}}
\def\cS{\mathcal{S}}
\def\cP{\mathcal{P}}
\def\cY{\mathcal{Y}}

\def\T{{\mathtt T}}

\def\Laurent{\mathscr{A}}

\def\forget{\mathrm{f}}

\def\PImn{P(m,n;I)}

\def\KImn{K(m,n;I)}
\def\OImn{\cO(m,n;I)}
\def\LaImn{\La(m,n;I)}

\def\deg{\operatorname{deg}}
\def\defect{\operatorname{def}}
\def\caps{\operatorname{caps}}
\def\cups{\operatorname{cups}}

\def\circles{\operatorname{circles}}

\def\op{\operatorname{op}}

\def\height{\operatorname{ht}}
\def\down{\vee}
\def\up{\wedge}
\def\Proj#1{\operatorname{Proj}(#1)}
\def\Rep#1{\operatorname{Rep}(#1)}
\def\rep#1{\operatorname{rep}(#1)}

\def\mod#1{\operatorname{Mod}_{l\!f}(#1)}

\def\id{\operatorname{id}}

\def\C{{\mathbb C}}
\def\Q{{\mathbb Q}}
\def\Z{{\mathbb Z}}

\def\pr{{\operatorname{pr}}}
\def\rad{{\operatorname{rad}\:}}
\def\hom{{\operatorname{Hom}}}

\def\End{{\operatorname{End}}}

\def\im{{\operatorname{Im}\:}}

\def\eps{{\varepsilon}}
\def\delt{{\delta}}
\def\phi{{\varphi}}

\def\la{{\lambda}}
\def\La{{\Lambda}}
\def\Ga{{\Gamma}}
\def\ga{{\gamma}}
\def\de{{\delta}}
\def\De{{\Delta}}

\def\De{{\Delta}}
\def\al{{\alpha}}

\def\@underbar#1{\settowidth{\@tempdimb}{#1}\@tempdimb=0.8\@tempdimb
                   \ooalign{#1\crcr
                         \hfil\rule[-.5mm]{\@tempdimb}{.4pt}\hfil}}

\def\bi{\text{\boldmath$i$}}
\def\bj{\text{\boldmath$j$}}
\def\bLa{\text{\boldmath$\La$}}
\def\bGa{\text{\boldmath$\Ga$}}
\def\bDe{\text{\boldmath$\De$}}
\def\bt{\text{\boldmath$t$}}
\def\bs{\text{\boldmath$s$}}
\def\br{\text{\boldmath$r$}}
\def\bu{\text{\boldmath$u$}}
\def\bla{\text{\boldmath$\la$}}
\def\bde{\text{\boldmath$\de$}}
\def\bga{\text{\boldmath$\ga$}}

\def\bsigma{\text{\boldmath$\sigma$}}
\def\btau{\text{\boldmath$\tau$}}
\def\bepsilon{\text{\boldmath$\epsilon$}}

\newdimen\hoogte    \hoogte=14pt    % hoogte  van hokje
\newdimen\breedte   \breedte=14pt   % breedte van hokje
\newdimen\dikte     \dikte=0.5pt    % dikte lijn

\newenvironment{young}{\begingroup
       \def\vr{\vrule height0.8\hoogte width\dikte depth 0.2\hoogte}
       \def\fbox##1{\vbox{\offinterlineskip
                    \hrule height\dikte
                    \hbox to \breedte{\vr\hfill##1\hfill\vr}
                    \hrule height\dikte}}
       \vbox\bgroup \offinterlineskip \tabskip=-\dikte \lineskip=-\dikte
            \halign\bgroup &\fbox{##\unskip}\unskip  \crcr }
       {\egroup\egroup\endgroup}
\def\diagram#1{\relax\ifmmode\vcenter{\,\begin{young}#1\end{young}\,}\else%
              $\vcenter{\,\begin{young}#1\end{young}\,}$\fi}

\begin{document}

\title[Khovanov's diagram algebra III]{\boldmath Highest weight categories
arising from Khovanov's diagram algebra III: category $\cO$}
\author{Jonathan Brundan and Catharina Stroppel}

\address{Department of Mathematics, University of Oregon, Eugene, OR 97403, USA}
\email{brundan@uoregon.edu}
\address{Department of Mathematics, University of Bonn, 53115 Bonn, Germany}
\email{stroppel@math.uni-bonn.de}

\thanks{First author supported in part by NSF grant no. DMS-0654147}
\thanks{Second author supported by the NSF and the Minerva Research Foundation DMS-0635607.}

\begin{abstract}
We prove that integral
blocks of parabolic category $\cO$
associated to the subalgebra $\mathfrak{gl}_m(\C) \oplus \mathfrak{gl}_n(\C)$
of $\mathfrak{gl}_{m+n}(\C)$
are Morita equivalent to
quasi-hereditary covers of generalised Khovanov algebras.
Although this result is in principle known, the existing proof
is quite indirect, going via perverse sheaves on
Grassmannians.
Our new approach is completely algebraic, exploiting Schur-Weyl duality for higher levels.
As a by-product we get a concrete combinatorial construction of $2$-Kac-Moody representations in the sense of Rouquier
corresponding to level two weights in finite type $A$.
\end{abstract}
\maketitle

\small
\tableofcontents
\normalsize

\section{Introduction}

The {\em generalised Khovanov algebra} $H^n_m$ is a certain
positively graded symmetric algebra
defined via an explicit calculus of diagrams. It
was introduced by Khovanov in the case $m=n$
as part of his ground-breaking
work categorifying the Jones polynomial  \cite{K,K2}.
In \cite{CK, S}, the definition was extended to obtain
another algebra $K^n_m$, known as the
{\em quasi-hereditary cover of $H^n_m$}. This extension also has a natural interpretation in terms of the geometry of Springer fibres and can be realized as a certain convolution algebra; see \cite{SW}.

In \cite{BS1, BS2}, we undertook a systematic study of the
representation theory of $K^n_m$,
showing in particular that $K^n_m$ is Koszul
and computing the various natural bases for its
graded Grothendieck group in an explicit combinatorial fashion
using the diagram calculus. We also set up a general
theory of {projective functors} for the algebras $K^n_m$, extending
ideas of Khovanov from \cite{K2}.

The category $\rep{K^n_m}$ of finite dimensional left $K^n_m$-modules
is equivalent to the category of perverse sheaves
(constructible with respect to the Schubert stratification)
on the Grassmannian $\operatorname{Gr}(m,m+n)$ of $m$-dimensional subspaces of an $(m+n)$-dimensional
complex vector space.
This statement is proved
in the case $m=n$
in \cite[Theorem 5.8.1]{S}
as an application of work of Braden
\cite{Br}, and it should be possible to obtain the general case by similar
arguments as indicated in \cite[Remark 5.8.2]{S}.
In turn, by the Beilinson-Bernstein localisation theorem
and the Riemann-Hilbert correspondence, this category of perverse
sheaves
is equivalent to
the principal block
of the parabolic analogue of the Bernstein-Gelfand-Gelfand
category $\cO$ associated to the subalgebra
$\mathfrak{gl}_m(\C) \oplus \mathfrak{gl}_n(\C)$ of
 $\mathfrak{gl}_{m+n}(\C)$.

The main goal of this article is to give
a direct
algebraic construction of an  equivalence between
$\rep{K^n_m}$ and the principal block of the
parabolic category $\cO$ just mentioned.
Our new approach actually does rather more.
For one thing, it applies
to {\em all} integral blocks, not just the principal block; these
are also equivalent to categories of $K^n_m$-modules
but for possibly smaller $m$ and $n$.
We are also able to prove for the first time
that the diagrammatically defined
projective functors from \cite[$\S$4]{BS2} correspond
under the equivalence of categories to
the projective functors
from \cite{BG} that arise by tensoring with finite dimensional irreducible
modules. This is the key identification needed to
verify \cite[Conjecture 2.9]{Stqft}, which relates Khovanov's
functorial tangle invariants from \cite{K2} to
the functorial tangle invariants defined in \cite{Sduke}.

To formulate some of the
results in more detail,
let $\g := \mathfrak{gl}_{m+n}(\C)$, let
$\h$ be the standard Cartan subalgebra
consisting of diagonal matrices,
and let $\b$ be the standard Borel subalgebra of upper triangular matrices.
The dual space $\h^*$ has orthonormal basis
$\eps_1,\dots,\eps_{m+n}$ with respect to the trace form $(.,.)$,
where
$\eps_i$ is the weight picking out the $i$th diagonal entry of a
diagonal matrix.
Let $\mathfrak{l}$ denote
the naturally embedded subalgebra $\mathfrak{gl}_m(\C)\oplus \mathfrak{gl}_n(\C)$, and $\mathfrak{p} := \mathfrak{l}+\mathfrak{b}$ be the corresponding standard parabolic subalgebra. Let $\cO(m,n)$ be the category of finitely generated $\mathfrak{g}$-modules
that are
locally finite over $\mathfrak{p}$, semisimple over $\mathfrak{h}$, and
have all weights
belonging to $\Z \eps_1\oplus\cdots\oplus \Z \eps_{m+n}$.
This is the sum of all integral blocks of the parabolic category $\cO$
associated to $\mathfrak{p} \subseteq \mathfrak{g}$.
A full set of representatives for the isomorphism classes of
irreducible modules in $\cO(m,n)$
is given by the modules $\{\cL(\la)\:|\:\la \in \La(m,n)\}$, where
\begin{align}\label{isd}
\La(m,n) &:= \left\{\la \in \h^*\:\:\Bigg|\:
\begin{array}{l}
(\la+\rho,\eps_i) \in \Z\text{ for all }1 \leq i\leq m+n,\\
(\la+\rho,\eps_1) > \cdots > (\la+\rho,\eps_m),\\
(\la+\rho,\eps_{m+1}) > \cdots > (\la+\rho,\eps_{m+n})
\end{array}\right\},\\
\rho &:= -\eps_2-2\eps_3-\cdots-(m+n-1)\eps_{m+n} \in \h^*,\label{rhodef}
\end{align}
and $\cL(\la)$ denotes the irreducible $\g$-module
of highest weight $\la$, i.e. the irreducible $\mathfrak{g}$-module
generated by a highest weight
vector $v_+$ such that $xv_+ = 0$ for all $x$ in the nilradical of $\mathfrak{b}$
and $h v_+ = \la(h) v_+$ for all $h \in \mathfrak{h}$.

In order to be able to
exploit the diagram calculus from \cite{BS1,BS2},
we need to identify
weights in $\La(m,n)$ with weights in the combinatorial
sense of \cite[$\S$2]{BS1} via the following {\em weight dictionary}.
Given $\la \in \La(m,n)$ we define
\begin{align}\label{Id}
I_\down(\la)&:=\{(\la+\rho,\eps_1), \dots, (\la+\rho,\eps_m)\},\\
I_\up(\la)&:= \{(\la+\rho,\eps_{m+1}), \dots, (\la+\rho,\eps_{m+n})\}.\label{Iu}
\end{align}
Then we identify $\la$ with the diagram consisting of a number line
whose vertices are indexed by $\Z$
such that the $i$th vertex is labelled
\begin{equation}\label{dict}
\left\{
\begin{array}{ll}
\circ&\text{if $i$ does not belong to either $I_\down(\la)$ or $I_\up(\la)$,}\\
{\scriptstyle\down} &\text{if $i$ belongs to $I_\down(\la)$ but not to
$I_\up(\la)$,}\\
{\scriptstyle\up} &\text{if $i$ belongs to $I_\up(\la)$ but not to
$I_\down(\la)$,}\\
\times&\text{if $i$ belongs to both $I_\down(\la)$ and $I_\up(\la)$.}
\end{array}\right.
\end{equation}
For example, the zero weight (which parametrises the trivial module)
is identified with
$$
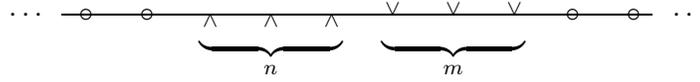
\begin{picture}(30,33)
\put(-118,17.25){$\cdots$}
\put(-98,20){\line(1,0){223}}
%\put(-41.9,2){\line(0,1){6}}
\put(92.5,17.3){$\circ$}
\put(115.5,17.3){$\circ$}
\put(-44.7,15.4){$\scriptstyle\up$}
\put(-21.7,15.4){$\scriptstyle\up$}
\put(1.3,15.4){$\scriptstyle\up$}
\put(24.3,20.1){$\scriptstyle\down$}
\put(47.3,20.1){$\scriptstyle\down$}
\put(70.3,20.1){$\scriptstyle\down$}
\put(-68.5,17.3){$\circ$}
\put(-91.5,17.3){$\circ$}
\put(133,17.5){$\cdots$}
\put(-46,11){$\underbrace{\phantom{hellow worl}}_{n}$}
\put(23,11){$\underbrace{\phantom{hellow worl}}_{m}$}
\end{picture}
$$
where the $\up$'s and $\down$'s are on the vertices
indexed $1-m-n,\dots,-1,0$.
Viewing $\la, \mu \in \La(m,n)$ as diagrams in this way,
the irreducible $\g$-modules
$\cL(\la)$ and $\cL(\mu)$ have the same central character
if and only if
$\la$
can be obtained from $\mu$ by permuting $\up$'s and $\down$'s.
This defines an equivalence relation $\sim$
on $\La(m,n)$.
We let $P(m,n)$ denote the set
$\La(m,n) / \!\!\sim$
of all $\sim$-equivalence
classes, and refer to elements of $P(m,n)$ as {\em blocks}.

For each block $\Ga \in P(m,n)$,
there is a finite dimensional algebra $K_\Ga$
exactly as defined in \cite[$\S$4]{BS1}.
As a vector space, $K_\Ga$ has basis
$$
\left\{(a \la b)\:|\:\text{for all oriented circle diagrams
$a \la b$ with $\la \in \Ga$}\right\},
$$
and its multiplication is defined by an explicit combinatorial procedure
in terms of such diagrams; see \cite[(6.3)]{BS1} for an example.
In particular, if $\Ga$ is the principal block, that is, the block generated
by the zero weight,
then the algebra
$K_\Ga$
is identified in an obvious way with
the quasi-hereditary cover $K^n_m$ of $H^n_m$.
Consider the (locally unital) algebra
\begin{equation}\label{KZ}
K(m,n) := \bigoplus_{\Ga \in P(m,n)} K_\Ga.
\end{equation}
Let $\rep{K(m,n)}$ denote the category
of (locally unital) finite dimensional left $K(m,n)$-modules.
According to \cite[$\S$5]{BS1},
the irreducible $K(m,n)$-modules are all one dimensional and their
isomorphism classes are
indexed in a canonical way by the set $\La(m,n)$.
Let $L(\la)$ denote the irreducible $K(m,n)$-module
associated to $\la \in \La(m,n)$.

\begin{Theorem}\label{thm1}
There is an equivalence of categories
$$\mathbb{E}:\cO(m,n) \stackrel{\sim}{\rightarrow}
\rep{K(m,n)}
$$
such that $\mathbb{E} (\cL(\la))\cong L(\la)$
for every $\la \in \La(m,n)$.
In particular, $\mathbb{E}$ restricts to an equivalence
between the principal block of $\cO(m,n)$
and $\rep{K^n_m}$.
\end{Theorem}

The next theorem is concerned
with projective functors.
On the diagram algebra side, these functors
arise from crossingless matchings.
More precisely, take $\Ga, \De \in P(m,n)$
and let $t$ be a proper $\De\Ga$-matching in the sense of
\cite[$\S$2]{BS2}. To this data there is associated a non-zero
$(K_\De,K_\Ga)$-bimodule $K_{\De\Ga}^t$; see \cite[$\S$3]{BS2}.
We view it as a $(K(m,n), K(m,n))$-bimodule
by extending the left action of $K_\De$ (resp.\ the right action of
$K_\Ga$) to
all of $K(m,n)$ by declaring that all the summands of (\ref{KZ})
different from $K_\De$ (resp.\ $K_\Ga$) act as zero.
Tensoring with this bimodule
defines an exact functor
\begin{equation*}\label{pfun}
K^t_{\De\Ga} \otimes_{K(m,n)}?:\rep{K(m,n)}
\rightarrow
\rep{K(m,n)}.
\end{equation*}
By \cite[Theorem 4.14]{BS2}, this functor
is indecomposable.
A {\em projective functor} on $\rep{K(m,n)}$ simply means
any endofunctor that is isomorphic to a finite direct sum of
such functors.

On the category $\cO$ side,
following \cite{BG}, a {\em projective functor}
means a functor that is isomorphic to a
summand of one of the exact endofunctors of $\cO(m,n)$
that arise
by tensoring with a finite dimensional rational $\g$-module.
In order to classify such projective functors, it suffices
by a variation on the Krull-Schmidt theorem to classify
the {\em indecomposable} projective functors; see e.g. \cite[$\S$3.1]{Sduke}.
The classification of
indecomposable projective functors on the principal block of
$\cO(m,n)$ can be deduced from \cite[Theorem 5.1]{Sduke}
(using \cite[Proposition 4.2]{Sduke} to determine which projective functors
have non-zero restrictions). This had been conjectured earlier
by Bernstein, Frenkel and Khovanov \cite[p.237]{BFK}.
The following theorem
gives an alternative approach to this classification, and
extends it to arbitrary integral blocks.

\begin{Theorem}\label{thm2}
Given blocks $\Ga,\De \in P(m,n)$
and a proper $\De\Ga$-matching $t$,
there is an indecomposable projective functor
$\cG^t_{\De\Ga}:\cO(m,n) \rightarrow \cO(m,n)$
and an isomorphism of functors
$$
\mathbb{E} \circ \cG^t_{\De\Ga} \cong
(K_{\De\Ga}^t \otimes_{K(m,n)} ?) \circ \mathbb{E}:
{\cO}(m,n) \rightarrow \rep{K(m,n)}.
$$
Every indecomposable projective functor on $\cO(m,n)$
is isomorphic to such a functor $\cG^t_{\De\Ga}$ for unique $\Ga,\De$ and $t$.
\end{Theorem}

One pleasant feature of the diagram algebra setup is that it makes
some natural but hard-to-see gradings
on the category $\cO$ side absolutely explicit.
Indeed, the algebra $K(m,n)$ carries an obvious grading
with respect to which its diagram basis is homogeneous.
This grading makes $K(m,n)$ into a
Koszul algebra; see \cite[Theorem 5.6]{BS2}.
In view of Theorem~\ref{thm1},
the category $\Rep{K(m,n)}$
of finite dimensional {\em graded}
left $K(m,n)$-modules can be interpreted as a graded version
of $\cO(m,n)$.
By the unicity of Koszul gradings \cite[$\S$2.5]{BGS}, this
 is equivalent to the
graded version of $\cO(m,n)$ constructed geometrically
(and in far greater generality)
in \cite[$\S$3.11]{BGS} and \cite{Back}.
The bimodules $K_{\De\Ga}^t$ are also naturally graded, so
in view of Theorem~\ref{thm2} they
define explicit graded lifts of the indecomposable
projective functors on $\cO(m,n)$.

For the proofs, the basic idea
is to exploit a special case of
the Schur-Weyl duality for higher levels
developed in \cite{BKschur}.
Recall that classical Schur-Weyl duality relates
polynomial representations of $\mathfrak{gl}_{m+n}(\C)$
to the representation theory of symmetric groups.
Schur-Weyl duality for level two
relates the category $\cO(m,n)$
to the representation theory of
degenerate cyclotomic Hecke algebras of level two.
By mimicking this Schur-Weyl duality on the diagram algebra
side, we obtain a natural realisation of another
(at first sight quite different) family of
Hecke algebras, namely, level two cyclotomic quotients $R^\La_\alpha$ of
certain algebras introduced independently by
Khovanov and Lauda in \cite{KLa}
and Rouquier in \cite{Rou} (for Cartan matrices of finite type $A$).
The bridge between parabolic category $\cO$
and the diagram algebra side finally comes as an application of
the main result of \cite{BKyoung}, which asserts that
these two sorts of Hecke algebra arising from the
two versions of Schur-Weyl duality are actually isomorphic algebras.

As a by-product of the argument, we also obtain an elementary
proof of the categorification
conjecture formulated by Khovanov and Lauda in \cite[$\S$3.4]{KLa}
for level two weights in finite type A;
see also \cite{BKllt, VV} which treat much more general situations
(but using geometric methods).
At the same time we give a conceptual interpretation of the
grading on $R^\La_\alpha$: for us this algebra arises naturally as the endomorphism algebra of a certain
projective module in $\mathcal O(m,n) \cong \rep{K(m,n)}$, and the
$\Z$-grading on it is the grading
induced by the Koszul grading on these projective modules suitably shifted in degree.
Finally our methods yield a special {\em graded cellular basis} for $R^\La_\alpha$  parametrised by some diagrams which are in bijection with
certain Young tableaux; see \cite[Remark 4.12]{BKW} where the existence of such bases is predicted
in more general situations.
In particular we deduce from this a graded dimension formula for the irreducible $R^\La_\alpha$-modules (in level two for  finite type $A$).

\vspace{2mm}

\noindent
{\em Acknowledgements.}
Both authors thank Alexander Kleshchev, Andrew Mathas and Rapha\"el Rouquier for useful conversations.
The second author acknowledges support from a Von Neumann Fellowship at the Institute for Advanced Study, Princeton,
where part of this research was carried out.

\vspace{2mm}

\noindent
{\em Notation.}\label{II}
For the remainder of the article, we fix
an index set $I$ that is a non-empty, bounded-below set of
consecutive integers and let $m, n$ be integers
with $0 \leq m,n \leq |I|+1$.
Set $I^+ := I \cup (I+1)$
and $o := \min(I) - 1$.
The reader will lose little in generality by assuming that
$I = I^+ = \{1,2,3,\dots\}$ and $o = 0$.

\section{Combinatorics of canonical bases}\label{sB}

This and the next two sections are concerned with some
essential combinatorial book-keeping at the level of Grothendieck groups.
We begin in this section by introducing an auxiliary space
$\textstyle \bigwedge^m V \otimes \bigwedge^n V$
together with three natural bases, namely, the monomial basis,
the dual-canonical basis and the quasi-canonical basis,
following the setup of \cite[section 2]{BKariki} closely.

\phantomsubsection{The quantised enveloping algebra}
We begin with some basic notions
related to the general linear Lie algebra of $I^+ \times I^+$
matrices. The underlying {\em weight lattice} $P$ is
the free $\Z$-module
on basis $\{\delt_i\:|\:i \in I^+\}$ equipped with a
bilinear form $(.,.)$ such that $(\delt_i, \delt_j)  = \delta_{i,j}$
(where $\delta_{i,j}$ is the usual Kronecker $\delta$-function).
For $i \in I$ and $j \in I^+$, let
\begin{equation*}\label{srp}
\alpha_i := \delt_i - \delt_{i+1},\qquad
\La_j := \sum_{I \ni i \leq j} \delt_i
\end{equation*}
denote the $i$th {\em simple root} and the $j$th {\em fundamental weight},
respectively; note that $(\alpha_i, \La_j) = \delta_{i,j}$.
The {\em root lattice} $Q$ is the $\Z$-submodule of $P$
generated by the simple roots. Let $Q_+$ be the subset of $Q$
consisting of all $\alpha$ that have non-negative
coefficients when expressed in terms of the simple roots,
and define the {\em height} $\height(\alpha)$
to be the sum of these coefficients.

Let $U$ denote the (generic) quantised enveloping algebra
associated to this root datum.\label{qea}
So $U$ is the associative algebra over the field of rational functions
$\Q(q)$ in an indeterminate $q$, with generators
$\{E_i, F_i\:|\:i \in I\} \cup \{D_i, D_i^{-1}\:|\:i \in I^+\}$
subject to the following well-known relations:
\begin{align*}
D_iD_i^{-1} &= D_i^{-1}D_i = 1,&E_iE_j &= E_jE_i&\hbox{if $|i-j|>1$},\\
D_iD_j &= D_jD_i,&\!\!\!\!E_i^2E_j  + E_jE_i^2 &= (q+q^{-1})E_iE_jE_i&\hbox{if $|i-j|=1$},\\
D_iE_jD_i^{-1} &= q^{(\delt_i, \alpha_j)} E_j,&F_iF_j &= F_jF_i&\hbox{if $|i-j| > 1$,}\\
D_iF_jD_i^{-1} &= q^{-(\delt_i, \alpha_j)} F_j,&\!\!\!\!F_i^2F_j+ F_jF_i^2 &=   (q+q^{-1})F_iF_jF_i&\hbox{if $|i-j|=1$,}
\end{align*}
\begin{equation*}
\quad
E_iF_j - F_jE_i = \delta_{i, j} \frac{D_i D_{i+1}^{-1} - D_{i+1} D_i^{-1}}{q - q^{-1}}.
\end{equation*}
We view $U$ as a Hopf algebra with comultiplication
$\Delta$ defined on generators by
\begin{align*}
\Delta(E_i) &= 1 \otimes E_i + E_i \otimes D_i D_{i+1}^{-1},\\
\Delta(F_i) &= F_i \otimes 1 + D_i^{-1}D_{i+1}  \otimes F_i,\\
\Delta(D_i^{\pm 1}) &= D_i^{\pm 1} \otimes D_i^{\pm 1}.
\end{align*}

\phantomsubsection{\boldmath The space $\bigwedge^m V \otimes \bigwedge^n V$}
Let $V$ denote the natural $U$-module
with basis
$\{v_i\:|\:i \in I^+\}$. The generators act on this basis by
the rules
$$
E_i v_j = \delta_{i+1, j} v_i,
\qquad
F_i v_j = \delta_{i,j} v_{i+1},
\qquad
D_i^{\pm 1} v_j = q^{\pm \delta_{i,j}} v_j.
$$
Following \cite[$\S$5]{B2} (noting the roles of $q$
and $q^{-1}$ are switched there),
we define the $n$th {\em quantum exterior power} $\bigwedge^n V$
to be the $U$-submodule of $\bigotimes^n V$
with basis given by the vectors
\begin{align}\label{bvs}
v_{i_1} \wedge\cdots \wedge v_{i_n} &:=
\sum_{w \in S_n} (-q)^{\ell(w)} v_{i_{w(1)}} \otimes\cdots\otimes v_{i_{w(n)}}
\end{align}
for all $i_1 > \cdots > i_n$ from the index set $I^+$.
Here, $\ell(w)$ denotes the usual length of a permutation
$w$ in the symmetric group $S_n$.

Consider the $U$-module
$\bigwedge^m V \otimes \bigwedge^n V$.
It has the obvious monomial basis
\begin{equation}\label{mb1}
\left\{
(v_{i_1}\wedge\cdots\wedge v_{i_m}) \otimes
(v_{j_1}\wedge\cdots\wedge v_{j_n})\:\bigg|\:
\begin{array}{l}
i_1,\dots,i_m,j_1,\dots,j_n \in I^+,\\
i_1>\cdots> i_m, j_1 > \cdots > j_n
\end{array}\right\}.
\end{equation}
Each vector $(v_{i_1}\wedge\cdots\wedge v_{i_m}) \otimes
(v_{j_1}\wedge\cdots\wedge v_{j_n})$ from this basis
is of weight
$(\delt_{i_1}+\cdots+\delt_{i_m})+(\delt_{j_1}+\cdots+\delt_{j_n}) \in P$.
Let
$\PImn$ denote
the set of all $\Ga \in P$ such that
\begin{itemize}
\item $0 \leq (\Ga,\delt_i) \leq 2$ for all $i \in I^+$;
\item $\sum_{i \in I^+} (\Ga,\delt_i) = m+n$;
\item the number of $i \in I^+$ such that $(\Ga,\delt_i) = 2$
is at most $\min(m,n)$.
\end{itemize}\label{pimn}
In other words, $\PImn$ is the set of weights that arise with non-zero multiplicity
in the module $\bigwedge^m V \otimes \bigwedge^n V$.
We reserve the notation $\La$ from now on for the
weight
\begin{equation}\label{gsw}
\La := \La_{o+m}+\La_{o+n},
\end{equation}
recalling that $o = \min(I)-1$ and $\La_i$ denotes the $i$th
fundamental weight.
This is the unique maximal element of $\PImn$ in
the dominance ordering, i.e.
all elements of $\PImn$ are of the form $\La - \alpha$
for $\alpha \in Q_+$.

\phantomsubsection{Combinatorics of weights and blocks}
Unfortunately the word ``weight'' gets over-used in this business.
In the remainder of the article, the terminology {\em weight} will always refer to
weights in the combinatorial sense of
\cite[$\S$2]{BS1}, namely,
diagrams consisting of a number line
with vertices labelled by the symbols
$\circ$, $\down$, $\up$ or $\times$.\label{wtsdef}
Recall also from \cite[$\S$2]{BS1}
the Bruhat order $\leq$ on weights,
which is generated by the basic operation of interchanging
$\down\up$ pairs of labels, and the equivalence relation
$\sim$, which arises by permuting
$\up$'s and $\down$'s.

Let $\LaImn$ denote the set of all weights drawn on a number line
with vertices indexed by $I^+$
such that exactly
$m$ of the vertices are labelled $\down$ or $\times$ and
exactly $n$ vertices are labelled $\up$ or $\times$.
By a {\em block} we mean a $\sim$-equivalence class of weights
from $\LaImn$.
It is often convenient to
represent such a block $\Ga$ diagrammatically
by replacing all the vertices labelled $\down$ or $\up$
in the weights from $\Ga$ by the symbol $\bullet$.
For example, taking $I= \{1,\dots,8\}$, $m=5$ and $n=4$,
the block $\Ga$ generated by the weight
\begin{equation}\label{eg}
\begin{picture}(54,10)
\put(-93,1){$\la=$}
\put(-64.8,3){\line(1,0){184}}
\put(-68.2,1){$\scriptstyle\times$}
\put(-44.8,.2){$\circ$}
\put(46.8,1){$\scriptstyle\times$}
\put(93.2,.2){$\circ$}
\put(-21.7,-1.6){$\scriptstyle\up$}
\put(70.3,-1.6){$\scriptstyle\up$}
\put(1.3,3.1){$\scriptstyle\down$}
\put(116.3,3.1){$\scriptstyle\down$}
\put(24.3,3.1){$\scriptstyle\down$}
\end{picture}
\end{equation}
is represented by the block diagram
\begin{equation}\label{egb}
% scriptstyles at (1,x) x = -68,-45,-22,1,24,47,70,93,...
% \times is -0.2
% \circ is +1
% \up is +0.3 and down 2.6
% \down is +0.3 and up 2.1
\begin{picture}(54,10)
\put(-93,1){$\Ga=$}
\put(-64.8,3){\line(1,0){184}}
\put(-68.2,1){$\scriptstyle\times$}
\put(-44.8,.2){$\circ$}
\put(46.8,1){$\scriptstyle\times$}
\put(93.2,0.2){$\circ$}
\put(-21.4,1){$\scriptstyle\bullet$}
\put(70.6,1){$\scriptstyle\bullet$}
\put(1.6,1){$\scriptstyle\bullet$}
\put(116.6,1){$\scriptstyle\bullet$}
\put(24.6,1){$\scriptstyle\bullet$}
\end{picture}
\end{equation}
Abusing notation further, we identify the set
$\PImn$ defined in the previous subsection
with the set
$\LaImn / \!\!\sim$ of all blocks
by identifying $\Ga \in \PImn$ with
the block diagram whose $i$th vertex is labelled
$\circ, \bullet$ or $\times$ according to whether
$(\Ga,\delt_i) = 0$, $1$ or $2$.
For example, the special element $\La \in \PImn$ from (\ref{gsw})
is identified in this way with the block
consisting of just one weight, namely, the weight
\vspace{4mm}
\begin{equation}\label{groundstate}
\iota := \left\{
\begin{array}{ll}
\hspace{86mm}
&\text{\!\!\!\!\!if $m \geq n$,}\\\\
&\text{\!\!\!\!\!if $m \leq n$.}\\
\end{array}
\right.
\begin{picture}(0,30)
\put(-93,12.2){$\cdots$}
\put(-294,14){$\overbrace{\phantom{hellow worl}}^{n}$}
\put(-227,14){$\overbrace{\phantom{hellow worl}}^{m-n}$}
\put(-291.8,15){\line(1,0){191}}
\put(-295.2,13){$\scriptstyle\times$}
\put(-272.2,13){$\scriptstyle\times$}
\put(-249.2,13){$\scriptstyle\times$}
\put(-225.7,15.1){$\scriptstyle\down$}
\put(-202.7,15.1){$\scriptstyle\down$}
\put(-179.7,15.1){$\scriptstyle\down$}
\put(-156.8,12.2){$\circ$}
\put(-133.8,12.2){$\circ$}
\put(-110.8,12.2){$\circ$}
\end{picture}
\begin{picture}(0,0)
\put(-93,-13.8){$\cdots$}
\put(-294,-18){$\underbrace{\phantom{hellow worl}}_{m}$}
\put(-227,-18){$\underbrace{\phantom{hellow worl}}_{n-m}$}
\put(-291.8,-11){\line(1,0){191}}
\put(-295.2,-13){$\scriptstyle\times$}
\put(-272.2,-13){$\scriptstyle\times$}
\put(-249.2,-13){$\scriptstyle\times$}
\put(-225.7,-15.6){$\scriptstyle\up$}
\put(-202.7,-15.6){$\scriptstyle\up$}
\put(-179.7,-15.6){$\scriptstyle\up$}
\put(-156.8,-13.8){$\circ$}
\put(-133.8,-13.8){$\circ$}
\put(-110.8,-13.8){$\circ$}
\end{picture}
\end{equation}
\vspace{4mm}

\noindent
We refer to this special weight as the {\em ground-state}.
Recall finally the
notion of {\em defect} $\defect(\Ga)$
of a block $\Ga$
from \cite[$\S$2]{BS1}. In our setting, we have simply that
\begin{equation}\label{defectdef}
\defect(\Ga)=\min(m,n) - \#\left(
\begin{array}{c}
\text{vertices labelled $\times$ in}\\
\text{the diagram for $\Ga$}
\end{array}
\right).
\end{equation}
For example, the block $\Ga$ from (\ref{egb})
is of defect $2$.
In Lie theoretic terms, we have equivalently
that
\begin{equation}
\defect(\Ga) = ((\La,\La)-(\Ga,\Ga))/2.
\end{equation}
This formula gives meaning to the notion of
defect for more general $\Ga \in P$ that do not necessarily belong to $\PImn$.

\phantomsubsection{The monomial basis}
With this combinatorial notation behind us,
given $\la \in \LaImn$,
define
$$\label{mon}
V_\la := (v_{i_1} \wedge \cdots \wedge v_{i_m}) \otimes
(v_{j_1} \wedge\cdots \wedge v_{j_n})
$$
where $i_1 > \cdots > i_m$ index the vertices of $\la$
labelled $\down$ or $\times$ and
$j_1 > \cdots > j_n$ index the vertices of $\la$ labelled
$\up$ or $\times$.
For example, if $\la$ is as in (\ref{eg})
then
$V_\la = (v_9 \wedge v_6 \wedge v_5 \wedge v_4 \wedge v_1) \otimes
(v_7 \wedge v_6 \wedge v_3 \wedge v_1)$.
The monomial basis for the space
$\bigwedge^m V \otimes \bigwedge^n V$
from (\ref{mb1})
is then the set
\begin{equation*}
\{V_\la\:|\:\la \in \LaImn\}.
\end{equation*}

\begin{Remark}\label{dinnertime}\rm
To help the reader to
make the connection with combinatorics elsewhere in the literature (e.g. as in
\cite{BKW, BKllt}), we note that there is an inclusion
$\LaImn \hookrightarrow \mathscr P^2$, where $\mathscr P^2$ denotes the
set of all {\em bipartitions}, meaning pairs $(\la^{(1)}, \la^{(2)})$
of partitions in the usual sense.
To define this,
take a weight $\la \in \LaImn$ and read off the sequences
$i_1 > \cdots > i_m$ and $j_1 > \cdots > j_n$ as above.
Then we associate to $\la$ the bipartition $(\la^{(1)}, \la^{(2)})$ where
$\la^{(1)} = (\la^{(1)}_1 \geq \la^{(1)}_2 \geq \cdots)$ and
$\la^{(2)} = (\la^{(2)}_1 \geq \la^{(2)}_2 \geq \cdots)$
are defined from
$$
\la_r^{(1)} :=
i_r-o-m+r-1,\qquad
\la_s^{(2)} :=
j_s-o-n+s-1
$$
for $1 \leq r \leq m$ and $1 \leq s \leq n$, with all other parts of $\la^{(1)}$ and $\la^{(2)}$ being zero.
Assuming $I$ is not bounded above, this map gives a bijection between
$\LaImn$ and the set of all bipartitions $(\la^{(1)}, \la^{(2)})$ where $\la^{(1)}$ has at most
$m$ and $\la^{(2)}$ has at most $n$ non-zero parts.
\end{Remark}

\phantomsubsection{\bf The dual-canonical basis}
As well as the monomial basis,
we need to introduce two other bases for the space
$\bigwedge^m V\otimes \bigwedge^n V$.
The first of these is
the so-called {\em dual-canonical basis}
\begin{equation*}\label{dcb}
\{L_\la\:|\:\la \in \LaImn\}.
\end{equation*}
To define $L_\la$ formally
following \cite[$\S$2.3]{BKariki},
we need some bar-involutions.
The {\em bar-involution} on $U$ is the
automorphism $-:U \rightarrow U$ that is
anti-linear with respect to the field automorphism
$\Q(q) \rightarrow \Q(q), f(q) \mapsto f(q^{-1})$
and satisfies
\begin{equation}\label{barinv}
\overline{E_i} = E_i,\qquad
\overline{F_i} = F_i,\qquad
\overline{K_i} = K_i^{-1}.
\end{equation}
By a {\em compatible bar-involution} on a $U$-module $M$
we mean an anti-linear involution
$-:M \rightarrow M$
such that $\overline{uv} = \overline{u}\,\overline{v}$
for each $u \in U, v \in M$.
The next lemma shows that our module $\bigwedge^m V \otimes \bigwedge^n V$
possesses a compatible bar-involution.

\begin{Lemma}\label{charac}
There is a unique compatible bar-involution on
$\bigwedge^m V \otimes \bigwedge^n V$
such that $\overline{V_\la} = V_\la$ for each weight
$\la \in \LaImn$ that is minimal with respect to
the Bruhat order.
Moreover:
$$
\overline{V_\la} = V_\la +
 \left(\text{a $\Z[q,q^{-1}]$-linear combination of $V_\mu$'s
for $\mu < \la$}\right)
$$
for any $\la \in \LaImn$.
\end{Lemma}

\begin{proof}
The module $\bigwedge^n V$ possesses a
compatible bar-involution, namely, the unique anti-linear involution
fixing the basis vectors
of the form (\ref{bvs}).
Similarly so does $\bigwedge^m V$.
Combining this with Lusztig's tensor product
construction from
\cite[$\S$27.3]{Lubook},
we obtain a compatible bar-involution on
$\bigwedge^m V \otimes \bigwedge^n V$
as in the statement of the lemma.
Uniqueness can be checked by induction on the Bruhat ordering;
see \cite[Proposition 4.5]{CWZ}.
\end{proof}

Now we can define the dual-canonical basis element $L_\la
\in \bigwedge^m V \otimes \bigwedge^n V$ for
any $\la \in \LaImn$: it is the unique bar invariant vector
such that
$$\label{ib}
L_\la = V_\la + \left(\text{a $q \Z[q]$-linear combination of $V_\mu$'s
for $\mu \in \LaImn$}\right).
$$
The existence and uniqueness of $L_\la$ follows from Lemma~\ref{charac}
by a general argument originating in \cite{KL},
sometimes known as Lusztig's lemma; see e.g.
\cite[1.2]{Du} for a concise formulation.
The polynomials $d_{\la,\mu}(q), p_{\la,\mu}(q) \in \Z[q]$ defined from
\begin{align}
V_\mu &= \sum_{\la \in \LaImn} d_{\la,\mu}(q) L_\la\label{e1},\\
L_\mu &= \sum_{\la \in \LaImn} p_{\la,\mu}(-q) V_\la\label{e2}
\end{align}
satisfy
$p_{\la,\la}(q) = d_{\la,\la}(q) = 1$ and
$p_{\la,\mu}(q) =  d_{\la,\mu}(q) = 0$ unless $\la \leq \mu$.

\begin{Remark}\rm
Although not needed explicitly here, it is important to note that
the polynomials $d_{\la,\mu}(q)$ and $p_{\la,\mu}(q)$
can be expressed in terms of certain
special Kazhdan-Lusztig polynomials associated to the
symmetric group $S_{m+n}$; see \cite[Remark 14]{B2}.
In particular, up to a trivial renormalisation, the $p_{\la,\mu}(q)$'s are
Deodhar's parabolic Kazhdan-Lusztig
polynomials associated to the subgroup $S_m \times S_n$ of $S_{m+n}$.
See also Remark~\ref{Exp} below for more about these polynomials.
\end{Remark}

\phantomsubsection{The quasi-canonical basis}
There is another basis
\begin{equation*}\label{qcb}
\{P_\la\:|\:\la \in \LaImn\}
\end{equation*}
for
$\bigwedge^m V \otimes \bigwedge^n V$ such that
\begin{align}
P_\la &= \sum_{\mu \in \LaImn} d_{\la,\mu}(q) V_\mu,\label{boo}\\
\label{boo2}
V_\la &= \sum_{\mu \in \LaImn} p_{\la,\mu}(-q) P_\mu.
\end{align}
(We have simply transposed the transition matrices
(\ref{e1})--(\ref{e2})).
Following \cite[$\S$2.6]{BKariki},
we call this the {\em quasi-canonical basis} to emphasise that it is
{\em not} the same as Lusztig's canonical basis
from
\cite[$\S$27.3]{Lubook}: our $P_\la$'s are not in general invariant under the
bar-involution.
Nevertheless, our quasi-canonical and dual-canonical bases are
dual to each other in a suitable sense.

To explain this,
let $\langle.,.\rangle$ be the sesquilinear form
on $\bigwedge^m V \otimes \bigwedge^n V$
(anti-linear in the first argument, linear in the second argument)
such that
\begin{equation}\label{e3}
\langle V_\la, \overline{V_\mu} \rangle = \delta_{\la,\mu}
\end{equation}
for each $\la,\mu \in \LaImn$.
Then, by a straightforward computation
using (\ref{boo}) and the formula obtained from (\ref{e2}) by applying
the bar-involution, we get that
\begin{equation}\label{e4}
\langle P_\la, L_\mu \rangle = \delta_{\la,\mu}
\end{equation}
for $\la,\mu \in \LaImn$.
For this to be useful,
we need to formulate one other basic property of the form $\langle.,.\rangle$.
Let $\tau:U \rightarrow U$ be the anti-linear anti-automorphism
such that
\begin{equation}\label{taudef}
\tau(E_i) = q F_i D_i^{-1}D_{i+1},\qquad
\tau(F_i) = q^{-1} D_iD_{i+1}^{-1} E_i,\qquad
\tau(D_i) = D_i^{-1}.
\end{equation}
Then, as can be checked directly from (\ref{e3}), we have that
\begin{equation}\label{berkeley}
\langle ux, y \rangle = \langle x, \tau(u) y \rangle
\end{equation}
for all $x, y \in \bigwedge^m V \otimes \bigwedge^n V$
and $u \in U$.

\phantomsubsection{\boldmath Specialisation at $q=1$}\label{spec}
Let $\Laurent := \Z[q,q^{-1}]$ and
$U_{\!\Laurent}$ denote
Lusztig's $\Laurent$-form for $U$.
This is the $\Laurent$-subalgebra of $U$ generated
by the quantum
divided powers $E_i^{(r)}$, $F_i^{(r)}$,
the elements $D_i, D_i^{-1}$, and the elements
$$
\left[\!\!\!\begin{array}{c}D_i\\r\end{array}\!\!\!\right]
:= \prod_{s=1}^r \frac{D_i q^{1-s} - D_i^{-1} q^{s-1}}{q^s-q^{-s}}
$$
for all $i$ and $r \geq 0$.
The Hopf algebra structure on $U$ makes
$U_{\!\Laurent}$ into a Hopf algebra over
$\Laurent$.
The maps (\ref{barinv}) and (\ref{taudef})
restrict to well-defined maps on $U_{\!\Laurent}$ too.
Let $V_\Laurent$ denote the $U_{\!\Laurent}$-submodule of $V$
generated as a free $\Laurent$-module by the basis vectors
$\{v_i\:|\:i \in I^+\}$.
Taking all tensor products over $\Laurent$ instead of $\Q(q)$,
we construct the $U_{\!\Laurent}$-module
$\bigwedge^m V_\Laurent \otimes \bigwedge^n V_\Laurent$ as above.
It is a free $\Laurent$-submodule
of $\bigwedge^m V \otimes \bigwedge^n V$ with the three distinguished
bases $\{V_\la\}$, $\{L_\la\}$ and $\{P_\la\}$,
all indexed by the set $\LaImn$.

Let $U_\Z$ and
$\bigwedge^m V_\Z \otimes \bigwedge^n V_\Z$ denote specialisations of
$U_{\!\Laurent}$ and $\bigwedge^m V_\Laurent \otimes \bigwedge^n V_\Laurent$
at $q=1$, i.e. we
apply the base change functor $\Z \otimes_{\Laurent} ?$
viewing $\Z$ as an $\Laurent$-module so that
$q$ acts as $1$.
It causes no problems here to
replace $U_\Z$ with the usual Kostant $\Z$-form
for the universal enveloping algebra of the general linear Lie algebra
of $I^+ \times I^+$ matrices,
with Chevalley generators $\cE_i$ and $\cF_i$ for $i \in I$.
The bases $\{V_\la\}, \{L_\la\}$ and $\{P_\la\}$
specialise to give bases
$\{\cV_\la\}, \{\cL_\la\}$ and $\{\cP_\la\}$
for
$\bigwedge^m V_\Z \otimes \bigwedge^n V_\Z$
as a free $\Z$-module,
with
\begin{align}\label{bgg}
\cV_\mu = \sum_{\la \in \LaImn} d_{\la,\mu}(1) \cL_\la,\qquad
\cP_\la = \sum_{\mu \in \LaImn} d_{\la,\mu}(1) \cV_\mu.
\end{align}
The form $\langle.,.\rangle$ specialises to a
symmetric bilinear form
on
$\bigwedge^m V_\Z \otimes \bigwedge^n V_\Z$
with respect to which the basis $\{\cV_\la\}$ is orthonormal
and the bases $\{\cP_\la\}$ and $\{\cL_\la\}$ are dual.
Moreover the Chevalley generators $\cE_i$ and $\cF_i$ are biadjoint
with respect to this form, i.e.
$\langle \cE_i v, w \rangle =
\langle v, \cF_i w \rangle$
and
$\langle \cF_i v, w \rangle =
\langle v, \cE_i w \rangle$.

\section{Categorification via diagram algebras}\label{sC1}

Following the ideas of
\cite{HK, Ch}
we next construct a graded diagram algebra
$\KImn$, and show that the Grothendieck group of the
category of graded $\KImn$-modules
can be identified with $\bigwedge^m V_\Laurent \otimes \bigwedge^n V_\Laurent$
so that
the standard modules, irreducible modules
and projective indecomposable modules correspond to the
monomial, dual-canonical and quasi-canonical bases, respectively.

\phantomsubsection{\boldmath The category $\Rep{\KImn}$}
If $A$ is any locally unital graded algebra, we write
$\rep{A}$ for the category of all
finite dimensional locally unital
left $A$-modules
and $\Rep{A}$ for the category of all
finite dimensional {\em graded} locally unital
left $A$-modules.
There an obvious functor
\begin{equation}\label{forget1}
\forget:\Rep{A} \rightarrow \rep{A}
\end{equation}
that forgets the grading on a module.
Recalling that $\Laurent$ denotes $\Z[q,q^{-1}]$,
the Grothendieck group $[\Rep{A}]$ of the category $\Rep{A}$
is naturally an $\Laurent$-module so that
$q^i [M] = [M \langle i \rangle]$ for each $i \in \Z$
and $M \in \Rep{A}$, where $\langle i \rangle$ denotes the
degree shift functor defined so that
$M \langle i \rangle_j =
M_{i-j}$.
We refer to \cite[$\S$5]{BS1} for other conventions regarding graded
modules over locally unital algebras.

Introduce the locally unital graded algebra
\begin{equation}\label{KI}
\KImn :=
\bigoplus_{\Ga \in \PImn} K_\Ga,
\end{equation}
where $K_\Ga$ is the
finite dimensional graded
algebra defined in \cite[$\S$4]{BS1}.
Here we are using the identification explained just after (\ref{egb})
of elements of $\PImn$ with
blocks of weights from $\LaImn$.
For $\Ga \in P(m,n;I)$ and $\la \in \Ga$,
we have the
$K_\Ga$-modules $L(\la), V(\la)$ and $P(\la)$, which are the
irreducible, cell and projective indecomposable modules from
\cite[$\S$5]{BS1}, respectively.\label{cellmod}
We always view them as $\KImn$-modules by extending the $K_\Ga$-action so that all the summands from (\ref{KI}) different from $K_\Ga$
act as zero.

The modules $L(\la)$, $V(\la)$ and $P(\la)$
are naturally graded so that $L(\la)$ is concentrated in degree zero
and the canonical quotient maps
$P(\la) \twoheadrightarrow V(\la) \twoheadrightarrow L(\la)$
are homogeneous of degree zero.
By \cite[Theorem 5.3]{BS1},
$\Rep{\KImn}$ is a graded
highest weight category and the cell modules are
its standard modules in the general sense of \cite{CPS}.
Because of this, we refer to $V(\la)$ as a
{\em standard module} from now on.
The isomorphism classes
$\{[L(\la)]\:|\:\la \in \LaImn\}$,
$\{[V(\la)]\:|\:\la \in \LaImn\}$
and
$\{[P(\la)]\:|\:\la \in \LaImn\}$
give three distinguished bases for
 $[\Rep{\KImn}]$ as a free $\Laurent$-module.

There is also a duality $\circledast$ on $\Rep{\KImn}$
which induces an anti-linear involution
\begin{equation*}
\circledast:[\Rep{\KImn}] \rightarrow [\Rep{\KImn}]
\end{equation*}
on the Grothendieck group fixing the $[L(\la)]$'s; see \cite[(5.4)]{BS1}.

\phantomsubsection{\boldmath Geometric bimodules and projective functors}\label{sgbpf}
Recall for blocks
$\Ga, \De \in \PImn$
and a $\De\Ga$-matching $t$ in the sense of \cite[$\S$2]{BS2}
that there is associated a graded $(K_\De, K_\Ga)$-bimodule
$K_{\De\Ga}^t$; see \cite[$\S$3]{BS2}.
This bimodule is non-zero if and only if $t$ is a {\em proper}
$\De\Ga$-matching,
i.e. at least one oriented $\De\Ga$-matching
$\de t \ga$ exists; here, $\de$ and $\ga$ are weights from
$\De$ and $\Ga$, respectively.
As in the introduction, we
always view $K_{\De\Ga}^t$ as a $(\KImn,\KImn)$-bimodule
by extending the $K_\De$- and $K_\Ga$-actions to all of $\KImn$
in the obvious way.
Writing $\caps(t)$ (resp.\ $\cups(t)$)\label{ncup}
for the number of caps (resp.\ cups) in the matching $t$, let
\begin{equation*}
G^t_{\De\Ga} := K^t_{\De\Ga} \langle -\caps(t) \rangle \otimes_{\KImn}?
:
\Rep{\KImn} \rightarrow \Rep{\KImn}.
\end{equation*}
Assuming $t$ is proper, $G^t_{\De\Ga}$ is an indecomposable functor;
see \cite[Theorem 4.14]{BS2}.
A {\em graded projective functor}
on $\Rep{\KImn}$ means an endofunctor that is isomorphic to a
finite direct
sum of $G^t_{\De\Ga}$'s, possibly shifted in degree.
By \cite[Theorem 4.10]{BS2}, each
$G^t_{\De\Ga}$ commutes with the duality $\circledast$, i.e.
there is a canonical
degree zero isomorphism
\begin{equation}\label{Comm}
G^t_{\De\Ga} \circ \circledast
\cong \circledast \circ G^t_{\De\Ga}.
\end{equation}
Let $t^*$ denote the mirror image of $t$ in a horizontal axis.
By \cite[Corollary 4.9]{BS2}, there is a canonical degree zero adjunction
making
\begin{equation}\label{Adjpar}
\left(G^{t^*}_{\Ga\De} \langle \defect(\Ga)-\defect(\De) \rangle, G^t_{\De\Ga}\right)
\end{equation}
into an adjoint pair.
In particular, this implies that each $G^t_{\De\Ga}$ is exact.

More generally, suppose that $\bGa = \Ga_d \cdots \Ga_0$ is any
sequence of blocks in $\PImn$ and $\bt = t_d \cdots t_1$
is a $\bGa$-matching in the sense of \cite[$\S$2]{BS2}.\label{compmatch}
We write $\bGa^*$ for the opposite block sequence
$\Ga_0 \cdots \Ga_d$ and $\bt^*$ for $t_1^* \cdots t_d^*$,
which is a $\bGa^*$-matching.
To this data, there is associated a graded $(\KImn,\KImn)$-bimodule
$K^\bt_\bGa$; see \cite[$\S$3]{BS2}.
By its
definition, it is non-zero if and only if $\bt$ is a proper $\bGa$-matching,
i.e. at least one
oriented $\bGa$-matching
\begin{equation}\label{tlanotation}
\bt[\bga] = \ga_d t_d \ga_{d-1}\cdots \ga_1 t_1 \ga_0
\end{equation}
exists, where
$\bga = \ga_d \cdots \ga_0$ is a sequence of weights with
$\ga_r \in \Ga_r$ for each $r$.
Moreover, $K^\bt_\bGa$ has an
explicit homogeneous basis
\begin{equation}\label{diagbasis}
\left\{(a \:\bt[\bga]\:b)\:|\:\text{for all oriented $\bGa$-circle diagrams
$a \:\bt[\bga]\:b$}\right\},
\end{equation}
in which the degree of $(a\:\bt[\bga]\:b)$ is equal to the total number
of clockwise cups and caps in the diagram.
Let
\begin{equation*}
G^{\bt}_{\bGa} :=
K^\bt_{\bGa} \langle -\caps(\bt) \rangle \otimes_{\KImn}?
:
\Rep{\KImn} \rightarrow \Rep{\KImn}.
\end{equation*}
By \cite[Theorem 3.5(iii)]{BS2}, the associative
multiplication from \cite[(3.12)]{BS2}
defines a canonical graded bimodule isomorphism
\begin{equation}\label{miso}
K^{t_d}_{\Ga_d \Ga_{d-1}}
 \otimes_{\KImn} \cdots \otimes_{\KImn}
K^{t_1}_{\Ga_1 \Ga_0} \stackrel{\sim}{\rightarrow}
K^{\bt}_{\bGa}.
\end{equation}
This induces a canonical isomorphism of functors
\begin{equation}\label{miso2}
G^{t_d}_{\Ga_d\Ga_{d-1}} \circ\cdots\circ G^{t_1}_{\Ga_1\Ga_0}\stackrel{\sim}{\rightarrow}
G^{\bt}_{\bGa}.
\end{equation}
Hence $G^\bt_\bGa$ is exact and commutes with duality, as that is true for
each individual
$G^{t_r}_{\Ga_r \Ga_{r-1}}$.
By \cite[Theorem 3.6]{BS2} we have for proper $\bt$ that
\begin{equation}\label{moso}
K^{\bt}_{\bGa}\langle-\caps(\bt)\rangle
\cong K^s_{\Ga_d \Ga_0}\langle-\caps(s)\rangle \otimes R^{\otimes \circles(\bt)}
\end{equation}
where $s$ denotes the reduction of $\bt$ in the sense of \cite[$\S$2]{BS2},
$\circles(\bt)$ is the number of internal circles in $\bt$,\label{ncirc}
and $R$ denotes
$\C[x] / (x^2)$ graded by declaring that
$1$ in degree $-1$ and $x$ in degree $1$.
This induces an isomorphism of functors
\begin{equation}\label{moso2}
G^{\bt}_{\bGa} \cong (? \otimes R^{\otimes \circles(\bt)}) \circ G^s_{\Ga_d \Ga_0}
\end{equation}
for proper $\bt$.
This explains how to decompose
$G^{\bt}_{\bGa}$ as a direct sum of indecomposable
projective functors.

\phantomsubsection{Special projective functors}
Now we define some important
functors $F_i$ and $E_i$
for each $i \in I$.
Given a block $\Ga\in\PImn$,
we say that $i \in I$ is {\em $\Ga$-admissible} if
$\Ga - \alpha_i$ belongs to $\PImn$.
Viewing blocks diagrammatically like in (\ref{egb}), this means that
the $i$th and $(i+1)$th vertices of $\Ga$
match the top number line
of a unique one of the following diagrams,
and $\defect(\Ga)$ is as indicated:
\begin{equation}
% scriptstyles at (1,x) x = -68,-45,-22,1,24,47,70,93,...
% \times is -0.2
% \circ is +1
% \up is +0.3 and down 2.6
% \down is +0.3 and up 2.1
\begin{picture}(75,33)
\put(159.5,12){$\Ga$}
\put(155,1){$t_i(\Ga)$}
\put(152,-11.4){$\Ga-\alpha_i$}
\put(-106,17){\vector(0,-1){28}}\put(-125,-.5){$F_i$}
\put(-86,28){$_\text{$\defect(\Ga) \geq 1$}$}
\put(-26,28){$_\text{$\defect(\Ga) \geq 0$}$}
\put(34,28){$_\text{$\defect(\Ga) \geq 0$}$}
\put(94,28){$_\text{$\defect(\Ga) \geq 0$}$}
%\put(-81,28){$\text{$_i$\quad\:\,$_{i+1}$}$}
%\put(-21,28){$\text{$_i$\quad\:\,$_{i+1}$}$}
%\put(36,28){$\text{$_i$\quad\:\,$_{i+1}$}$}
%\put(99.5,28){$\text{$_i$\quad\:\,$_{i+1}$}$}
\put(-82,-24){$\text{``cup''}$}
\put(-21,-24){$\text{``cap''}$}
\put(21,-24){$\text{``right-shift''}$}
\put(89,-24){$\text{``left-shift''}$}

\put(-84,15){\line(1,0){33}}
\put(-84,-9){\line(1,0){33}}
\put(-67.5,15){\oval(23,23)[b]}
\put(-59.3,-10.9){$\scriptstyle\times$}
\put(-81.8,-11.6){$\circ$}
\put(-80.9,13.1){{$\scriptstyle\bullet$}}
\put(-57.9,13.1){{$\scriptstyle\bullet$}}

\put(-24,15){\line(1,0){33}}
\put(-24,-9){\line(1,0){33}}
\put(-7.5,-9){\oval(23,23)[t]}
\put(1.2,12.4){$\circ$}
\put(-22.2,13.1){$\scriptstyle\times$}
\put(-21.1,-10.9){{$\scriptstyle\bullet$}}
\put(1.9,-10.9){{$\scriptstyle\bullet$}}

\put(36,15){\line(1,0){33}}
\put(36,-9){\line(1,0){33}}
\put(61.2,12.4){$\circ$}
\put(38.2,-11.6){$\circ$}
\put(64,-8.2){\line(-1,1){22.9}}
\put(61.9,-10.9){{$\scriptstyle\bullet$}}
\put(38.9,13.1){{$\scriptstyle\bullet$}}

\put(96,15){\line(1,0){33}}
\put(96,-9){\line(1,0){33}}
\put(120.7,-10.9){$\scriptstyle\times$}
\put(97.7,13.1){$\scriptstyle\times$}
\put(101.1,-8.2){\line(1,1){22.9}}
\put(98.9,-10.9){{$\scriptstyle\bullet$}}
\put(121.9,13.1){{$\scriptstyle\bullet$}}
\end{picture}
\vspace{8mm}\label{CKLR}
\end{equation}
Also define a $(\Ga-\alpha_i)\Ga$-matching $t_i(\Ga)$
so that the strip between the $i$th and $(i+1)$th vertices
is as in the diagram, and there are only vertical ``identity'' line segments
elsewhere.
The {\em special projective functors} are
the functors
\begin{equation}\label{gFi}
F_i := \bigoplus_{\Ga}
G_{(\Ga-\alpha_i)\Ga}^{t_i(\Ga)},
\qquad\qquad
E_i := \bigoplus_\Ga
G_{\Ga(\Ga-\alpha_i)}^{t_i(\Ga)^*},
\end{equation}
where the direct sums are over all
$\Ga \in \PImn$ such that $i$ is $\Ga$-admissible.
The following lemma makes precise the sense in which
these functors generate all
other projective functors on $\Rep{\KImn}$.

\begin{Lemma}\label{gen}
Suppose we are given blocks
$\Ga, \De \in \PImn$
and a proper $\Ga\De$-matching $t$. Up to a degree shift, the
indecomposable projective functor
$G^t_{\Ga\De}$ is
a summand of a
composition of finitely many special projective functors.
\end{Lemma}

\begin{proof}
The key point is that $t$
can be obtained as the reduction of a composite matching
built from diagrams of the form (\ref{CKLR}) and their duals;
we omit some combinatorial details here.
Given this, the lemma follows from  (\ref{miso2}) and (\ref{moso2}).
\end{proof}

\phantomsubsection{Properties of special projective functors}
We proceed to record some other
basic properties of the functors $F_i$ and $E_i$.

\begin{Lemma}\label{duality}
The functors $F_i$ and $E_i$ commute with the duality $\circledast$,
i.e. there are canonical degree zero isomorphisms
$F_i \circ \circledast \cong \circledast \circ F_i$ and
$E_i \circ \circledast \cong \circledast \circ E_i$
of functors on $\Rep{\KImn}$.
\end{Lemma}

\begin{proof}
This is immediate from (\ref{Comm}).
\end{proof}

For $i \in I^+$,
we let
\begin{equation}\label{Ki}
D_i^{\pm 1}:\Rep{\KImn} \rightarrow \Rep{\KImn}
\end{equation}
be the degree shift functor mapping a module $M
\in\Rep{K_\Ga}$ to $M\langle \pm(\Ga, \delt_i) \rangle$.
The following lemma should be compared with (\ref{taudef})--(\ref{berkeley})
and \cite[Proposition 4.2]{FKS}.

\begin{Lemma}\label{adjunctions}
There are degree zero adjunctions making
$(F_i \circ D_i D_{i+1}^{-1} \langle -1 \rangle, E_i)$
and $(E_i, F_i \circ D_i^{-1} D_{i+1} \langle 1 \rangle)$
into adjoint pairs of functors on $\Rep{\KImn}$.
\end{Lemma}

\begin{proof}
We just derive the adjunction for
$(F_i \circ D_i D_{i+1}^{-1} \langle -1 \rangle, E_i)$,
the other being similar.
Let $\Ga \in \PImn$ such that $i$ is $\Ga$-admissible.
By looking at the diagrams in (\ref{CKLR}), we have that
\begin{equation}\label{defede}
\defect(\Ga-\alpha_i)-\defect(\Ga)=(\Ga,\delt_i-\delt_{i+1}) - 1.
\end{equation}
Hence applying (\ref{Adjpar}), we get a canonical degree
zero adjunction making
$$
\left(G^{t_i(\Ga)}_{(\Ga-\alpha_i) \Ga}\langle (\Ga,\delt_i-\delt_{i+1})-1\rangle,
G^{t_i(\Ga)^*}_{\Ga (\Ga-\alpha_i)}\right)
$$
into an adjoint pair.
Now use the definitions
(\ref{gFi}) and (\ref{Ki}).
\end{proof}

\begin{Lemma}\label{ga}
Let $\la \in \LaImn$ and $i \in I$.
For symbols $x,y \in \{\circ,\up,\down,\times\}$
we write $\la_{xy}$ for the diagram obtained from $\la$
by relabelling the $i$th and $(i+1)$th vertices
by $x$ and $y$, respectively.
\begin{itemize}
\item[\rm(i)]
If $\la = \la_{{\scriptscriptstyle\down}\circ}$ then
$F_i P(\la) \cong P(\la_{\circ{\scriptscriptstyle\down}})$,
$F_i V(\la) \cong V(\la_{\circ{\scriptscriptstyle\down}})$,
$F_i L(\la) \cong L(\la_{\circ{\scriptscriptstyle\down}})$.
\item[\rm(ii)]
If $\la = \la_{{\scriptscriptstyle\up}\circ}$ then
$F_i P(\la) \cong P(\la_{\circ{\scriptscriptstyle\up}})$,
$F_i V(\la) \cong V(\la_{\circ{\scriptscriptstyle\up}})$,
$F_i L(\la) \cong L(\la_{\circ{\scriptscriptstyle\up}})$.
\item[\rm(iii)]
If $\la = \la_{\scriptscriptstyle\times\down}$ then
$F_i P(\la) \cong P(\la_{\scriptscriptstyle\down\times})$,
$\!F_i V(\la) \cong V(\la_{\scriptscriptstyle\down\times})$,
$\!F_i L(\la) \cong L(\la_{\scriptscriptstyle\down\times})$.
\item[\rm(iv)]
If $\la = \la_{\scriptscriptstyle\times\up}$ then
$F_i P(\la) \cong P(\la_{\scriptscriptstyle\up\times})$,
$\!F_i V(\la) \cong V(\la_{\scriptscriptstyle\up\times})$,
$\!F_i L(\la) \cong L(\la_{\scriptscriptstyle\up\times})$.
\item[\rm(v)]
If $\la = \la_{{\scriptscriptstyle\times}\circ}$ then:
\begin{itemize}
\item[(a)]
$F_i P(\la) \cong P(\la_{\scriptscriptstyle\down\up})\langle -1\rangle$;
\item[(b)]
there is a short exact sequence
$$
0 \rightarrow V(\la_{\scriptscriptstyle\up\down}) \rightarrow F_i V(\la)
\rightarrow V(\la_{\scriptscriptstyle\down\up}) \langle -1 \rangle \rightarrow 0;
$$
\item[(c)]
$F_i L(\la)$ has irreducible
socle $L(\la_{\scriptscriptstyle\down\up})\langle 1 \rangle$ and
head $L(\la_{\scriptscriptstyle\down\up})\langle -1 \rangle$, and all other composition
factors are of the form $L(\mu)$ for $\mu \in\LaImn$
such that
$\mu = \mu_{\scriptscriptstyle\down\down}$,
$\mu = \mu_{\scriptscriptstyle\up\up}$ or
$\mu = \mu_{\scriptscriptstyle\up\down}$.
\end{itemize}
\item[\rm(vi)]
If $\la = \la_{{\scriptscriptstyle\down\up}}$ then
$F_i P(\la) \cong
P(\la_{\circ\scriptscriptstyle{\times}}) \oplus
P(\la_{\circ\scriptscriptstyle{\times}}) \langle 2\rangle$,
$F_i V(\la) \cong V(\la_{\circ{\scriptscriptstyle\times}})$
and
$F_i L(\la) \cong L(\la_{\circ{\scriptscriptstyle\times}})$.
\item[\rm(vii)]
If $\la = \la_{{\scriptscriptstyle\up\down}}$ then
$F_i V(\la) \cong V(\la_{\circ{\scriptscriptstyle\times}})\langle 1 \rangle$
and $F_i L(\la) = \{0\}$.
\item[\rm(viii)]
If
$\la = \la_{{\scriptscriptstyle\down\down}}$
then
$F_i V(\la) = F_i L(\la) = \{0\}$.
\item[\rm(ix)]
If
$\la = \la_{{\scriptscriptstyle\up\up}}$
then
$F_i V(\la) = F_i L(\la) = \{0\}$.
\item[\rm(x)]
For all other $\la$ we have that
$F_i P(\la) = F_i V(\la) = F_i L(\la) = \{0\}$.
\end{itemize}
For the dual statement about $E_i$,
interchange all occurrences of $\circ$ and $\times$.
\end{Lemma}

\begin{proof}
Apply \cite[Theorem 4.2]{BS2} for $P(\la)$,
\cite[Theorem 4.5]{BS2} for $V(\la)$, and
\cite[Theorem 4.11]{BS2} for $L(\la)$.
\end{proof}

\phantomsubsection{\bf The first categorification theorem}
The following theorem explains the connection between the
Grothendieck group of
$\Rep{\KImn}$ and the $\Laurent$-form
of the module $\bigwedge^m V
\otimes \bigwedge^n V$ from $\S$\ref{sB}.

\begin{Theorem}\label{cat1}
Identify the Grothendieck group $[\Rep{\KImn}]$
with the $U_{\!\Laurent}$-module $\bigwedge^m V_\Laurent
\otimes \bigwedge^n V_\Laurent$
by identifying
$[V(\la)]$ with $V_\la$ for each $\la \in \LaImn$.
\begin{enumerate}
\item[\rm(i)] We have that $[L(\la)] = L_\la$
and $[P(\la)] = P_\la$ for each $\la \in \LaImn$.
\item[\rm(ii)] The
endomorphisms of the Grothendieck group induced by the
exact functors
$E_i, F_i$ and $D_i^{\pm 1}$
coincide with the action of the
generators $E_i, F_i$ and $D_i^{\pm 1}$
 of $U_{\!\Laurent}$ for each $i \in I$.
\item[\rm(iii)]
We have that
$$
\sum_{j \in \Z} q^j \dim \hom_{\KImn}(P, M)_j=
\left\langle [P], [M] \right\rangle
$$
for $M, P \in \Rep{\KImn}$
with $P$ projective.
\item[\rm(iv)]
We have that $[M^\circledast] = \overline{[M]}$ for each $M \in \Rep{\KImn}$.
\end{enumerate}
\end{Theorem}

\begin{proof}
We first check (ii), explaining the argument
just in the case of $F_i$; a similar argument establishes
the statement for $E_i$ and the statement for $D_i^{\pm 1}$ is
obvious.
By the definition of the action of $F_i$
on $\bigwedge^m V \otimes \bigwedge^n V$ it maps
\begin{align*}
(- v_{i} -)\otimes (-)
&\mapsto
(- v_{i+1} -) \otimes (-),\\
 (-) \otimes (- v_{i} -)
&\mapsto
(-) \otimes (- v_{i+1} -),\\
(- v_{i+1} \wedge v_i -) \otimes (- v_{i}-)
&\mapsto (- v_{i+1} \wedge v_i -) \otimes (- v_{i+1} -),\\
(- v_{i} -) \otimes(- v_{i+1} \wedge v_i -)
&\mapsto (- v_{i+1} -) \otimes (- v_{i+1}\wedge v_i -),\\
(- v_{i}-) \otimes (- v_{i} -)
&\mapsto
(- v_{i+1}-) \otimes (- v_i -)+ q^{-1}
(- v_i -) \otimes (- v_{i+1} -),\\
(- v_{i}-) \otimes (- v_{i+1} -)
&\mapsto
(- v_{i+1} -) \otimes (- v_{i+1} -),\\
(- v_{i+1}-) \otimes (- v_{i}-)
&\mapsto
q(- v_{i+1}-) \otimes (- v_{i+1} -),
\end{align*}
where $-$ denotes a wedge product of basis vectors
$v_j$ for $j \neq i, i+1$.
Moreover $F_i$ acts as zero on all other $V_\la$'s.
Comparing with Lemma~\ref{ga}, this is the
same as the action of the functor $F_i$ on the
basis for the Grothendieck group coming from the standard modules.

Next we consider (iv).
If $\la$
is minimal in the Bruhat order
then $V(\la) = L(\la)$ by \cite[Theorem 5.2]{BS1},
so $V(\la)^\circledast \cong V(\la)$.
Thus, $\circledast$ induces an anti-linear endomorphism
of the Grothendieck group
that fixes $V_\la$
for each minimal $\la$.
Moreover by (ii) and Lemma~\ref{duality} this
induced endomorphism commutes
with the actions of $E_i, F_i$ for all $i \in I$.
It follows easily that the induced endomorphism
is a compatible bar-involution.
Hence it coincides with the bar-involution
from Lemma~\ref{charac} by the uniqueness from that lemma.

Using (ii) and (iv), we can now establish (i).
As $L(\la)^\circledast \cong L(\la)$,
we get from (iv) that the vector $[L(\la)]$ is bar invariant
for each $\la \in \LaImn$.
Also the inverse of the $q$-decomposition matrix from \cite[(5.14)]{BS1}
has $1$'s on the diagonal and all other entries belong to
$q \Z[q]$. So:
$$
[L(\la)] = [V(\la)] + (\text{a $q \Z[q]$-linear
combination of $[V(\mu)]$'s}).
$$
This verifies that $[L(\la)]$ satisfies the defining properties
of the dual-canonical basis vector $L_\la$.
Hence $[L(\la)] = L_\la$.
Then the fact that $[P(\la)] = P_\la$
follows on comparing (\ref{boo})
and \cite[(5.15)]{BS1}.

Finally, (iii) is clear from (\ref{e4}), (i) and the sesquilinearity
of the form $\langle.,.\rangle$, since
$\dim \hom_{\KImn}(P(\la), L(\mu))_j = \delta_{\la,\mu} \delta_{j,0}$
for each $\la,\mu \in \LaImn$.
\end{proof}

\begin{Remark}\label{Exp}\rm
In view of Theorem~\ref{cat1}, the
polynomials $d_{\la,\mu}(q)$ and
$p_{\la,\mu}(q)$ from (\ref{e1})--(\ref{e2}) are the same
as the polynomials defined by explicit closed formulae in
\cite[(5.12)]{BS1} and \cite[(5.3)]{BS2}, the latter
going back to \cite{LS};
see also \cite{Brenti} and \cite[Theorem 3]{LM}.
Hence the quasi-canonical and dual-canonical bases
for
 $\bigwedge^m V \otimes \bigwedge^n V$
are known exactly.
Moreover \cite[Theorems 4.2 and 4.11]{BS2} give explicit formulae
for the action of $F_i$ and $E_i$ on these bases; in almost all cases
this is described already by Lemma~\ref{ga}.
This gives an explicit diagram calculus
for working with the various bases
of  $\bigwedge^m V \otimes \bigwedge^n V$, which is closely related to
the diagram calculus
developed by Frenkel and Khovanov in \cite{FK}.
\end{Remark}

\phantomsubsection{The crystal graph}
Given $i \in I$ and $\la,\mu \in \LaImn$, we write
$\la = \tilde f_i (\mu)$ if
the $i$th and $(i+1)$th vertices of $\la$ and $\mu$
are labelled
according to one of the
six cases in the following table,
and all other vertices of $\la$ and $\mu$ are labelled in the
same way:
\begin{equation}\label{cg}
\begin{array}{c|c|c|c|c|c|c}
\mu&\quad{\scriptstyle\down}\:\circ\quad&\quad{\scriptstyle\up}\:\circ\quad&\quad{\scriptstyle\times}\:{\scriptstyle\down}\quad&\quad{\scriptstyle\times}\:{\scriptstyle\up}\quad&\quad{\scriptstyle\times}\:\circ\quad&\quad{\scriptstyle\down}\:{\scriptstyle\up}\quad\\
\hline
\la&\circ\:{\scriptstyle\down}&\circ\:{\scriptstyle\up}&{\scriptstyle\down}\:{\scriptstyle\times}&{\scriptstyle\up}\:{\scriptstyle\times}&{\scriptstyle\down}\:{\scriptstyle\up}&\circ\:{\scriptstyle\times}
\end{array}
\end{equation}
Define the {\em crystal graph}
to be the directed coloured graph with vertex set equal to $\LaImn$
and a directed edge $\mu \stackrel{i}{\rightarrow} \la$ of colour $i \in I$
whenever $\la = \tilde f_i(\mu)$.
This graph is isomorphic to the crystal graph that is the tensor product
of the crystal graphs associated to
the irreducible $U$-modules $\bigwedge^m V$ and
$\bigwedge^n V$ in the sense of Kashiwara.
The representation theoretic significance of the crystal graph is clear
from Lemma~\ref{ga}:
$\la = \tilde f_i(\mu)$ if and only if
$L(\la)$ is a quotient of
$F_i L(\mu)$ (possibly shifted in degree).

\begin{Lemma}\label{cgl}
Let $\la \in \LaImn$.
Then there exists $\mu \in \LaImn$ that is maximal in the Bruhat
ordering and a sequence $i_1,\dots,i_k \in I$ for some $k \geq 0$
such that $\la = \tilde f_{i_k} \cdots \tilde f_{i_1}(\mu)$.
\end{Lemma}

\begin{proof}
If $\la$ is maximal in the Bruhat ordering, there is nothing to do.
So assume it is not maximal.
Then we can find $i < j$ such that the $i$th vertex
of $\la$ is labelled $\down$, the $j$th vertex of $\la$ is
labelled $\up$,
and all vertices of $\la$ in between are labelled $\circ$ or $\times$.
By using crystal graph edges
of the form ${\scriptstyle\times}\circ \rightarrow
{\scriptstyle\down\up}\rightarrow \circ{\scriptstyle\times}$,
we reduce to the situation that the vertices $i+1,\dots,j-1$
are labelled so all the $\circ$'s are to the right of the $\times$'s.
Then by using edges of the form
${\scriptstyle\up}\circ \rightarrow \circ{\scriptstyle \up}$
and ${\scriptstyle\times\down} \rightarrow
{\scriptstyle \down\times}$ we reduce to the situation that
the vertices $\down$ and $\up$ are neighbours.
Finally use an edge of the form ${\scriptstyle\times}\circ
\rightarrow {\scriptstyle \down\up}$ to eliminate this $\down\up$ pair
altogether. Then iterate.
\end{proof}

\begin{Theorem}\label{cgt}
Take $\la \in \LaImn$ and let
$\mu, i_1,\dots,i_k$ be as in Lemma~\ref{cgl}.
Then
$$
F_{i_k} \cdots F_{i_1} V_\mu= (q+q^{-1})^rq^{r-s} P_\la,
$$
where $r$ (resp.\ $s$) is the number of
crystal graph edges in the
path $\mu \stackrel{i_1}{\rightarrow} \cdots \stackrel{i_d}{\rightarrow}
\la$ that are
of the form ${\scriptstyle \down\up}
\rightarrow \circ{\scriptstyle\times}$ (resp.\ ${\scriptstyle \times}\circ
\rightarrow {\scriptstyle \down\up}$).
\end{Theorem}

\begin{proof}
By applying Lemma~\ref{ga}(vi) a total of $r$ times and
Lemma~\ref{ga}(v) a total of $s$ times, we get that
$[F_{i_k} \circ \cdots \circ F_{i_1} (P(\mu))]
=(q+q^{-1})^r q^{r-s} [P(\la)]$
in the Grothendieck group $[\Rep{\KImn}]$.
As $\mu$ is maximal in the Bruhat ordering
we have by \cite[Theorem 5.1]{BS1} that $P(\mu) \cong V(\mu)$.
Hence
\begin{align*}
F_{i_k} \cdots F_{i_1} V_\mu &=
[F_{i_k} \circ \cdots \circ F_{i_1} (V(\mu))]
=
[F_{i_k} \circ \cdots \circ F_{i_1} (P(\mu))]\\
&=(q+q^{-1})^r q^{r-s} [P(\la)] = (q+q^{-1})^r q^{r-s} P_\la,
\end{align*}
using Theorem~\ref{cat1}(i)--(ii).
\end{proof}

As in \cite[(6.7)]{BS1}, given a block $\Ga \in \PImn$,
we let $\Ga^\circ$ denote the
set of all weights $\ga \in \Ga$ that are of
maximal defect, i.e. the associated cup diagram $\underline{\ga}$
from \cite[$\S$2]{BS1}
has $\defect(\Ga)$ cups.
By a {\em prinjective module} we mean a module that is both projective
and injective.

\begin{Lemma}\label{princ1}
Up to shifts in degree,
the modules $\{P(\la)\:|\:\la \in \Ga^\circ\}$
give a complete set of representatives for the isomorphism classes of
prinjective indecomposable modules
in $\Rep{K_\Ga}$.
Moreover for any $\la \in \Ga^\circ$ the
module $P(\la)\langle -\defect(\Ga)\rangle$ is self-dual.
\end{Lemma}

\begin{proof}
This follows from \cite[Theorem 6.1]{BS2}.
\end{proof}

Our final lemma gives an alternative description of the set
\begin{equation}\label{circirc}
\LaImn^\circ := \bigcup_{\Ga \in \PImn} \Ga^\circ.
\end{equation}
in terms of the crystal graph: it is the connected component of the
crystal graph generated by the ground-state $\iota$ from (\ref{groundstate}).

\begin{Lemma}\label{stupidcomb}
For $\la \in \LaImn$, we have that
$\la \in \LaImn^\circ$ if and only if
there exists a sequence $i_1,\dots,i_d \in I$ such that
$\la = \tilde f_{i_d} \cdots \tilde f_{i_1} (\iota)$.
\end{Lemma}

\begin{proof}
Suppose first that there exists $\mu \in \LaImn$
such that $\la = \tilde f_i(\mu)$ for some
$i \in I$.
By inspecting (\ref{cg}),
we have that $\la$ is of maximal defect in its block
if and only if $\mu$ is of maximal defect in its block.
Hence we are reduced to the case that
$\la$ is extremal in the crystal graph in the sense that it
cannot be written as $\tilde f_i(\mu)$
for any $\mu\in \LaImn$ or $i \in I$.
Then by (\ref{cg}) again
the weight $\la$ consists of $\times$'s then
$\up$'s then $\down$'s then $\circ$'s. For such a weight it is
clear that $\la$ is of maximal defect in its block if and only if
$\la = \iota$.
\end{proof}

\section{Categorification via parabolic category $\cO$}\label{sC2}

In this section, we give a self-contained account of
another known categorification theorem
this time categorifying
$\bigwedge^m V_\Z \otimes \bigwedge^n V_\Z$ (the same
space as in the previous section but specialised at $q=1$)
using
a certain sum of blocks of the
parabolic category $\cO$ corresponding to the Grassmannian
$\operatorname{Gr}(m,m+n)$. (To keep $q$ generic one could work with a graded version of category $\cO$ as in \cite{Sussan, MSslk}, but 
it is enough for the purposes of this paper to stick to the specialised version
on the category $\cO$ side,
making the grading explicit only on the diagrammatical side.)
The arguments in this section provide an elementary proof of
the Kazhdan-Lusztig conjecture in this very special case;
the possibility of doing this goes back to work of Enright and Shelton
\cite{ES} though they used a different strategy.

\phantomsubsection{\boldmath The category $\OImn$}
Let $\g := \mathfrak{gl}_{m+n}(\C)$\label{lieg}
with its standard Cartan subalgebra $\h$ of
diagonal matrices and its standard Borel subalgebra
$\b$ of upper triangular matrices.
We define the standard coordinates
$\eps_1,\dots,\eps_{m+n}$ for $\mathfrak{h}^*$, the weight $\rho$,
the subalgebras $\mathfrak{l}$ and $\mathfrak{p}$,
and the category $\cO(m,n)$ as in the introduction.
We refer the reader to \cite[Chapter 9]{Hbook} for
a detailed treatment of the basic properties of
parabolic category $\mathcal O$ (for any semisimple Lie algebra).

The category $\cO(m,n)$ is a highest weight category in the sense of \cite{CPS}\label{cato}
with irreducible modules $\{\cL(\la)\:|\:\la\in \La(m,n)\}$,
standard modules $\{\cV(\la)\:|\:\la \in \La(m,n)\}$
and projective indecomposable
modules $\{\cP(\la)\:|\:\la \in \La(m,n)\}$,
where $\La(m,n)$ is the subset of $\mathfrak{h}^*$
defined by (\ref{isd}).
The standard module $\cV(\la)$ can be constructed explicitly
as a
{\em parabolic Verma module}:
\begin{equation}\label{feet}
\cV(\la) = U(\g) \otimes_{U(\mathfrak{p})} \cS(\la)
\end{equation}
where $\cS(\la)$ is the finite dimensional irreducible $\mathfrak{l}$-module of highest weight $\la$, viewed as a $\mathfrak{p}$-module via the
natural projection $\mathfrak{p}\twoheadrightarrow \mathfrak{l}$.
The irreducible module $\cL(\la)$ is the unique irreducible quotient
of $\cV(\la)$, and the projective indecomposable module $\cP(\la)$
is its projective cover in $\cO(m,n)$.

Let $\circledast$ denote the standard duality on
$\cO(m,n)$, namely, $M^\circledast$ is the direct sum
of the duals of all the weight spaces of $M$, with
$x \in \mathfrak{g}$ acting on $f \in M^{\circledast}$
by $(xf)(v) := f(x^T v)$ (matrix transposition).
This duality fixes irreducible modules, i.e.
$\cL(\la)^\circledast \cong \cL(\la)$ for each $\la \in \La(m,n)$.

Two irreducible modules $\cL(\la)$ and $\cL(\mu)$
have the same central character if and only if
$\la+\rho$ and $\mu+\rho$ lie in the same orbit under the
natural action of the symmetric group $S_{m+n}$
permuting the $\eps_i$'s.
We denote this equivalence relation on $\La(m,n)$
by $\sim$, and let $P(m,n)$ denote the set
$\La(m,n) / \!\!\sim$ of equivalence classes.
The category $\cO(m,n)$ decomposes
according to generalised central characters as
\begin{equation}\label{allblocks}
\cO(m,n)= \bigoplus_{\Ga \in P(m,n)}
\cO_\Ga
\end{equation}
where $\cO_\Ga$ denotes the Serre subcategory of
$\cO(m,n)$ generated by the irreducible
objects $\cL(\la)$ with $\la \in \Ga$.
For $\Ga \in P(m,n)$, we
let \begin{equation*}
\pr_\Ga:\cO(m,n) \rightarrow \cO(m,n)
\end{equation*}
denote the projection onto the summand $\cO_\Ga$ along
(\ref{allblocks}).

\begin{Remark}\rm
In fact it is known by
a special case of \cite[Theorem 2]{cyclo} (see also (\cite[2.4.4 Corollary B and 2.9 Proposition B]{BoeNakano})
that each $\cO_\Ga$ is a single
block of $\cO(m,n)$, i.e. it is an indecomposable
category, though we will not need to use this.
\end{Remark}

Now recall the set $\LaImn$ of weights
defined in diagrammatic terms in $\S$\ref{sB}.
Using the weight dictionary from (\ref{dict}), we can identify
$\LaImn$ with the following subset of $\La(m,n)$:
\begin{equation}
\LaImn = \left\{\la \in \h^*\:\:\Bigg|\:
\begin{array}{l}
(\la+\rho,\eps_i) \in I^+\text{ for all }1 \leq i\leq m+n,\\
(\la+\rho,\eps_1) > \cdots > (\la+\rho,\eps_m),\\
(\la+\rho,\eps_{m+1}) > \cdots > (\la+\rho,\eps_{m+n})
\end{array}\right\}.
\end{equation}
For example, taking $I = \{1,\dots,8\}, m=5$ and $n=4$, the weight
\begin{equation*}
\la = 9\eps_1+ 7\eps_2+ 7\eps_3+ 7\eps_4+ 5\eps_5+ 12\eps_6+ 12\eps_7+
10\eps_8+9\eps_9\in \mathfrak{h}^*
\end{equation*}
is an element of $\LaImn$. The corresponding
sets $I_\down(\la)$ and $I_\up(\la)$ from (\ref{Id})--(\ref{Iu})
are $\{9,6,5,4,1\}$ and $\{7,6,3,1\}$, respectively. Hence
via the weight dictionary $\la$ is identified with the
weight displayed in (\ref{eg}).

Recall also that $\PImn$ denotes the $\sim$-equivalence classes
in $\LaImn$.
Under the identification just made,
$\PImn$ becomes a subset of the set $P(m,n)$ appearing
in (\ref{allblocks}).
So it makes sense to
consider the following sum of blocks in $\cO(m,n)$:
\begin{equation}\label{blocks}
\OImn := \bigoplus_{\Ga \in \PImn} \cO_\Ga.
\end{equation}
Equivalently, this is the
category of all $\mathfrak{g}$-modules that are
semisimple over $\mathfrak{h}$ and possess a composition series
with composition factors of the form
$\cL(\la)$ for $\la \in \LaImn$.
The irreducible, standard and projective indecomposable modules
in $\OImn$ are the
modules $\cL(\la), \cV(\la)$ and $\cP(\la)$
for $\la \in \LaImn$.
Their isomorphism classes
$\{[\cL(\la)]\}$, $\{[\cV(\la)]\}$ and $\{[\cP(\la)]\}$
give three natural bases for the Grothendieck
group $[\OImn]$.

The following lemma originates in work of Irving \cite{I}.

\begin{Lemma}\label{princ2}
Recalling (\ref{circirc}),
the modules $\{\cP(\la)\:|\:\la \in \LaImn^\circ\}$
give a complete set of representatives for the isomorphism classes of
prinjective indecomposable
modules in $\OImn$.
\end{Lemma}

\begin{proof}
If $m \geq n$ then this follows by a special case of
\cite[Theorem 4.8]{BKschur}.
A similar
argument establishes the result if $m < n$ too.
\end{proof}

\phantomsubsection{Special projective functors}
Now we introduce the {special projective functors}
on $\OImn$ following \cite[$\S$4.4]{BKrep}
and \cite[$\S$7.4]{CR}. It is convenient
to work first
on all of $\cO(m,n)$, defining functors
$\cF_i$ and $\cE_i$ for all $i \in \Z$, before restricting attention
to $\OImn$.

Given $\Ga \in P(m,n)$ and $i \in \Z$, we say that $i$ is
{\em $\Ga$-admissible} if there exists $\la \in \Ga$
and $1 \leq j \leq m+n$ such that $\la+\eps_j\in\La(m,n)$
and $(\la+\rho,\eps_j) = i$.
In that case, we let $\Ga - \alpha_i \in P(m,n)$ denote the
$\sim$-equivalence class generated by the weight
$\la+\eps_j$, for any $\la$ and $j$ as in the previous sentence.
If $\Ga \in \PImn$ and $i \in I$ then these notions
agree with the ones introduced in diagrammatic terms in the preceeding
sections.

Let $\cV$ be the natural $\mathfrak{g}$-module of column vectors
and $\cV^*$ be its dual in the usual sense of Lie algebras.
The {\em special projective functors} on $\cO(m,n)$
are the endofunctors $\cF_i$ and $\cE_i$
defined for each $i \in \Z$ by
\begin{align}\label{donely}
\cF_i &:= \bigoplus_{\Ga}
\pr_{\Ga - \alpha_i} \circ (? \otimes \cV) \circ \pr_\Ga,\qquad
\cE_i := \bigoplus_{\Ga}
\pr_{\Ga} \circ (? \otimes \cV^*) \circ \pr_{\Ga-\alpha_i},
\end{align}
where the direct sums are over all $\Ga \in P(m,n)$ such that
$i$ is $\Ga$-admissible.
Because the functors $? \otimes \cV$ and $? \otimes \cV^*$
commute with the duality $\circledast$, so do the functors $\cF_i$ and $\cE_i$.
Moreover $\cF_i$ and $\cE_i$ are biadjoint, hence they are both exact and send
projectives to projectives.

\begin{Lemma}\label{tid}
For $\la \in \La(m,n)$,
$\cV(\la) \otimes \cV$ has a filtration
with sections isomorphic to
$\cV(\la+\eps_j)$ for all
$j=1,\dots,m+n$ such that
$\la+\eps_j \in \La(m,n)$, arranged in order from bottom
to top.
Dually,
$\cV(\la) \otimes \cV^*$ has a filtration
with sections isomorphic to
$\cV(\la-\eps_j)$ for all
$j=1,\dots,m+n$ such that
$\la-\eps_j \in \La(m,n)$, arranged in order from top
to bottom.
\end{Lemma}

\begin{proof}
This is a standard consequence
of the definition (\ref{feet}) and the tensor identity;
see e.g. \cite[Theorem 3.6]{Hbook}.
\end{proof}

\begin{Corollary}\label{tid2}
For $\la \in \La(m,n)$ and $i\in\Z$,
$\cF_i \cV(\la)$ has a filtration
with sections isomorphic to
$\cV(\la+\eps_j)$ for all
$j=1,\dots,m+n$ such that
$\la+\eps_j \in \La(m,n)$
and $(\la+\rho,\eps_j) = i$,
arranged in order from bottom
to top.
Dually,
$\cV(\la) \otimes \cV^*$ has a filtration
with sections isomorphic to
$\cV(\la-\eps_j)$ for all
$j=1,\dots,m+n$ such that
$\la-\eps_j \in \La(m,n)$ and $(\la+\rho,\eps_j)= i+1$,
arranged in order from top
to bottom.
\end{Corollary}

\begin{Corollary}\label{all}
The functors $? \otimes \cV$ and $? \otimes \cV^*$
on $\mathcal O(m,n)$ decompose as
$$
? \otimes \cV = \bigoplus_{i \in \Z} \cF_i,
\qquad
? \otimes \cV^* = \bigoplus_{i \in \Z} \cE_i.
$$
\end{Corollary}

Much later on we will also need the following lemma
first observed in \cite[$\S$7.4]{CR}
which gives
an alternative description of the functors $\cF_i$
and $\cE_i$.
Let
\begin{equation}\label{traceform}
\Omega := \sum_{j,k=1}^{m+n} e_{j,k} \otimes e_{k,j} \in \mathfrak{g}
\otimes \mathfrak{g}.
\end{equation}
This
corresponds to the (invariant) trace form on $\mathfrak{g}$.

\begin{Lemma}\label{crd}
For any $M \in \mathcal O(m,n)$,
$\cF_i M$ (resp.\ $\cE_i M$)
is the generalised $i$-eigenspace
(resp. the generalised $-(m+n+i)$-eigenspace)
of the operator $\Omega$
acting on $M \otimes \cV$ (resp.\ $M \otimes \cV^*$).
\end{Lemma}

\begin{proof}
We just prove the statement about $\cF_i$, a similar argument
treating $\cE_i$.
By classical theory, the center of
$U(\mathfrak{g})$ is a free polynomial algebra with
generators $z_1,\dots,z_{m+n}$, where $z_r$ is the central element
determined uniquely by the property
that it acts on all highest weight modules of highest weight
$\la \in \mathfrak{h}^*$ by multiplication by the scalar
$$
e_r(\la)
:= \sum_{1 \leq i_1 < \cdots < i_r \leq m+n}
(\la+\rho,\eps_{i_1})(\la+\rho,\eps_{i_2}) \cdots (\la+\rho,\eps_{i_r}).
$$
Now fix $\la \in \La(m,n)$. If we can check that the statement of the
lemma holds in the special case that $M = \cV(\la)$, then it follows
at once that it is true on every irreducible module in $\cO(m,n)$,
hence it is true on any module.
By \cite[Lemma 5.1]{cyclo}, $\Omega$ acts on $\cV(\la) \otimes \cV$
in the same way as the central element
\begin{equation*}
e_1(\la)+e_2(\la) - z_2.
\end{equation*}
Hence, fixing a filtration of $\cV(\la) \otimes \cV$
as in Lemma~\ref{tid}, $\Omega$ respects the filtration
and the induced action on the section isomorphic
to $\cV(\la+\eps_j)$ is by multiplication by the scalar
\begin{equation*}
e_1(\la)+e_2(\la) - e_2(\la+\eps_j) = (\la+\rho,\eps_j).
\end{equation*}
Comparing with Corollary~\ref{tid2},
we deduce that $\cF_i \cV(\la)$ is the generalised $i$-eigenspace of
$\Omega$, as required.
\end{proof}

Following \cite{BG}, a
{\em projective functor} on $\cO(m,n)$ means
any endofunctor that is isomorphic to a summand
of a functor arising from tensoring with a finite dimensional
rational $\mathfrak{g}$-module.

\begin{Lemma}\label{gen2}
Any indecomposable projective functor on $\cO(m,n)$
is a summand of a composition of finitely many special
projective functors.
\end{Lemma}

\begin{proof}
Observe that any irreducible rational $\mathfrak{g}$-module
is a summand of a tensor product of finitely many copies of
$\cV$ and $\cV^*$.
Given this, the lemma follows from Corollary~\ref{all}.
\end{proof}

For $i \in I$, the functors $\cF_i$ and $\cE_i$
restrict to well-defined endofunctors of the subcategory $\OImn$;
the resulting
restrictions are given explicitly by the same formulae as (\ref{donely}),
but summing now only over $\Ga \in \PImn$ such that $i$ is $\Ga$-admissible.
We call these the {\em special projective functors} on $\OImn$.

\phantomsubsection{Properties of special projective functors}
The goal in the remainder of the section
is to prove an analogue of Theorem~\ref{cat1} in the
present setting. Our approach is entirely elementary, based just
on Corollary~\ref{tid2} and the following technical result; we include a
simple computational proof
in order to make the exposition self-contained.

\begin{Lemma}\label{jsf}
Let $i \in I$ and $\la
\in \LaImn$ be a weight such that the
$i$th and $(i+1)$th vertices of $\la$ are labelled
$\up$ and $\down$, respectively.
Let $\mu$ be the weight obtained from $\la$
by interchanging the labels on these two vertices.
Then $\cL(\mu)$
is
a composition factor of
$\cV(\la)$.
\end{Lemma}

\begin{proof}
Let $a_j := (\la+\rho,\eps_j)$ and $v_+$ be a non-zero highest weight vector
in $\cV(\la)$.
For $1 \leq j < k \leq m+n$, let
$$
s_{k,j} := \operatorname{cdet}
\left(
\begin{array}{cccccc}
e_{j+1,j}&a_{j+1}-a_j&0&\hdots&0\\
e_{j+2,j}&e_{j + 2, j + 1}&a_{j+2}-a_j&\hdots&0\\
\vdots&\vdots&\ddots&\ddots&\vdots\\
e_{k-2,j}&e_{k-2,j+1}&\hdots&a_{k-2}-a_j&0\\
e_{k-1,j}&e_{k-1,j+1}&\hdots&e_{k-1,k-2}&a_{k-1}-a_j\\
e_{k,j}&e_{k,j+1}&\hdots&e_{k,k-2}&e_{k,k-1}
\end{array}
\right)
\in U(\g)
$$
where $\operatorname{cdet}$ means the usual Laplace expansion of determinant,
ordering monomials in column order.
These lowering operators were introduced originally (in a slightly
different form) in \cite{NM}.
The following key properties are easily checked by direct calculation
from the above matrix:
\begin{itemize}
\item $e_{r,r+1} s_{k,j} v_+ = 0$ for $1 \leq r < m+n$ with $r \neq k-1$;
\item $e_{k-1,k} s_{k,j} v_+ = (a_j-a_k-1) s_{k-1,j} v_+$ (interpreting
$s_{j,j}$ as $1$).
\end{itemize}
Now, the assumptions on $\la$ mean that there are integers
$1 \leq j \leq m$ and $m+1 \leq k \leq m+n$ such that
$a_j = i+1, a_k = i$ and none
of the numbers $a_{j+1},\dots,a_{k-1}$
are equal to $i$ or $i+1$.
The assumptions on $\mu$ mean that $\mu$ is the weight obtained
from $\la$ by subtracting the positive root $\eps_j - \eps_k$.
We claim that the vector
$s_{k,j} v_+$
is a non-zero highest weight vector in $\cV(\la)$. Since it has weight
$\mu$ this claim proves the lemma.

For the claim, the two properties from the previous paragraph
and the fact that $a_j-a_k-1=0$
give at once that $s_{k,j} v_+$ is a highest weight vector.
The problem is to show that it is non-zero. For this, we can expand
$$
s_{k,j} v_+ = \sum_{\substack{j \leq q \leq m \\ m+1 \leq p \leq k}}
e_{p,q} \otimes s_{k,j}^{p,q} v_+ \in U(\g)
\otimes_{U(\mathfrak{p})} \cS(\la)
$$
for unique vectors $s_{k,j}^{p,q} v_+ \in \cS(\la)$.
In particular:
$$
s_{k,j}^{k,m} v_+ = (a_{m+1}-a_j)(a_{m+2}-a_j)\cdots (a_{k-1}-a_j)s_{m,j} v_+.
$$
To complete the proof we show that
$s_{k,j}^{k,m} v_+ \neq 0$.
Since none of $a_{m+1},\dots,a_{k-1}$ equal $a_j = i+1$, we just need to show that
$s_{m,j} v_+ \neq 0$. For this we apply the operator
$e_{j,j+1} e_{j+1,j+2} \cdots e_{m-1,m}$ using the second property from the
previous paragraph to get $(a_j-a_{j+1}-1)\cdots(a_j-a_m-1) v_+$,
which is non-zero as none of $a_{j+1},\dots,a_m$ equal $a_j-1 = i$.
\end{proof}

The next lemma is obviously the same statement as Lemma~\ref{ga},
except that there are no degree shifts to keep track of in the present
ungraded setting.

\begin{Lemma}\label{a}
Let $\la \in \LaImn$ and $i \in I$.
For symbols $x,y \in \{\circ,\up,\down,\times\}$
we write $\la_{xy}$ for the diagram obtained from $\la$
by relabelling the $i$th and $(i+1)$th vertices
by $x$ and $y$, respectively.
\begin{itemize}
\item[\rm(i)]
If $\la = \la_{{\scriptscriptstyle\down}\circ}$ then
$\cF_i \cP(\la) \cong \cP(\la_{\circ{\scriptscriptstyle\down}})$,
$\cF_i \cV(\la) \cong \cV(\la_{\circ{\scriptscriptstyle\down}})$,
$\cF_i \cL(\la) \cong \cL(\la_{\circ{\scriptscriptstyle\down}})$.
\item[\rm(ii)]
If $\la = \la_{{\scriptscriptstyle\up}\circ}$ then
$\cF_i \cP(\la) \cong \cP(\la_{\circ{\scriptscriptstyle\up}})$,
$\cF_i \cV(\la) \cong \cV(\la_{\circ{\scriptscriptstyle\up}})$,
$\cF_i \cL(\la) \cong \cL(\la_{\circ{\scriptscriptstyle\up}})$.
\item[\rm(iii)]
If $\la = \la_{\scriptscriptstyle\times\!\down}$ then
$\cF_i \cP(\la) \cong \cP(\la_{\scriptscriptstyle\down\times})$,
$\cF_i \cV(\la) \cong \cV(\la_{\scriptscriptstyle\down\times})$,
$\cF_i \cL(\la) \cong \cL(\la_{\scriptscriptstyle\down\times})$.
\item[\rm(iv)]
If $\la = \la_{\scriptscriptstyle\times\!\up}$ then
$\cF_i \cP(\la) \cong \cP(\la_{\scriptscriptstyle\up\times})$,
$\cF_i \cV(\la) \cong \cV(\la_{\scriptscriptstyle\up\times})$,
$\cF_i \cL(\la) \cong \cL(\la_{\scriptscriptstyle\up\times})$.
\item[\rm(v)]
If $\la = \la_{{\scriptscriptstyle\times}\circ}$ then:
\begin{itemize}
\item[(a)]
$\cF_i \cP(\la) \cong \cP(\la_{\scriptscriptstyle\down\up})$;
\item[(b)]
there is a short exact sequence
$$
0 \rightarrow \cV(\la_{\scriptscriptstyle\up\down}) \rightarrow \cF_i \cV(\la)
\rightarrow \cV(\la_{\scriptscriptstyle\down\up}) \rightarrow 0;
$$
\item[(c)]
$[\cF_i \cL(\la):\cL(\la_{\scriptscriptstyle\down\up})] = 2$
and all other composition
factors are of the form $\cL(\mu)$ for $\mu$
with
$\mu = \mu_{\scriptscriptstyle\down\down}$,
$\mu = \mu_{\scriptscriptstyle\up\up}$ or
$\mu = \mu_{\scriptscriptstyle\up\down}$;
\item[(d)]
$\cF_i \cL(\la)$ has irreducible
socle and
head isomorphic to $\cL(\la_{\scriptscriptstyle\down\up})$.
\end{itemize}
\item[\rm(vi)]
If $\la = \la_{{\scriptscriptstyle\down\up}}$ then
$\cF_i \cP(\la) \cong
\cP(\la_{\circ\scriptscriptstyle{\times}}) \oplus
\cP(\la_{\circ\scriptscriptstyle{\times}})$,
$\cF_i \cV(\la) \cong \cV(\la_{\circ{\scriptscriptstyle\times}})$
and
$\cF_i \cL(\la) \cong \cL(\la_{\circ{\scriptscriptstyle\times}})$.
\item[\rm(vii)]
If $\la = \la_{{\scriptscriptstyle\up\down}}$ then
$\cF_i \cV(\la) \cong \cV(\la_{\circ{\scriptscriptstyle\times}})$
and $\cF_i \cL(\la) = \{0\}$.
\item[\rm(viii)]
If
$\la = \la_{{\scriptscriptstyle\down\down}}$
then
$\cF_i \cV(\la) = \cF_i \cL(\la) = \{0\}$.
\item[\rm(ix)]
If
$\la = \la_{{\scriptscriptstyle\up\up}}$
then
$\cF_i \cV(\la) = \cF_i \cL(\la) = \{0\}$.
\item[\rm(x)]
For all other $\la$ we have that
$\cF_i \cP(\la) = \cF_i \cV(\la) = \cF_i \cL(\la) = \{0\}$.
\end{itemize}
For the dual statement about $\cE_i$,
interchange all occurrences of $\circ$ and $\times$.
\end{Lemma}

\begin{proof}
The statements (i)--(x) for $\cV(\la)$ follow directly from
Corollary~\ref{tid2}
on translating into the diagrammatic language.

We next check (viii), (ix) and (x) for $\cL(\la)$.
In all these cases, we know already that
$\cF_i \cV(\la) = \{0\}$.
As $\cL(\la)$ is a quotient of $\cV(\la)$ and $\cF_i$ is exact,
it follows immediately that $\cF_i \cL(\la) = \{0\}$ as required.

The proofs of (i), (ii), (iii) and (iv) for $\cL(\la)$
are not much harder. For example, if
$\la = \la_{{\scriptscriptstyle \down}\circ}$ as in (i),
then $\cF_i \cL(\la)$ is a quotient of
$\cF_i \cV(\la) \cong
\cV(\la_{\circ{\scriptscriptstyle \down}})$.
Moreover it is self-dual
as $\cL(\la)$ is self-dual and $\cF_i$ commutes with duality.
Hence we either have that
$\cF_i \cL(\la) = \{0\}$ or
$\cF_i \cL(\la) \cong \cL(\la_{\circ{\scriptscriptstyle \down}})$,
as these are the only self-dual quotients of
$\cV(\la_{\circ{\scriptscriptstyle \down}})$.
To rule out the possibility that it is zero, let $\Ga$ be the block
generated by $\la$, and note that
$\cF_i$ maps $\cO_\Ga$
to $\cO_{\Ga-\alpha_i}$.
Moreover it induces a $\Z$-module isomorphism
$[\cO_\Ga] \stackrel{\sim}{\rightarrow}
[\cO_{\Ga-\alpha_i}]$
because it defines a bijection between the bases
of these Grothendieck groups arising from the standard modules.
Hence $\cF_i$ is non-zero on every non-zero module in
$\cO_\Ga$.
This proves (i) for $\cL(\la)$, and the proofs of (ii), (iii)
and (iv) are similar.

Next we check (vi) and (vii) for $\cL(\la)$, i.e.
we show that
$\cF_i \cL(\la_{\scriptscriptstyle \down\up})
\cong \cL(\la_{\circ{\scriptscriptstyle\times}})$
and
$\cF_i \cL(\la_{\scriptscriptstyle \up\down})
= \{0\}$.
We know that $\cF_i \cV(\la_{\scriptscriptstyle \down\up}) \cong \cF_i
\cV(\la_{\scriptscriptstyle \up\down}) \cong
\cV(\la_{\circ{\scriptscriptstyle\times}})$.
So by an argument from the previous paragraph, we either have that
$\cF_i \cL(\la_{\scriptscriptstyle \down\up})
\cong \cL(\la_{\circ{\scriptscriptstyle\times}})$
or $\cF_i \cL(\la_{\scriptscriptstyle\down\up}) = \{0\}$.
Similarly, either $\cF_i \cL(\la_{\scriptscriptstyle \up\down})
\cong \cL(\la_{\circ{\scriptscriptstyle\times}})$
or $\cF_i \cL(\la_{\scriptscriptstyle\up\down}) = \{0\}$.
As $[\cF_i \cV(\la_{\scriptscriptstyle\down\up}):
\cL(\la_{\circ{\scriptscriptstyle\times}})] = 1$,
there must be some composition factor
$\cL(\mu)$ of $\cV(\la_{\scriptscriptstyle\down\up})$
such that
$[\cF_i \cL(\mu):
\cL(\la_{\circ{\scriptscriptstyle\times}})] = 1$.
The facts proved so far
imply either that
$\mu = \la_{\scriptscriptstyle\down\up}$
or that
$\mu = \la_{\scriptscriptstyle\up\down}$.
But the latter case cannot occur as
$\la_{\scriptscriptstyle\up\down}$ is strictly bigger than
$\la_{\scriptscriptstyle\down\up}$ in the Bruhat ordering.
Hence
$\mu =
\la_{\scriptscriptstyle\down\up}$
and we have proved that
$[\cF_i \cL(\la_{\scriptscriptstyle\down\up}):\cL(\la_{\circ{\scriptscriptstyle\times}})]=1$.
This gives $\cF_i \cL(\la_{\scriptscriptstyle \down\up})
\cong \cL(\la_{\circ{\scriptscriptstyle\times}})$ as required for (vi).
It remains for (vii) to show that
$\cF_i \cL(\la_{\scriptscriptstyle \up\down}) = \{0\}$.
Suppose for a contradiction that it is non-zero,
hence
$\cF_i \cL(\la_{\scriptscriptstyle \up\down})
\cong \cL(\la_{\circ{\scriptscriptstyle\times}})$.
By Lemma~\ref{jsf},
$\cV(\la_{\scriptscriptstyle \up\down})$
has both
$\cL(\la_{\scriptscriptstyle \up\down})$
and
$\cL(\la_{\scriptscriptstyle \down\up})$ as composition factors,
so we deduce that
$[\cF_i \cV(\la_{\scriptscriptstyle \up\down}):
\cL(\la_{\circ{\scriptscriptstyle\times}})]
\geq 2$, which is the desired
 contradiction.

In this paragraph, we check (v)(c).
Take $\la$ with $\la = \la_{{\scriptscriptstyle \times}\circ}$.
Let $\Ga$ be the block generated by $\la$.
Note $\cF_i$ maps
$\cO_\Ga$ to $\cO_{\Ga-\alpha_i}$
and $\cO_{\Ga-\alpha_i}$ to
$\cO_{\Ga-2\alpha_i}$.
We know for any $\nu \in \Ga$ that $\cF_i^2 \cV(\nu)
\cong
\cV(\nu_{\circ{\scriptscriptstyle\times}})
\oplus \cV(\nu_{\circ{\scriptscriptstyle\times}})$.
Hence $\cF_i^2$ induces a $\Z$-module isomorphism
between $[\cO_\Ga]$ and
$2 [\cO_{\Ga-2\alpha_i}]$.
We deduce for any non-zero module $M \in \cO_\Ga$
that $\cF_i^2 M$
is non-zero and its class is
divisible by two in $[\cO_{\Ga-2\alpha_i}]$.
In particular,  $\cF_i^2 \cL(\la)$
 is a non-zero self-dual quotient of
$\cV(\la_{\circ{\scriptscriptstyle\times}})
\oplus \cV(\la_{\circ{\scriptscriptstyle\times}})$
whose class is divisible by two.
This implies that
\begin{equation}\label{drop}
\cF_i^2 \cL(\la)
\cong
\cL(\la_{\circ{\scriptscriptstyle\times}})\oplus
\cL(\la_{\circ{\scriptscriptstyle\times}}).
\end{equation}
Now take any $\mu \in \Ga-\alpha_i$.
We know already that
$\cF_i \cL(\mu) \cong \cL(\mu_{\circ{\scriptscriptstyle \times}})$
if $\mu = \mu_{\scriptscriptstyle \down\up}$, and
$\cF_i \cL(\mu) = \{0\}$ otherwise.
Assuming now that $\mu =
\mu_{\scriptscriptstyle \down\up}$,
we deduce from this that
$[\cF_i \cL(\la): \cL(\mu)] =
[\cF_i^2 \cL(\la): \cL(\mu_{\circ{\scriptscriptstyle \times}})]$.
Using (\ref{drop}), we conclude for $\mu = \mu_{\scriptscriptstyle \down\up}$
that
$[\cF_i \cL(\la):\cL(\mu)] = 0$ unless $\mu =
\la_{\scriptscriptstyle \down\up}$, and
$[\cF_i \cL(\la):\cL(\la_{\scriptscriptstyle \down\up})] = 2$.

Now we deduce all the statements (i)--(x) for $\cP(\la)$
by using the fact that $(\cF_i, \cE_i)$ is an adjoint pair of functors.
We just explain the argument in case (vi), since the other cases
are similar (actually, easier).
As $\cF_i$ sends projectives to projectives, $\cF_i \cP(\la)$ is a direct
sum of projective indecomposables. To compute the multiplicity of $\cP(\mu)$
in this decomposition we calculate
$$
\hom_{\g}(\cF_i \cP(\la), \cL(\mu))
\cong \hom_{\g}(\cP(\la), \cE_i \cL(\mu)) = [\cE_i \cL(\mu): \cL(\la)].
$$
By (v)(c) (or rather, its analogue for $\cE_i$)
this multiplicity is zero unless $\mu = \la_{\circ{\scriptscriptstyle\times}}$,
when it is two.
Hence $\cF_i \cP(\la) \cong \cP(\la_{\circ{\scriptscriptstyle\times}}) \oplus
\cP(\la_{\circ{\scriptscriptstyle\times}})$.

It just remains to deduce (v)(d).
By (v)(a), (v)(c) and exactness of $\cF_i$, we get that
$\cF_i \cL(\la)$ is a non-zero quotient of
$\cP(\la_{\scriptscriptstyle \down\up})$,
hence it has irreducible head
isomorphic to
$\cL(\la_{\scriptscriptstyle\down\up})$.
Since it is self-dual it also has irreducible socle
isomorphic to
$\cL(\la_{\scriptscriptstyle\down\up})$.
\end{proof}

\phantomsubsection{\boldmath The second categorification theorem}
The following theorem should be compared with Theorem~\ref{cat1}.

\begin{Theorem}\label{cat2}
Identify $[\OImn]$
with  $\bigwedge^m V_\Z \otimes \bigwedge^n V_\Z$
by identifying $[\cV(\la)]$ with $\cV_\la$ for each
$\la \in \LaImn$.
\begin{enumerate}
\item[\rm(i)] We have that $[\cL(\la)] = \cL_\la$
and $[\cP(\la)] = \cP_\la$ for each $\la \in \LaImn$.
\item[\rm(ii)] The
endomorphisms of the Grothendieck group induced by the
exact functors
$\cE_i$ and $\cF_i$
coincide with the action of the Chevalley
generators $\cE_i$ and $\cF_i$
 of $U_\Z$ for each $i \in I$.
\item[\rm(iii)]
We have that
$$
\dim \hom_{\mathfrak{g}}(P, M)=
\left\langle [P], [M] \right\rangle
$$
for $M, P \in \OImn$
with $P$ projective.
\end{enumerate}
\end{Theorem}

\begin{proof}
In view of Theorem~\ref{cat1},
Lemma~\ref{ga} can be re-interpreted as describing
how the generators $E_i$ and $F_i$ of $U_{\!\Laurent}$ act on
the basis elements $V_\la, L_\la$ and $P_\la$
of $\bigwedge^m V_\Laurent \otimes \bigwedge^n V_\Laurent$.
Specializing at $q=1$, we get analogous descriptions of how
the generators $\cE_i$ and $\cF_i$ of $U_\Z$
act on $\cV_\la, \cL_\la$ and $\cP_\la$.
In particular, we see that the Chevalley generators $\cE_i$ and $\cF_i$ act on
$\cV_\la$ in exactly the same way as the
functors $\cE_i$ and $\cF_i$ act on $[\cV(\la)]$ as described by Lemma~\ref{a}.
This proves (ii).

To deduce (i),
take any $\la \in \LaImn$ and
let $\mu, i_1,\dots,i_k$ and $r$ be as in Lemma~\ref{cgl}.
In view of Theorem~\ref{cgt} specialised at $q=1$, we know already that
$$
\cF_{i_k} \cdots \cF_{i_1} \cV_\mu = 2^r \cP_\la.
$$
On the other hand by Lemma~\ref{a} we have that
$$
[\cF_{i_k}\circ \cdots\circ \cF_{i_1} (\cP(\mu))] = 2^r [\cP(\la)].
$$
As $\mu$ is maximal in the Bruhat ordering,
the parabolic Verma module $\cV(\mu)$ is projective,
i.e. $[\cP(\mu)] = [\cV(\mu)] = \cV_\mu$.
Hence combining the above two equations, we get that
$[\cP(\la)] =  \cP_\la$.
It then follows that
 $[\cL(\la)] = \cL_\la$ too, by
(\ref{bgg}) and the usual BGG reciprocity in the highest weight
category $\OImn$.

Finally (iii) follows because the bases $\{\cP_\la\}$ and
$\{\cL_\la\}$ are dual with respect to the form
$\langle.,.\rangle$ and also
$\dim \hom_{\mathfrak{g}}(\cP(\la),\cL(\mu)) = \delta_{\la,\mu}$.
\end{proof}

\begin{Remark}\rm
Theorem~\ref{cat2} is certainly not new; for example, essentially
this theorem appears already in \cite[Theorem 5.5]{CWZ}.
Its generalisation from 2-block to $k$-block
parabolic subalgebras in type A
is recorded in \cite[Theorem 4.5]{BKrep}, where it is deduced from
the Kazhdan-Lusztig conjecture; see also \cite[$\S$3.1]{BKariki}.
The graded version of this result gives a categorification of
a $k$-fold tensor product of quantum exterior powers,
which has been  used in
\cite{Sussan} and \cite{MSslk} to define functorial knot and
tangle invariants which decategorify to the $\mathfrak{sl}_k$-version of the
HOMFLY-PT polynomial \cite{MOY}.
\end{Remark}

\section{Local analysis of special projective functors}\label{sL}

The goal in the remainder of the article is to explain the
combinatorial coincidence between Theorems~\ref{cat1} and \ref{cat2}
by showing that the categories $\rep{\KImn}$ and $\OImn$ are equivalent.
Most of the new work needed to establish this takes place on the diagram
algebra side. We begin in this section by studying
some locally-defined natural transformations between
compositions of special projective functors.

\phantomsubsection{Admissible sequences and associated composite matchings}
For {any} sequence $\bi = (i_1,\dots,i_d) \in I^d$, we can consider
the compositions
\begin{align}\label{cf}
F_{\bi} &:= F_{i_d} \circ \cdots \circ F_{i_1}:\Rep{\KImn}
\rightarrow \Rep{\KImn},\\
E_{\bi} &:= E_{i_d} \circ \cdots \circ E_{i_1}:\Rep{\KImn}
\rightarrow \Rep{\KImn}.\label{cf2}
\end{align}
We are mainly going to be interested
here in the properties of the first of these.

Suppose we are
given a block $\Ga \in \PImn$.
Generalising the definition made just before (\ref{CKLR}),
we say that
$\bi = (i_1,\dots,i_d) \in I^d$ is a {\em $\Ga$-admissible sequence}
if $\Ga - \alpha_{i_1}-\cdots-\alpha_{i_r} \in \PImn$ for each $r=1,\dots,d$.
The restriction $F_{\bi}|_{\Rep{K_\Ga}}$ of the functor $F_\bi$ to the subcategory
$\Rep{K_\Ga}$ is obviously zero unless $\bi$ is a $\Ga$-admissible sequence.

Given a $\Ga$-admissible sequence $\bi \in I^d$,
we define
the {\em associated block sequence}
$\bGa = \Ga_d \cdots \Ga_0$
by setting $\Ga_r := \Ga - \alpha_{i_1}-\cdots-\alpha_{i_r}$
for each $r=0,\dots,d$.
Then define the
 {\em associated composite matching}
$\bt = t_d \cdots t_1$
by setting
$t_r := t_{i_r}(\Ga_{r-1})$
for each $r=1,\dots,d$;
see Figure~\ref{fig1} for an example.
We say that $\bi$ is a
{\em proper} $\Ga$-admissible sequence if $\bt$
is a proper $\bGa$-matching in the sense of \cite[$\S$2]{BS2}.

\begin{figure}
$$
\begin{picture}(-230,295)
%\put(-195,316){$_1$}
%\put(-175,316){$_2$}
%\put(-155,316){$_3$}
%\put(-135,316){$_4$}
%\put(-115,316){$_5$}
%\put(-95,316){$_6$}
%\put(-75,316){$_7$}
%\put(-55,316){$_8$}
%\put(-35,316){$_9$}
%\put(-15.5,316){$_{10}$}

\put(-214,286){$_1$}
\put(-214,266){$_2$}
\put(-214,246){$_3$}
\put(-214,226){$_4$}
\put(-214,206){$_5$}
\put(-214,186){$_6$}
\put(-214,166){$_7$}
\put(-214,146){$_8$}
\put(-214,126){$_9$}
\put(-218,106){$_{10}$}
\put(-218,86){$_{11}$}
\put(-218,66){$_{12}$}
\put(-218,46){$_{13}$}
\put(-218,26){$_{14}$}
\put(-218,6){$_{15}$}

\put(-193,297){\line(1,0){180}}
\put(-193,277){\line(1,0){180}}
\put(-193,257){\line(1,0){180}}
\put(-193,237){\line(1,0){180}}
\put(-193,217){\line(1,0){180}}
\put(-193,197){\line(1,0){180}}
\put(-193,177){\line(1,0){180}}
\put(-193,157){\line(1,0){180}}
\put(-193,137){\line(1,0){180}}
\put(-193,117){\line(1,0){180}}
\put(-193,97){\line(1,0){180}}
\put(-193,77){\line(1,0){180}}
\put(-193,57){\line(1,0){180}}
\put(-193,37){\line(1,0){180}}
\put(-193,17){\line(1,0){180}}
\put(-193,-3){\line(1,0){180}}

\put(-176,295){{$\scriptstyle\times$}}
\put(-156,295){{$\scriptstyle\times$}}
\put(-136,295){{$\scriptstyle\times$}}
\put(-115.7,294.3){{$\circ$}}
\put(-95.7,294.3){{$\circ$}}
\put(-75,295){{$\scriptstyle\bullet$}}
\put(-55,295){{$\scriptstyle\bullet$}}
\put(-35,295){{$\scriptstyle\bullet$}}
\put(-43,297){\oval(20,20)[b]}

\put(-196,295){{$\scriptstyle\times$}}
\put(-15.7,294.3){{$\circ$}}
\put(-176,275){{$\scriptstyle\times$}}
\put(-156,275){{$\scriptstyle\times$}}
\put(-136,275){{$\scriptstyle\times$}}
\put(-115.7,274.3){{$\circ$}}
\put(-95.7,274.3){{$\circ$}}
\put(-75,275){{$\scriptstyle\bullet$}}
\put(-55.7,274.3){{$\circ$}}
\put(-36,275){{$\scriptstyle\times$}}
\put(-73,277){\line(0,1){20}}

\put(-196,275){{$\scriptstyle\times$}}
\put(-15.7,274.3){{$\circ$}}
\put(-176,255){{$\scriptstyle\times$}}
\put(-156,255){{$\scriptstyle\times$}}
\put(-95.7,254.3){{$\circ$}}
\put(-55.7,254.3){{$\circ$}}
\put(-75,255){{$\scriptstyle\bullet$}}
\put(-36,255){{$\scriptstyle\times$}}
\put(-53,237){\line(-1,1){20}}
\put(-135,255){{$\scriptstyle\bullet$}}
\put(-115,255){{$\scriptstyle\bullet$}}

\put(-196,255){{$\scriptstyle\times$}}
\put(-15.7,254.3){{$\circ$}}
\put(-176,235){{$\scriptstyle\times$}}
\put(-156,235){{$\scriptstyle\times$}}
\put(-135,235){{$\scriptstyle\bullet$}}
\put(-115,235){{$\scriptstyle\bullet$}}
\put(-95.7,234.3){{$\circ$}}
\put(-75.7,234.3){{$\circ$}}
\put(-123,257){\oval(20,20)[t]}
\put(-36,235){{$\scriptstyle\times$}}
\put(-73,257){\line(0,1){20}}
\put(-55,235){{$\scriptstyle\bullet$}}

\put(-196,235){{$\scriptstyle\times$}}
\put(-15.7,234.3){{$\circ$}}
\put(-176,215){{$\scriptstyle\times$}}
\put(-156,215){{$\scriptstyle\times$}}
\put(-135,215){{$\scriptstyle\bullet$}}
\put(-115.7,214.3){{$\circ$}}
\put(-95,215){{$\scriptstyle\bullet$}}
\put(-75.7,214.3){{$\circ$}}
\put(-93,217){\line(-1,1){20}}
\put(-133,217){\line(0,1){20}}
\put(-133,237){\line(0,1){20}}
\put(-113,237){\line(0,1){20}}
\put(-36,215){{$\scriptstyle\times$}}
\put(-53,217){\line(0,1){20}}
\put(-55,215){{$\scriptstyle\bullet$}}

\put(-196,215){{$\scriptstyle\times$}}
\put(-15.7,214.3){{$\circ$}}
\put(-176,195){{$\scriptstyle\times$}}
\put(-155,195){{$\scriptstyle\bullet$}}
\put(-135,195){{$\scriptstyle\times$}}
\put(-115.8,194.3){{$\circ$}}
\put(-95,195){{$\scriptstyle\bullet$}}
\put(-75.7,194.3){{$\circ$}}
\put(-153,197){\line(1,1){20}}
\put(-93,197){\line(0,1){20}}
\put(-53,197){\line(0,1){20}}
\put(-36,195){{$\scriptstyle\times$}}
\put(-55,195){{$\scriptstyle\bullet$}}

\put(-196,195){{$\scriptstyle\times$}}
\put(-15.7,194.3){{$\circ$}}
\put(-176,175){{$\scriptstyle\times$}}
\put(-155,175){{$\scriptstyle\bullet$}}
\put(-135,175){{$\scriptstyle\bullet$}}
\put(-115,175){{$\scriptstyle\bullet$}}
\put(-95,175){{$\scriptstyle\bullet$}}
\put(-75.7,174.3){{$\circ$}}
\put(-153,177){\line(0,1){20}}
\put(-93,177){\line(0,1){20}}
\put(-123,177){\oval(20,20)[t]}
\put(-53,177){\line(0,1){20}}
\put(-36,175){{$\scriptstyle\times$}}
\put(-55,175){{$\scriptstyle\bullet$}}

\put(-196,175){{$\scriptstyle\times$}}
\put(-15.7,174.3){{$\circ$}}
\put(-176,155){{$\scriptstyle\times$}}
\put(-155,155){{$\scriptstyle\bullet$}}
\put(-135,155){{$\scriptstyle\bullet$}}
\put(-115,155){{$\scriptstyle\bullet$}}
\put(-95.7,154.3){{$\circ$}}
\put(-75,155){{$\scriptstyle\bullet$}}
\put(-153,157){\line(0,1){20}}
\put(-133,157){\line(0,1){20}}
\put(-113,157){\line(0,1){20}}
\put(-73,157){\line(-1,1){20}}
\put(-53,157){\line(0,1){20}}
\put(-36,155){{$\scriptstyle\times$}}
\put(-55,155){{$\scriptstyle\bullet$}}

\put(-196,155){{$\scriptstyle\times$}}
\put(-15.7,154.3){{$\circ$}}
\put(-176,135){{$\scriptstyle\times$}}
\put(-155,135){{$\scriptstyle\bullet$}}
\put(-135,135){{$\scriptstyle\bullet$}}
\put(-115.7,134.3){{$\circ$}}
\put(-95,135){{$\scriptstyle\bullet$}}
\put(-75,135){{$\scriptstyle\bullet$}}
\put(-153,137){\line(0,1){20}}
\put(-133,137){\line(0,1){20}}
\put(-93,137){\line(-1,1){20}}
\put(-73,137){\line(0,1){20}}
\put(-53,137){\line(0,1){20}}
\put(-36,135){{$\scriptstyle\times$}}
\put(-55,135){{$\scriptstyle\bullet$}}

\put(-196,135){{$\scriptstyle\times$}}
\put(-15.7,134.3){{$\circ$}}
\put(-175,115){{$\scriptstyle\bullet$}}
\put(-156,115){{$\scriptstyle\times$}}
\put(-135,115){{$\scriptstyle\bullet$}}
\put(-115.7,114.3){{$\circ$}}
\put(-95,115){{$\scriptstyle\bullet$}}
\put(-75,115){{$\scriptstyle\bullet$}}
\put(-173,97){\line(0,1){20}}
\put(-113,97){\line(-1,1){20}}
\put(-93,117){\line(0,1){20}}
\put(-73,117){\line(0,1){20}}
\put(-53,117){\line(0,1){20}}
\put(-36,115){{$\scriptstyle\times$}}
\put(-55,115){{$\scriptstyle\bullet$}}

\put(-196,115){{$\scriptstyle\times$}}
\put(-15.7,114.3){{$\circ$}}
\put(-175,95){{$\scriptstyle\bullet$}}
\put(-156,95){{$\scriptstyle\times$}}
\put(-135.7,94.3){{$\circ$}}
\put(-115,95){{$\scriptstyle\bullet$}}
\put(-95,95){{$\scriptstyle\bullet$}}
\put(-75,95){{$\scriptstyle\bullet$}}
\put(-173,77){\line(0,1){20}}
\put(-93,97){\line(0,1){20}}
\put(-133,117){\line(0,1){20}}
\put(-73,97){\line(0,1){20}}
\put(-103,97){\oval(20,20)[b]}
\put(-53,97){\line(0,1){20}}
\put(-36,95){{$\scriptstyle\times$}}
\put(-55,95){{$\scriptstyle\bullet$}}

\put(-196,95){{$\scriptstyle\times$}}
\put(-15.7,94.3){{$\circ$}}
\put(-175,75){{$\scriptstyle\bullet$}}
\put(-156,75){{$\scriptstyle\times$}}
\put(-135.7,74.3){{$\circ$}}
\put(-115.7,74.3){{$\circ$}}
\put(-96,75){{$\scriptstyle\times$}}
\put(-75,75){{$\scriptstyle\bullet$}}
\put(-173,117){\line(1,1){20}}
\put(-73,77){\line(0,1){20}}
\put(-53,77){\line(0,1){20}}
\put(-36,75){{$\scriptstyle\times$}}
\put(-55,75){{$\scriptstyle\bullet$}}

\put(-196,75){{$\scriptstyle\times$}}
\put(-15.7,74.3){{$\circ$}}
\put(-175,55){{$\scriptstyle\bullet$}}
\put(-155,55){{$\scriptstyle\bullet$}}
\put(-135,55){{$\scriptstyle\bullet$}}
\put(-115.7,54.3){{$\circ$}}
\put(-96,55){{$\scriptstyle\times$}}
\put(-75,55){{$\scriptstyle\bullet$}}
\put(-173,57){\line(0,1){20}}
\put(-73,57){\line(0,1){20}}
\put(-143,57){\oval(20,20)[t]}
\put(-53,57){\line(0,1){20}}
\put(-36,55){{$\scriptstyle\times$}}
\put(-55,55){{$\scriptstyle\bullet$}}

\put(-196,55){{$\scriptstyle\times$}}
\put(-15.7,54.3){{$\circ$}}
\put(-175.7,34.3){{$\circ$}}
\put(-156,35){{$\scriptstyle\times$}}
\put(-135,35){{$\scriptstyle\bullet$}}
\put(-115.7,34.3){{$\circ$}}
\put(-96,35){{$\scriptstyle\times$}}
\put(-75,35){{$\scriptstyle\bullet$}}
\put(-133,37){\line(0,1){20}}
\put(-73,37){\line(0,1){20}}
\put(-163,57){\oval(20,20)[b]}
\put(-53,37){\line(0,1){20}}
\put(-36,35){{$\scriptstyle\times$}}
\put(-55,35){{$\scriptstyle\bullet$}}

\put(-196,35){{$\scriptstyle\times$}}
\put(-15.7,34.3){{$\circ$}}
\put(-175.7,14.3){{$\circ$}}
\put(-156,15){{$\scriptstyle\times$}}
\put(-135,15){{$\scriptstyle\bullet$}}
\put(-115.7,14.3){{$\circ$}}
\put(-95,15){{$\scriptstyle\bullet$}}
\put(-76,15){{$\scriptstyle\times$}}
\put(-133,17){\line(0,1){20}}
\put(-73,37){\line(-1,-1){20}}
\put(-53,17){\line(0,1){20}}
\put(-36,15){{$\scriptstyle\times$}}
\put(-55,15){{$\scriptstyle\bullet$}}

\put(-196,15){{$\scriptstyle\times$}}
\put(-15.7,14.3){{$\circ$}}
\put(-196,-5){{$\scriptstyle\times$}}
\put(-175.7,-5.7){{$\circ$}}
\put(-155,-5){{$\scriptstyle\bullet$}}
\put(-136,-5){{$\scriptstyle\times$}}
\put(-115.7,-5.7){{$\circ$}}
\put(-95,-5){{$\scriptstyle\bullet$}}
\put(-76,-5){{$\scriptstyle\times$}}
\put(-55,-5){{$\scriptstyle\bullet$}}
\put(-36,-5){{$\scriptstyle\times$}}
\put(-15.7,-5.7){{$\circ$}}
\put(-153,-3){\line(1,1){20}}
\put(-93,-3){\line(0,1){20}}
\put(-53,-3){\line(0,1){20}}
\end{picture}
$$
\caption{An associated composite matching}\label{fig1}
\end{figure}
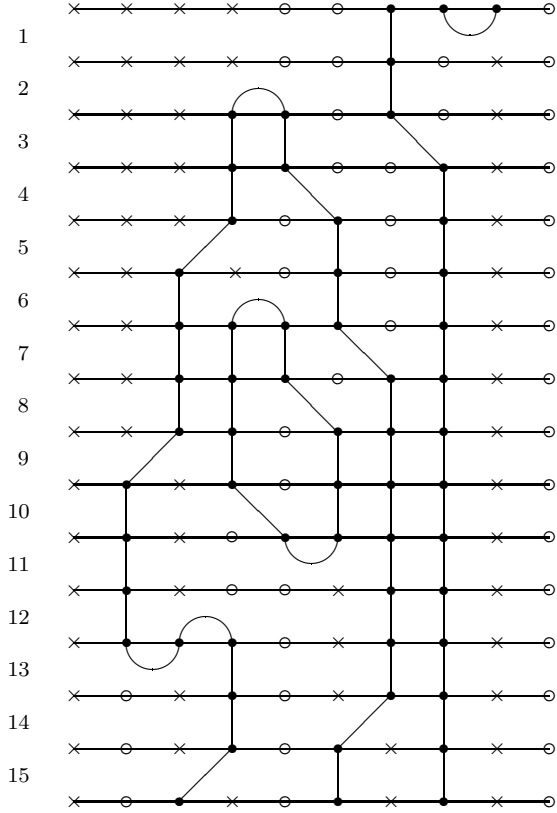

\begin{Lemma}\label{caniso}
Let $\bi \in I^d$ be a $\Ga$-admissible sequence,
and $\bGa$ and $\bt$ be the associated block
sequence and composite matching.
There is a
canonical isomorphism
between $F_{\bi}|_{\Rep{K_\Ga}}$
and the functor $G^{\bt}_{\bGa}$ that arises by tensoring with the
bimodule $K^\bt_{\bGa}\langle-\caps(\bt)\rangle$.
\end{Lemma}

\begin{proof}
By the definition (\ref{gFi}), the restriction of $F_{\bi}$ to $\Rep{K_\Ga}$
is equal to the composition
$G^{t_d}_{\Ga_d\Ga_{d-1}} \circ\cdots\circ
G^{t_1}_{\Ga_1\Ga_0}$. Now apply (\ref{miso2}).
\end{proof}

\begin{Corollary}
The restriction
$F_\bi|_{\Rep{K_\Ga}}$ is non-zero
if and only if $\bi$ is a proper $\Ga$-admissible sequence.
\end{Corollary}

\begin{proof}
This follows from
Lemma~\ref{caniso} and the fact that $G^\bt_\bGa$ is non-zero
if and only if $\bt$ is a proper $\bGa$-matching
(recall the discussion immediately before (\ref{tlanotation})).
\end{proof}

Suppose that $\bt = t_d \cdots t_1$
is the composite matching associated to some
$\Ga$-admissible sequence $\bi \in I^d$.
The diagram $\bt$ involves various different
connected components.
We refer to such a component $C$ as
a {\em generalised cap}
if it is connected to the bottom number line but not the top number line,
a {\em generalised cup}
if it is connected to the top number line but not the bottom number line,
a {\em propagating line}
if it is connected to both the bottom and the top number lines,
and an {\em internal circle} if it is not connected to
either the top or the bottom number lines.
The example in Figure~\ref{fig1}
contains one of each of these sorts of components.
An internal circle is called a {\em small circle} if it consists
just of a single cap and a single cup.

Recalling (\ref{CKLR}), each level $t_r$ of $\bt$
contains exactly one of the following:
a cup, a cap, a right-shift or a left-shift.
We say that a component $C$ of $\bt$ has a cup at level $r$, a cap
at level $r$, a right-shift at level $r$
or a left-shift at level $r$ if
there is such a cup, cap, right-shift or left-shift in $t_r$
that lies on the component $C$.
We say that $C$ is {\em non-trivial} at level $r$ if one of these
four things occurs.
Define the {\em height} of $C$ to be the number of $r=1,\dots,d$
such that $C$ is non-trivial at level $r$.
The sum of the heights of all the components of $\bt$
is equal to the height $d$ of the composite matching $\bt$ itself.
For example, if $C$ is
the internal circle in Figure~\ref{fig1}
then $C$ is non-trivial
at levels 6, 8, 10 and 11, hence it
is of height four, and it
has a cap at level $6$.

\begin{Lemma}\label{as}
Let $\Ga \in \PImn$.
\begin{itemize}
\item[(i)]
$(i,i)$ is a $\Ga$-admissible sequence
if and only if the $i$th and $(i+1)$th vertices of $\Ga$ are
labelled $\times$ and
$\circ$, respectively.
The asssociated composite matching contains a small internal circle:
$$
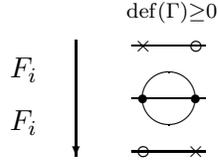
\begin{picture}(50,62)
\put(-84,55){$_{\defect(\Ga) \geq 0}$}

\put(-103,44){\vector(0,-1){44}}
\put(-128,30){$F_i$}
\put(-128,10){$F_{i}$}
\put(-82,42){\line(1,0){28}}
\put(-82,22){\line(1,0){28}}
\put(-82,2){\line(1,0){28}}

\put(-80,20){{$\scriptstyle\bullet$}}
\put(-60,20){{$\scriptstyle\bullet$}}
\put(-60.3,39.3){{$\circ$}}
\put(-80.3,-0.7){{$\circ$}}
\put(-81,40){{$\scriptstyle\times$}}
\put(-61,0){{$\scriptstyle\times$}}
\put(-68,22){\oval(20,20)[t]}
\put(-68,22){\oval(20,20)[b]}
\end{picture}
$$
\item[(ii)]
$(i,i+1)$ and $(i+1,i)$ are both $\Ga$-admissible sequences
if and only if both of the following conditions hold:
\begin{itemize}
\item[(a)]
the $(i+1)$th vertex of $\Ga$ is labelled $\bullet$;
\item[(b)]
  if the $i$th or the $(i+2)$th vertex of $\Ga$ is labelled
$\bullet$ then $\defect(\Ga) \geq 1$.
\end{itemize}
The associated composite matchings fall into the following
four families (displaying only the strip between $i$
and $i+2$):
$$
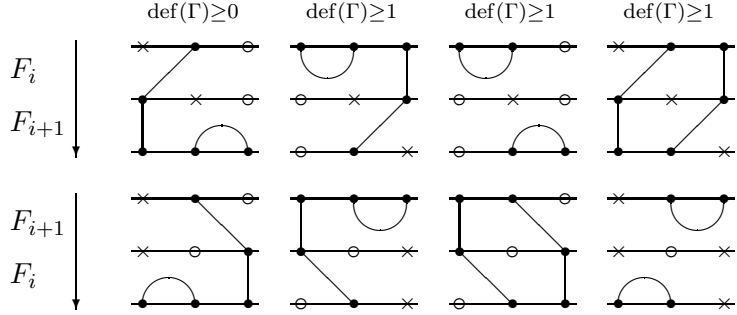
\begin{picture}(50,61)
\put(-103,44){\vector(0,-1){44}}
\put(-128,30){$F_i$}
\put(-128,10){$F_{i+1}$}

\put(-76,55){$_{\defect(\Ga) \geq 0}$}
\put(-16,55){$_{\defect(\Ga) \geq 1}$}
\put(44,55){$_{\defect(\Ga) \geq 1}$}
\put(104,55){$_{\defect(\Ga) \geq 1}$}
\put(-82,42){\line(1,0){48}}
\put(-82,22){\line(1,0){48}}
\put(-82,2){\line(1,0){48}}
\put(-58,42){\line(-1,-1){20}}
\put(-81,40){{$\scriptstyle\times$}}
\put(-60,40){{$\scriptstyle\bullet$}}
\put(-40.7,39.3){{$\circ$}}
\put(-78,22){\line(0,-1){20}}
\put(-80,20){$\scriptstyle\bullet$}
\put(-61,20){$\scriptstyle\times$}
\put(-40.7,19.3){{$\circ$}}
\put(-80,0){$\scriptstyle\bullet$}
\put(-60,0){$\scriptstyle\bullet$}
\put(-40,0){{$\scriptstyle\bullet$}}
\put(-48,2){\oval(20,20)[t]}

\put(-22,42){\line(1,0){48}}
\put(-22,22){\line(1,0){48}}
\put(-22,2){\line(1,0){48}}
\put(-8,40){\oval(20,20)[b]}
\put(-20,40){{$\scriptstyle\bullet$}}
\put(0,40){{$\scriptstyle\bullet$}}
\put(20,40){{$\scriptstyle\bullet$}}
\put(22,42){\line(0,-1){20}}
\put(-20.7,19.3){$\circ$}
\put(-1,20){$\scriptstyle\times$}
\put(20,20){{$\scriptstyle\bullet$}}
\put(-20.7,-0.7){$\circ$}
\put(0,0){$\scriptstyle\bullet$}
\put(19,0){{$\scriptstyle\times$}}
\put(22,22){\line(-1,-1){20}}
\put(38,42){\line(1,0){48}}
\put(38,22){\line(1,0){48}}
\put(38,2){\line(1,0){48}}
\put(52,40){\oval(20,20)[b]}
\put(40,40){{$\scriptstyle\bullet$}}
\put(60,40){{$\scriptstyle\bullet$}}
\put(79.3,39.3){{$\circ$}}
\put(39.3,19.3){$\circ$}
\put(59,20){$\scriptstyle\times$}
\put(79.3,19.3){{$\circ$}}
\put(39.3,-0.7){$\circ$}
\put(60,0){$\scriptstyle\bullet$}
\put(80,0){{$\scriptstyle\bullet$}}
\put(72,2){\oval(20,20)[t]}
\put(98,42){\line(1,0){48}}
\put(98,22){\line(1,0){48}}
\put(98,2){\line(1,0){48}}
\put(122,42){\line(-1,-1){20}}
\put(99,40){{$\scriptstyle\times$}}
\put(120,40){{$\scriptstyle\bullet$}}
\put(140,40){{$\scriptstyle\bullet$}}
\put(142,42){\line(0,-1){20}}
\put(102,22){\line(0,-1){20}}
\put(100,20){$\scriptstyle\bullet$}
\put(119,20){$\scriptstyle\times$}
\put(140,20){{$\scriptstyle\bullet$}}
\put(142,22){\line(-1,-1){20}}
\put(100,0){$\scriptstyle\bullet$}
\put(120,0){$\scriptstyle\bullet$}
\put(139,0){{$\scriptstyle\times$}}
\end{picture}
$$
$$
\begin{picture}(50,52)
\put(-103,44){\vector(0,-1){44}}
\put(-128,30){$F_{i+1}$}
\put(-128,10){$F_{i}$}

\put(-82,42){\line(1,0){48}}
\put(-82,22){\line(1,0){48}}
\put(-82,2){\line(1,0){48}}
\put(-58,42){\line(1,-1){20}}
\put(-81,40){{$\scriptstyle\times$}}
\put(-60,40){{$\scriptstyle\bullet$}}
\put(-40.7,39.3){{$\circ$}}
\put(-38,22){\line(0,-1){20}}
\put(-81,20){$\scriptstyle\times$}
\put(-60.7,19.3){$\circ$}
\put(-40,20){{$\scriptstyle\bullet$}}
\put(-80,0){$\scriptstyle\bullet$}
\put(-60,0){$\scriptstyle\bullet$}
\put(-40,0){{$\scriptstyle\bullet$}}
\put(-68,2){\oval(20,20)[t]}

\put(-22,42){\line(1,0){48}}
\put(-22,22){\line(1,0){48}}
\put(-22,2){\line(1,0){48}}
\put(12,40){\oval(20,20)[b]}
\put(-20,40){{$\scriptstyle\bullet$}}
\put(0,40){{$\scriptstyle\bullet$}}
\put(20,40){{$\scriptstyle\bullet$}}
\put(-18,42){\line(0,-1){20}}
\put(-20,20){$\scriptstyle\bullet$}
\put(-0.7,19.3){$\circ$}
\put(19,20){{$\scriptstyle\times$}}
\put(-20.7,-0.7){$\circ$}
\put(0,0){$\scriptstyle\bullet$}
\put(19,0){{$\scriptstyle\times$}}
\put(-18,22){\line(1,-1){20}}
\put(38,42){\line(1,0){48}}
\put(38,22){\line(1,0){48}}
\put(38,2){\line(1,0){48}}
\put(62,42){\line(1,-1){20}}
\put(40,40){{$\scriptstyle\bullet$}}
\put(60,40){{$\scriptstyle\bullet$}}
\put(79.3,39.3){{$\circ$}}
\put(42,42){\line(0,-1){20}}
\put(82,22){\line(0,-1){20}}
\put(40,20){$\scriptstyle\bullet$}
\put(59.3,19.3){$\circ$}
\put(80,20){{$\scriptstyle\bullet$}}
\put(42,22){\line(1,-1){20}}
\put(39.3,-0.7){$\circ$}
\put(60,0){$\scriptstyle\bullet$}
\put(80,0){{$\scriptstyle\bullet$}}
\put(98,42){\line(1,0){48}}
\put(98,22){\line(1,0){48}}
\put(98,2){\line(1,0){48}}
\put(132,40){\oval(20,20)[b]}
\put(99,40){{$\scriptstyle\times$}}
\put(120,40){{$\scriptstyle\bullet$}}
\put(140,40){{$\scriptstyle\bullet$}}
\put(99,20){$\scriptstyle\times$}
\put(119.3,19.3){$\circ$}
\put(139,20){{$\scriptstyle\times$}}
\put(100,0){$\scriptstyle\bullet$}
\put(120,0){$\scriptstyle\bullet$}
\put(139,0){{$\scriptstyle\times$}}
\put(112,2){\oval(20,20)[t]}
\end{picture}
\vspace{2mm}
$$
\item[(iii)] $(i,i+1)$ is a $\Ga$-admissible sequence but
$(i+1,i)$ is not if and only if exactly one of the following
conditions holds:
\begin{itemize}
\item[(a)] the $(i+1)$th vertex of $\Ga$ is labelled $\circ$,
and moreover if the $i$th and $(i+2)$th vertices
are labelled $\bullet$ then $\defect(\Ga) \geq 1$;
\item[(b)]
the $i$th, $(i+1)$th and $(i+2)$th vertices
are labelled $\times,\bullet$ and $\bullet$, respectively,
and $\defect(\Ga) = 0$.
\end{itemize}
The associated composite matchings are as follows:
$$
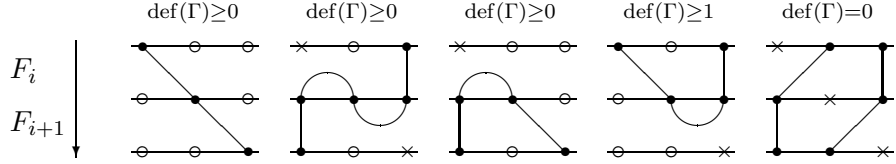
\begin{picture}(50,64)
\put(-76,55){$_{\defect(\Ga) \geq 0}$}
\put(-16,55){$_{\defect(\Ga) \geq 0}$}
\put(44,55){$_{\defect(\Ga) \geq 0}$}
\put(104,55){$_{\defect(\Ga) \geq 1}$}
\put(164,55){$_{\defect(\Ga) = 0}$}

\put(-103,44){\vector(0,-1){44}}
\put(-128,30){$F_i$}
\put(-128,10){$F_{i+1}$}
\put(-82,42){\line(1,0){48}}
\put(-82,22){\line(1,0){48}}
\put(-82,2){\line(1,0){48}}
\put(-80,40){{$\scriptstyle\bullet$}}
\put(-60.7,39.3){{$\circ$}}
\put(-40.7,39.3){{$\circ$}}
\put(-78,42){\line(1,-1){20}}
\put(-80.7,19.3){$\circ$}
\put(-60,20){$\scriptstyle\bullet$}
\put(-40.7,19.3){{$\circ$}}
\put(-80.7,-0.7){$\circ$}
\put(-60.7,-0.7){$\circ$}
\put(-40,0){{$\scriptstyle\bullet$}}
\put(-58,22){\line(1,-1){20}}
\put(38,42){\line(1,0){48}}
\put(38,22){\line(1,0){48}}
\put(38,2){\line(1,0){48}}
\put(39,40){{$\scriptstyle\times$}}
\put(59.3,39.3){{$\circ$}}
\put(79.3,39.3){{$\circ$}}
\put(42,22){\line(0,-1){20}}
\put(62,22){\line(1,-1){20}}
\put(40,20){$\scriptstyle\bullet$}
\put(60,20){$\scriptstyle\bullet$}
\put(79.3,19.3){{$\circ$}}
\put(40,0){$\scriptstyle\bullet$}
\put(59.3,-0.7){$\circ$}
\put(80,0){{$\scriptstyle\bullet$}}
\put(52,22){\oval(20,20)[t]}
\put(-22,42){\line(1,0){48}}
\put(-22,22){\line(1,0){48}}
\put(-22,2){\line(1,0){48}}
\put(22,42){\line(-1,0){20}}
\put(-21,40){{$\scriptstyle\times$}}
\put(-0.7,39.3){{$\circ$}}
\put(20,40){{$\scriptstyle\bullet$}}
\put(22,42){\line(0,-1){20}}
\put(-18,22){\line(0,-1){20}}
\put(-20,20){$\scriptstyle\bullet$}
\put(0,20){$\scriptstyle\bullet$}
\put(20,20){{$\scriptstyle\bullet$}}
\put(-20,0){$\scriptstyle\bullet$}
\put(-0.7,-0.7){$\circ$}
\put(19,0){{$\scriptstyle\times$}}
\put(-8,22){\oval(20,20)[t]}
\put(12,22){\oval(20,20)[b]}
\put(98,42){\line(1,0){48}}
\put(98,22){\line(1,0){48}}
\put(98,2){\line(1,0){48}}
%\put(122,42){\line(-1,-1){20}}
\put(100,40){{$\scriptstyle\bullet$}}
\put(119.3,39.3){{$\circ$}}
\put(140,40){{$\scriptstyle\bullet$}}
\put(142,42){\line(0,-1){20}}
\put(102,42){\line(1,-1){20}}
\put(99.3,19.3){$\circ$}
\put(120,20){$\scriptstyle\bullet$}
\put(140,20){{$\scriptstyle\bullet$}}
%\put(142,22){\line(-1,-1){20}}
\put(99.3,-0.7){$\circ$}
\put(119.3,-0.7){$\circ$}
\put(139,0){{$\scriptstyle\times$}}
\put(132,22){\oval(20,20)[b]}
\put(158,42){\line(1,0){48}}
\put(158,22){\line(1,0){48}}
\put(158,2){\line(1,0){48}}
\put(182,42){\line(-1,-1){20}}
\put(159,40){{$\scriptstyle\times$}}
\put(180,40){{$\scriptstyle\bullet$}}
\put(200,40){{$\scriptstyle\bullet$}}
\put(202,42){\line(0,-1){20}}
\put(162,22){\line(0,-1){20}}
\put(160,20){$\scriptstyle\bullet$}
\put(179,20){$\scriptstyle\times$}
\put(200,20){{$\scriptstyle\bullet$}}
\put(202,22){\line(-1,-1){20}}
\put(160,0){$\scriptstyle\bullet$}
\put(180,0){$\scriptstyle\bullet$}
\put(199,0){{$\scriptstyle\times$}}
\end{picture}
$$
\item[(iv)]
 $(i+1,i)$ is a $\Ga$-admissible sequence but
$(i,i+1)$ is not if and only exactly one of the following conditions holds:
\begin{itemize}
\item[(a)] the $(i+1)$th vertex of $\Ga$ is labelled $\times$, and moreover
if the $i$th and $(i+2)$th vertices
are labelled $\bullet$ then $\defect(\Ga) \geq 1$;
\item[(b)]
the $i$th, $(i+1)$th and $(i+2)$th vertices
are labelled $\bullet,\bullet$ and $\circ$, respectively,
and $\defect(\Ga) = 0$.
\end{itemize}
The associated composite matchings are as follows:
$$
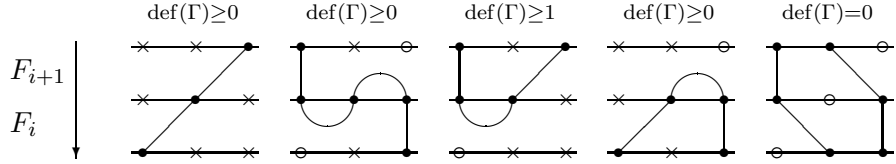
\begin{picture}(50,64)
\put(-76,55){$_{\defect(\Ga) \geq 0}$}
\put(-16,55){$_{\defect(\Ga) \geq 0}$}
\put(44,55){$_{\defect(\Ga) \geq 1}$}
\put(104,55){$_{\defect(\Ga) \geq 0}$}
\put(164,55){$_{\defect(\Ga) = 0}$}

\put(-103,44){\vector(0,-1){44}}
\put(-128,30){$F_{i+1}$}
\put(-128,10){$F_{i}$}
\put(-82,42){\line(1,0){48}}
\put(-82,22){\line(1,0){48}}
\put(-82,2){\line(1,0){48}}
\put(-81,40){{$\scriptstyle\times$}}
\put(-61,40){{$\scriptstyle\times$}}
\put(-40,40){{$\scriptstyle\bullet$}}
\put(-38,42){\line(-1,-1){20}}
\put(-81,20){$\scriptstyle\times$}
\put(-60,20){$\scriptstyle\bullet$}
\put(-41,20){{$\scriptstyle\times$}}
\put(-80,0){$\scriptstyle\bullet$}
\put(-61,0){$\scriptstyle\times$}
\put(-41,0){{$\scriptstyle\times$}}
\put(-58,22){\line(-1,-1){20}}
\put(98,42){\line(1,0){48}}
\put(98,22){\line(1,0){48}}
\put(98,2){\line(1,0){48}}
\put(99,40){{$\scriptstyle\times$}}
\put(119,40){{$\scriptstyle\times$}}
\put(139.3,39.3){{$\circ$}}
\put(142,22){\line(0,-1){20}}
\put(122,22){\line(-1,-1){20}}
\put(99,20){$\scriptstyle\times$}
\put(120,20){$\scriptstyle\bullet$}
\put(140,20){{$\scriptstyle\bullet$}}
\put(100,0){$\scriptstyle\bullet$}
\put(119,0){$\scriptstyle\times$}
\put(140,0){{$\scriptstyle\bullet$}}
\put(132,22){\oval(20,20)[t]}
\put(-22,42){\line(1,0){48}}
\put(-22,22){\line(1,0){48}}
\put(-22,2){\line(1,0){48}}
\put(22,42){\line(-1,0){20}}
\put(-20,40){{$\scriptstyle\bullet$}}
\put(-1,40){{$\scriptstyle\times$}}
\put(19.3,39.3){{$\circ$}}
\put(-18,42){\line(0,-1){20}}
\put(22,22){\line(0,-1){20}}
\put(-20,20){$\scriptstyle\bullet$}
\put(0,20){$\scriptstyle\bullet$}
\put(20,20){{$\scriptstyle\bullet$}}
\put(-20.7,-0.7){$\circ$}
\put(-1,0){$\scriptstyle\times$}
\put(20,0){{$\scriptstyle\bullet$}}
\put(-8,22){\oval(20,20)[b]}
\put(12,22){\oval(20,20)[t]}
\put(38,42){\line(1,0){48}}
\put(38,22){\line(1,0){48}}
\put(38,2){\line(1,0){48}}
%\put(122,42){\line(-1,-1){20}}
\put(40,40){{$\scriptstyle\bullet$}}
\put(59,40){{$\scriptstyle\times$}}
\put(80,40){{$\scriptstyle\bullet$}}
\put(42,42){\line(0,-1){20}}
\put(62,22){\line(1,1){20}}
\put(40,20){$\scriptstyle\bullet$}
\put(60,20){$\scriptstyle\bullet$}
\put(79,20){{$\scriptstyle\times$}}
\put(39.3,-0.7){$\circ$}
\put(59,0){$\scriptstyle\times$}
\put(79,0){{$\scriptstyle\times$}}
\put(52,22){\oval(20,20)[b]}
\put(158,42){\line(1,0){48}}
\put(158,22){\line(1,0){48}}
\put(158,2){\line(1,0){48}}
\put(182,42){\line(1,-1){20}}
\put(160,40){{$\scriptstyle\bullet$}}
\put(180,40){{$\scriptstyle\bullet$}}
\put(199.3,39.3){{$\circ$}}
\put(162,42){\line(0,-1){20}}
\put(202,22){\line(0,-1){20}}
\put(160,20){$\scriptstyle\bullet$}
\put(179.3,19.3){$\circ$}
\put(200,20){{$\scriptstyle\bullet$}}
\put(162,22){\line(1,-1){20}}
\put(159.3,-0.7){$\circ$}
\put(180,0){$\scriptstyle\bullet$}
\put(200,0){{$\scriptstyle\bullet$}}
\end{picture}
$$
\item[(v)] For $|i-j| > 1$,
$(i,j)$ and $(j,i)$ are both $\Ga$-admissible sequences if and only if
all of the following conditions hold:
\begin{itemize}
\item[(a)] the $i$th vertex of $\Ga$ is labelled $\bullet$ or $\times$,
the $(i+1)$th vertex of $\Ga$ is labelled $\circ$ or $\bullet$,
and if both are labelled $\bullet$ then $\defect(\Ga) \geq 1$;
\item[(b)] the $j$th vertex of $\Ga$ is labelled $\bullet$ or $\times$,
the $(j+1)$th vertex of $\Ga$ is labelled $\circ$ or $\bullet$,
and if both are labelled $\bullet$ then $\defect(\Ga) \geq 1$;
\item[(c)] if the $i$th, $(i+1)$th, $j$th and $(j+1)$th vertices
of $\Ga$ are all labelled $\bullet$
then $\defect(\Ga) \geq 2$.
\end{itemize}
In all cases, $(i,j)$ and $(j,i)$ are proper
$\Ga$-admissible sequences, and the associated composite matchings
have the same reductions and contain
no internal circles.
\item[(vi)]
For $|i-j| > 1$,
$(i,j)$ is a $\Ga$-admissible sequence but $(j,i)$ is not if and only if
 the $i$th, $(i+1)$th, $j$th and $(j+1)$th vertices
of $\Ga$ are labelled $\times,\circ,\bullet$ and $\bullet$,
respectively, and  $\defect(\Ga) = 0$.
In this case $(i,j)$ is not a proper $\Ga$-admissible sequence:
\end{itemize}
$$
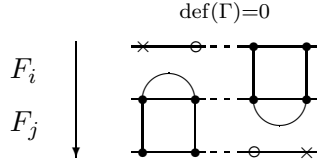
\begin{picture}(50,64)
\put(-64,55){$_{\defect(\Ga) = 0}$}

\put(-103,44){\vector(0,-1){44}}
\put(-128,30){$F_i$}
\put(-128,10){$F_{j}$}
\put(-82,42){\line(1,0){28}}
\put(-82,22){\line(1,0){28}}
\put(-82,2){\line(1,0){28}}
\dashline{2.5}(-52,42)(-40,42)
\dashline{2.5}(-52,2)(-40,2)
\dashline{2.5}(-52,22)(-40,22)
\put(-40,42){\line(1,0){28}}
\put(-40,22){\line(1,0){28}}
\put(-40,2){\line(1,0){28}}

\put(-38,40){{$\scriptstyle\bullet$}}
\put(-18,40){{$\scriptstyle\bullet$}}
\put(-38,20){{$\scriptstyle\bullet$}}
\put(-18,20){{$\scriptstyle\bullet$}}
\put(-80,0){{$\scriptstyle\bullet$}}
\put(-60,0){{$\scriptstyle\bullet$}}
\put(-80,20){{$\scriptstyle\bullet$}}
\put(-60,20){{$\scriptstyle\bullet$}}

\put(-36,42){\line(0,-1){20}}
\put(-16,42){\line(0,-1){20}}
\put(-78,22){\line(0,-1){20}}
\put(-58,22){\line(0,-1){20}}
\put(-38.3,-0.7){{$\circ$}}
\put(-60.3,39.3){{$\circ$}}
\put(-81,40){{$\scriptstyle\times$}}
\put(-19,0){{$\scriptstyle\times$}}
\put(-26,22){\oval(20,20)[b]}
\put(-68,22){\oval(20,20)[t]}
\end{picture}
$$
\end{Lemma}

\begin{proof}
This follows from the definitions.
\end{proof}

\phantomsubsection{\boldmath The natural transformations $y(i)$ and $\psi(ij)$}
In this subsection, we
construct natural transformations
\begin{align}\label{cold}
y(i) :\qquad\:\: F_i &\rightarrow F_i,\\
\psi(ij):\:\:F_j \circ F_i
&\rightarrow F_i \circ F_j,\label{bold}
\end{align}
for each $i,j \in I$.
To do this, it suffices by additivity to define natural transformations
\begin{align*}
y(i)_\Ga :\quad\qquad F_i|_{\Rep{K_\Ga}} &\rightarrow F_i|_{\Rep{K_\Ga}},\\
\psi(ij)_\Ga:\:(F_j \circ F_i)|_{\Rep{K_\Ga}}
&\rightarrow (F_i \circ F_j)|_{\Rep{K_\Ga}},
\end{align*}
for each block $\Ga \in \PImn$.

\subsubsection*{The definition of $y(i)_\Ga$}
If $i$ is not $\Ga$-admissible, then $F_i|_{\Rep{K_\Ga}}$ is the
zero functor, and we have to take
$y(i)_\Ga := 0$.
Now assume that $i$ is $\Ga$-admissible.
By the definition (\ref{gFi}),
$F_i|_{\Rep{K_\Ga}}$ is the functor
defined by tensoring with the bimodule
$K^{t}_{\De\Ga}\langle - \caps(t)\rangle$,
where $t := t_i(\Ga)$ and $\De := \Ga - \alpha_i$.
In the next paragraph, we define a bimodule endomorphism
\begin{equation}\label{bimd}
\overline y = \overline y(i)_\Ga:K^{t}_{\De\Ga}\langle-\caps(t)\rangle \rightarrow K^{t}_{\De\Ga}
\langle-\caps(t)\rangle.
\end{equation}
Given this, we let the desired natural transformation $y(i)_\Ga$
on $M \in \Rep{K_\Ga}$ be the
homomorphism
$y(i)_M:F_i M \rightarrow F_i M$
defined by
\begin{equation}\label{bimd2}
y(i)_M := (-1)^{(\Ga,\La_i)} \overline y(i)_\Ga \otimes \id_M.
\end{equation}
The sign $(-1)^{(\Ga,\La_i)}$ here may be computed in practise by counting the
number of vertices to the left or equal to the $i$th vertex that are labelled
$\bullet$
in the block diagram for $\Ga$.
We point out according to the definition in the next paragraph
that the bimodule endomorphism
$\overline y(i)_\Ga$ is homogeneous of degree two,
hence each $y(i)_M$ is of degree two as well.

It remains to define the bimodule endomorphism
$\overline y(i)_\Ga$, which we often refer to as
a {\em positive circle move} (``positive''
because it is of positive degree).
Recall that $K^t_{\De\Ga}\langle -\caps(t)\rangle$ has
a basis consisting of vectors $(a \de t \ga b)$ for
each oriented circle diagram $a \de t \ga b$ with $\de \in \De$
and $\ga \in \Ga$.
Define $\overline y(i)_\Ga$ to be the linear map
sending the basis vector $(a \de t \ga b)$
to $(a \de' t \ga' b)$
if the component of $t$ between vertices $i$ and $i+1$
lies on an anti-clockwise circle in $a \de t \ga b$,
or to $0$ otherwise; here $\de'$ and $\ga'$ are the weights obtained
on switching the orientation of this circle so that it becomes a
clockwise circle.

Unfortunately,
it is not obvious from the definition in the previous paragraph
that $\overline y(i)_\Ga$
is actually a bimodule homomorphism. To see this, we reinterpret the
map $\overline y(i)_\Ga$
in terms of the surgery procedure from \cite[$\S$3]{BS1}:
roughly, it is ``multiplication by a clockwise circle''
at the component of $t$ between vertices $i$ and $i+1$.
We explain precisely what we mean by this just in the case
that $t$ is a right-shift; the other three cases
from (\ref{CKLR}) are interpreted in a similar way.
In this case, the definition of $\overline y(i)_\Ga$
is summarised by the following
diagram (we display only the strip between vertices $i$ and $i+1$):
$$
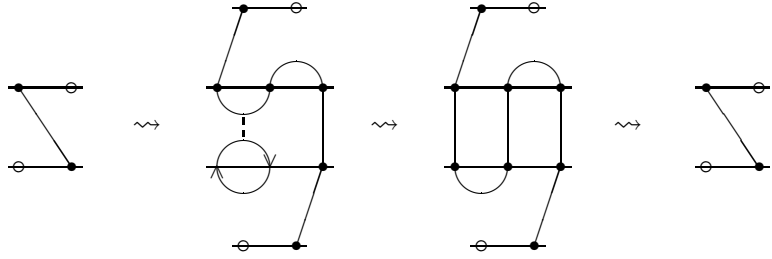
\begin{picture}(108,95)
\put(-85,59){\line(1,0){28}}
\put(-85,29){\line(1,0){28}}
\put(-83,57){{$\scriptstyle\bullet$}}
\put(-63,27){{$\scriptstyle\bullet$}}
\put(-83.7,26.3){{$\circ$}}
\put(-63.7,56.3){{$\circ$}}
\put(-81,59){\line(2,-3){20}}

\put(0,89){\line(1,0){28}}
\put(-10,59){\line(1,0){48}}
\put(-10,29){\line(1,0){48}}
\put(0,-1){\line(1,0){28}}
\put(24,-1){\line(1,3){10}}
\put(-6,59){\line(1,3){10}}
\put(2,87){{$\scriptstyle\bullet$}}
\put(-8,57){{$\scriptstyle\bullet$}}
\put(12,57){{$\scriptstyle\bullet$}}
\put(32,57){{$\scriptstyle\bullet$}}
\put(32,27){{$\scriptstyle\bullet$}}
\put(22,-3){{$\scriptstyle\bullet$}}
\put(1.3,-3.7){{$\circ$}}
\put(21.3,86.3){{$\circ$}}
\put(34,59){\line(0,-1){30}}
\put(4,59){\oval(20,20)[b]}
\put(24,59){\oval(20,20)[t]}
\put(4,29){\oval(20,20)[t]}
\put(4,29){\oval(20,20)[b]}
\put(-8.8,24.8){{$\scriptstyle\up$}}
\put(11.2,28.8){{$\scriptstyle\down$}}
\dashline{8.2}(4,40.5)(4,47.5)

\put(90,89){\line(1,0){28}}
\put(80,59){\line(1,0){48}}
\put(80,29){\line(1,0){48}}
\put(90,-1){\line(1,0){28}}
\put(114,-1){\line(1,3){10}}
\put(84,59){\line(1,3){10}}
\put(92,87){{$\scriptstyle\bullet$}}
\put(82,57){{$\scriptstyle\bullet$}}
\put(82,27){{$\scriptstyle\bullet$}}
\put(102,27){{$\scriptstyle\bullet$}}
\put(102,57){{$\scriptstyle\bullet$}}
\put(122,57){{$\scriptstyle\bullet$}}
\put(122,27){{$\scriptstyle\bullet$}}
\put(112,-3){{$\scriptstyle\bullet$}}
\put(91.3,-3.7){{$\circ$}}
\put(111.3,86.3){{$\circ$}}
\put(124,59){\line(0,-1){30}}
\put(84,59){\line(0,-1){30}}
\put(104,59){\line(0,-1){30}}
\put(114,59){\oval(20,20)[t]}
\put(94,29){\oval(20,20)[b]}

\put(-38,42){$\rightsquigarrow$}
\put(52,42){$\rightsquigarrow$}
\put(144,42){$\rightsquigarrow$}

\put(175,59){\line(1,0){28}}
\put(175,29){\line(1,0){28}}
\put(177,57){{$\scriptstyle\bullet$}}
\put(197,27){{$\scriptstyle\bullet$}}
\put(176.3,26.3){{$\circ$}}
\put(196.3,56.3){{$\circ$}}
\put(179,59){\line(2,-3){20}}
\end{picture}
$$
Formally, we take a basis vector $(a \de t \ga b) \in K^{t}_{\De\Ga}$ and
proceed as follows:
\begin{itemize}
\item
apply the closure operation from \cite[$\S$3]{BS2}
to convert $a \de t \ga b$
into a closed oriented circle diagram;
\item
extend the diagram by
inserting additional number lines and an internal clockwise circle
as indicated in the above diagram;
\item
apply the surgery procedure from \cite[$\S$3]{BS1}
at the position indicated by the dotted line in the
diagram;
\item reduce the result by removing the additional internal number lines;
\item finally apply the inverse of the closure operation
 to get back to an element of $K^{t}_{\De\Ga}$.
\end{itemize}
This procedure gives
the same linear map $\overline y(i)_\Ga$ as defined
in the previous paragraph.
Moreover the new description implies that $\overline y(i)_\Ga$
is a bimodule homomorphism,
because we know that
any sequence of surgery procedures produces the same result
independent of the order chosen. This is the same observation as
used to justify that the algebra multiplication is well defined and associative
in \cite[$\S$3]{BS1};
its proof involves reformulating the surgery procedure in the language of
TQFT's.

\subsubsection*{The definition of $\psi(ij)_\Ga$}
If $(i,j)$ (resp.\ $(j,i)$) is not a $\Ga$-admissible sequence
then the functor
$(F_j \circ F_i)|_{\Rep{K_\Ga}}$
(resp.\ $(F_i \circ F_j)|_{\Rep{K_\Ga}}$) is the zero functor,
in which case we have to take $\psi(ij)_\Ga := 0$.
Now assume that both $(i,j)$ and $(j,i)$ are $\Ga$-admissible sequences.
Lemma~\ref{caniso} gives us canonical isomorphisms
$$
c':
(F_j \circ F_i)|_{\Rep{K_\Ga}}
\stackrel{\sim}{\rightarrow}
G^{\bt}_{\bGa},\qquad
c:(F_i \circ F_j)|_{\Rep{K_\Ga}}
\stackrel{\sim}{\rightarrow}
G^{\bu}_{\bDe},
$$
where $\bGa$ and
$\bt$ (resp.\ $\bDe$ and $\bu$)
denote the block sequence and
composite matching
associated to $(i,j)$ (resp.\ $(j,i)$).
Recalling $G^{\bt}_{\bGa}$ (resp.\ $G^\bu_\bDe$) is the functor defined
by tensoring with the bimodule $K^\bt_\bGa\langle-\caps(\bt)\rangle$
(resp. $K^\bu_\bDe\langle-\caps(\bu)\rangle$),
the plan is to define a bimodule homomorphism
\begin{equation}\label{betadef}
\overline\psi = \overline\psi(ij)_\Ga:
K^{\bt}_{\bGa} \langle-\caps(\bt)\rangle
\rightarrow K^{\bu}_{\bDe} \langle -\caps(\bu)\rangle.
\end{equation}
Given this, we let
the desired natural transformation
$\psi(ij)_\Ga$ on $M \in \Rep{K_\Ga}$ be the homomorphism
$\psi(ij)_M:F_jF_iM \rightarrow F_iF_jM$ defined by
\begin{equation}\label{sff}
\psi(ij)_M :=
\left\{
\begin{array}{ll}
-(-1)^{(\Ga,\La_i)}
c_M^{-1}\circ (\overline\psi(ij)_\Ga \otimes \id_M) \circ c_M'
&\text{if $j=i$ or $i+1$,}\\
c_M^{-1} \circ (\overline\psi(ij)_\Ga \otimes \id_M) \circ c_M'
&\text{otherwise.}
\end{array}\right.
\end{equation}
To define
the bimodule homomorphism $\overline\psi(ij)_\Ga$,
 we split into cases according
to whether $i=j$, $|i-j|=1$
or $|i-j|>1$.
We will refer to $\overline\psi(ij)_\Ga$ in
these three cases as a
{\em negative circle move}, a {\em crossing move} or a
{\em height move}, respectively.
It will turn out that negative circle moves are homogeneous
of degree $-2$, crossing moves are homogeneous of degree $1$,
and height moves are homogeneous of degree $0$.

\vspace{1mm}

\noindent{\em Negative circle moves.} Suppose first that $i=j$.
By Lemma~\ref{as}(i)
the matching $\bt = \bu$ contains a small internal circle.
We define $\overline\psi(ii)_\Ga$ on a basis vector $(a\: \bt[\bga]\: b) \in
K^\bt_\bGa\langle -1 \rangle$ as follows.
If the internal circle in $a\: \bt [\bga] \:b$
is anti-clockwise then we map $(a\:\bt[\bga]\:b)$ to zero;
if it is clockwise then we map
$(a \: \bt[\bga]\:b)$ to the basis vector
$(a\:\bt[\bga']\:b)$ where $\bga'$ is
obtained by
switching the orientation of this circle to anti-clockwise.

\vspace{1mm}

\noindent{\em Crossing moves.}
Next suppose that $|i-j|=1$.
The eight possibilities for $\bt$ and the corresponding
possibilities for $\bu$ are listed in Lemma~\ref{as}(ii).
In all cases, the part of the matching displayed involves two
distinct components. Define $\overline\psi(ij)_\Ga$
by applying the surgery procedure to cut these two
components
in $\bt$ and rejoin them as in $\bu$.
We again make this precise just in one case,
in which the
definition of $\overline\psi(ij)_\Ga$
is as summarised by the following diagram:
$$
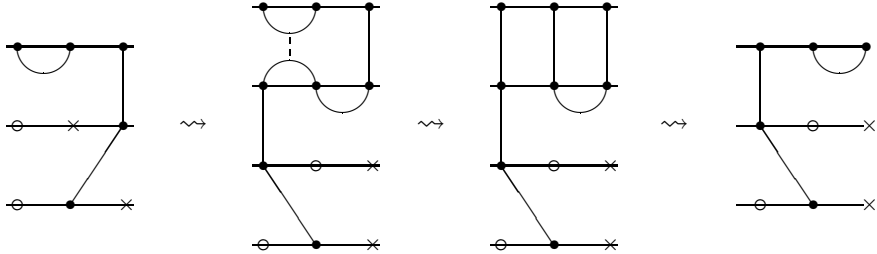
\begin{picture}(123,95)
\put(-103,14){\line(1,0){48}}
\put(-103,74){\line(1,0){48}}
\put(-103,44){\line(1,0){48}}
\put(-101,72){{$\scriptstyle\bullet$}}
\put(-81,72){{$\scriptstyle\bullet$}}
\put(-61,72){{$\scriptstyle\bullet$}}
\put(-61,42){{$\scriptstyle\bullet$}}
\put(-80.9,42.1){{$\scriptstyle\times$}}
\put(-60.9,12.1){{$\scriptstyle\times$}}
\put(-81,12){{$\scriptstyle\bullet$}}
\put(-101.7,11.3){{$\circ$}}
\put(-101.7,41.3){{$\circ$}}
\put(-79,14){\line(2,3){20}}
\put(-59,44){\line(0,1){30}}
\put(-89,74){\oval(20,20)[b]}

\put(-10,89){\line(1,0){48}}
\put(-10,59){\line(1,0){48}}
\put(-10,29){\line(1,0){48}}
\put(-10,-1){\line(1,0){48}}
\put(14,-1){\line(-2,3){20}}
\put(-8,87){{$\scriptstyle\bullet$}}
\put(12,87){{$\scriptstyle\bullet$}}
\put(32,87){{$\scriptstyle\bullet$}}
\put(-8,57){{$\scriptstyle\bullet$}}
\put(12,57){{$\scriptstyle\bullet$}}
\put(32,57){{$\scriptstyle\bullet$}}
\put(-8,27){{$\scriptstyle\bullet$}}
\put(12,-3){{$\scriptstyle\bullet$}}
\put(32.1,-2.9){{$\scriptstyle\times$}}
\put(-8.7,-3.7){{$\circ$}}
\put(32.1,27.1){{$\scriptstyle\times$}}
\put(11.3,26.3){{$\circ$}}
\put(34,59){\line(0,1){30}}
\put(-6,59){\line(0,-1){30}}
\put(4,59){\oval(20,20)[t]}
\put(24,59){\oval(20,20)[b]}
\put(4,89){\oval(20,20)[b]}
\dashline{8.2}(4,70.5)(4,77.5)

\put(80,89){\line(1,0){48}}
\put(80,59){\line(1,0){48}}
\put(80,29){\line(1,0){48}}
\put(80,-1){\line(1,0){48}}
\put(104,-1){\line(-2,3){20}}
\put(82,87){{$\scriptstyle\bullet$}}
\put(102,87){{$\scriptstyle\bullet$}}
\put(122,87){{$\scriptstyle\bullet$}}
\put(82,57){{$\scriptstyle\bullet$}}
\put(102,57){{$\scriptstyle\bullet$}}
\put(122,57){{$\scriptstyle\bullet$}}
\put(82,27){{$\scriptstyle\bullet$}}
\put(102,-3){{$\scriptstyle\bullet$}}
\put(122.1,-2.9){{$\scriptstyle\times$}}
\put(81.3,-3.7){{$\circ$}}
\put(122.1,27.1){{$\scriptstyle\times$}}
\put(101.3,26.3){{$\circ$}}
\put(124,59){\line(0,1){30}}
\put(84,59){\line(0,-1){30}}
\put(84,89){\line(0,-1){30}}
\put(104,89){\line(0,-1){30}}
\put(114,59){\oval(20,20)[b]}

\put(-38,42){$\rightsquigarrow$}
\put(52,42){$\rightsquigarrow$}
\put(144,42){$\rightsquigarrow$}

\put(173,74){\line(1,0){48}}
\put(173,44){\line(1,0){48}}
\put(173,14){\line(1,0){48}}
\put(202,14){\line(-2,3){20}}
\put(180,72){{$\scriptstyle\bullet$}}
\put(200,72){{$\scriptstyle\bullet$}}
\put(220,72){{$\scriptstyle\bullet$}}
\put(180,42){{$\scriptstyle\bullet$}}
\put(200,12){{$\scriptstyle\bullet$}}
\put(220.1,12.1){{$\scriptstyle\times$}}
\put(179.3,11.3){{$\circ$}}
\put(220.1,42.1){{$\scriptstyle\times$}}
\put(199.3,41.3){{$\circ$}}
\put(182,74){\line(0,-1){30}}
\put(212,74){\oval(20,20)[b]}
\end{picture}
$$
The interpretation of this follows
the same steps as in
the earlier definition of $\overline y(i)_\Ga$.
Note in this case that the map $\overline\psi(ij)_\Ga$
is homogeneous of degree 1:
the first map in the above diagram
is of degree 1 (as one additional clockwise cap or cup gets added),
the second map is of degree 0 (as the surgery procedure
 preserves the number of clockwise cups and caps),
and the final map is obviously of degree 0 too.
In fact in all eight of these cases, $\overline\psi(ij)_\Ga$
is homogeneous of degree 1.

\vspace{1mm}

\noindent{\em Height moves.}
Finally assume that $|i-j|>1$.
By Lemma~\ref{as}(v),
the composite matchings $\bt$ and $\bu$ are proper, have the same
reductions and no internal circles.
Hence we get an {\em isomorphism}
$\overline\psi(ij)_\Ga:
K^{\bt}_{\bGa}\langle -\caps(\bt)\rangle
\stackrel{\sim}{\rightarrow}
K^{\bu}_{\bDe}\langle-\caps(\bu)\rangle$
by composing two isomorphisms of the form (\ref{moso}).
In all these cases, $\overline\psi(ij)_\Ga$ is of degree 0.

\phantomsubsection{\boldmath
Composite natural transformations}
Suppose now we are given a $d$-tuple
$\bi = (i_1,\dots,i_d) \in I^d$, and recall the composite functor $F_\bi$
from (\ref{cf}).
We will often use the natural left action of the symmetric
group $S_d$ on $I^d$ by permuting the entries.
Let $e(\bi):F_\bi \rightarrow F_\bi$
denote the identity endomorphism of the
functor $F_\bi$.
Note that
$$
e(\bi) = e(i_d) \cdots e(i_1).
$$
The multiplication being used in this expression
is the ``horizontal'' composition of natural
transformations from \cite[$\S$II.5]{Maclane}; we reserve the notation
$\circ$ for the ``vertical'' composition
from \cite[$\S$II.4]{Maclane}.

The natural transformation $y(i):F_i \rightarrow F_i$ from (\ref{cold})
induces
\begin{equation}\label{ci}
y_r(\bi) := e(i_d) \cdots e(i_{r+1}) y(i_r) e(i_{r-1})\cdots e(i_1)
:F_\bi \rightarrow F_\bi
\end{equation}
for $r=1,\dots,d$.
In other words,
for a module $M$,
$y_r(\bi)_M:F_\bi M \rightarrow F_\bi M$
is the homomorphism
$F_{i_d} \cdots F_{i_{r+1}} y(i_r)_{F_{i_{r-1}}\cdots F_{i_1} M}$.

Similarly the natural transformation $\psi(ij) :F_j \circ F_i
\rightarrow F_i \circ F_j$ from (\ref{bold}) induces
\begin{equation}\label{bi}
\psi_r(\bi) := e(i_d) \cdots e(i_{r+2})
\psi(i_r i_{r+1}) e(i_{r-1}) \cdots e(i_1)
:F_\bi \rightarrow F_{s_r \cdot \bi}
\end{equation}
for $r=1,\dots,d-1$.
In other words, for a module $M$,
$\psi_r(\bi)_M:F_\bi M \rightarrow F_{s_r \cdot \bi} M$
is the homomorphism
$F_{i_d} \cdots F_{i_{r+2}} \psi(i_r i_{r+1})_{F_{i_{r-1}} \cdots F_{i_1} M}$.

Given a block $\Ga \in \PImn$, we also use the notation
$e(\bi)_\Ga, y_r(\bi)_\Ga$ and $\psi_r(\bi)_\Ga$ for the
natural transformations obtained from $e(\bi)$, $y_r(\bi)$
and $\psi_r(\bi)$ by restricting to $\Rep{K_\Ga}$.
For computational purposes,
it is important to have available a more concrete description
of $y_r(\bi)_\Ga$ and $\psi_r(\bi)_\Ga$ in terms of bimodule homomorphisms.

To explain this for $y_r(\bi)_\Ga$, we may assume that
$\bi$ is a $\Ga$-admissible sequence, so that according to
Lemma~\ref{caniso} there is a canonical isomorphism
$$
c:F_\bi|_{\Rep{K_\Ga}}
\stackrel{\sim}{\rightarrow} G^\bt_\bGa
$$
where $\bGa$ and $\bt$ are the block sequence and composite matching
associated to $\bi$.
We define a bimodule endomorphism
\begin{equation}\label{expy}
\overline y_r = \overline y_r(\bi)_\Ga:K^\bt_\bGa\langle-\caps(\bt)\rangle
\rightarrow K^\bt_\bGa\langle-\caps(\bt)\rangle
\end{equation}
exactly like in (\ref{bimd}), but
performing the positive circle move
to the component of $\bt$ that is non-trivial at level $r$.

\begin{Lemma}\label{inb}
With the above notation, we have that
\begin{equation*}
y_r(\bi)_M = (-1)^{(\Ga_{r-1}, \La_{i_r})}
c_M^{-1} \circ (\overline y_r(\bi)_\Ga \otimes \id_M) \circ c_M,
\end{equation*}
for any $M \in \Rep{K_\Ga}$.
\end{Lemma}

\begin{proof}
There is nothing to prove if $d=1$, as this is just the original
definition of $y(i)_M$ from (\ref{bimd2}) in this case.
If $d > 1$ then it suffices to show that
$$
c_M \circ y_r(\bi)_M =
(-1)^{(\Ga_{r-1}, \La_{i_r})}
(\overline y_r(\bi)_\Ga \otimes \id_M) \circ c_M.
$$
Recalling the definition of $c_M$ which goes back to (\ref{miso}),
both sides amount to applying the same two sequences of surgery procedures,
but in different orders, then multiplying by the same sign.
So they are equal because the order does not matter when applying
sequences of surgery procedures.
\end{proof}

The bimodule interpretation of $\psi_r(\bi)_\Ga$ is similar.
We may assume that both $\bi$ and $s_r \cdot \bi$ are $\Ga$-admissible
sequences, and let
$\bGa$ and
$\bt$ (resp.\ $\bDe$ and $\bu$)
denote the block sequence and
composite matching
associated to $\bi$ (resp.\ $s_r\cdot\bi$).
Lemma~\ref{caniso} gives us canonical isomorphisms
$$
c':
F_\bi|_{\Rep{K_\Ga}}
\stackrel{\sim}{\rightarrow}
G^{\bt}_{\bGa},\qquad
c:F_{s_r\cdot\bi}|_{\Rep{K_\Ga}}
\stackrel{\sim}{\rightarrow}
G^{\bu}_{\bDe}.
$$
We define a bimodule homomorphism
\begin{equation}\label{betadef2}
\overline\psi_r = \overline\psi_r(\bi)_\Ga:
K^{\bt}_{\bGa} \langle-\caps(\bt)\rangle
\rightarrow K^{\bu}_{\bDe} \langle -\caps(\bu)\rangle
\end{equation}
in similar fashion to (\ref{betadef}),
making either a negative circle move,
a crossing move or a height move
at levels $r$ and $(r+1)$
of the matching $\bt$ according to whether
$i_r = i_{r+1}$, $|i_r-i_{r+1}| = 1$ or $|i_r - i_{r+1}|> 1$.

\begin{Lemma}\label{sff2}
With the above notation, we have that
\begin{equation*}
\psi_r(\bi)_M =
\left\{
\begin{array}{ll}
-(-1)^{(\Ga_{r-1},\La_{i_r})}
c_M^{-1}\circ (\overline\psi_r(\bi)_\Ga \otimes \id_M) \circ c_M'
&\text{if $i_{r+1}=i_r$ or $i_r+1$,}\\
c_M^{-1} \circ (\overline\psi_r(\bi)_\Ga \otimes \id_M) \circ c_M'
&\text{otherwise,}
\end{array}\right.
\end{equation*}
for any $M \in \Rep{K_\Ga}$.
\end{Lemma}

\begin{proof}
This follows by a similar argument to the proof of Lemma~\ref{inb}.
\end{proof}

\phantomsubsection{\boldmath Local relations}
Now we can prove the following key result, which
describes the relations that hold betwen
the natural transformations $y(i)_\Ga$ and $\psi(ij)_\Ga$.

\begin{Theorem}\label{localrels}
The following hold
for $\Ga \in \PImn$ and $i,j,k \in I$.
\begin{itemize}
\item[(i)]
\begin{itemize}
\item[(a)]
$y(i)_\Ga\circ y(i)_\Ga = 0$;
\item[(b)]
$y(i)_\Ga = 0$
if $\defect(\Ga) = 0$
and either the $i$th or the $(i+1)$th vertex of $\Ga$ is labelled $\bullet$.
\end{itemize}
\item[(ii)]
\begin{itemize}
\item[(a)] $\psi(ii)_\Ga\circ\psi(ii)_\Ga =0$;
\item[(b)] $\psi(ii)_\Ga \circ y_2(ii)_\Ga = y_1(ii)_\Ga\circ \psi(ii)_\Ga +
e(ii)_\Ga$;
\item[(c)]
$y_2(ii)_\Ga\circ \psi(ii)_\Ga = \psi(ii)_\Ga\circ y_1(ii)_\Ga + e(ii)_\Ga$;
\item[(d)] $y_1(ii)_\Ga + y_2(ii)_\Ga = 0$.
\end{itemize}
\item[(iii)] If $i \neq j$ then
\begin{itemize}
\item[(a)] $\psi(ij)_\Ga \circ y_2(ij)_\Ga = y_1(ji)_\Ga \circ \psi(ij)_\Ga$;
\item[(b)] $y_2(ji)_\Ga\circ\psi(ij)_\Ga  = \psi(ij)_\Ga\circ y_1(ij)_\Ga $.
\end{itemize}
\item[(iv)] If $|i-j|> 1$ then $\psi(ji)_\Ga\circ \psi(ij)_\Ga = e(ij)_\Ga$.
\item[(v)] If $|i-j|=1$ and the $max(i,j)$th vertex of $\Ga$ is
$\bullet$ then
\begin{itemize}
\item[(a)] $\psi(ji)_\Ga \circ\psi(ij)_\Ga =
(i-j)y_1(ij)_\Ga + (j-i)y_2(ij)_\Ga$;
\item[(b)] $\psi(ij)_\Ga\circ y_1(ij)_\Ga + \psi(ij)_\Ga\circ y_2(ij)_\Ga = 0$;
\item[(c)]
$\psi_1(jii)_\Ga\circ \psi_2(jii)_\Ga\circ \psi_1(iji)_\Ga =
(j-i)e(iji)_\Ga$;
\item[(d)]
$\psi_2(iij)_\Ga\circ \psi_1(iij)_\Ga\circ \psi_2(iji)_\Ga = 0$.
\end{itemize}
\item[(vi)] If $|i-j|=1$ and the $max(i,j)$th vertex of $\Ga$ is
$\circ$ or $\times$ then
\begin{itemize}
\item[(a)] $\psi(ji)_\Ga \circ \psi(ij)_\Ga = 0$;
\item[(b)] $y_1(ij)_\Ga = y_2(ij)_\Ga$.
\item[(c)]
$\psi_1(jii)_\Ga\circ \psi_2(jii)_\Ga\circ \psi_1(iji)_\Ga = 0$;
\item[(d)]
$\psi_2(iij)_\Ga \circ\psi_1(iij)_\Ga\circ \psi_2(iji)_\Ga = (i-j) e(iji)_\Ga$.
\end{itemize}
\item[(vii)]
$\psi_1(jki)_\Ga \circ\psi_2(jik)_\Ga\circ \psi_1(ijk)_\Ga
= \psi_2(kij)_\Ga \circ\psi_1(ikj)_\Ga \circ
\psi_2(ijk)_\Ga$
either if $i \neq k$ or if $|i-j| \neq 1$.
\end{itemize}
\end{Theorem}

\begin{proof}
In all cases,
the strategy is to translate into a statement about
bimodule homomorphisms using Lemmas \ref{inb} and \ref{sff2}, then to
verify that statement by direct computations with the diagram bases.
To get started, consider (i). We trivially have that $y(i)_\Ga = 0$ unless
$i$ is $\Ga$-admissible. So assume that $i$ is $\Ga$-admissible.
By the definition of the bimodule homomorphism
(\ref{bimd}), it is clear that $\overline y(i)_\Ga^2= 0$,
and moreover $\overline y(i)_\Ga = 0$ if
$\defect(\Ga) = 0$ and either the $i$th or $(i+1)$th vertex of
$\Ga$ is labelled $\bullet$ (see the last two diagrams from (\ref{CKLR})).
In view of (\ref{bimd2}), this implies the desired statement about
the natural transformation $y(i)_\Ga$.

Next consider (ii).
The desired relations are all trivially true if $(i,i)$ is not a $\Ga$-admissible sequence. So assume that $(i,i)$ is $\Ga$-admissible. Then we are in the situation of Lemma~\ref{as}(i). By (\ref{moso2}),
the functor
$(F_i \circ F_i)|_{\Rep{K_\Ga}}$ can be identified with the functor
defined by tensoring with the
vector space $R$ (in which $1$ corresponds to an anti-clockwise circle
and $x$ corresponds to a clockwise circle). The bimodule endomorphisms
$\overline y_1(ii)_\Ga$ and $\overline y_2(ii)_\Ga$
are both equal to the same positive circle move
coming from the map $R \rightarrow R, 1 \mapsto x, x \mapsto 0$,
and the endomorphism $\overline\psi(ii)_\Ga$ is the negative
circle move
coming from the map
$R \rightarrow R, x \mapsto 1, 1 \mapsto 0$. Using this it is trivial to check that
$$
\overline\psi(ii)_\Ga\circ\overline\psi(ii)_\Ga = 0,\:
\overline\psi(ii)_\Ga \circ\overline y_1(ii)_\Ga + \overline
y_2(ii)_\Ga \circ\overline\psi(ii)_\Ga = e(ii)_\Ga,
\:
\overline y_1(ii)_\Ga = \overline y_2(ii)_\Ga.
$$
Incorporating the signs from Lemma~\ref{inb} and (\ref{sff}),
these equations imply
the desired identities (a), (c) and (d). Then (b) follows from
(c) and (d).

For (iii), we may assume that both $(i,j)$ and $(j,i)$ are $\La$-admissible sequences,
as both sides of the desired identities are trivially zero if they are not.
Adopting the same notation as in (\ref{betadef}),
and noting that
the additional signs coming from (\ref{sff}) and Lemma~\ref{inb}
are the same on both sides,
it suffices to check that
$$
\overline\psi(ij)_\Ga \circ \overline y_2(ij)_\Ga
=
\overline y_1(ji)_\Ga \circ \overline\psi(ij)_\Ga,
\qquad
\overline\psi(ij)_\Ga \circ \overline y_2(ij)_\Ga
=
\overline y_1(ji)_\Ga \circ \overline\psi(ij)_\Ga,
$$
as bimodule homomorphisms from
$K^\bt_\bGa\langle-\caps(\bt)\rangle$ to
$K^\bu_\bDe\langle-\caps(\bu)\rangle$.
These identities are obvious if $|i-j| > 1$.
If $|i-j|=1$ then
the possibilities for $\bt$ and $\bu$ are listed in Lemma~\ref{as}(ii).
In (the closure of) a diagram basis vector from
$K^{\bt}_{\bGa}\langle-\caps(\bt)\rangle$
(resp.\ $K^\bu_\bDe\langle-\caps(\bu)\rangle$)
the two components of $\bt$ (resp.\ $\bu$) from the
diagrams in Lemma~\ref{as}(ii)
could either be joined
into one component in the big picture
or they could remain as
two distinct components in the big picture.
In the former case we denote the basis vector by $1$ or $x$
according to whether this single component is anti-clockwise or clockwise;
in the latter case we denote the basis vector by
$1 \otimes 1$, $1 \otimes x$, $x \otimes 1$ or $x \otimes x$
according to the orientations of the two components.
With this notation, we can
represent our bimodule homomorphisms as
\begin{align*}
\overline\psi&:1 \mapsto 1 \otimes x + x \otimes 1, \:x \mapsto x\otimes x,\\
\overline y_1&: 1 \mapsto x, \:x \mapsto 0,\\
\overline y_2&: 1 \mapsto x, \:x \mapsto 0.\\\intertext{in the one component case, or}
\overline\psi&:1 \otimes 1 \mapsto 1,\: x \otimes 1 \mapsto x,\: 1 \otimes x \mapsto x,\:
x \otimes x \mapsto 0,\\
\overline y_1&:1 \otimes 1 \mapsto x \otimes 1, \:
1 \otimes x \mapsto x \otimes x,
\:x \otimes 1 \mapsto 0, \:x \otimes x \mapsto 0,\\
\overline y_2&:1 \otimes 1 \mapsto 1 \otimes x, \:
1 \otimes x \mapsto 0,\:
x \otimes 1 \mapsto x\otimes x,\: x \otimes x \mapsto 0
\end{align*}
in the two component case.
Now it is easy to check that
$\overline\psi \circ\overline y_1 = \overline y_2 \circ\overline\psi$
and
$\overline\psi \circ\overline y_2 = \overline y_1 \circ\overline\psi$,
as required.

For (iv), we may assume $(i,j)$ is a proper $\Ga$-admissible sequence.
In view of Lemma~\ref{as}(vi), we get that $(j,i)$ is admissible too, and we are
in the situation of Lemma~\ref{as}(v). By the definition (\ref{betadef}),
the height moves
$\overline\psi(ij)_\Ga$ and $\overline\psi(ji)_\Ga$ are inverses of each other,
and there are no additional signs, so we are done.

Next consider (v)(a), (b).
We may assume that $(i,j)$ is a $\Ga$-admissible sequence.
Then we are either in the situation of Lemma~\ref{as}(ii) or the fifth
cases from Lemma~\ref{as}(iii),(iv). In the latter two cases all of
$\psi(ij)_\Ga, y_1(ij)_\Ga$ and $y_2(ij)_\Ga$ are zero, so we are done.
In the former case, $(j,i)$ is also $\Ga$-admissible and,
noting that $(\Ga,\La_i) = -(\Ga,\La_j)$, we reduce to checking the
following identities
$$
\overline\psi(ji)_\Ga \circ \overline\psi(ij)_\Ga = \overline y_1(ij)_\Ga+
\overline y_2(ij)_\Ga,\qquad
\overline\psi(ij)_\Ga\circ \overline y_1(ij)_\Ga = \overline\psi(ij)_\Ga
\circ \overline y_2(ij)_\Ga
$$
at the level of bimodules.
This is easy to do using the formulae
for $\overline\psi, \overline y_1$ and $\overline y_2$
from the previous paragraph,
considering the one component and two component
cases separately.

Next consider (v)(c), (d). We explain just the situation
when $j=i+1$, the case $i=j+1$ being entirely similar.
We may assume that $(i,j,i)$ is a
proper $\Ga$-admissible
sequence, and deduce by Lemma~\ref{as}
that there are only two possibilities for the associated matching
$\bt$. The two possibilities are as displayed on the left hand side of the
following (the second possibility arises only for $\defect(\Ga) > 0$):
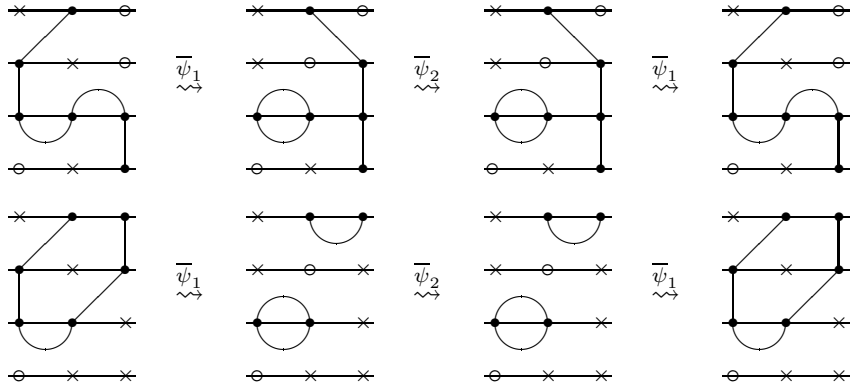
$$
\begin{picture}(-40,72)
\put(-177,62){\line(1,0){48}}
\put(-177,42){\line(1,0){48}}
\put(-177,22){\line(1,0){48}}
\put(-177,2){\line(1,0){48}}
\put(-153,62){\line(-1,-1){20}}
\put(-176,60){{$\scriptstyle\times$}}
\put(-155,60){{$\scriptstyle\bullet$}}
\put(-135.7,59.3){{$\circ$}}
\put(-173,42){\line(0,-1){20}}
\put(-175,40){$\scriptstyle\bullet$}
\put(-156,40){$\scriptstyle\times$}
\put(-156,0){$\scriptstyle\times$}
\put(-175.7,-0.7){$\circ$}
\put(-135.7,39.3){{$\circ$}}
\put(-175,20){$\scriptstyle\bullet$}
\put(-155,20){$\scriptstyle\bullet$}
\put(-135,20){{$\scriptstyle\bullet$}}
\put(-135,0){{$\scriptstyle\bullet$}}
\put(-143,22){\oval(20,20)[t]}
\put(-163,22){\oval(20,20)[b]}
\put(-133,22){\line(0,-1){20}}
\put(66,30){$\stackrel{\overline\psi_1}{\rightsquigarrow}$}
\put(-24,30){$\stackrel{\overline\psi_2}{\rightsquigarrow}$}
\put(-114,30){$\stackrel{\overline\psi_1}{\rightsquigarrow}$}

\put(-87,62){\line(1,0){48}}
\put(-87,42){\line(1,0){48}}
\put(-87,22){\line(1,0){48}}
\put(-87,2){\line(1,0){48}}
\put(-63,62){\line(1,-1){20}}
\put(-86,60){{$\scriptstyle\times$}}
\put(-65,60){{$\scriptstyle\bullet$}}
\put(-45.7,59.3){{$\circ$}}
\put(-43,42){\line(0,-1){20}}
\put(-45,40){$\scriptstyle\bullet$}
\put(-86,40){$\scriptstyle\times$}
\put(-66,0){$\scriptstyle\times$}
\put(-85.7,-0.7){$\circ$}
\put(-65.7,39.3){{$\circ$}}
\put(-85,20){$\scriptstyle\bullet$}
\put(-65,20){$\scriptstyle\bullet$}
\put(-45,20){{$\scriptstyle\bullet$}}
\put(-45,0){{$\scriptstyle\bullet$}}
\put(-73,22){\oval(20,20)[t]}
\put(-73,22){\oval(20,20)[b]}
\put(-43,22){\line(0,-1){20}}

\put(3,62){\line(1,0){48}}
\put(3,42){\line(1,0){48}}
\put(3,22){\line(1,0){48}}
\put(3,2){\line(1,0){48}}
\put(27,62){\line(1,-1){20}}
\put(4,60){{$\scriptstyle\times$}}
\put(25,60){{$\scriptstyle\bullet$}}
\put(44.3,59.3){{$\circ$}}
\put(47,42){\line(0,-1){20}}
\put(45,40){$\scriptstyle\bullet$}
\put(4,40){$\scriptstyle\times$}
\put(24,0){$\scriptstyle\times$}
\put(3.3,-0.7){$\circ$}
\put(23.3,39.3){{$\circ$}}
\put(5,20){$\scriptstyle\bullet$}
\put(25,20){$\scriptstyle\bullet$}
\put(45,20){{$\scriptstyle\bullet$}}
\put(45,0){{$\scriptstyle\bullet$}}
\put(17,22){\oval(20,20)[t]}
\put(17,22){\oval(20,20)[b]}
\put(47,22){\line(0,-1){20}}

\put(93,62){\line(1,0){48}}
\put(93,42){\line(1,0){48}}
\put(93,22){\line(1,0){48}}
\put(93,2){\line(1,0){48}}
\put(117,62){\line(-1,-1){20}}
\put(94,60){{$\scriptstyle\times$}}
\put(115,60){{$\scriptstyle\bullet$}}
\put(134.3,59.3){{$\circ$}}
\put(97,42){\line(0,-1){20}}
\put(95,40){$\scriptstyle\bullet$}
\put(114,40){$\scriptstyle\times$}
\put(114,0){$\scriptstyle\times$}
\put(94.3,-0.7){$\circ$}
\put(134.3,39.3){{$\circ$}}
\put(95,20){$\scriptstyle\bullet$}
\put(115,20){$\scriptstyle\bullet$}
\put(135,20){{$\scriptstyle\bullet$}}
\put(135,0){{$\scriptstyle\bullet$}}
\put(127,22){\oval(20,20)[t]}
\put(107,22){\oval(20,20)[b]}
\put(137,22){\line(0,-1){20}}
\end{picture}
$$
$$
\begin{picture}(-40,75)
\put(66,30){$\stackrel{\overline\psi_1}{\rightsquigarrow}$}
\put(-24,30){$\stackrel{\overline\psi_2}{\rightsquigarrow}$}
\put(-114,30){$\stackrel{\overline\psi_1}{\rightsquigarrow}$}
\put(93,62){\line(1,0){48}}
\put(93,42){\line(1,0){48}}
\put(93,22){\line(1,0){48}}
\put(93,2){\line(1,0){48}}
\put(117,62){\line(-1,-1){20}}
\put(137,42){\line(-1,-1){20}}
\put(94,60){{$\scriptstyle\times$}}
\put(115,60){{$\scriptstyle\bullet$}}
\put(97,42){\line(0,-1){20}}
\put(137,62){\line(0,-1){20}}
\put(95,40){$\scriptstyle\bullet$}
\put(135,40){$\scriptstyle\bullet$}
\put(135,60){$\scriptstyle\bullet$}
\put(134,20){$\scriptstyle\times$}
\put(114,40){$\scriptstyle\times$}
\put(114,0){$\scriptstyle\times$}
\put(134,0){$\scriptstyle\times$}
\put(94.3,-0.7){$\circ$}
\put(95,20){$\scriptstyle\bullet$}
\put(115,20){$\scriptstyle\bullet$}
\put(107,22){\oval(20,20)[b]}

\put(3,62){\line(1,0){48}}
\put(3,42){\line(1,0){48}}
\put(3,22){\line(1,0){48}}
\put(3,2){\line(1,0){48}}
\put(4,60){{$\scriptstyle\times$}}
\put(4,40){{$\scriptstyle\times$}}
\put(44,40){{$\scriptstyle\times$}}
\put(25,60){{$\scriptstyle\bullet$}}
\put(45,60){$\scriptstyle\bullet$}
\put(44,20){$\scriptstyle\times$}
\put(24,0){$\scriptstyle\times$}
\put(44,0){$\scriptstyle\times$}
\put(4.3,-0.7){$\circ$}
\put(24.3,39.3){$\circ$}
\put(5,20){$\scriptstyle\bullet$}
\put(25,20){$\scriptstyle\bullet$}
\put(17,22){\oval(20,20)[b]}
\put(17,22){\oval(20,20)[t]}
\put(37,62){\oval(20,20)[b]}

\put(-87,62){\line(1,0){48}}
\put(-87,42){\line(1,0){48}}
\put(-87,22){\line(1,0){48}}
\put(-87,2){\line(1,0){48}}
\put(-86,60){{$\scriptstyle\times$}}
\put(-86,40){{$\scriptstyle\times$}}
\put(-46,40){{$\scriptstyle\times$}}
\put(-65,60){{$\scriptstyle\bullet$}}
\put(-45,60){$\scriptstyle\bullet$}
\put(-46,20){$\scriptstyle\times$}
\put(-66,0){$\scriptstyle\times$}
\put(-46,0){$\scriptstyle\times$}
\put(-85.7,-0.7){$\circ$}
\put(-65.7,39.3){$\circ$}
\put(-85,20){$\scriptstyle\bullet$}
\put(-65,20){$\scriptstyle\bullet$}
\put(-73,22){\oval(20,20)[b]}
\put(-73,22){\oval(20,20)[t]}
\put(-53,62){\oval(20,20)[b]}

\put(-177,62){\line(1,0){48}}
\put(-177,42){\line(1,0){48}}
\put(-177,22){\line(1,0){48}}
\put(-177,2){\line(1,0){48}}
\put(-153,62){\line(-1,-1){20}}
\put(-133,42){\line(-1,-1){20}}
\put(-176,60){{$\scriptstyle\times$}}
\put(-155,60){{$\scriptstyle\bullet$}}
\put(-173,42){\line(0,-1){20}}
\put(-133,62){\line(0,-1){20}}
\put(-175,40){$\scriptstyle\bullet$}
\put(-135,40){$\scriptstyle\bullet$}
\put(-135,60){$\scriptstyle\bullet$}
\put(-136,20){$\scriptstyle\times$}
\put(-156,40){$\scriptstyle\times$}
\put(-156,0){$\scriptstyle\times$}
\put(-136,0){$\scriptstyle\times$}
\put(-175.7,-0.7){$\circ$}
\put(-175,20){$\scriptstyle\bullet$}
\put(-155,20){$\scriptstyle\bullet$}
\put(-163,22){\oval(20,20)[b]}
\end{picture}
$$
In either case it is clear that $(i,i,j)$ is not $\Ga$-admissible,
hence $\psi_2(iji)_\Ga = 0$ proving (d).
For (c), the usual consideration of signs reduces to checking
at the level of bimodule homomorphisms that
$\overline\psi_1 \circ \overline\psi_2\circ \overline\psi_1 = 1$;
we have indicated the maps involved in the above diagram.
To check this identity, we write $1$ for
an anti-clockwise circle and $x$ for a clockwise circle
 as usual, and get that
the composition $\overline\psi_1 \circ \overline\psi_2 \circ \overline\psi_1$ is
$$
1 \mapsto 1 \otimes x + x \otimes 1 \mapsto 1 \otimes 1 \mapsto 1,\qquad
x \mapsto x \otimes x \mapsto 1 \otimes x \mapsto x,
$$
i.e. it is the identity map.

The proof of (vi) is similar to (v) so we omit it.
Finally (vii) can be checked by analogous techniques, though it is
somewhat more lengthy.
We omit the details.
\end{proof}

\begin{Remark}\rm
The functors $F_i$ and $E_i$ for $i \in I$ together with the
natural transformations (\ref{cold})--(\ref{bold})
and the adjunction $(F_i, E_i \circ D_i^{-1} D_{i+1} \langle 1\rangle)$
from Lemma~\ref{adjunctions}
make the graded abelian category
$\Rep{K(m,n;I)}$ into an integrable representation
of the $2$-Kac-Moody algebra
$\mathfrak{A}(\mathfrak{sl}_{I^+})$
in the sense of \cite{Rou} (except that we have interchanged the roles of
$E_i$ and $F_i$).
This is easy to verify using Theorem~\ref{localrels},
Theorem~\ref{cat1}(ii),
Lemma~\ref{adjunctions} and \cite[Theorem 5.27]{Rou}.
We will not pursue this connection
further here.
\end{Remark}

\section{Homogeneous Schur-Weyl duality and a graded cellular basis}\label{sD}

It is time to give a rough sketch of the strategy
 in the remainder of the article.
For each block $\Ga = \La-\alpha \in \PImn$,
we will define modules $T^\La_\alpha \in \Rep{K_\Ga}$ on the diagram
algebra side
and $\mathcal T^\La_\alpha \in \mathcal O_\Ga$ on the category $\mathcal O$ side,
both of which are built using the respective
special projective functors;
see (\ref{dc}) and (\ref{cT}).
These modules both satisfy a
{\em double centraliser property} which
ensures that
the categories $\Rep{K_\Ga}$ and $\mathcal O_\Ga$
can be reconstructed from the
endomorphism algebras
$\End_{K_\Ga}(T^\La_\alpha)^{\op}$
and $\End_{\mathfrak{g}}(\cT^\La_\alpha)^{\op}$,
respectively; see Corollaries~\ref{mapj} and \ref{mapi}.
As we will explain in detail later on, this
reduces the problem of proving
our two categories are equivalent to showing that
$$
\End_{K_\Ga}(T^\La_\alpha)^{\op}
\cong
\End_{\mathfrak{g}}(\cT^\La_\alpha)^{\op}
$$
as algebras.
The right hand endomorphism algebra (from the category $\mathcal O$ side)
is already well understood thanks to a special case of the
{\em Schur-Weyl duality for higher levels} from \cite{BKschur}: it is
a certain block of a degenerate cyclotomic Hecke algebra of level two.

In this section we are going to focus instead on the left hand endomorphism
algebra (i.e. from the diagram algebra side), which we denote by $E^\La_\alpha$.
The most important result of the section gives an explicit graded
cellular basis for $E^\La_\alpha$ parametrised by certain
diagrams called {\em oriented stretched circle diagrams}; see Theorem~\ref{cell}.
We loosely refer to all this
as {\em homogeneous Schur-Weyl duality}, as it gives rise to a
naturally graded analogue on the diagram algebra side
of the Schur-Weyl duality from \cite{BKschur}.

The other important ingredient in the section is the definition of another
algebra $R^\La_\alpha$ by certain generators and relations originating in
\cite{KLa, Rou}, which is already known by \cite{BKyoung} to be isomorphic
to the Hecke algebra block mentioned in the opening paragraph; see
Theorem~\ref{id2}.
As an application of the local relations from Theorem~\ref{localrels},
we will
construct a homomorphism
$$
\omega:R^\La_\alpha \rightarrow E^\La_\alpha;
$$
see Theorem~\ref{phi}.
This homomorphism will eventually turn out to be an isomorphism;
see Corollary~\ref{id3}. Hence we will have proved that the above two
endomorphism algebras are isomorphic, and the all-important
link between the two sides
will be made.

\phantomsubsection{\boldmath The prinjective generator $T^\La_\alpha$}
Given $\alpha \in Q_+$ of height $d$,
let
\begin{equation*}\label{mind}
I^\alpha := \{\bi = (i_1,\dots,i_d) \in I^d\:|\:\alpha_{i_1}+\cdots+\alpha_{i_d} = \alpha\}.
\end{equation*}
This is a single orbit under the action of the symmetric group $S_d$ on
$I^d$.
Recall also the special block $\La \in \PImn$ fixed in (\ref{gsw});
it is a block of defect zero containing
only one weight $\iota$, namely,
the ground-state from (\ref{groundstate}).
Moreover every block $\Ga \in \PImn$ can be written as
$\Ga = \La-\alpha$ for some $\alpha \in Q_+$.
It is convenient at this point to set
\begin{equation*}\label{altk}
K^\La_\alpha := \left\{
\begin{array}{ll}
K_{\La-\alpha}&\text{if $\La - \alpha \in \PImn$,}\\
{\mathbf 0}\:\text{(the zero algebra)}&\text{if $\La - \alpha \notin \PImn$.}
\end{array}\right.
\end{equation*}
Let
\begin{equation}\label{dc}
T^\La_\alpha := \bigoplus_{\bi \in I^\alpha} F_\bi L(\iota) \langle \defect(\La-\alpha)\rangle.
\end{equation}
If $\La - \alpha \notin\PImn$, there are no $\La$-admissible
sequences $\bi \in I^\alpha$, so $T^\La_\alpha$ is the zero module.
If $\La - \alpha \in \PImn$ then
$T^\La_\alpha$ is a $K_{\La-\alpha}$-module.
Hence, in all cases, it makes sense to regard
$T^\La_\alpha$ as a $K^\La_\alpha$-module.

To justify the importance of this module, recall a prinjective
module is a module that is both projective and injective.
A {\em prinjective generator} for $\Rep{K^\La_\alpha}$
means a finite dimensional
prinjective $K^\La_\alpha$-module
such that
every prinjective indecomposable module from Lemma~\ref{princ1}
appears as a summand (possibly shifted in degree).

\begin{Lemma}\label{pring1}
The module $T^\La_\alpha$ is a prinjective generator for $\Rep{K^\La_\alpha}$.
It is non-zero if and only if $\La-\alpha \in \PImn$.
\end{Lemma}

\begin{proof}
The fact that $T^\La_\alpha$ is a prinjective module follows easily
because $L(\iota)$ is and
projective functors send prinjectives
to prinjectives.
To see it is a generator, we may assume that
$\Ga := \La-\alpha \in \PImn$ and
take any $\la \in \Ga^\circ$.
By Lemma~\ref{stupidcomb}
we can write $\lambda = \tilde f_{i_d} \cdots \tilde f_{i_1} (\iota)$.
Applying Lemma~\ref{ga}, we deduce that
$L(\la)\langle j \rangle$ appears in the head
of $F_{\bi} L(\iota)$
for $\bi = (i_1,\dots,i_d)$ and some $j \in \Z$.
Hence $P(\la)\langle k\rangle$
is a summand of $T^\La_\alpha$ for some $k \in \Z$.
For the final statement about the non-zeroness of $T^\La_\alpha$,
it just remains to observe that
$\Ga^\circ$ is non-empty for $\Ga \in \PImn$.
\end{proof}

\phantomsubsection{\boldmath The endomorphism algebra $E^\La_\alpha$}
We are interested in the remainder of this section in the
(naturally graded) endomorphism algebra
\begin{equation}\label{Thend}
E^\La_\alpha :=
\End_{K^{\La}_\alpha}(T^\La_\alpha)^{\op}.
\end{equation}
The $\op$ here indicates that we view $T^\La_\alpha$
as a {\em right} $E^\La_\alpha$-module, i.e. it is
a graded $(K^\La_\alpha, E^\La_\alpha)$-bimodule.

The algebra $E^\La_\alpha$ is Morita equivalent to a
generalised Khovanov algebra. To explain this precisely,
introduce the idempotent
\begin{equation}\label{truncid}
e :=
\left\{\begin{array}{ll}
\sum_{\la \in (\La-\alpha)^\circ} e_\la
&\text{if $\La-\alpha \in \PImn$,}\\
0&\text{if $\La-\alpha \notin \PImn$.}
\end{array}
\right.
\end{equation}
As in
\cite[(6.8)]{BS1},
the {\em generalised Khovanov algebra}
is the
subalgebra
\begin{equation}\label{genk}
H^\La_\alpha := e K^\La_\alpha e.
\end{equation}
Moreover by \cite[Corollary 6.3]{BS2}, the familiar truncation functor
\begin{equation}\label{truncf}
e:\mod{K_\alpha^\La} \rightarrow \mod{H^\La_\alpha}
\end{equation}
defined by left multiplication by the idempotent $e$
is fully faithful on projectives.

\begin{Theorem}\label{itskhovanov}
The module
$e T^\La_\alpha$ is a projective generator
for $H^\La_\alpha$.
Moreover the natural restriction map gives an isomorphism
between $E^\La_\alpha$ and the endomorphism algebra
$\End_{H^\La_\alpha}(e T^\La_\alpha)^{\op}$.
Hence
$E^\La_\alpha$ is the endomorphism algebra of a projective generator
for $H^\La_\alpha$, so $E^\La_\alpha$
and $H^\La_\alpha$ are Morita equivalent.
\end{Theorem}

\begin{proof}
The fact that $e T^\La_\alpha$ is a projective generator for $H^\La_\alpha$
follows from Lemmas~\ref{pring1} and \ref{princ1}, together with
standard facts about truncation functors of the form (\ref{truncf}).
The fact that $E^\La_\alpha \cong \End_{H^\La_\alpha}(eT^\La_\alpha)^{\op}$
is a consequence of the definition (\ref{Thend})
and the fact that the functor $e$ is fully faithful on projectives.
\end{proof}

\phantomsubsection{\boldmath The cyclotomic
Khovanov-Lauda-Rouquier algebra $R^\La_\alpha$}
Introduce another algebra $R^\La_\alpha$ defined by generators
\begin{equation}\label{EKLGens}
\{e(\bi)\:|\: \bi\in I^\al\}\cup\{y_1,\dots,y_{d}\}\cup
\{\psi_1, \dots,\psi_{d-1}\},
\end{equation}
where $d = \height(\alpha)$ as before,
subject to the following relations
for $\bi,\bj\in I^\al$ and all
admissible $r, s$:
\begin{align}
y_1^{(\al_{i_1},\La)}e(\bi)&=0;\hspace{82mm}
\label{ERCyc}\end{align}
\begin{align}
\hspace{6mm}e(\bi) e(\bj) &= \de_{\bi,\bj} e(\bi);
\hspace{17.3mm}{\textstyle\sum_{\bi \in I^\alpha}} e(\bi) = 1;\label{R1}\\
y_r e(\bi) &= e(\bi) y_r;
\hspace{26mm}\psi_r e(\bi) = e(s_r{\cdot}\bi) \psi_r;\hspace{17mm}\label{R2PsiE}
\end{align}\begin{align}
\label{newc}
y_r y_s  &= y_s y_r;\\
\label{R3YPsi}
\hspace{9mm}\psi_r y_s  &= y_s \psi_r\hspace{51.4mm}\text{if $s \neq r,r+1$};\\
\psi_r \psi_s &= \psi_s \psi_r\hspace{50.8mm}\text{if $|r-s|>1$};\label{R3Psi}\end{align}
\begin{align}
\psi_r y_{r+1} e(\bi)
&=
\left\{
\begin{array}{ll}
(y_r\psi_r+1)e(\bi) &\hbox{if $i_r=i_{r+1}$},\\
y_r\psi_r e(\bi) \hspace{37.4mm}&\hbox{if $i_r\neq i_{r+1}$};
\end{array}
\right.
\label{R6}\\
y_{r+1} \psi_re(\bi) &=
\left\{
\begin{array}{ll}
(\psi_r y_r+1) e(\bi)
&\hbox{if $i_r=i_{r+1}$},\\
\psi_r y_r e(\bi)  \hspace{37.4mm}&\hbox{if $i_r\neq i_{r+1}$};\hspace{2mm}
\end{array}
\right.
\label{R5}\end{align}\begin{align}
\hspace{11.9mm}\psi_r^2e(\bi) &=
\left\{
\begin{array}{ll}
0&\text{if $i_r = i_{r+1}$},\\
(i_{r+1}-i_r)(y_{r+1}-y_r)e(\bi)\hspace{8.5mm}&\text{if $i_r=i_{r+1}\pm 1$,}\\
e(\bi)&\text{otherwise};
\end{array}
\right.
 \label{R4}\\
\:\:\psi_{r}\psi_{r+1} \psi_{r} e(\bi)
&=
\left\{\begin{array}{ll}
(\psi_{r+1} \psi_{r} \psi_{r+1} +(i_{r+1}-i_r))e(\bi)&\text{if $i_{r+2}=i_r = i_{r+1}\pm 1$},\\
\psi_{r+1} \psi_{r} \psi_{r+1} e(\bi)&\text{otherwise}.
\end{array}\right.
\label{R7}
\end{align}
This algebra is isomorphic in an obvious way to the
algebra denoted $R(\alpha;\La)$ in \cite[$\S$3.4]{KLa}, which is
a level two cyclotomic quotient of the Khovanov-Lauda-Rouquier
algebra associated to the Dynkin
diagram of type $A$ with vertices indexed by $I$.
By inspecting the relations, it follows that
there is a $\Z$-grading on $R^\La_\alpha$
defined by declaring that each $e(\bi)$ is of degree 0,
each $y_r$ is of degree $2$, and $\psi_r e(\bi)$ is of
degree $-2, 1$ or $0$ according to whether
$i_r = i_{r+1}$, $|i_r-i_{r+1}| =1$ or $|i_r - i_{r+1}| > 1$.

The connection between the endomorphism algebra
$E^\La_\alpha$
and the Khovanov-Lauda-Rouquier
algebra $R^\La_\alpha$ is suggested by the following theorem,
which makes $T^\La_\alpha$ into a right $R^\La_\alpha$-module.

\begin{Theorem}\label{phi}
For any $\alpha \in Q_+$,
there is a homomorphism of graded algebras
$$
\omega:R^\La_\alpha \rightarrow E^\La_\alpha
$$
mapping $e(\bi)$ to the
projection of $T^\La_\alpha$ onto the summand
$F_\bi L(\iota) \langle \defect(\La-\alpha)\rangle$ along the decomposition (\ref{dc}),
$y_r$ to the endomorphism $\sum_{\bi \in I^\alpha} y_r(\bi)_{L(\iota)}$,
and $\psi_r$ to the endomorphism
$\sum_{\bi \in I^\alpha} \psi_r(\bi)_{L(\iota)}$.
\end{Theorem}

\begin{proof}
The relations (\ref{R1}) are immediate, while (\ref{R2PsiE}) follows
because $y_r$ leaves each summand $F_\bi L(\iota)\langle\defect(\La-\alpha)\rangle$ invariant
and $\psi_r$ maps $F_\bi L(\iota)\langle\defect(\La-\alpha)\rangle$ to $F_{s_r\cdot \bi} L(\iota)\langle\defect(\La-\alpha)\rangle$
according to (\ref{ci}) and (\ref{bi}).
Also the commuting relations (\ref{newc}), (\ref{R3YPsi})
and (\ref{R3Psi}) follow from the local nature of the
definitions (\ref{ci}) and (\ref{bi}).
All the remaining relations follow from Theorem~\ref{localrels}.
\end{proof}

\phantomsubsection{Stretched cap, cup and circle diagrams}
Continue with $\alpha \in Q_+$ of height $d$.
In the subsequent two subsections, we are going to introduce
explicit diagram bases for $T^\La_\alpha$ and its endomorphism
algebra $E^\La_\alpha$. In this subsection we define the various
sorts of diagrams needed to do  this.

We refer to the composite matchings $\bt = t_d \cdots t_1$
that are associated to $\La$-admissible sequences
$\bi \in I^\alpha$ as {\em stretched cap diagrams}
of type $\alpha$, calling
$\bt$ {\em proper} if
$\bi$ is a proper $\La$-admissible sequence.
Here are some examples, the last of which is not proper:
$$
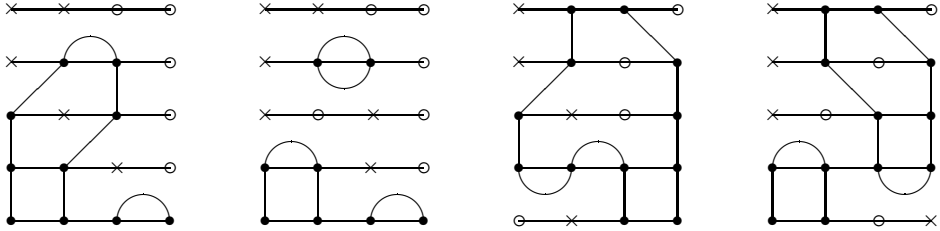
\begin{picture}(157,84)
\put(-96,0){\line(1,0){60}}
\put(-96,20){\line(1,0){60}}
\put(-96,40){\line(1,0){60}}
\put(-96,60){\line(1,0){60}}
\put(-96,80){\line(1,0){60}}
\put(-76,20){\line(0,-1){20}}
\put(-76,60){\line(-1,-1){20}}
\put(-56,40){\line(-1,-1){20}}
\put(-56,60){\line(0,-1){20}}
\put(-96,40){\line(0,-1){40}}
\put(-66,60){\oval(20,20)[t]}
\put(-46,0){\oval(20,20)[t]}
\put(-58.2,-2){$\scriptstyle\bullet$}
\put(-38.2,-2){$\scriptstyle\bullet$}
\put(-78.2,-2){$\scriptstyle\bullet$}
\put(-98.2,-2){$\scriptstyle\bullet$}
\put(-98.2,18){$\scriptstyle\bullet$}
\put(-78.2,18){$\scriptstyle\bullet$}
\put(-59.3,18.2){$\scriptstyle\times$}
\put(-38.5,17.3){$\circ$}
\put(-98.2,38){$\scriptstyle\bullet$}
\put(-58.2,38){$\scriptstyle\bullet$}
\put(-79.3,38.2){$\scriptstyle\times$}
\put(-38.5,37.3){$\circ$}
\put(-78.2,58){$\scriptstyle\bullet$}
\put(-58.2,58){$\scriptstyle\bullet$}
\put(-99.3,58.2){$\scriptstyle\times$}
\put(-38.5,57.3){$\circ$}
\put(-99.3,78.2){$\scriptstyle\times$}
\put(-79.3,78.2){$\scriptstyle\times$}
\put(-58.5,77.3){$\circ$}
\put(-38.5,77.3){$\circ$}

\put(0,0){\line(1,0){60}}
\put(0,20){\line(1,0){60}}
\put(0,40){\line(1,0){60}}
\put(0,60){\line(1,0){60}}
\put(0,80){\line(1,0){60}}
\put(20,20){\line(0,-1){20}}
\put(0,20){\line(0,-1){20}}
\put(30,60){\oval(20,20)[t]}
\put(10,20){\oval(20,20)[t]}
\put(30,60){\oval(20,20)[b]}
\put(50,0){\oval(20,20)[t]}
\put(37.8,-2){$\scriptstyle\bullet$}
\put(57.8,-2){$\scriptstyle\bullet$}
\put(17.8,-2){$\scriptstyle\bullet$}
\put(-2.2,-2){$\scriptstyle\bullet$}
\put(-2.2,18){$\scriptstyle\bullet$}
\put(17.8,18){$\scriptstyle\bullet$}
\put(36.7,18.2){$\scriptstyle\times$}
\put(57.5,17.3){$\circ$}
\put(-3.3,38.2){$\scriptstyle\times$}
\put(37.7,38.2){$\scriptstyle\times$}
\put(57.5,37.3){$\circ$}
\put(17.5,37.3){$\circ$}
\put(17.8,58){$\scriptstyle\bullet$}
\put(37.8,58){$\scriptstyle\bullet$}
\put(-3.3,58.2){$\scriptstyle\times$}
\put(57.5,57.3){$\circ$}
\put(-3.3,78.2){$\scriptstyle\times$}
\put(16.7,78.2){$\scriptstyle\times$}
\put(37.5,77.3){$\circ$}
\put(57.5,77.3){$\circ$}

\put(96,0){\line(1,0){60}}
\put(96,20){\line(1,0){60}}
\put(96,40){\line(1,0){60}}
\put(96,60){\line(1,0){60}}
\put(96,80){\line(1,0){60}}
\put(106,20){\oval(20,20)[b]}
\put(126,20){\oval(20,20)[t]}
\put(136,20){\line(0,-1){20}}
\put(156,60){\line(0,-1){60}}
\put(116,60){\line(-1,-1){20}}
\put(156,60){\line(-1,1){20}}
\put(96,40){\line(0,-1){20}}
\put(116,80){\line(0,-1){20}}
\put(133.8,-2){$\scriptstyle\bullet$}
\put(153.8,-2){$\scriptstyle\bullet$}
\put(112.7,-1.8){$\scriptstyle\times$}
\put(93.5,-2.7){$\circ$}
\put(93.8,18){$\scriptstyle\bullet$}
\put(113.8,18){$\scriptstyle\bullet$}
\put(133.8,18){$\scriptstyle\bullet$}
\put(153.8,18){$\scriptstyle\bullet$}
\put(93.8,38){$\scriptstyle\bullet$}
\put(153.8,38){$\scriptstyle\bullet$}
\put(112.7,38.2){$\scriptstyle\times$}
\put(133.5,37.3){$\circ$}
\put(113.8,58){$\scriptstyle\bullet$}
\put(153.8,58){$\scriptstyle\bullet$}
\put(92.7,58.2){$\scriptstyle\times$}
\put(133.5,57.3){$\circ$}
\put(113.8,78){$\scriptstyle\bullet$}
\put(133.8,78){$\scriptstyle\bullet$}
\put(92.7,78.2){$\scriptstyle\times$}
\put(153.5,77.3){$\circ$}

\put(192,0){\line(1,0){60}}
\put(192,20){\line(1,0){60}}
\put(192,40){\line(1,0){60}}
\put(192,60){\line(1,0){60}}
\put(192,80){\line(1,0){60}}
\put(242,20){\oval(20,20)[b]}
\put(202,20){\oval(20,20)[t]}
\put(212,20){\line(0,-1){20}}
\put(192,20){\line(0,-1){20}}
\put(252,60){\line(0,-1){40}}
\put(232,40){\line(0,-1){20}}
\put(212,80){\line(0,-1){20}}
\put(252,60){\line(-1,1){20}}
\put(232,40){\line(-1,1){20}}
\put(189.8,-2){$\scriptstyle\bullet$}
\put(209.8,-2){$\scriptstyle\bullet$}
\put(248.7,-1.8){$\scriptstyle\times$}
\put(229.5,-2.7){$\circ$}
\put(189.8,18){$\scriptstyle\bullet$}
\put(209.8,18){$\scriptstyle\bullet$}
\put(229.8,18){$\scriptstyle\bullet$}
\put(249.8,18){$\scriptstyle\bullet$}
\put(229.8,38){$\scriptstyle\bullet$}
\put(249.8,38){$\scriptstyle\bullet$}
\put(188.7,38.2){$\scriptstyle\times$}
\put(209.5,37.3){$\circ$}
\put(209.8,58){$\scriptstyle\bullet$}
\put(249.8,58){$\scriptstyle\bullet$}
\put(188.7,58.2){$\scriptstyle\times$}
\put(229.5,57.3){$\circ$}
\put(209.8,78){$\scriptstyle\bullet$}
\put(229.8,78){$\scriptstyle\bullet$}
\put(188.7,78.2){$\scriptstyle\times$}
\put(249.5,77.3){$\circ$}
\end{picture}
$$
If $\bt$ is a stretched cap diagram, its
{\em upper reduction} means the (ordinary) cap diagram obtained by
removing all its number lines except for the
bottom one together with
any internal circles and generalised cups it may contain;
propagating lines become rays up to infinity.
A {\em (proper) stretched cup diagram}
means the mirror image $\bt^* = t_1^* \cdots t_d^*$
of a (proper) stretched cap diagram $\bt = t_d \cdots t_1$.
Its {\em lower reduction} is the mirror image of the upper reduction
of $\bt$.

Given a stretched cap diagram $\bt$ and a cup diagram $a$
whose top number line matches the bottom number line of $\bt$
(i.e. their free vertices in all the same positions)
we can glue $a$ under $\bt$ to obtain
a new diagram $a \bt$. We call this an {\em upper-stretched circle diagram}.
The mirror image of an upper-stretched circle diagram
is a {\em lower-stretched circle diagram}.
Given a pair of stretched cap diagrams $\bt$ and $\bu$
of type $\alpha$, we can glue $\bu^*$ under $\bt$
to obtain a {\em stretched circle diagram} $\bu^* \bt$ of type $\alpha$.
Note such a diagram
has a distinguished number line in the middle, which we call
the {\em boundary line}.
We refer to the internal circles of $\bt$ (resp. $\bu^*$)
as the {\em upper circles} (resp. {\em lower circles})
of $\bu^* \bt$.
All remaining internal circles in $\bu^* \bt$ cross the boundary line,
so we call them {\em boundary circles}.

All the diagrams so far have oriented versions too.
Given a stretched cap diagram $\bt$ of type $\alpha$
we can uniquely recover the underlying
$\La$-admissible sequence $\bi \in I^\alpha$ from it, hence $\bt$ also
determines the
{\em associated block sequence} $\bGa = \Ga_d \cdots \Ga_0$
in which $\Ga_r = \La - \alpha_{i_1}-\cdots-\alpha_{i_{r}}$.
For $\bGa$ arising from $\bt$
in this way, we refer to an oriented $\bGa$-matching
$\bt[\bga]$ as in (\ref{tlanotation})\label{oscd}
as an {\em oriented stretched cap diagram} of type $\alpha$.
Note
it is necessarily the case that $\ga_0 = \iota$, the ground-state.
An {\em oriented stretched cup diagram} means the mirror image
$\bt^*[\bga^*]$ of an oriented stretched cup diagram;
here, for $\bga = \ga_d \cdots \ga_0$ we write
$\bga^*$ for the opposite weight sequence
$\ga_0 \cdots \ga_d$.

An {\em oriented upper-stretched circle diagram}
means a diagram of the form
$a \:\bt[\bga]$ where $a \ga_d$ is an oriented cup diagram
and $\bt[\bga]$ is an oriented stretched cap diagram.
The mirror image
$\bt^*[\bga^*]\:a^*$
of such a diagram is an {\em oriented lower-stretched
circle diagram}.
An {\em oriented stretched circle diagram} is a diagram
 obtained
by gluing an oriented stretched cup diagram
$\bu^*[\bde^*]$ underneath an oriented stretched cap diagram
$\bt[\bga]$ assuming that $\ga_d = \de_d$; we denote this
composite diagram by
\begin{equation*}
\bu^*[\bde^*] \wr \bt[\bga] =
\de_0 u_1^*\de_1 \cdots \de_{d-1} u_d^* \ga_d t_d\ga_{d-1} \cdots \ga_1 t_1 \ga_0.
\end{equation*}
Here is
 an example of an oriented stretched circle diagram,
taking $m=n=2$ and $I = \{1,2\}$,
together with its lower  and upper reductions:
\begin{equation}
\begin{picture}(300,68)
\put(119,-50){\text{$u_1^*$}}
\put(119,-30){\text{$u_2^*$}}
\put(119,-10){\text{$u_3^*$}}
\put(119,12){\text{$t_3$}}
\put(119,32){\text{$t_2$}}
\put(119,52){\text{$t_1$}}
\put(-5,-59){\text{$\de_0=\iota$}}
\put(-5,-39){\text{$\de_1$}}
\put(-5,-19){\text{$\de_2$}}
\put(-5,1){\text{$\de_3=\ga_3$}}
\put(-1,21){\text{$\phantom{\de_3=}\ga_2$}}
\put(-1,41){\text{$\phantom{\de_3=}\ga_1$}}
\put(-1,61){\text{$\iota=\ga_0$}}
\put(52,63){\line(1,0){40}}
\put(52,43){\line(1,0){40}}
\put(52,23){\line(1,0){40}}
\thicklines\put(52,3){\line(1,0){40}}\thinlines
\put(52,-17){\line(1,0){40}}
\put(52,-37){\line(1,0){40}}
\put(52,-57){\line(1,0){40}}
\put(49.3,-21.1){{$\scriptstyle\up$}}
\put(69,-19){{$\scriptstyle\times$}}
\put(89.1,-17.1){{$\scriptstyle\down$}}
\put(92,-17){\line(0,-1){20}}
\put(52,-17){\line(1,-1){20}}
\put(49,-39){$\scriptstyle\times$}
\put(89.1,38.9){$\scriptstyle\up$}
\put(69.1,42.9){$\scriptstyle\down$}
\put(68.7,-41.1){$\scriptstyle\up$}
\put(89.1,-37.1){{$\scriptstyle\down$}}
\put(49,-59){$\scriptstyle\times$}
\put(69,-59){$\scriptstyle\times$}
\put(89.3,-59.7){{$\circ$}}
\put(89,1){$\scriptstyle\times$}
\put(82,-37){\oval(20,20)[b]}
\put(62,3){\oval(20,20)[t]}
\put(82,43){\oval(20,20)[b]}
\put(82,43){\oval(20,20)[t]}
\put(49,61){$\scriptstyle\times$}
\put(69,61){$\scriptstyle\times$}
\put(89.3,60.7){{$\circ$}}
\put(89,21){$\scriptstyle\times$}
\put(49,21){$\scriptstyle\times$}
\put(49,41){$\scriptstyle\times$}
\put(69.3,20.4){{$\circ$}}
\put(69.5,2.9){$\scriptstyle\down$}
\put(49.3,-1.1){$\scriptstyle\up$}
\put(52,3){\line(0,-1){20}}
\put(72,3){\line(1,-1){20}}
\thicklines\put(252,3){\line(1,0){40}}\thinlines
\put(252,-17){\line(1,0){40}}
\put(252,-37){\line(1,0){40}}
\put(252,-57){\line(1,0){40}}
\put(249.3,-21.1){{$\scriptstyle\up$}}
\put(269,-19){{$\scriptstyle\times$}}
\put(289.1,-17.1){{$\scriptstyle\down$}}
\put(292,-17){\line(0,-1){20}}
\put(252,-17){\line(1,-1){20}}
\put(249,-39){$\scriptstyle\times$}
\put(268.7,-41.1){$\scriptstyle\up$}
\put(289.1,-37.1){{$\scriptstyle\down$}}
\put(249,-59){$\scriptstyle\times$}
\put(269,-59){$\scriptstyle\times$}
\put(289.3,-59.7){{$\circ$}}
\put(289,1){$\scriptstyle\times$}
\put(282,-37){\oval(20,20)[b]}
\put(262,3){\oval(20,20)[t]}
\put(269.5,2.9){$\scriptstyle\down$}
\put(249.3,-1.1){$\scriptstyle\up$}
\put(252,3){\line(0,-1){20}}
\put(272,3){\line(1,-1){20}}
\put(172,63){\line(1,0){40}}
\put(172,43){\line(1,0){40}}
\put(172,23){\line(1,0){40}}
\thicklines
\put(172,3){\line(1,0){40}}
\thinlines
\put(209.1,38.9){$\scriptstyle\up$}
\put(189.1,42.9){$\scriptstyle\down$}
\put(209,1){$\scriptstyle\times$}
\put(182,3){\oval(20,20)[t]}
\put(202,43){\oval(20,20)[b]}
\put(182,3){\oval(20,20)[b]}
\put(202,43){\oval(20,20)[t]}
\put(169,61){$\scriptstyle\times$}
\put(189,61){$\scriptstyle\times$}
\put(209.3,60.7){{$\circ$}}
\put(209,21){$\scriptstyle\times$}
\put(169,21){$\scriptstyle\times$}
\put(169,41){$\scriptstyle\times$}
\put(189.3,20.7){{$\circ$}}
\put(189.1,2.9){$\scriptstyle\down$}
\put(169.3,-1.1){$\scriptstyle\up$}
\end{picture}\label{fig2}
\end{equation}
\vspace{50pt}

\noindent
We say a stretched circle diagram $\bu^*\bt$ is {\em proper}
if oriented diagrams of the form $\bu^*[\bde^*]\wr\bt[\bga]$
exist; this implies that both $\bu$ and $\bt$ are proper too.

Finally we introduce some degrees. If $\bt[\bga]$, $\bu[\bde^*]$
and $\bu^*[\bde^*]\wr\bt[\bga]$ are oriented stretched cap, cup and circle diagrams, respectively,
we define their degrees from
\begin{align}
\deg(\bt[\bga]) &:= \#\text{(clockwise caps)} - \#\text{(anti-clockwise cups)},\label{degz}\\
\deg(\bu^*[\bde^*]) &:= \#\text{(clockwise cups)} - \#\text{(anti-clockwise caps)},\label{degzz}\\
\deg(\bu^*[\bde^*]\wr\bt[\bga])
&:= \deg(\bu^*[\bde^*])+\deg(\bt[\bga])=
\deg(\bu[\bde])+\deg(\bt[\bga]).\label{maindegdef}
\end{align}
The following lemma gives an alternative description of the last one of these;
for instance, the oriented stretched circle diagram in (\ref{fig2}) is of degree 1.

\begin{Lemma}
The degree of the oriented stretched circle diagram $\bu^*[\bde^*]\wr\bt[\bga]$ is
equal to
$\defect(\La-\alpha)+
\#\text{(clockwise circles)} - \#\text{(anti-clockwise circles)}.$
\end{Lemma}

\begin{proof}
Using the observation that $\defect(\La-\alpha) = \caps(\bt) - \cups(\bt)$,
it is easy to see from the definition (\ref{maindegdef}) that
$\deg(\bu^*[\bde^*]\wr \bt[\bga])$ is equal to
$\defect(\La-\alpha)$ plus the number of clockwise cups and caps in the diagram
$\bu^*[\bde^*]\wr\bt[\bga]$ minus the total
number of caps in $\bu^* \bt$.
Now apply \cite[Lemma 2.2]{BS2}.
\end{proof}

\begin{Remark}\label{newr1}\rm
There is a natural bijection between the set of oriented stretched cap diagrams
and a special case of the standard tableaux from \cite[$\S$3.2]{BKW} (taking $e:=0$, $l:=2$
and $(k_1,k_2) := (o+m,o+n)$).
To make this precise, define a {\em standard bitableau}
to be a pair $\T = (\T^{(1)}, \T^{(2)})$ of Young tableaux in the usual sense,
with boxes filled with the distinct integers $1,\dots,d$ so that within each $\T^{(i)}$ the entries are strictly increasing
along rows and down columns, like in the following example:
\begin{equation}\label{atab}
\T = \left(\:\diagram{5\cr 8\cr}\:,\:\:\:
\diagram{1&2&3\cr 4&6&7\cr}\:\right).
\end{equation}
In such a diagram, we define the {\em residue} of the node in row $r$ and column $c$ of $\T^{(1)}$
(resp.\ $T^{(2)}$) to be the integer $(o+m+c-r)$ (resp.\ $(o+n+c-r)$).
The {\em residue sequence} $\bi^\T$ of $\T$ means the sequence $(i_1,\dots,i_d)$ where $i_k$ is the residue of the node with entry $k$, and the {\em type} of $\T$ is $\alpha_{i_1}+\cdots+\alpha_{i_d} \in Q_+$; for instance, taking $o=0$ and $m=n=2$,
the tableau in (\ref{atab}) has $\bi^\T = (2,3,4,1,2,2,3,1)$ and is of type $2\alpha_1+3\alpha_2+2\alpha_3+\alpha_4$.

Now suppose we are given an oriented stretched cap diagram $\bt[\bga]$ with underlying admissible sequence $\bi = (i_1,\dots,i_d) \in I^\alpha$.
We let $\T$ be the standard bitableau
obtained by starting from the empty bitableau, then
adding a node of residue $i_k$ labelled by the entry $k$ for each $k=1,\dots,d$, so that the diagram obtained
after each step is itself a standard bitableau.
At the $k$th
step, there is a unique way to
add such a node with one exception: when
the $k$th level $t_k$ of $\bt$ is a cap there are two possible places to add a node of residue $i_k$;
in that case we add the node into $\T^{(1)}$ if the cap is clockwise
or into $\T^{(2)}$ if the cap is anti-clockwise.
In this way, we obtain a well-defined map from oriented stretched cap diagrams $\bt[\bga]$ of type $\alpha$
to standard bitableaux
of type $\alpha$. Moreover the underlying shape of the bitableau $\T$
is the bipartition $(\la^{(1)}, \la^{(2)})$ associated to the bottom
weight $\la := \ga_d$ according to the map from Remark~\ref{dinnertime}.

It is quite easy to reverse the construction, hence the map $\bt[\bga] \mapsto \T$ just defined is actually a bijection.
Moreover, the notion of degree from (\ref{degz}) coincides under this bijection to the notion of degree of a standard
bitableau from \cite[(3.5)]{BKW}.
For example, taking $m=n=2$ again, the tableau (\ref{atab}) arises via our bijection from the following oriented stretched
cap diagram
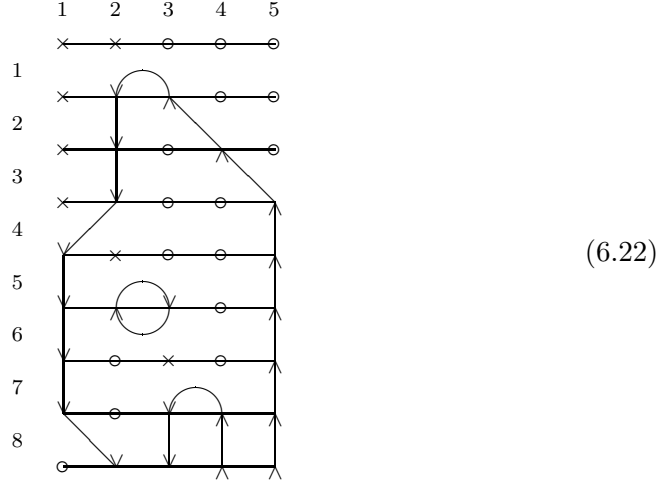
\begin{equation}
\begin{picture}(-310,100)
\put(-195,94){$_1$}
\put(-175,94){$_2$}
\put(-155,94){$_3$}
\put(-135,94){$_4$}
\put(-115,94){$_5$}
\put(-212,71){$_1$}
\put(-212,51){$_2$}
\put(-212,31){$_3$}
\put(-212,11){$_4$}
\put(-212,-9){$_5$}
\put(-212,-29){$_6$}
\put(-212,-49){$_7$}
\put(-212,-69){$_8$}
\put(-192.5,82){\line(1,0){80}}
\put(-192.5,62){\line(1,0){80}}
\put(-192.5,42){\line(1,0){80}}
\put(-192.5,22){\line(1,0){80}}
\put(-192.5,2){\line(1,0){80}}
\put(-192.5,-18){\line(1,0){80}}
\put(-192.5,-38){\line(1,0){80}}
\put(-192.5,-58){\line(1,0){80}}
\put(-192.5,-78){\line(1,0){80}}

\put(-196,80){{$\scriptstyle\times$}}
\put(-176,80){{$\scriptstyle\times$}}
\put(-155.7,79.3){{$\circ$}}
\put(-135.7,79.3){{$\circ$}}
\put(-115.7,79.3){{$\circ$}}
\put(-162.5,62){\oval(20,20)[t]}
\put(-172.5,62){\line(0,-1){40}}
\put(-162.5,-18){\oval(20,20)[t]}
\put(-192.5,2){\line(0,-1){60}}
\put(-112.5,22){\line(0,-1){100}}
\put(-132.5,-58){\line(0,-1){20}}
\put(-152.5,-58){\line(0,-1){20}}
\put(-142.5,-58){\oval(20,20)[t]}
\put(-162.5,-18){\oval(20,20)[b]}
\put(-192.5,-58){\line(1,-1){20}}
\put(-172.5,22){\line(-1,-1){20}}
\put(-152.5,62){\line(1,-1){40}}

\put(-196,60){{$\scriptstyle\times$}}
\put(-175.2,62){{$\scriptstyle\down$}}
\put(-155.2,57.4){{$\scriptstyle\up$}}
\put(-135.7,59.3){{$\circ$}}
\put(-115.7,59.3){{$\circ$}}

\put(-196,40){{$\scriptstyle\times$}}
\put(-175.2,42){{$\scriptstyle\down$}}
\put(-135.2,37.4){{$\scriptstyle\up$}}
\put(-155.7,39.3){{$\circ$}}
\put(-115.7,39.3){{$\circ$}}

\put(-196,20){{$\scriptstyle\times$}}
\put(-175.2,22){{$\scriptstyle\down$}}
\put(-115.2,17.4){{$\scriptstyle\up$}}
\put(-155.7,19.3){{$\circ$}}
\put(-135.7,19.3){{$\circ$}}

\put(-195.2,2){{$\scriptstyle\down$}}
\put(-115.2,-2.6){{$\scriptstyle\up$}}
\put(-176,0){{$\scriptstyle\times$}}
\put(-155.7,-0.7){{$\circ$}}
\put(-135.7,-0.7){{$\circ$}}

\put(-195.2,-18){{$\scriptstyle\down$}}
\put(-115.2,-22.6){{$\scriptstyle\up$}}
\put(-135.7,-20.7){{$\circ$}}
\put(-155.2,-18){{$\scriptstyle\down$}}
\put(-175.2,-22.6){{$\scriptstyle\up$}}

\put(-195.2,-38){{$\scriptstyle\down$}}
\put(-115.2,-42.6){{$\scriptstyle\up$}}
\put(-135.7,-40.7){{$\circ$}}
\put(-156,-40){{$\scriptstyle\times$}}
\put(-175.7,-40.7){{$\circ$}}

\put(-195.2,-58){{$\scriptstyle\down$}}
\put(-175.7,-60.7){{$\circ$}}
\put(-155.2,-58){{$\scriptstyle\down$}}
\put(-135.2,-62.6){{$\scriptstyle\up$}}
\put(-115.2,-62.6){{$\scriptstyle\up$}}

\put(-195.7,-80.7){{$\circ$}}
\put(-175.2,-78){{$\scriptstyle\down$}}
\put(-155.2,-78){{$\scriptstyle\down$}}
\put(-135.2,-82.6){{$\scriptstyle\up$}}
\put(-115.2,-82.6){{$\scriptstyle\up$}}
\end{picture}\label{itmightbewrong}
\end{equation}

\vspace{28mm}

\noindent
which is of degree $1$.

In a similar way, there is a bijection from
stretched oriented cup diagrams of type $\alpha$ to standard bitableaux of type $\alpha$.
Putting the two bijections together, we obtain a bijection from the set of
stretched oriented circle diagrams of type $\alpha$ to the set of all
pairs $({\mathtt S},\T)$ of standard bitableaux of type $\alpha$ such that ${\mathtt S}$ and $\T$
have the same underlying shape.
\end{Remark}

\phantomsubsection{\boldmath Stretched circle diagram bases for $T^\La_\alpha$ and its dual}
As $\La$ is a block of defect zero,
the algebra $K_\La$ is trivial, so the irreducible module $L(\iota)$
is just a copy of the regular module $K_\La$.
Hence we can rewrite the definition (\ref{dc}) as
\begin{align*}
T^\La_\alpha &:= \bigoplus_{\bi \in I^\alpha} F_\bi K_\La \langle\defect(\La-\alpha)\rangle.\\\intertext{Also introduce a dual object}
\widehat T^\La_\alpha &:= \bigoplus_{\bi \in I^\alpha}  E_{\bi^*}K^\La_\alpha
\end{align*}
where for $\bi = (i_1,\dots,i_d)$ we set $\bi^* := (i_d,\dots,i_1)$.
Note $T^\La_\alpha$ is a graded left $K^\La_\alpha$-module, and
$\widehat T^\La_\alpha$ is a graded right $K^\La_\alpha$-module
(arising from the right regular action of $K^\La_\alpha$ on itself).

Using Lemma~\ref{caniso}, noting also that
$\defect(\La-\alpha) = \caps(\bt)-\cups(\bt)$,
we get a canonical isomorphism
\begin{equation}\label{te}
T^\La_\alpha \stackrel{\sim}{\rightarrow}
\bigoplus K_\bGa^\bt \langle-\cups(\bt)\rangle
\end{equation}
of graded left $K^\La_\alpha$-modules,
where the direct sum is over all proper $\La$-admissible sequences
$\bi \in I^\alpha$, and $\bGa$ and $\bt$ denote the associated block
sequence and composite matching, respectively.
Similarly, using $\caps(\bt^*) = \cups(\bt)$, there is a canonical isomorphism
\begin{equation}\label{te2}
\widehat T^\La_\alpha \stackrel{\sim}{\rightarrow}
\bigoplus K_{\bGa^*}^{\bt^*} \langle -\cups(\bt)\rangle,
\end{equation}
where the direct sum is over the same triples $(\bi, \bGa, \bt)$ as in (\ref{te}).
Pulling the diagram bases
from (\ref{diagbasis})
back through these canonical isomorphisms, we deduce that
$T^\La_\alpha$ and $\widehat T^\La_\alpha$ have
distinguished homogeneous bases denoted
\begin{align}\label{thes}
\bigg\{(a\:\bt[\bga]\,|\:&\bigg|\:\begin{array}{l}
\text{for all oriented upper-stretched}\\
\text{circle diagrams
$a\:\bt[\bga]$ of type $\alpha$}
\end{array}\bigg\},\\
\bigg\{|\,\bu^*[\bde^*]\:b)\:&\bigg|\:\begin{array}{l}
\text{for all oriented lower-stretched}\\
\text{circle diagrams
$\bu^*[\bde^*]\:b$ of type $\alpha$}
\end{array}\bigg\}.\label{thes2}
\end{align}

The bases (\ref{thes})--(\ref{thes2}) are
quite convenient for computations.
For example, the natural right action of the
generators $y_r$ and $\psi_r$ of $R^\La_\alpha$ on
a basis element $(a\:\bt[\bga]\,|$
of $T^\La_\alpha$
can be computed quickly by making the appropriate
circle move, crossing move or height move
to the diagram $a \:\bt[\bga]$
as in the definitions
(\ref{expy}) and (\ref{betadef2}) (remembering also the signs
from Lemmas~\ref{inb} and \ref{sff2}). The left $K^\La_\alpha$-module
structure on $T^\La_\alpha$ can be computed in terms of the basis as follows.
Take
$(a \la b) \in K^\La_\alpha$ and
$(c \:\bt[\bga]\,| \in T^\La_\alpha$.
If $b^*\neq c$,
we have that $(a \la b)(c \:\bt[\bga]| = 0$.
If $b^*=c$, proceed as follows:
\begin{itemize}
\item draw $a \la b$ underneath $c \:\bt[\bga]$;
\item iterate the generalised surgery procedure
from \cite[$\S$6]{BS1} to the symmetric section of the diagram
 containing $b$ under $c$ to eliminate all cap-cup pairs;
\item
finally remove the bottom number line
to obtain a
linear combination
of basis vectors of $T^\La_\alpha$.
\end{itemize}
The right $K^\La_\alpha$-module structure on $\widehat T^\La_\alpha$
can be computed similarly.

There is a precise sense in which $\widehat T^\La_\alpha$
is dual to $T^\La_\alpha$.
To explain this,
define another multiplication map
\begin{equation}\label{varphi}
\varphi:T^\La_\alpha \otimes \widehat T^\La_\alpha
\rightarrow K^\La_\alpha
\end{equation}
as follows.
Take basis vectors
$(a \:\bt[\bga]\,| \in T^\La_\alpha$
and $|\,\bu^*[\bde^*]\:b) \in \widehat T^\La_\alpha$.
If $\bt = \bu$ and all mirror image pairs of
internal circles
in $\bt[\bga]$ and $\bu^*[\bde^*]$ are oriented so that one is clockwise,
the other anti-clockwise,
then we define
\begin{equation*}
\phi((a \:\bt[\bga]\,| \otimes |\,\bu^*[\bde^*]\:b)) := (a \,\ga_d\, c) (c^* \de_d\, b) \in K^\La_\alpha,
\end{equation*}
where $c$ is the upper reduction of $\bt = \bu$.
Otherwise, we set
\begin{equation*}
\phi((a \:\bt[\bga]\,| \otimes |\,\bu^*[\bde^*]\:b)) := 0.
\end{equation*}
This definition is a slight
variation on the definition of the adjunction
morphism $\varphi$ from \cite[(4.3)]{BS2}.
The following lemma establishes in particular that
$\varphi$ is a degree zero
homomorphism of graded $(K^\La_\alpha, K^\La_\alpha)$-bimodules.

\begin{Lemma}\label{precd}
There is a degree zero isomorphism of graded right $K^\La_\alpha$-modules
$$
\widehat T^\La_\alpha
\stackrel{\sim}{\rightarrow}
\hom_{K^\La_\alpha}(T^\La_\alpha,
K^\La_\alpha)
$$
mapping $y \in \widehat T^\La_\alpha$ to the
homomorphism $T^\La_\alpha
\rightarrow
K^\La_\alpha, \:x \mapsto \phi(x \otimes y)$.
\end{Lemma}

\begin{proof}
It suffices to show that the restriction of the given map is
a degree zero $K^\La_\alpha$-module
isomorphism
$$
K_{\bGa^*}^{\bt^*}\langle - \cups(\bt)\rangle
\stackrel{\sim}{\rightarrow}
\hom_{K^\La_\alpha}(K^\bt_\bGa \langle -\cups(\bt)\rangle, K^\La_\alpha)
$$
for $\bGa$ and $\bt$ associated
to a fixed proper $\La$-admissible sequence $\bi \in I^\alpha$.
Let $s$ be the reduction of $\bt$, and note that
$\caps(s)- \cups(s) = \caps(\bt)-\cups(\bt)$.
Using (\ref{moso})
and \cite[Theorem 4.7]{BS2},
we have canonical isomorphisms
\begin{align*}
K^{\bt^*}_{\bGa^*} \langle - \cups(\bt) \rangle
&\cong
K^{s^*}_{\La(\La-\alpha)} \langle -\cups(s)\rangle
\otimes R^{\circles(\bt)}\\
&\cong
\hom_{K^\La_\alpha}(K^s_{(\La-\alpha) \La}, K^\La_\alpha)\langle\cups(s)\rangle
\otimes R^{\circles(\bt)}\\
&
\cong
\hom_{K^\La_\alpha}(K^s_{(\La-\alpha) \La}\langle-\cups(s)\rangle \otimes R^{\circles(\bt)}, K^\La_\alpha)\\
&\cong
\hom_{K^\La_\alpha}(K^\bt_\bGa \langle -\cups(\bt)\rangle, K^\La_\alpha).
\end{align*}
The composite of all of these isomorphisms
is exactly the map
defined in the statement of the lemma.
\end{proof}

Note via the isomorphism from
Lemma~\ref{precd} that $\widehat{T}^\La_\alpha$
becomes a graded left $E^\La_\alpha$-module,
because $\hom_{K^\La_\alpha}(T^\La_\alpha,
K^\La_\alpha)$
is naturally such.
Applying the homomorphism $\omega$ from Theorem~\ref{phi},
we deduce $\widehat{T}^\La_\alpha$ is a graded left $R^\La_\alpha$-module
too.

\begin{Remark}\label{bhat}\rm
Let $e$ be the truncation functor from (\ref{truncf}).
The isomorphism in Lemma~\ref{precd} restricts to an isomorphism
$$
\widehat{T}^\La_\alpha e \stackrel{\sim}{\rightarrow}
\hom_{K^\La_\alpha}(T^\La_\alpha,
K^\La_\alpha)e =
\hom_{K^\La_\alpha}(T^\La_\alpha,
K^\La_\alpha e).
$$
Since the truncation functor is fully faithful on projectives,
it defines an isomorphism
$\hom_{K^\La_\alpha}(T^\La_\alpha,
K^\La_\alpha e) \cong
\hom_{H^\La_\alpha}(e T^\La_\alpha,
H^\La_\alpha).$
In this way, we deduce also the existence of an isomorphism
$$
\widehat{T}^\La_\alpha e \stackrel{\sim}{\rightarrow}
\hom_{H^\La_\alpha}(e T^\La_\alpha,
H^\La_\alpha)
$$
as graded $(E^\La_\alpha, H^\La_\alpha)$-bimodules.
\end{Remark}

\phantomsubsection{\boldmath The stretched circle diagram basis
for $E^\La_\alpha$}
Now we turn our attention to the endomorphism algebra
$E^\La_\alpha$ from (\ref{Thend}),
still assuming we are given $\alpha \in Q_+$ of height $d$.
The following theorem gives an explicit description of this algebra.

\begin{Theorem}\label{dis}
There is a graded vector space isomorphism
$$
\widehat T^\La_\alpha \otimes_{K^\La_\alpha} T^\La_\alpha
\stackrel{\sim}{\rightarrow}
E^\La_\alpha
$$
mapping $y \otimes z$
to the endomorphism $T^\La_\alpha \rightarrow T^\La_\alpha,
\:x \mapsto \phi(x \otimes y) z$.
Using this isomorphism, the multiplication on
 $E^\La_\alpha$
lifts to define a graded multiplication
on $\widehat T^\La_\alpha \otimes_{K^\La_\alpha} T^\La_\alpha$
with $(u \otimes x)(y \otimes z) = u \phi(x \otimes y) \otimes z
= u \otimes \phi(x \otimes y) z$.
\end{Theorem}

\begin{proof}
The map in the statement of the theorem
is the composite of the canonical isomorphisms
$$
\widehat{T}^\La_\alpha \otimes_{K^\La_\alpha} T^\La_\alpha
\stackrel{\sim}{\rightarrow}
\hom_{K^\La_\alpha}(T^\La_\alpha, K^\La_\alpha) \otimes_{K^\La_\alpha}
T^\La_\alpha
\stackrel{\sim}{\rightarrow}
\hom_{K^\La_\alpha}(T^\La_\alpha, T^\La_\alpha).
$$
The first of these arises from Lemma~\ref{precd}, and the
second is the map sending $f \otimes t$ to the map
$x \mapsto f(x) t$ (which
is an isomorphism
because $T^\La_\alpha$ is a projective left $K^\La_\alpha$-module
by \cite[Corollary 4.4]{BS2}).
\end{proof}

Finally,
we need one more canonical isomorphism
arising from the bimodule multiplication from (\ref{miso}):
\begin{equation*}
\widehat{T}^\La_\alpha \otimes_{K^\La_\alpha} T^\La_\alpha
\stackrel{\sim}{\rightarrow}
\bigoplus
K^{\bu^*\bt}_{\bDe^* \wr \bGa} \langle-\cups(\bu)-\cups(\bt)\rangle,
\end{equation*}
where the direct sum is over all proper
$\La$-admissible sequences $\bi, \bj \in I^\alpha$,
and $\bGa$ and $\bt$ (resp. $\bDe$ and $\bu$)
denote the block sequence and composite matching associated to $\bi$
(resp.\ $\bj$).
Pulling back the diagram basis for each $K^{\bu^* \bt}_{\bDe^* \wr \bGa}$
from (\ref{diagbasis}) through this isomorphism, then pushing forward
from there to $E^\La_\alpha$ using the isomorphism from Theorem~\ref{dis},
we deduce that $E^\La_\alpha$ has a distinguished homogeneous basis
\begin{equation}\label{Hbasis}
\left\{|\,\bu^*[\bde^*]\wr \bt[\bga]\,|\:\bigg|\:\begin{array}{l}
\text{for all oriented stretched circle}\\
\text{diagrams
$\bu^*[\bde^*]\wr\bt[\bga]$ of type $\alpha$}.
\end{array}\right\}
\end{equation}
The degree of the vector
$|\,\bu^*[\bde^*]\wr \bt[\bga]\,|$ is given by the formula (\ref{maindegdef}).

\phantomsubsection{Algorithms for computing with stretched circle diagrams}
We next spell out how to compute the various sorts of products in terms of the
diagram basis (\ref{Hbasis}).
 These explicit descriptions all follow from Theorem~\ref{dis}
and the appropriate associativity.

First of all, given  two basis vectors
$|\,\bs^*[\btau^*]\wr \br[\bsigma]\,|$ and
$|\,\bu^*[\bde^*]\wr \bt[\bga]\,|$ in $E^\La_\alpha$,
their
product $|\,\bs^*[\btau^*]\wr \br[\bsigma]\,|
\cdot|\,\bu^*[\bde^*]\wr \bt[\bga]\,|$ is zero unless
$\br = \bu$ and
all mirror image pairs of internal circles in
$\br[\bsigma]$ and $\bu^*[\bde^*]$ are oriented so that one
is clockwise, the other anti-clockwise.
Assuming these conditions hold, we let $b$ be the upper reduction of
$\br = \bu$ and proceed as follows:
\begin{itemize}
\item
draw $\bs^*[\btau^*] \:b$ under $b^* \:\bt[\bga]$;
\item
iterate the generalised surgery procedure from \cite[$\S$6]{BS1}
to eliminate all cap-cup pairs in the symmetric middle section of
the diagram;
\item collapse the middle section to obtained the
desired linear combination of basis vectors of $E^\La_\alpha$.
\end{itemize}
The right action of $E^\La_\alpha$ on $T^\La_\alpha$
and the left action of $E^\La_\alpha$ on $\widehat T^\La_\alpha$
are both computed on basis vectors by analogous
procedures.

Let us describe explicitly the image of the idempotent
$e(\bi) \in R^\La_\alpha$ under the homomorphism $\omega$ from Theorem~\ref{phi}.
If $\bi$ is not $\La$-admissible then $\omega(e(\bi)) = 0$.
If $\bi$ is $\La$-admissible, we let $\bGa$ and $\bt$ be the associated
block sequence and composite matching.
Then
\begin{equation}\label{piei}
\omega(e(\bi)) = \sum |\,\bt^*[\bde^*]\wr \bt[\bga]\,|
\end{equation}
where the sum is over all $\bga = \ga_d \cdots \ga_0$
and $\bde = \de_d \cdots \de_0$ such that
\begin{itemize}
\item $\ga_r, \de_r \in \Ga_r$ for all $r=0,\dots,d$,
and $\ga_d = \de_d$;
\item
$\bt^*[\bde^*]\wr \bt[\bga]$ is an oriented stretched circle diagram;
\item every boundary circle in $\bt^*[\bde^*]\wr \bt[\bga]$
is anti-clockwise;
\item every mirror image pair
of upper and lower circles in $\bt^*[\bde^*]\wr \bt[\bga]$ is oriented
so one is clockwise, the other anti-clockwise.
\end{itemize}
Note in particular that this element is homogeneous of degree zero.
The set $\{\omega(e(\bi))\:|\:\bi \in I^\alpha\}$
is a system of mutually orthogonal idempotents in $E^\La_\alpha$
whose sum is the identity
endomorphism.

Via the homomorphism from Theorem~\ref{phi}
we get right and left actions of the generators $y_r$ and $\psi_r$ of
$R^\La_\alpha$ on $E^\La_\alpha$. These can be computed directly on the
basis vectors of $E^\La_\alpha$ by mimicking the definitions
of (\ref{expy}) and (\ref{betadef2}) (and incorporating
a sign as in Lemmas~\ref{inb} and \ref{sff2}).
So multiplying a basis vector $|\,\bu^*[\bde^*]\wr \bt[\bga]\,|$
on the right (resp. left) by $y_r$ involves making a signed
positive circle move at the $r$th level of $\bt$
(resp. $\bu$).
Similarly multiplying a basis vector $|\,\bu^*[\bde^*]\wr \bt[\bga]\,|$
on the right (resp. left) by $\psi_r$ involves making a signed
negative circle move, crossing move or height move
at the $r$th and $(r+1)$th levels of $\bt$
(resp. $\bu$).

The image of $y_r e(\bi) = e(\bi) y_r$ under the homomorphism $\omega$
can be obtained explicitly in terms of the diagram basis
by starting from $\omega(e(\bi))$
as in (\ref{piei}) then acting by $y_r$, either on the left or the right.
Similarly the image of $\psi_r  e(\bi) = e(s_r\cdot\bi) \psi_r$
can be obtained either by starting from $\omega(e(\bi))$ and acting on the left
by $\psi_r$, or by starting with $e(s_r\cdot \bi)$ and acting on
the right by $\psi_r$.

\phantomsubsection{\boldmath Cellularity of $E^\La_\alpha$}
There are involutory graded algebra anti-automorphisms
\begin{equation}\label{Stars}
*:R^\La_\alpha \rightarrow R^\La_\alpha,
\qquad
*:E^\La_\alpha \rightarrow E^\La_\alpha.
\end{equation}
On $R^\La_\alpha$, $*$ is the anti-automorphism with
$e(\bi)^* = e(\bi), y_r^* = y_r$ and $\psi_r^* = \psi_r$
for each $\bi$ and $r$.
On $E^\La_\alpha$, $*$ is the linear map sending
$|\,\bu^*[\bde^*]\wr\bt[\bga]\,|$
to its mirror image $|\,\bt^*[\bga^*]\wr \bu[\bde]\,|$.
Recalling the homomorphism $\omega$ from Theorem~\ref{phi},
the following diagram commutes:
\begin{equation}
\begin{CD}
R^\La_\alpha&@>*>>&R^\La_\alpha\\
@V\omega VV&&@VV\omega V\\
E^\La_\alpha&@>*>>&E^\La_\alpha
\end{CD}
\end{equation}

\begin{Theorem}\label{cell}
Assume that $\Ga := \La - \alpha \in \PImn$.
Then the algebra $E^\La_\alpha$ is a cellular algebra
in the sense of Graham and Lehrer \cite{GL} with cell datum
$(\Ga, M, C, *)$ where
\begin{itemize}
\item[(i)] $\Ga$ is viewed as a poset with respect
to the Bruhat order;
\item[(ii)]
$M(\la)$ for each $\la \in \Ga$ denotes the set of all
oriented upper-stretched cap diagrams $\bt[\bga]$
of type $\al$ such that $\ga_d = \la$;
\item[(iii)] $C$ is defined by setting $C^\la_{\bu[\bde],\bt[\bga]}
:= |\,\bu^*[\bde^*]\wr\bt[\bga]\,|$ for each $\la \in \Ga$
and $\bu[\bde], \bt[\bga] \in M(\la)$;
\item[(iv)] $*:E^\La_\alpha \rightarrow E^\La_\alpha$ is the algebra
anti-automorphism from (\ref{Stars}).
\end{itemize}
\end{Theorem}

\noindent
(For our conventions regarding cellular algebras, see the paragraph
after the statement of \cite[Corollary 3.3]{BS1}.)

\begin{proof}
This follows by similar arguments to \cite[Corollary 3.3]{BS1}.
\end{proof}

\begin{Remark}\rm\label{newr2}
(1) In fact our cellular basis makes $E^\La_\alpha$ into a {\em graded} cellular algebra in the obvious sense, with the
degree of $\bt[\bga] \in M(\la)$ being
defined by (\ref{degz}). We show in Corollary~\ref{id3} below
that $E^\La_\alpha \cong R^\La_\alpha$, so that $E^\La_\alpha$
is a particular example of a
cyclotomic Khovanov-Lauda-Rouquier algebra of type $A$.
The latter algebras were conjectured to be graded cellular algebras in
\cite[Remark 4.12]{BKW},
with a graded cellular basis parametrised by pairs of bitableaux of the same shape and type $\alpha$.
Combined with the observations made in Remark~\ref{newr1}, Theorem~\ref{cell} verifies a particular instance of this conjecture.
Subsequently, the conjecture has been proved in
general by Hu and Mathas \cite{HM}.

(2)
By arguments completely analogous to \cite[Theorem 6.3]{BS1}, we get also that $E^\La_\alpha$ is a {\em graded symmetric algebra}:
it possesses a homogeneous symmetrizing form $\tau:E^\La_\alpha\rightarrow \C$ of degree $-2\defect(\La-\alpha)$.
Explicitly, $\tau$ is the linear map defined on a basis vector $|\bu^*[\bde^*]\wr \bt[\bga]|$  by declaring that
it is zero unless $\bu=\bt$, all boundary circles are clockwise, and all mirror image pairs of internal circles are
oppositely oriented, in which case it is $1$.

(3)
For $\Ga := \La-\alpha \in \PImn$ as in Theorem~\ref{cell},
we showed already in Theorem~\ref{itskhovanov}
that $E^\La_\alpha$ is
Morita equivalent to the generalised Khovanov algebra
$H_\Ga$ from \cite[(6.2)]{BS1}.
The latter algebra
was already shown to be both cellular and symmetric
in \cite[Theorems 6.2--6.3]{BS1}.
\end{Remark}

\phantomsubsection{\boldmath Addition of a simple root}
Suppose we are given $\alpha \in Q_+$ of height $d$
as before, and also $i \in I$.
For $\bi = (i_1,\dots,i_d) \in I^\alpha$, we write $\bi+i$
for the tuple $(i_1,\dots, i_d,i) \in I^{\alpha+\alpha_i}$.
By the relations, there is a (non-unital) algebra homomorphism
\begin{equation}\label{thetai}
\theta_i:R^\La_\alpha \rightarrow R^\La_{\alpha+\alpha_i}
\end{equation}
such that $\theta_i(e(\bi)) = e(\bi+i)$,
$\theta_i(y_r e(\bi)) = y_r e(\bi+i)$
and $\theta_i (\psi_r e(\bi)) = \psi_r e(\bi+i)$.
The image of the identity element of $R^\La_\alpha$ under this homomorphism
is the idempotent
\begin{equation}\label{baa}
\sum_{\substack{\bi \in I^{\alpha+\alpha_i} \\ i_{d+1} = i}} e(\bi)
\in R^\La_{\alpha+\alpha_i}.
\end{equation}
We want to construct a corresponding homomorphism at the level
of the endomorphism algebra $E^\La_\alpha$, which will be a key tool
for inductive arguments.

We begin by defining a degree-preserving linear map
\begin{equation}\label{kwon}
\theta_i:E^\La_\alpha \rightarrow E^\La_{\alpha+\alpha_i}.
\end{equation}
Take a basis vector $|\,\bu^*[\bde^*]\wr \bt[\bga]\,| \in E^\La_\alpha$
as in (\ref{Hbasis}).
If $i$ is not $(\La-\alpha)$-admissible then we set
$\theta_i(|\,\bu^*[\bde^*]\wr \bt[\bga]\,|) := 0$.
If $i$ is $(\La-\alpha)$-admissible, there are various cases
according to whether $s := t_i(\La-\alpha)$ is a cup, a cap,
a right-shift or a left-shift.
In all cases the basic idea
is to
replace the matching $\bu^* \bt$ in the diagram by
$\bu^* s^* s \bt$, i.e. we want to insert
two extra levels at the boundary line:
\begin{equation}
% scriptstyles at (1,x) x = -68,-45,-22,1,24,47,70,93,...
% \times is -0.2
% \circ is +1
% \up is +0.3 and down 2.6
% \down is +0.3 and up 2.1
\begin{picture}(145,39)
\put(162.5,2){$\La-\alpha-\alpha_i$}
\put(172,-10){$s^*$}
\put(172,15){$s$}
\put(162.5,-21.4){$\La-\alpha$}
\put(162.5,26.6){$\La-\alpha$}
\put(-82,-34){$\text{``cup''}$}
\put(-21,-34){$\text{``cap''}$}
\put(21,-34){$\text{``right-shift''}$}
\put(89,-34){$\text{``left-shift''}$}

\put(-24,29){\line(1,0){33}}
\thicklines\put(-24,5){\line(1,0){33}}\thinlines
\put(-24,-19){\line(1,0){33}}
\put(-7.5,5){\oval(23,23)[b]}
\put(-7.5,5){\oval(23,23)[t]}
\put(-22.3,-20.9){$\scriptstyle\times$}
\put(1.2,-21.6){$\circ$}
\put(-22.3,27.1){$\scriptstyle\times$}
\put(1.2,26.4){$\circ$}
\put(-20.9,3.1){{$\scriptstyle\bullet$}}
\put(2.1,3.1){{$\scriptstyle\bullet$}}

\put(-84,29){\line(1,0){33}}
\put(-67.5,29){\oval(23,23)[b]}
\put(-80.9,27.1){{$\scriptstyle\bullet$}}
\put(-57.9,27.1){{$\scriptstyle\bullet$}}
\thicklines\put(-84,5){\line(1,0){33}}\thinlines
\put(-84,-19){\line(1,0){33}}
\put(-67.5,-19){\oval(23,23)[t]}
\put(-81.8,2.4){$\circ$}
\put(-59.2,3.1){$\scriptstyle\times$}
\put(-81.1,-20.9){{$\scriptstyle\bullet$}}
\put(-58.1,-20.9){{$\scriptstyle\bullet$}}

\thicklines\put(96,5){\line(1,0){33}}\thinlines
\put(96,-19){\line(1,0){33}}
\put(120.7,3.1){$\scriptstyle\times$}
\put(97.7,-20.9){$\scriptstyle\times$}
\put(124.5,-19.2){\line(-1,1){22.9}}
\put(121.9,-20.9){{$\scriptstyle\bullet$}}
\put(98.9,3.1){{$\scriptstyle\bullet$}}
\put(96,29){\line(1,0){33}}
\put(97.7,27.1){$\scriptstyle\times$}
\put(101.1,5.8){\line(1,1){22.9}}
\put(121.9,27.1){{$\scriptstyle\bullet$}}

\put(36,29){\line(1,0){33}}
\put(61.2,26.4){$\circ$}
\put(64,5.8){\line(-1,1){22.9}}
\put(38.9,27.1){{$\scriptstyle\bullet$}}
\thicklines\put(36,5){\line(1,0){33}}\thinlines
\put(36,-19){\line(1,0){33}}
\put(61.2,-21.6){$\circ$}
\put(38.2,2.4){$\circ$}
\put(41.5,-18.2){\line(1,1){22.9}}
\put(38.9,-20.9){{$\scriptstyle\bullet$}}
\put(61.9,3.1){{$\scriptstyle\bullet$}}
\end{picture}
\vspace{12mm}\label{CKLR2}
\end{equation}
Here, as in (\ref{CKLR}), we display only the strip between the $i$th
and $(i+1)$th vertices, all other strips being the ``identity''.
Somewhat informally, we proceed as follows in the various cases:

\vspace{1mm}

\noindent{\em Case one: $s$ is a right-shift or a left-shift.}
We compute
$\theta_i(|\,\bu^*[\bde^*]\wr \bt[\bga]\,|)$
simply by inserting the matching $s^* s$
into the middle of the
 diagram $\bu^*[\bde^*]\wr \bt[\bga]$ then orienting the new boundary line
in the only sensible way, leaving all other orientations unchanged.

\vspace{1mm}

\noindent{\em Case two: $s$ is a cap.}
This time,
$s^* s$ contains a small circle. Again we simply insert $s^* s$
into the diagram, orienting the new boundary line so that the
new small circle is anti-clockwise
and leaving all other orientations unchanged.

\vspace{1mm}

\noindent
{\em Case three: $s$ is a cup.}
If $s$ is a cup,
then inserting $s^*s$ into the diagram either splits one component
into two or joins two components into one.
In this case we define
$\theta_i(|\,\bu^*[\bde^*]\wr \bt[\bga]\,|)$
to be a sum of zero, one or two basis vectors of $E^\La_{\alpha+\alpha_i}$,
with the new component(s)
oriented following exactly the same rules as in the
generalised surgery procedure from \cite[$\S$6]{BS1}.

\vspace{1mm}

Put more formally, this means in all cases that
\begin{equation}\label{lastyuk}
\theta_i(|\,\bu^*[\bde^*]\wr \bt[\bga]\,|) = \sum
|\,\bu^*[\hat\bde^*] s^* \la s \bt[\hat\bga]\,|
\end{equation}
where the sum is over all
$\hat\bde = \hat \de_d \cdots \hat \de_0,
\hat\bga = \hat\ga_d \cdots \hat\ga_0$ and $\la\in\La-\alpha-\alpha_i$
such that
\begin{itemize}
\item
each
$\hat\ga_r$ (resp.\ $\hat\de_r$)
lies in the same block as $\ga_r$ (resp.\ $\de_r$);
\item $\bu^*[\hat\bde^*] s^* \la s \bt[\hat\bga]$ is an oriented stretched
circle diagram of the same degree as $\bu^*[\bde^*] \wr \bt[\bga]$;
\item
the components of $\bu^*[\hat\bde^*] s^* \la s \bt [\hat\bga]$
that do not pass through the $i$th or $(i+1)$th vertices
at the top or bottom of $s$ or $s^*$ are oriented in the same way as
the corresponding components of $\bu^*[\bde^*] \wr \bt[\bga]$.
\end{itemize}

For example, writing $e$ for the identity element of the one-dimensional
algebra $K_\La$, we have according to this definition that
\begin{equation}\label{imei}
\theta_{i_d} \circ \cdots \circ \theta_{i_1}(e) = \omega(e(\bi))
\end{equation}
as computed by (\ref{piei}), for any $\bi \in I^\alpha$.

\begin{Theorem}\label{thetaithm}
The linear map $\theta_i$ just defined is equal to the composition
of the following maps:
\begin{align*}
E^\La_\alpha &=
\End_{K^\La_\alpha}\bigg(\bigoplus_{\bi \in I^\alpha} F_\bi L(\iota)\bigg)^{\op}
&&\!\!\!\!\!\!\!\!\rightarrow
\End_{K^\La_{\alpha+\alpha_i}}\bigg(\bigoplus_{\bi \in I^{\alpha}} F_{\bi+i}
L(\iota)\bigg)^{\op}\\
&=\End_{K^\La_{\alpha+\alpha_i}}\bigg(\bigoplus_{\substack{\bi \in I^{\alpha+\alpha_i} \\ \alpha_{d+1} = i}} F_{\bi}
L(\iota)\bigg)^{\op}
&&\!\!\!\!\!\!\!\!\hookrightarrow \End_{K^\La_{\alpha+\alpha_i}}\bigg(\bigoplus_{\bi \in I^{\alpha+\alpha_i}} F_\bi L(\iota)\bigg)^{\op} = E^\La_{\alpha+\alpha_i},
\end{align*}
where the first map is defined by applying the functor $F_i$
and the second map is the natural inclusion.
In particular, $\theta_i$ is a graded (non-unital) algebra homomorphism.
\end{Theorem}

\begin{proof}
We may assume that $i$ is $(\La-\alpha)$-admissible
and set $s := t_i(\La-\alpha)$.
Let $\bar\theta_i$ denote the composition of the maps from the
statement of the theorem.
Take any basis vector $|\,\bu^*[\bde^*]\wr \bt[\bga]\,| \in E^\La_\alpha$.
We want to show that \begin{equation*}
\theta_i(|\,\bu^*[\bde^*]\wr \bt[\bga]\,|)= \bar\theta_i(|\,\bu^*[\bde^*]\wr \bt[\bga]\,|)
\end{equation*}
as endomorphisms
of $T^\La_{\alpha+\alpha_i}$. It is clear that both sides
are zero on all basis vectors of $T^\La_{\alpha+\alpha_i}$ that
are {\em not} of the special form $(a \mu s \bu[\bepsilon]\,|$, so take such a
special basis vector. Then we need to show that
\begin{equation}\label{reallybad}
(a \mu s \bu[\bepsilon]\,| \cdot \theta_i(|\,\bu^*[\bde^*]\wr \bt[\bga]\,|)
=
(a \mu s \bu[\bepsilon]\,| \cdot \bar\theta_i(|\,\bu^*[\bde^*]\wr \bt[\bga]\,|).
\end{equation}
The vector on the right hand side of (\ref{reallybad}) can be computed
explicitly as follows. It is zero unless all
mirror image pairs of internal circles
in $\bu[\bepsilon]$ and $\bu^*[\bde^*]$ are oppositely oriented.
In that case, the product is obtained by applying a sequence of
generalised surgery procedures to remove all cap-cup pairs from
the $bb^*$-part of the diagram
\begin{equation}\label{lhd2}
a \mu s \eps_d b b^* \bt[\bga],
\end{equation}
where $b$ is the upper reduction of $\bu$.
On the other hand, the vector on the left hand side
of (\ref{reallybad}) is equal to
\begin{equation}\label{lhd}
\sum
(a \mu s \bu[\bepsilon]\,|\cdot |\,[\bu^*[\hat\bde^*]s^* \la s \bt[\hat\bga]\,|
\end{equation}
summing over the same tuples as in (\ref{lastyuk}).
Recall moreover to compute each
product in (\ref{lhd}),
the product is zero unless mirror image pairs of internal circles
in $s \bu[\bepsilon]$ and $\bu^*[\hat\bde^*] s^*$ are oppositely oriented.
In that case, the product is obtained by applying a sequence of
generalised surgery procedure
to remove all cap-cup pairs from the $cc^*$-part of the diagram
\begin{equation}\label{lhd1}
a \mu c c^* \la s \bt[\hat\bga],
\end{equation}
where
$c$ is the upper reduction of $s \bu$ ($=$ the upper reduction of $sb$).

Comparing (\ref{lhd1}) and (\ref{lhd2}),
it follows easily that (\ref{reallybad}) holds if $s$ is a right-shift
or a left-shift: both procedures amount to performing exactly the same
sequence of surgery procedures.
The case when $s$ is a cap is not much harder: this time the
computation of (\ref{lhd1}) involves one more surgery procedure than
(\ref{lhd2}), but this extra surgery procedure corresponds
to ``multiplication by an anti-clockwise circle'' which is the identity map.
Finally we are left with the case when $s$ is a cup, which is
more difficult to explain.
Recall from \cite[$\S$3]{BS1}
that our surgery procedures are associated to the TQFT
defined by the commutative Frobenius algebra $R = \C[x] / (x^2)$,
which has multiplication $m$, comultiplication $\delta$ and counit
$\eps$.
We split into three sub-cases.

\vspace{1mm}

\noindent{\em Sub-case one: $b^*s^* s \bt$ has one less component
than $b^* \bt$.}
We only need to consider components that interact with $s$ or $s^*$,
as the surgery procedures on all other components clearly correspond
in (\ref{lhd2}) and (\ref{lhd1}). We will assume moreover that the relevant
components are all circles rather than propagating lines,
leaving the appropriate
modifications when propagating lines are involved to the reader.
With these reductions,
there is just one basic configuration to be considered,
represented
by the following diagram (in which for simplicity we have displayed
only the lower
reduction $d$ of $\bt$ rather than $\bt$ itself):
$$
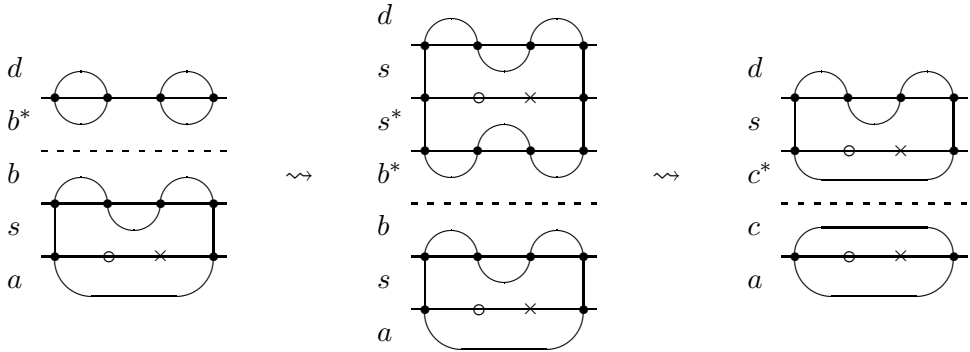
\begin{picture}(172,132)
\put(220.8,29.3){{$\circ$}}
\put(239.7,30.1){{$\scriptstyle\times$}}
\put(220.8,69.3){{$\circ$}}
\put(239.7,70.1){{$\scriptstyle\times$}}
\put(-59.2,29.3){{$\circ$}}
\put(-40.3,30.1){{$\scriptstyle\times$}}
\put(80.8,9.3){{$\circ$}}
\put(99.7,10.1){{$\scriptstyle\times$}}
\put(80.8,89.3){{$\circ$}}
\put(99.7,90.1){{$\scriptstyle\times$}}

\put(61,90){{$\scriptstyle\bullet$}}
\put(61,70){{$\scriptstyle\bullet$}}
\put(101,70){{$\scriptstyle\bullet$}}
\put(81,70){{$\scriptstyle\bullet$}}
\put(121,90){{$\scriptstyle\bullet$}}
\put(121,70){{$\scriptstyle\bullet$}}
\put(121,110){{$\scriptstyle\bullet$}}
\put(101,110){{$\scriptstyle\bullet$}}
\put(81,110){{$\scriptstyle\bullet$}}
\put(61,110){{$\scriptstyle\bullet$}}
\put(101,30){{$\scriptstyle\bullet$}}
\put(81,30){{$\scriptstyle\bullet$}}
\put(61,30){{$\scriptstyle\bullet$}}
\put(121,30){{$\scriptstyle\bullet$}}
\put(61,10){{$\scriptstyle\bullet$}}
\put(121,10){{$\scriptstyle\bullet$}}

\put(261,30){{$\scriptstyle\bullet$}}
\put(201,30){{$\scriptstyle\bullet$}}
\put(261,70){{$\scriptstyle\bullet$}}
\put(201,70){{$\scriptstyle\bullet$}}
\put(261,90){{$\scriptstyle\bullet$}}
\put(201,90){{$\scriptstyle\bullet$}}
\put(241,90){{$\scriptstyle\bullet$}}
\put(221,90){{$\scriptstyle\bullet$}}

\put(-19,30){{$\scriptstyle\bullet$}}
\put(-79,30){{$\scriptstyle\bullet$}}
\put(-19,50){{$\scriptstyle\bullet$}}
\put(-79,50){{$\scriptstyle\bullet$}}
\put(-39,50){{$\scriptstyle\bullet$}}
\put(-59,50){{$\scriptstyle\bullet$}}
\put(-19,90){{$\scriptstyle\bullet$}}
\put(-79,90){{$\scriptstyle\bullet$}}
\put(-59,90){{$\scriptstyle\bullet$}}
\put(-39,90){{$\scriptstyle\bullet$}}

\put(-95,100){$d$}
\put(-95,80){$b^*$}
\put(-95,60){$b$}
\put(-95,40){$s$}
\put(-95,20){$a$}
\put(-82,92){\line(1,0){70}}
\put(-82,52){\line(1,0){70}}
\put(-82,32){\line(1,0){70}}
\dashline{2.5}(-82,72)(-12,72)
\put(-67,92){\oval(20,20)[t]}
\put(-27,92){\oval(20,20)[t]}
\put(-67,92){\oval(20,20)[b]}
\put(-27,92){\oval(20,20)[b]}
\put(-27,52){\oval(20,20)[t]}
\put(-67,52){\oval(20,20)[t]}
\put(-47,52){\oval(20,20)[b]}
\put(-77,52){\line(0,-1){20}}
\put(-17,52){\line(0,-1){20}}
\put(-47,32){\oval(60,30)[b]}

\put(10,60){$\rightsquigarrow$}
\put(149,60){$\rightsquigarrow$}

\put(45,120){$d$}
\put(45,100){$s$}
\put(45,80){$s^*$}
\put(45,60){$b^*$}
\put(45,40){$b$}
\put(45,20){$s$}
\put(45,0){$a$}
\put(58,112){\line(1,0){70}}
\put(58,92){\line(1,0){70}}
\put(58,72){\line(1,0){70}}
\put(58,32){\line(1,0){70}}
\put(58,12){\line(1,0){70}}
\dashline{2.5}(58,52)(128,52)
\put(113,112){\oval(20,20)[t]}
\put(73,112){\oval(20,20)[t]}
\put(93,112){\oval(20,20)[b]}
\put(63,112){\line(0,-1){40}}
\put(123,112){\line(0,-1){40}}
\put(93,72){\oval(20,20)[t]}
\put(73,72){\oval(20,20)[b]}
\put(113,72){\oval(20,20)[b]}
\put(73,32){\oval(20,20)[t]}
\put(113,32){\oval(20,20)[t]}
\put(93,32){\oval(20,20)[b]}
\put(63,32){\line(0,-1){20}}
\put(123,32){\line(0,-1){20}}
\put(93,12){\oval(60,30)[b]}

\put(185,100){$d$}
\put(185,80){$s$}
\put(185,60){$c^*$}
\put(185,40){$c$}
\put(185,20){$a$}
\put(198,92){\line(1,0){70}}
\put(198,72){\line(1,0){70}}
\put(198,32){\line(1,0){70}}
\dashline{2.5}(198,52)(268,52)
\put(253,92){\oval(20,20)[t]}
\put(213,92){\oval(20,20)[t]}
\put(233,92){\oval(20,20)[b]}
\put(203,92){\line(0,-1){20}}
\put(263,92){\line(0,-1){20}}
\put(233,72){\oval(60,22)[b]}
\put(233,32){\oval(60,22)[t]}
\put(233,32){\oval(60,30)[b]}
\end{picture}
$$
The point now is that to compute (\ref{lhd2}) involves two surgery
procedures to the $bb^*$-part of the leftmost diagram.
Each of these involves two circles combining into one circle, represented
by the multiplication $m$ of the
Frobenius algebra $R$.
So the computation of (\ref{lhd2})
is a map represented by $m \circ (m \otimes 1)$.
On the other hand the passage from (\ref{lhd2}) to (\ref{lhd}) involves
applying the map $1 \otimes m$ to join the top two circles, then one surgery
procedure gets performed to the $cc^*$-part of the rightmost diagram,
which is the map $m$. Now we are done because $m \circ (m \otimes 1) = m \circ (1 \otimes m)$.

\vspace{1mm}

\noindent{\em Sub-case two: $b^*s^* s \bt$ has one more component
than $b^* \bt$ and $sb$ does not have an internal circle.}
Making the same reductions as before, there is essentially only one basic configuration to consider, represented by the following diagram:
$$
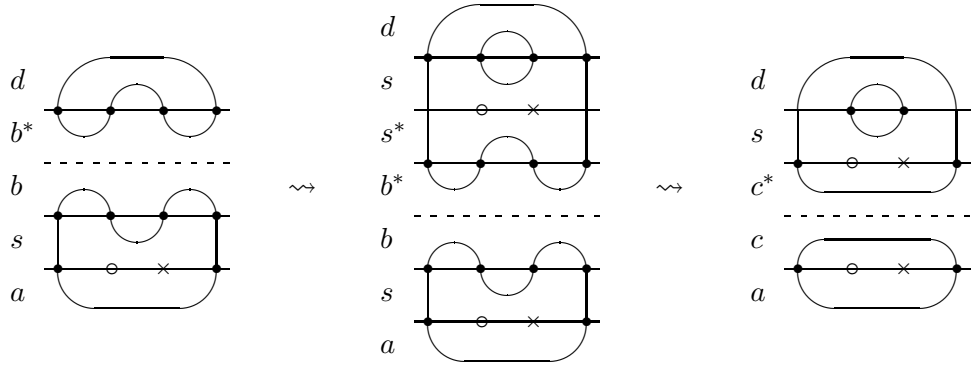
\begin{picture}(172,132)
\put(220.8,29.3){{$\circ$}}
\put(239.7,30.1){{$\scriptstyle\times$}}
\put(220.8,69.3){{$\circ$}}
\put(239.7,70.1){{$\scriptstyle\times$}}

\put(-59.2,29.3){{$\circ$}}
\put(-40.3,30.1){{$\scriptstyle\times$}}
\put(80.8,9.3){{$\circ$}}
\put(99.7,10.1){{$\scriptstyle\times$}}
\put(80.8,89.3){{$\circ$}}
\put(99.7,90.1){{$\scriptstyle\times$}}

\put(-19,30){{$\scriptstyle\bullet$}}
\put(-79,30){{$\scriptstyle\bullet$}}
\put(-19,50){{$\scriptstyle\bullet$}}
\put(-79,50){{$\scriptstyle\bullet$}}
\put(-39,50){{$\scriptstyle\bullet$}}
\put(-59,50){{$\scriptstyle\bullet$}}
\put(-19,90){{$\scriptstyle\bullet$}}
\put(-79,90){{$\scriptstyle\bullet$}}
\put(-59,90){{$\scriptstyle\bullet$}}
\put(-39,90){{$\scriptstyle\bullet$}}

\put(61,10){{$\scriptstyle\bullet$}}
\put(121,10){{$\scriptstyle\bullet$}}
\put(81,30){{$\scriptstyle\bullet$}}
\put(101,30){{$\scriptstyle\bullet$}}
\put(61,30){{$\scriptstyle\bullet$}}
\put(121,30){{$\scriptstyle\bullet$}}
\put(121,110){{$\scriptstyle\bullet$}}
\put(81,110){{$\scriptstyle\bullet$}}
\put(61,110){{$\scriptstyle\bullet$}}
\put(101,110){{$\scriptstyle\bullet$}}
\put(121,70){{$\scriptstyle\bullet$}}
\put(81,70){{$\scriptstyle\bullet$}}
\put(61,70){{$\scriptstyle\bullet$}}
\put(101,70){{$\scriptstyle\bullet$}}

\put(221,90){{$\scriptstyle\bullet$}}
\put(241,90){{$\scriptstyle\bullet$}}
\put(261,70){{$\scriptstyle\bullet$}}
\put(201,70){{$\scriptstyle\bullet$}}
\put(261,30){{$\scriptstyle\bullet$}}
\put(201,30){{$\scriptstyle\bullet$}}

\put(-95,100){$d$}
\put(-95,80){$b^*$}
\put(-95,60){$b$}
\put(-95,40){$s$}
\put(-95,20){$a$}
\put(-82,92){\line(1,0){70}}
\put(-82,52){\line(1,0){70}}
\put(-82,32){\line(1,0){70}}
\dashline{2.5}(-82,72)(-12,72)
\put(-47,92){\oval(20,20)[t]}
\put(-47,92){\oval(60,40)[t]}
\put(-47,32){\oval(60,30)[b]}
\put(-47,52){\oval(20,20)[b]}
\put(-67,52){\oval(20,20)[t]}
\put(-27,52){\oval(20,20)[t]}
\put(-67,92){\oval(20,20)[b]}
\put(-27,92){\oval(20,20)[b]}
\put(-77,52){\line(0,-1){20}}
\put(-17,52){\line(0,-1){20}}

\put(10,60){$\rightsquigarrow$}
\put(149,60){$\rightsquigarrow$}

\put(45,120){$d$}
\put(45,100){$s$}
\put(45,80){$s^*$}
\put(45,60){$b^*$}
\put(45,40){$b$}
\put(45,20){$s$}
\put(45,0){$a$}
\put(58,112){\line(1,0){70}}
\put(58,92){\line(1,0){70}}
\put(58,72){\line(1,0){70}}
\put(58,32){\line(1,0){70}}
\put(58,12){\line(1,0){70}}
\dashline{2.5}(58,52)(128,52)
\put(93,112){\oval(20,20)[t]}
\put(93,112){\oval(60,40)[t]}
\put(93,12){\oval(60,30)[b]}
\put(93,32){\oval(20,20)[b]}
\put(73,32){\oval(20,20)[t]}
\put(113,32){\oval(20,20)[t]}
\put(73,72){\oval(20,20)[b]}
\put(113,72){\oval(20,20)[b]}
\put(93,72){\oval(20,20)[t]}
\put(93,112){\oval(20,20)[b]}
\put(63,32){\line(0,-1){20}}
\put(123,32){\line(0,-1){20}}
\put(63,112){\line(0,-1){40}}
\put(123,112){\line(0,-1){40}}

\put(185,100){$d$}
\put(185,80){$s$}
\put(185,60){$c^*$}
\put(185,40){$c$}
\put(185,20){$a$}
\put(198,92){\line(1,0){70}}
\put(198,72){\line(1,0){70}}
\put(198,32){\line(1,0){70}}
\dashline{2.5}(198,52)(268,52)
\put(233,92){\oval(20,20)[t]}
\put(233,92){\oval(60,40)[t]}
\put(233,32){\oval(60,30)[b]}
\put(233,32){\oval(60,22)[t]}
\put(233,72){\oval(60,22)[b]}
\put(233,92){\oval(20,20)[b]}
\put(203,92){\line(0,-1){20}}
\put(263,92){\line(0,-1){20}}
\end{picture}
$$
The direct computation of (\ref{lhd2}) from the leftmost diagram
involves two surgery procedures
encoded by the map $\delta \circ m$.
On the other hand the passage from (\ref{lhd2}) to (\ref{lhd1})
is encoded by the map $1 \otimes \delta$, then the computation of
(\ref{lhd1}) is encoded by $m \otimes 1$.
So we are done by the
identity $(m \otimes 1) \circ (1 \otimes \delta) = \delta \circ m$, which is
the basic defining relation of a Frobenius algebra.
\vspace{1mm}

\noindent{\em Sub-case three: $b^*s^* s \bt$ has one more component
than $b^* \bt$ and $sb$ has an internal circle.}
The basic configuration to be considered is as follows:
$$
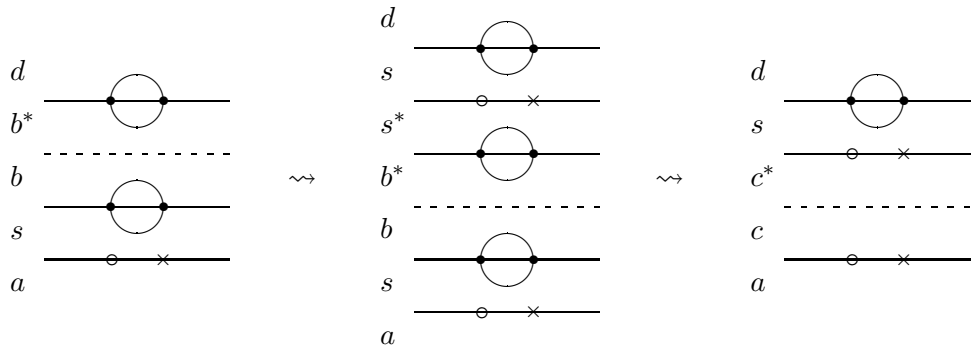
\begin{picture}(172,132)
\put(221,90){{$\scriptstyle\bullet$}}
\put(241,90){{$\scriptstyle\bullet$}}

\put(81,110){{$\scriptstyle\bullet$}}
\put(101,110){{$\scriptstyle\bullet$}}
\put(81,30){{$\scriptstyle\bullet$}}
\put(101,70){{$\scriptstyle\bullet$}}
\put(81,70){{$\scriptstyle\bullet$}}
\put(101,30){{$\scriptstyle\bullet$}}

\put(-59,90){{$\scriptstyle\bullet$}}
\put(-39,90){{$\scriptstyle\bullet$}}
\put(-59,50){{$\scriptstyle\bullet$}}
\put(-39,50){{$\scriptstyle\bullet$}}

\put(-59.2,29.3){{$\circ$}}
\put(-40.3,30.1){{$\scriptstyle\times$}}
\put(80.8,9.3){{$\circ$}}
\put(99.7,10.1){{$\scriptstyle\times$}}
\put(80.8,89.3){{$\circ$}}
\put(99.7,90.1){{$\scriptstyle\times$}}
\put(220.8,69.3){{$\circ$}}
\put(239.7,70.1){{$\scriptstyle\times$}}
\put(220.8,29.3){{$\circ$}}
\put(239.7,30.1){{$\scriptstyle\times$}}

\put(-95,100){$d$}
\put(-95,80){$b^*$}
\put(-95,60){$b$}
\put(-95,40){$s$}
\put(-95,20){$a$}
\put(-82,92){\line(1,0){70}}
\put(-82,52){\line(1,0){70}}
\put(-82,32){\line(1,0){70}}
\dashline{2.5}(-82,72)(-12,72)
\put(-47,52){\oval(20,20)[b]}
\put(-47,52){\oval(20,20)[t]}
\put(-47,92){\oval(20,20)[b]}
\put(-47,92){\oval(20,20)[t]}

\put(10,60){$\rightsquigarrow$}
\put(149,60){$\rightsquigarrow$}

\put(45,120){$d$}
\put(45,100){$s$}
\put(45,80){$s^*$}
\put(45,60){$b^*$}
\put(45,40){$b$}
\put(45,20){$s$}
\put(45,0){$a$}
\put(58,112){\line(1,0){70}}
\put(58,92){\line(1,0){70}}
\put(58,72){\line(1,0){70}}
\put(58,32){\line(1,0){70}}
\put(58,12){\line(1,0){70}}
\dashline{2.5}(58,52)(128,52)
\put(93,72){\oval(20,20)[b]}
\put(93,72){\oval(20,20)[t]}
\put(93,112){\oval(20,20)[b]}
\put(93,112){\oval(20,20)[t]}
\put(93,32){\oval(20,20)[b]}
\put(93,32){\oval(20,20)[t]}

\put(185,100){$d$}
\put(185,80){$s$}
\put(185,60){$c^*$}
\put(185,40){$c$}
\put(185,20){$a$}
\put(198,92){\line(1,0){70}}
\put(198,72){\line(1,0){70}}
\put(198,32){\line(1,0){70}}
\dashline{2.5}(198,52)(268,52)
\put(233,92){\oval(20,20)[b]}
\put(233,92){\oval(20,20)[t]}
\end{picture}
$$
Here, the direct computation of (\ref{lhd2})
corresponds to the map $m$.
On the other hand
passing from (\ref{lhd2}) to (\ref{lhd})
corresponds to the map $1 \otimes \delta$, then passing
from (\ref{lhd}) to (\ref{lhd1}) involves removing a
mirror image pair of internal circles, which is encoded by the
map $(\eps \circ m) \otimes 1$,
where $(\eps \circ m)$ is the map $R \otimes R \rightarrow \C,
1 \otimes 1 \mapsto 0, 1 \otimes x \mapsto 1,
x \otimes 1 \mapsto 1, x \otimes x \mapsto 0$.
Again we are done since
$((\eps \circ m) \bar\otimes 1) \circ (1 \otimes \delta) = m$.
\end{proof}

\begin{Corollary}\label{cddcd}
The following diagram commutes:
$$
\begin{CD}
R^\La_\alpha&@>\theta_i>>&R^\La_{\alpha+\alpha_i}\\
@V\omega VV&&@VV\omega V\\
E^\La_\alpha&@>>\theta_i>&E^\La_{\alpha+\alpha_i}.
\end{CD}
$$
\end{Corollary}

\begin{proof}
In view of the theorem, it suffices to check the commutativity on the generators
of $R^\La_\alpha$. For $e(\bi)$ this follows from (\ref{imei}).
It then suffices to check that
$\theta_i(y_r \omega(e(\bi))) = y_r\omega(e(\bi+i))$
and
$\theta_i(\psi_r\omega(e(\bi))) = \psi_r\omega(e(\bi+i))$.
This is routine for circle moves or height moves.
Finally if $\overline\psi_r(\bi)_\La$ is a crossing move
there are various
cases to be considered, all of which
 reduce to the
associativity/coassociativity of the multiplication/comultiplication in $R$
or to the identity
$(m \otimes 1) \circ (1 \otimes \delta) = \delta \circ m$
like in the proof of Theorem~\ref{thetaithm}.
\end{proof}

\section{Surjectivity}

In this section we prove that the homomorphism
$\omega:R^\La_\alpha \rightarrow E^\La_\alpha$
from Theorem~\ref{phi} is surjective.

\phantomsubsection{Raising and lowering moves}
Assume in this subsection that we are given a block $\Ga\in\PImn$ and a
proper
$\Ga$-admissible sequence
$\bi = (i_1,\dots,i_d)$,
and let $\bGa = \Ga_d \cdots \Ga_0$ and $\bt = t_d \cdots t_1$ be the
associated block sequence and composite matching.
By composing sequences of the local circle moves, crossing moves and
height moves
from (\ref{expy}) and (\ref{betadef2}),
we can define various ``global''
bimodule homomorphisms
$K^{\bt}_{\bGa}\langle-\caps(\bt)\rangle
\rightarrow K^{\bu}_{\bDe}\langle -\caps(\bu)\rangle$
for other
block sequences and composite matchings $\bDe$ and $\bu$.
The following lemmas describe some useful
raising and lowering moves which can be obtained in this
way.
The proofs of all of these are straightforward consequences of
Lemma~\ref{as}.
The reader may find Figure~\ref{fig1} helpful when trying to visualise the statements;
for example consider applying
 Lemma~\ref{raisingshifts} to raise the left-shift at level $14$
of this diagram.

\begin{Lemma}[Raising caps]\label{raisingcaps}
Let $C$ be a component of $\bt$ and $1 \leq q \leq d$
such that $C$ has a cap at level $q$.
Exactly one of the following two statements holds:
\begin{itemize}
\item[(i)] $|i_p-i_q| > 1$ for all $1 \leq p < q$;
\item[(ii)] there exists $1 \leq p < q$ such that
$|i_p-i_{q}| = 1$ and $|i_s-i_q| > 1$ for all $p < s < q$.
\end{itemize}
In case {\rm (i)}, the composition $\overline\psi_1
\circ \cdots \circ \overline\psi_{q-1}$
is a sequence of height moves raising the cap
to level $1$.
In case {\rm (ii)}, $\overline\psi_{p+1} \circ\cdots \circ \overline\psi_{q-1}$
is a sequence  of height moves
raising the cap to level $p+1$,
and after that $\overline\psi_p$ is a
crossing move of one of the following four types:
$$
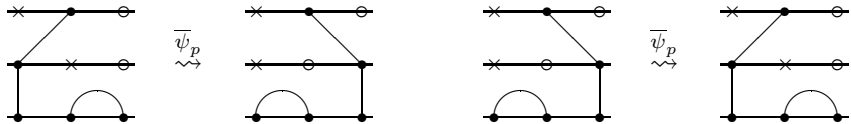
\begin{picture}(160,51)
\put(-82,42){\line(1,0){48}}
\put(-82,22){\line(1,0){48}}
\put(-82,2){\line(1,0){48}}
\put(-58,42){\line(-1,-1){20}}
\put(-81,40){{$\scriptstyle\times$}}
\put(-60,40){{$\scriptstyle\bullet$}}
\put(-40.7,39.3){{$\circ$}}
\put(-78,22){\line(0,-1){20}}
\put(-80,20){$\scriptstyle\bullet$}
\put(-61,20){$\scriptstyle\times$}
\put(-40.7,19.3){{$\circ$}}
\put(-80,0){$\scriptstyle\bullet$}
\put(-60,0){$\scriptstyle\bullet$}
\put(-40,0){{$\scriptstyle\bullet$}}
\put(-48,2){\oval(20,20)[t]}

\put(-19,19.5){$\stackrel{\overline\psi_p}{\rightsquigarrow}$}

\put(8,42){\line(1,0){48}}
\put(8,22){\line(1,0){48}}
\put(8,2){\line(1,0){48}}
\put(32,42){\line(1,-1){20}}
\put(9,40){{$\scriptstyle\times$}}
\put(30,40){{$\scriptstyle\bullet$}}
\put(49.3,39.3){{$\circ$}}
\put(52,22){\line(0,-1){20}}
\put(9,20){$\scriptstyle\times$}
\put(29.3,19.3){$\circ$}
\put(50,20){{$\scriptstyle\bullet$}}
\put(10,0){$\scriptstyle\bullet$}
\put(30,0){$\scriptstyle\bullet$}
\put(50,0){{$\scriptstyle\bullet$}}
\put(22,2){\oval(20,20)[t]}

\put(188,42){\line(1,0){48}}
\put(188,22){\line(1,0){48}}
\put(188,2){\line(1,0){48}}
\put(212,42){\line(-1,-1){20}}
\put(189,40){{$\scriptstyle\times$}}
\put(210,40){{$\scriptstyle\bullet$}}
\put(229.3,39.3){{$\circ$}}
\put(192,22){\line(0,-1){20}}
\put(190,20){$\scriptstyle\bullet$}
\put(209,20){$\scriptstyle\times$}
\put(229.3,19.3){{$\circ$}}
\put(190,0){$\scriptstyle\bullet$}
\put(210,0){$\scriptstyle\bullet$}
\put(230,0){{$\scriptstyle\bullet$}}
\put(222,2){\oval(20,20)[t]}

\put(161,19.5){$\stackrel{\overline\psi_p}{\rightsquigarrow}$}

\put(98,42){\line(1,0){48}}
\put(98,22){\line(1,0){48}}
\put(98,2){\line(1,0){48}}
\put(122,42){\line(1,-1){20}}
\put(99,40){{$\scriptstyle\times$}}
\put(120,40){{$\scriptstyle\bullet$}}
\put(139.3,39.3){{$\circ$}}
\put(142,22){\line(0,-1){20}}
\put(99,20){$\scriptstyle\times$}
\put(119.3,19.3){$\circ$}
\put(140,20){{$\scriptstyle\bullet$}}
\put(100,0){$\scriptstyle\bullet$}
\put(120,0){$\scriptstyle\bullet$}
\put(140,0){{$\scriptstyle\bullet$}}
\put(112,2){\oval(20,20)[t]}
\end{picture}
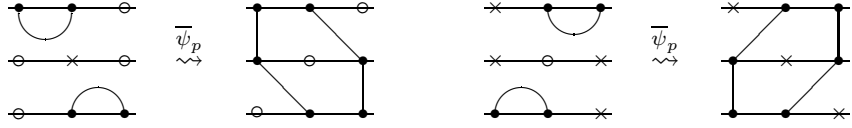
$$
$$
\begin{picture}(160,55)

\put(-82,42){\line(1,0){48}}
\put(-82,22){\line(1,0){48}}
\put(-82,2){\line(1,0){48}}
\put(-68,40){\oval(20,20)[b]}
\put(-80,40){{$\scriptstyle\bullet$}}
\put(-60,40){{$\scriptstyle\bullet$}}
\put(-40.7,39.3){{$\circ$}}
\put(-80.7,19.3){$\circ$}
\put(-61,20){$\scriptstyle\times$}
\put(-40.7,19.3){{$\circ$}}
\put(-80.7,-0.7){$\circ$}
\put(-60,0){$\scriptstyle\bullet$}
\put(-40,0){{$\scriptstyle\bullet$}}
\put(-48,2){\oval(20,20)[t]}

\put(188,42){\line(1,0){48}}
\put(188,22){\line(1,0){48}}
\put(188,2){\line(1,0){48}}
\put(212,42){\line(-1,-1){20}}
\put(189,40){{$\scriptstyle\times$}}
\put(210,40){{$\scriptstyle\bullet$}}
\put(230,40){{$\scriptstyle\bullet$}}
\put(232,42){\line(0,-1){20}}
\put(192,22){\line(0,-1){20}}
\put(190,20){$\scriptstyle\bullet$}
\put(209,20){$\scriptstyle\times$}
\put(230,20){{$\scriptstyle\bullet$}}
\put(232,22){\line(-1,-1){20}}
\put(190,0){$\scriptstyle\bullet$}
\put(210,0){$\scriptstyle\bullet$}
\put(229,0){{$\scriptstyle\times$}}
\put(161,19.5){$\stackrel{\overline\psi_p}{\rightsquigarrow}$}

\put(98,42){\line(1,0){48}}
\put(98,22){\line(1,0){48}}
\put(98,2){\line(1,0){48}}
\put(99,40){{$\scriptstyle\times$}}
\put(139,20){{$\scriptstyle\times$}}
\put(139,0){{$\scriptstyle\times$}}
\put(120,40){{$\scriptstyle\bullet$}}
\put(140,40){{$\scriptstyle\bullet$}}
\put(99,20){$\scriptstyle\times$}
\put(119.3,19.3){$\circ$}
\put(100,0){$\scriptstyle\bullet$}
\put(120,0){$\scriptstyle\bullet$}
\put(112,2){\oval(20,20)[t]}
\put(132,42){\oval(20,20)[b]}

\put(-19,19.5){$\stackrel{\overline\psi_p}{\rightsquigarrow}$}

\put(8,42){\line(1,0){48}}
\put(8,22){\line(1,0){48}}
\put(8,2){\line(1,0){48}}
\put(32,42){\line(1,-1){20}}
\put(12,22){\line(1,-1){20}}
\put(30,40){{$\scriptstyle\bullet$}}
\put(10,40){{$\scriptstyle\bullet$}}
\put(49.3,39.3){{$\circ$}}
\put(52,22){\line(0,-1){20}}
\put(12,42){\line(0,-1){20}}
\put(29.3,19.3){$\circ$}
\put(50,20){{$\scriptstyle\bullet$}}
\put(10,20){{$\scriptstyle\bullet$}}
\put(9.3,0){$\circ$}
\put(30,0){$\scriptstyle\bullet$}
\put(50,0){{$\scriptstyle\bullet$}}
\end{picture}
$$
\end{Lemma}

\begin{Lemma}[Lowering cups]\label{loweringcups}
Let $C$ be a component of $\bt$ and $1 \leq p \leq d$
such that $C$ has a cup at level $p$.
Exactly one of the following two statements holds:
\begin{itemize}
\item[(i)] $|i_p-i_q| > 1$ for all $p < q \leq d$;
\item[(ii)] there exists $p < q \leq d$ such that
$|i_p-i_{q}| = 1$ and $|i_p-i_s| > 1$ for all $p < s < q$.
\end{itemize}
In case {\rm (i)}, the composition $\overline\psi_{d-1} \circ \cdots \circ \overline\psi_{p}$
is a sequence of height moves lowering the cup
to level $d$.
In case {\rm (ii)}, $\overline\psi_{q-2} \circ\cdots \circ \overline\psi_{p}$
is a sequence of height moves
lowering the cup to level $q-1$,
and after that $\overline\psi_{q-1}$ is a crossing move
of one of the following four types:
$$
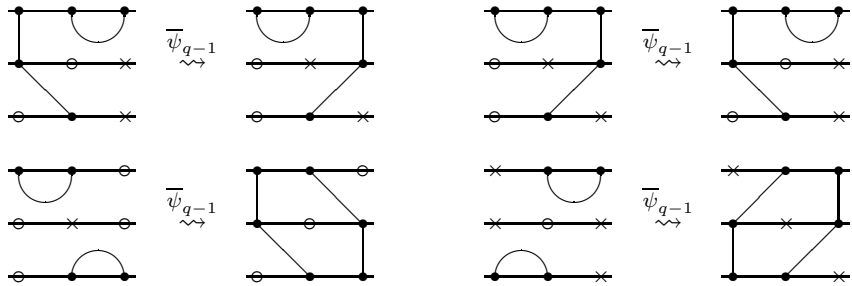
\begin{picture}(160,51)
\put(-82,42){\line(1,0){48}}
\put(-82,22){\line(1,0){48}}
\put(-82,2){\line(1,0){48}}
\put(-48,40){\oval(20,20)[b]}
\put(-80,40){{$\scriptstyle\bullet$}}
\put(-60,40){{$\scriptstyle\bullet$}}
\put(-40,40){{$\scriptstyle\bullet$}}
\put(-78,42){\line(0,-1){20}}
\put(-80,20){$\scriptstyle\bullet$}
\put(-60.7,19.3){$\circ$}
\put(-41,20){{$\scriptstyle\times$}}
\put(-80.7,-0.7){$\circ$}
\put(-60,0){$\scriptstyle\bullet$}
\put(-41,0){{$\scriptstyle\times$}}
\put(-78,22){\line(1,-1){20}}

\put(-22,19.5){$\stackrel{\overline\psi_{q-1}}{\rightsquigarrow}$}

\put(8,42){\line(1,0){48}}
\put(8,22){\line(1,0){48}}
\put(8,2){\line(1,0){48}}
\put(22,40){\oval(20,20)[b]}
\put(10,40){{$\scriptstyle\bullet$}}
\put(30,40){{$\scriptstyle\bullet$}}
\put(50,40){{$\scriptstyle\bullet$}}
\put(52,42){\line(0,-1){20}}
\put(9.3,19.3){$\circ$}
\put(29,20){$\scriptstyle\times$}
\put(50,20){{$\scriptstyle\bullet$}}
\put(9.3,-0.7){$\circ$}
\put(30,0){$\scriptstyle\bullet$}
\put(49,0){{$\scriptstyle\times$}}
\put(52,22){\line(-1,-1){20}}

\put(188,42){\line(1,0){48}}
\put(188,22){\line(1,0){48}}
\put(188,2){\line(1,0){48}}
\put(222,40){\oval(20,20)[b]}
\put(190,40){{$\scriptstyle\bullet$}}
\put(210,40){{$\scriptstyle\bullet$}}
\put(230,40){{$\scriptstyle\bullet$}}
\put(192,42){\line(0,-1){20}}
\put(190,20){$\scriptstyle\bullet$}
\put(209.3,19.3){$\circ$}
\put(229,20){{$\scriptstyle\times$}}
\put(189.3,-0.7){$\circ$}
\put(210,0){$\scriptstyle\bullet$}
\put(229,0){{$\scriptstyle\times$}}
\put(192,22){\line(1,-1){20}}

\put(158,19.5){$\stackrel{\overline\psi_{q-1}}{\rightsquigarrow}$}

\put(98,42){\line(1,0){48}}
\put(98,22){\line(1,0){48}}
\put(98,2){\line(1,0){48}}
\put(112,40){\oval(20,20)[b]}
\put(100,40){{$\scriptstyle\bullet$}}
\put(120,40){{$\scriptstyle\bullet$}}
\put(140,40){{$\scriptstyle\bullet$}}
\put(142,42){\line(0,-1){20}}
\put(99.3,19.3){$\circ$}
\put(119,20){$\scriptstyle\times$}
\put(140,20){{$\scriptstyle\bullet$}}
\put(99.3,-0.7){$\circ$}
\put(120,0){$\scriptstyle\bullet$}
\put(139,0){{$\scriptstyle\times$}}
\put(142,22){\line(-1,-1){20}}
\end{picture}
$$
$$
\begin{picture}(160,55)

\put(-82,42){\line(1,0){48}}
\put(-82,22){\line(1,0){48}}
\put(-82,2){\line(1,0){48}}
\put(-68,40){\oval(20,20)[b]}
\put(-80,40){{$\scriptstyle\bullet$}}
\put(-60,40){{$\scriptstyle\bullet$}}
\put(-40.7,39.3){{$\circ$}}
\put(-80.7,19.3){$\circ$}
\put(-61,20){$\scriptstyle\times$}
\put(-40.7,19.3){{$\circ$}}
\put(-80.7,-0.7){$\circ$}
\put(-60,0){$\scriptstyle\bullet$}
\put(-40,0){{$\scriptstyle\bullet$}}
\put(-48,2){\oval(20,20)[t]}

\put(-22,19.5){$\stackrel{\overline\psi_{q-1}}{\rightsquigarrow}$}

\put(8,42){\line(1,0){48}}
\put(8,22){\line(1,0){48}}
\put(8,2){\line(1,0){48}}
\put(32,42){\line(1,-1){20}}
\put(10,40){{$\scriptstyle\bullet$}}
\put(30,40){{$\scriptstyle\bullet$}}
\put(49.3,39.3){{$\circ$}}
\put(12,42){\line(0,-1){20}}
\put(52,22){\line(0,-1){20}}
\put(10,20){$\scriptstyle\bullet$}
\put(29.3,19.3){$\circ$}
\put(50,20){{$\scriptstyle\bullet$}}
\put(12,22){\line(1,-1){20}}
\put(9.3,-0.7){$\circ$}
\put(30,0){$\scriptstyle\bullet$}
\put(50,0){{$\scriptstyle\bullet$}}

\put(98,42){\line(1,0){48}}
\put(98,22){\line(1,0){48}}
\put(98,2){\line(1,0){48}}
\put(132,40){\oval(20,20)[b]}
\put(99,40){{$\scriptstyle\times$}}
\put(120,40){{$\scriptstyle\bullet$}}
\put(140,40){{$\scriptstyle\bullet$}}
\put(99,20){$\scriptstyle\times$}
\put(119.3,19.3){$\circ$}
\put(139,20){{$\scriptstyle\times$}}
\put(100,0){$\scriptstyle\bullet$}
\put(120,0){$\scriptstyle\bullet$}
\put(139,0){{$\scriptstyle\times$}}
\put(112,2){\oval(20,20)[t]}

\put(158,19.5){$\stackrel{\overline\psi_{q-1}}{\rightsquigarrow}$}

\put(188,42){\line(1,0){48}}
\put(188,22){\line(1,0){48}}
\put(188,2){\line(1,0){48}}
\put(212,42){\line(-1,-1){20}}
\put(189,40){{$\scriptstyle\times$}}
\put(210,40){{$\scriptstyle\bullet$}}
\put(230,40){{$\scriptstyle\bullet$}}
\put(232,42){\line(0,-1){20}}
\put(192,22){\line(0,-1){20}}
\put(190,20){$\scriptstyle\bullet$}
\put(209,20){$\scriptstyle\times$}
\put(230,20){{$\scriptstyle\bullet$}}
\put(232,22){\line(-1,-1){20}}
\put(190,0){$\scriptstyle\bullet$}
\put(210,0){$\scriptstyle\bullet$}
\put(229,0){{$\scriptstyle\times$}}
\end{picture}
$$
\end{Lemma}

\begin{Lemma}[Raising shifts]\label{raisingshifts}
Let $C$ be a component of $\bt$ and $i \in I$
be minimal (resp.\ maximal) such that $C$
passes through the $i$th
(resp. $(i+1)$th)
vertex of some number line.
Suppose we are given $1 \leq q \leq d$ such that
$i_q = i$ and $C$
has a right-shift (resp.\ left-shift) at level $q$.
Exactly one of the following two statements hold:
\begin{itemize}
\item[(i)] $|i_p-i_q| > 1$ for all $1 \leq p < q$;
\item[(ii)] there exists $1 \leq p < q$ such that $i_p = i_q+1$
(resp.\ $i_p = i_q-1$)
and $|i_s-i_q| > 1$ for all $p < s < q$.
\end{itemize}
In case (i), the composition $\overline\psi_1 \circ \cdots \circ \overline\psi_{q-1}$
is a sequence of height moves raising the
right-shift (resp.\ left-shift) to level $1$.
In case (ii),
the composition $\overline\psi_{p+1}\circ\cdots\circ\overline\psi_{q-1}$
is a sequence of height moves raising the
right-shift (resp.\ left-shift) to level $p+1$.
After that, with one exception, $\overline\psi_p$
is a crossing move
of one of the following types:
$$
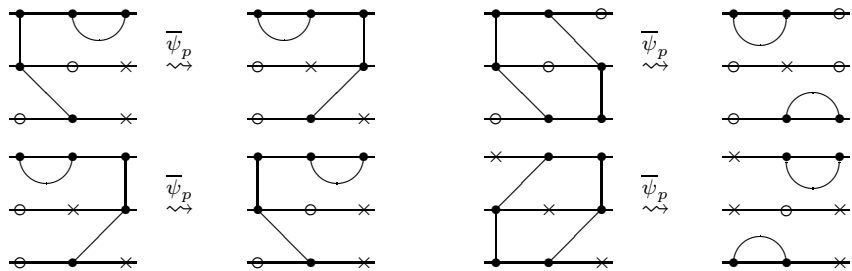
\begin{picture}(120,51)
\put(-82,42){\line(1,0){48}}
\put(-82,22){\line(1,0){48}}
\put(-82,2){\line(1,0){48}}
\put(-58,2){\line(-1,1){20}}
\put(-41,0){{$\scriptstyle\times$}}
\put(-60,0){{$\scriptstyle\bullet$}}
\put(-80.7,-0.7){{$\circ$}}
\put(-78,42){\line(0,-1){20}}
\put(-80,20){$\scriptstyle\bullet$}
\put(-41,20){$\scriptstyle\times$}
\put(-60.7,19.3){{$\circ$}}
\put(-80,40){$\scriptstyle\bullet$}
\put(-60,40){$\scriptstyle\bullet$}
\put(-40,40){{$\scriptstyle\bullet$}}
\put(-48,42){\oval(20,20)[b]}

\put(-23,19.5){$\stackrel{\overline\psi_{p}}{\rightsquigarrow}$}

\put(8,42){\line(1,0){48}}
\put(8,22){\line(1,0){48}}
\put(8,2){\line(1,0){48}}
\put(32,2){\line(1,1){20}}
\put(49,0){{$\scriptstyle\times$}}
\put(30,0){{$\scriptstyle\bullet$}}
\put(9.3,-0.7){{$\circ$}}
\put(52,42){\line(0,-1){20}}
\put(29,20){$\scriptstyle\times$}
\put(9.3,19.3){$\circ$}
\put(50,20){{$\scriptstyle\bullet$}}
\put(10,40){$\scriptstyle\bullet$}
\put(30,40){$\scriptstyle\bullet$}
\put(50,40){{$\scriptstyle\bullet$}}
\put(22,42){\oval(20,20)[b]}

\put(98,42){\line(1,0){48}}
\put(98,22){\line(1,0){48}}
\put(98,2){\line(1,0){48}}
\put(122,2){\line(-1,1){20}}
\put(139.3,39.3){{$\circ$}}
\put(120,40){{$\scriptstyle\bullet$}}
\put(100,40){{$\scriptstyle\bullet$}}
\put(142,22){\line(0,-1){20}}
\put(102,42){\line(0,-1){20}}
\put(100,20){$\scriptstyle\bullet$}
\put(119.3,19.3){$\circ$}
\put(140,20){{$\scriptstyle\bullet$}}
\put(142,22){\line(-1,1){20}}
\put(140,0){$\scriptstyle\bullet$}
\put(120,0){$\scriptstyle\bullet$}
\put(99.3,-0.7){{$\circ$}}

\put(157,19.5){$\stackrel{\overline\psi_{p}}{\rightsquigarrow}$}

\put(188,42){\line(1,0){48}}
\put(188,22){\line(1,0){48}}
\put(188,2){\line(1,0){48}}
\put(202,40){\oval(20,20)[b]}
\put(210,40){{$\scriptstyle\bullet$}}
\put(190,40){{$\scriptstyle\bullet$}}
\put(209,20){$\scriptstyle\times$}
\put(189.3,19.3){$\circ$}
\put(229.3,19.3){$\circ$}
\put(189.3,-0.7){$\circ$}
\put(229.3,39.3){$\circ$}
\put(230,0){$\scriptstyle\bullet$}
\put(210,0){$\scriptstyle\bullet$}
\put(222,2){\oval(20,20)[t]}
\end{picture}
$$
$$
\begin{picture}(120,51)
\put(-82,42){\line(1,0){48}}
\put(-82,22){\line(1,0){48}}
\put(-82,2){\line(1,0){48}}
\put(-58,2){\line(1,1){20}}
\put(-41,0){{$\scriptstyle\times$}}
\put(-60,0){{$\scriptstyle\bullet$}}
\put(-80.7,-0.7){{$\circ$}}
\put(-38,42){\line(0,-1){20}}
\put(-40,20){$\scriptstyle\bullet$}
\put(-61,20){$\scriptstyle\times$}
\put(-80.7,19.3){{$\circ$}}
\put(-40,40){$\scriptstyle\bullet$}
\put(-60,40){$\scriptstyle\bullet$}
\put(-80,40){{$\scriptstyle\bullet$}}
\put(-68,42){\oval(20,20)[b]}

\put(-23,19.5){$\stackrel{\overline\psi_{p}}{\rightsquigarrow}$}

\put(8,42){\line(1,0){48}}
\put(8,22){\line(1,0){48}}
\put(8,2){\line(1,0){48}}
\put(32,2){\line(-1,1){20}}
\put(49,0){{$\scriptstyle\times$}}
\put(30,0){{$\scriptstyle\bullet$}}
\put(9.3,-0.7){{$\circ$}}
\put(12,42){\line(0,-1){20}}
\put(49,20){$\scriptstyle\times$}
\put(29.3,19.3){$\circ$}
\put(10,20){{$\scriptstyle\bullet$}}
\put(10,40){$\scriptstyle\bullet$}
\put(30,40){$\scriptstyle\bullet$}
\put(50,40){{$\scriptstyle\bullet$}}
\put(42,42){\oval(20,20)[b]}

\put(98,42){\line(1,0){48}}
\put(98,22){\line(1,0){48}}
\put(98,2){\line(1,0){48}}
\put(122,2){\line(1,1){20}}
\put(99,40){{$\scriptstyle\times$}}
\put(120,40){{$\scriptstyle\bullet$}}
\put(140,40){{$\scriptstyle\bullet$}}
\put(102,22){\line(0,-1){20}}
\put(142,42){\line(0,-1){20}}
\put(140,20){$\scriptstyle\bullet$}
\put(119,20){$\scriptstyle\times$}
\put(100,20){{$\scriptstyle\bullet$}}
\put(102,22){\line(1,1){20}}
\put(100,0){$\scriptstyle\bullet$}
\put(120,0){$\scriptstyle\bullet$}
\put(139,0){{$\scriptstyle\times$}}

\put(157,19.5){$\stackrel{\overline\psi_{p}}{\rightsquigarrow}$}

\put(188,42){\line(1,0){48}}
\put(188,22){\line(1,0){48}}
\put(188,2){\line(1,0){48}}
\put(222,40){\oval(20,20)[b]}
\put(210,40){{$\scriptstyle\bullet$}}
\put(230,40){{$\scriptstyle\bullet$}}
\put(189,20){$\scriptstyle\times$}
\put(229,20){$\scriptstyle\times$}
\put(209.3,19){$\circ$}
\put(229,0){$\scriptstyle\times$}
\put(189,40){$\scriptstyle\times$}
\put(190,0){$\scriptstyle\bullet$}
\put(210,0){$\scriptstyle\bullet$}
\put(202,2){\oval(20,20)[t]}
\end{picture}
$$
The exception is if
$\Ga_{q}$
is a block of defect zero and
there is a right-shift (resp.\ left-shift)
at level $p$; this can happen only if the components of $\bt$
that are non-trivial at levels $p$ and $q$ are two
different propagating lines.
\end{Lemma}

\begin{Lemma}[Lowering shifts]\label{loweringshifts}
Let $C$ be a component of $\bt$ and $i \in I$
be minimal (resp.\ maximal) such that $C$ passes through the $i$th
(resp. $(i+1)$th)
vertex of some number line.
Suppose we are given $1 \leq p \leq d$ such that $i_p = i$ and $C$
has a left-shift (resp.\ right-shift) at level $p$.
Exactly one of the following two statements hold:
\begin{itemize}
\item[(i)] $|i_p-i_q| > 1$ for all $p < q \leq d$;
\item[(ii)] there exists $p < q \leq d$ such that $i_q = i_p+1$
(resp.\ $i_q = i_p-1$)
and $|i_p-i_s| > 1$ for all $p < s < q$.
\end{itemize}
In case (i), the composition $\overline\psi_{d-1}\circ\cdots\circ\overline\psi_p$
is a sequence of height moves lowering the
left-shift (resp.\ right-shift) to level $d$.
In case (ii),
the composition $\overline\psi_{q-2}\circ\cdots\circ\overline\psi_{p}$
is a sequence of height moves lowering the
left-shift (resp.\ right-shift) to level $q-1$.
After that, with one exception, $\overline\psi_{q-1}$
is a crossing move
of one of the following types:
$$
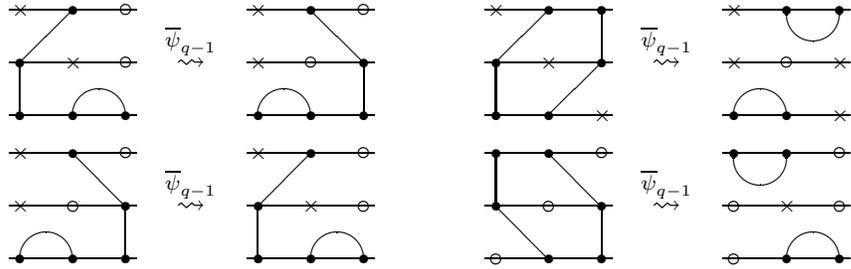
\begin{picture}(120,51)
\put(-82,42){\line(1,0){48}}
\put(-82,22){\line(1,0){48}}
\put(-82,2){\line(1,0){48}}
\put(-58,42){\line(-1,-1){20}}
\put(-81,40){{$\scriptstyle\times$}}
\put(-60,40){{$\scriptstyle\bullet$}}
\put(-40.7,39.3){{$\circ$}}
\put(-78,22){\line(0,-1){20}}
\put(-80,20){$\scriptstyle\bullet$}
\put(-61,20){$\scriptstyle\times$}
\put(-40.7,19.3){{$\circ$}}
\put(-80,0){$\scriptstyle\bullet$}
\put(-60,0){$\scriptstyle\bullet$}
\put(-40,0){{$\scriptstyle\bullet$}}
\put(-48,2){\oval(20,20)[t]}

\put(-23,19.5){$\stackrel{\overline\psi_{q-1}}{\rightsquigarrow}$}

\put(8,42){\line(1,0){48}}
\put(8,22){\line(1,0){48}}
\put(8,2){\line(1,0){48}}
\put(32,42){\line(1,-1){20}}
\put(9,40){{$\scriptstyle\times$}}
\put(30,40){{$\scriptstyle\bullet$}}
\put(49.3,39.3){{$\circ$}}
\put(52,22){\line(0,-1){20}}
\put(9,20){$\scriptstyle\times$}
\put(29.3,19.3){$\circ$}
\put(50,20){{$\scriptstyle\bullet$}}
\put(10,0){$\scriptstyle\bullet$}
\put(30,0){$\scriptstyle\bullet$}
\put(50,0){{$\scriptstyle\bullet$}}
\put(22,2){\oval(20,20)[t]}

\put(98,42){\line(1,0){48}}
\put(98,22){\line(1,0){48}}
\put(98,2){\line(1,0){48}}
\put(122,42){\line(-1,-1){20}}
\put(99,40){{$\scriptstyle\times$}}
\put(120,40){{$\scriptstyle\bullet$}}
\put(140,40){{$\scriptstyle\bullet$}}
\put(142,42){\line(0,-1){20}}
\put(102,22){\line(0,-1){20}}
\put(100,20){$\scriptstyle\bullet$}
\put(119,20){$\scriptstyle\times$}
\put(140,20){{$\scriptstyle\bullet$}}
\put(142,22){\line(-1,-1){20}}
\put(100,0){$\scriptstyle\bullet$}
\put(120,0){$\scriptstyle\bullet$}
\put(139,0){{$\scriptstyle\times$}}

\put(157,19.5){$\stackrel{\overline\psi_{q-1}}{\rightsquigarrow}$}

\put(188,42){\line(1,0){48}}
\put(188,22){\line(1,0){48}}
\put(188,2){\line(1,0){48}}
\put(222,40){\oval(20,20)[b]}
\put(189,40){{$\scriptstyle\times$}}
\put(210,40){{$\scriptstyle\bullet$}}
\put(230,40){{$\scriptstyle\bullet$}}
\put(189,20){$\scriptstyle\times$}
\put(209.3,19.3){$\circ$}
\put(229,20){{$\scriptstyle\times$}}
\put(190,0){$\scriptstyle\bullet$}
\put(210,0){$\scriptstyle\bullet$}
\put(229,0){{$\scriptstyle\times$}}
\put(202,2){\oval(20,20)[t]}
\end{picture}
$$
$$
\begin{picture}(120,51)
\put(-82,42){\line(1,0){48}}
\put(-82,22){\line(1,0){48}}
\put(-82,2){\line(1,0){48}}
\put(-58,42){\line(1,-1){20}}
\put(-81,40){{$\scriptstyle\times$}}
\put(-60,40){{$\scriptstyle\bullet$}}
\put(-40.7,39.3){{$\circ$}}
\put(-38,22){\line(0,-1){20}}
\put(-40,20){$\scriptstyle\bullet$}
\put(-81,20){$\scriptstyle\times$}
\put(-60.7,19.3){{$\circ$}}
\put(-40,0){$\scriptstyle\bullet$}
\put(-60,0){$\scriptstyle\bullet$}
\put(-80,0){{$\scriptstyle\bullet$}}
\put(-68,2){\oval(20,20)[t]}

\put(-23,19.5){$\stackrel{\overline\psi_{q-1}}{\rightsquigarrow}$}

\put(8,42){\line(1,0){48}}
\put(8,22){\line(1,0){48}}
\put(8,2){\line(1,0){48}}
\put(32,42){\line(-1,-1){20}}
\put(9,40){{$\scriptstyle\times$}}
\put(30,40){{$\scriptstyle\bullet$}}
\put(49.3,39.3){{$\circ$}}
\put(12,22){\line(0,-1){20}}
\put(29,20){$\scriptstyle\times$}
\put(49.3,19.3){$\circ$}
\put(10,20){{$\scriptstyle\bullet$}}
\put(50,0){$\scriptstyle\bullet$}
\put(30,0){$\scriptstyle\bullet$}
\put(10,0){{$\scriptstyle\bullet$}}
\put(42,2){\oval(20,20)[t]}

\put(98,42){\line(1,0){48}}
\put(98,22){\line(1,0){48}}
\put(98,2){\line(1,0){48}}
\put(122,42){\line(1,-1){20}}
\put(139.3,39.3){{$\circ$}}
\put(120,40){{$\scriptstyle\bullet$}}
\put(100,40){{$\scriptstyle\bullet$}}
\put(102,42){\line(0,-1){20}}
\put(142,22){\line(0,-1){20}}
\put(140,20){$\scriptstyle\bullet$}
\put(119.3,19.3){$\circ$}
\put(100,20){{$\scriptstyle\bullet$}}
\put(102,22){\line(1,-1){20}}
\put(140,0){$\scriptstyle\bullet$}
\put(120,0){$\scriptstyle\bullet$}
\put(99.3,-0.7){{$\circ$}}

\put(157,19.5){$\stackrel{\overline\psi_{q-1}}{\rightsquigarrow}$}

\put(188,42){\line(1,0){48}}
\put(188,22){\line(1,0){48}}
\put(188,2){\line(1,0){48}}
\put(202,40){\oval(20,20)[b]}
\put(210,40){{$\scriptstyle\bullet$}}
\put(190,40){{$\scriptstyle\bullet$}}
\put(209,20){$\scriptstyle\times$}
\put(189.3,-0.7){$\circ$}
\put(189.3,19.3){$\circ$}
\put(229.3,19.3){$\circ$}
\put(229.3,39.3){$\circ$}
\put(230,0){$\scriptstyle\bullet$}
\put(210,0){$\scriptstyle\bullet$}
\put(222,2){\oval(20,20)[t]}
\end{picture}
$$
The exception is if
$\Ga_{p-1}$ is a block of defect zero
and there is a left-shift (resp.\ right-shift)
at level $q$; this can happen only if the components of $\bt$
that are non-trivial at levels $p$ and $q$ are
two different propagating lines.
\end{Lemma}

\begin{Lemma}[Small circles]\label{smallcircles}
Let $C$ be a component of $\bt$ and $i \in I$
be minimal (resp.\ maximal) such that $C$ passes through the $i$th
(resp. $(i+1)$th)
vertex of some number line.
Suppose we are given $1 \leq p < q \leq d$ such that
$i_p = i_q = i$,
$C$ has a cap at level $p$ and
a cup at level $q$, and
$i_s \neq i$ for $p < s < q$.
Then
$\overline\psi_{p+1} \circ \cdots \circ \overline\psi_{q-1}$
is a sequence of height moves
transforming $C$ into a small circle at levels $p$ and $p+1$.
\end{Lemma}

\phantomsubsection{Negative circle moves}
Suppose again
that we are given a block $\Ga$ and a proper $\Ga$-admissible
sequence $\bi = (i_1,\dots,i_d)$, with associated block
sequence $\bGa$ and composite matching $\bt$.
Let $C$ be an internal circle in $\bt$.
Define
\begin{equation}
z_C: K^\bt_{\bLa}\langle-\caps(\bt)\rangle
\rightarrow
K^{\bt}_{\bLa}\langle-\caps(\bt)\rangle
\end{equation}
to be the endomorphism mapping
$(a\:\bt[\bga]\:b)$ to $(a\:\bt[\bga']\:b)$ if $C$ is clockwise
in $\bt[\bga]$
or to zero otherwise;
here, $\bga'$
is the weight sequence obtained from $\bga$ on
re-orienting $C$ so it is anti-clockwise.
It is obvious that this is a well-defined bimodule endomorphism.
We call it a {\em negative circle move} (``negative'' because it is of
 degree $-2$).
Note if $C$ is a small circle that is non-trivial at levels $r$ and $r+1$
then $z_C$ coincides with the local
negative circle move $\overline\psi_r$
introduced already in (\ref{betadef2}).

\begin{Lemma}\label{nonest}
Suppose that $C$ is an internal circle of $\bt$ containing no other
nested circles.
Then the negative
circle move $z_C$ can be expressed as
a composition of local moves of the form
$\overline\psi_r$ for various $r$.
\end{Lemma}

\begin{proof}
We proceed by induction on the height of $C$.
In the base case, $C$ has height 2.
Suppose it has a cap at level $p$ and a cup at level $q$.
By Lemma~\ref{smallcircles},
$\overline\psi_{p+1} \circ \cdots \circ \overline\psi_{q-1}$
is a sequence of
height moves transforming $C$ into a small circle at levels $p$ and $p+1$.
After that $\overline\psi_p$ is an negative circle move.
Then $\overline\psi_{q-1} \circ \cdots \circ \overline\psi_{p+1}$ is another
sequence of height moves stretching the small circle back
to $C$. Putting it together,  we have that
$$
z_C =
\overline\psi_{q-1} \circ \cdots \circ \overline\psi_{p+1} \circ \overline\psi_p
\circ \overline\psi_{p+1} \circ \cdots \circ \overline\psi_{q-1},
$$
hence the base of the induction is checked.

Now suppose for the induction step that $C$ has height greater than 2.
Let $i \in I$ be minimal such that $C$ passes through
the $i$th vertex of some number line.
Let $1 \leq p < d$ be minimal such that $C$ is non-trivial at level $p$
and $i_p = i$.
Let $p < s \leq d$ be minimal such that $C$ is non-trivial at level $s$
and $i_s = i$.
Obviously $C$ must either have a cap or a left-shift at level $p$
and either a cup or a right-shift at level $s$.
In view of Lemma~\ref{smallcircles} and the assumption that $C$
has height greater than 2, it cannot have both a cap at level $p$
and a cup at level $s$.
Hence either $C$ has a left-shift at level $p$ or a right-shift at level $s$.

Suppose in this paragraph that $C$ has a left-shift at level $p$.
Because $i_p = i_s$ the hypotheses of Lemma~\ref{loweringshifts}(i)
are not satisfied. Hence as in Lemma~\ref{loweringshifts}(ii)
there exists $p < q < s$ such that
$i_q = i_p+1$ and $|i_p-i_k| > 1$ for $p < k < q$.
Then
$$
\delta := \overline\psi_{q-1} \circ
\overline\psi_{q-2} \circ \cdots \circ \overline\psi_p
$$
is a sequence of height moves lowering the left-shift
to level $q-1$ followed by one crossing move. Because of the minimality of the
choice of
$p$ and the fact that $C$ contains no nested circles,
$\delta$ must cut $C$ into two circles $C'$ and $C''$.
Both of these are of strictly smaller height than $C$ and
neither contains any nested circles.
In the reverse direction,
$$
m := \overline\psi_p \circ \cdots \circ \overline\psi_{q-2} \circ \overline\psi_{q-1}
$$
joins $C'$ and $C''$ back together to recover the original circle $C$.
By induction, the negative circle moves $z_{C'}$
and $z_{C''}$
can both be written as
compositions of local moves of the desired form.
Now we claim that
$$
z_C =
m\circ z_{C'} \circ z_{C''} \circ \delta.
$$
To see this, take a basis vector
$(a\:\bt[\bla]\:b)\in K^{\bt}_{\bLa}\langle-\caps(\bt)\rangle$.
If the circle $C$ is anti-clockwise
in $a \:\bt[\bla]\:b$
then $\delta$
maps it to a sum of two basis vectors in
which one of $C'$ and $C''$ is anti-clockwise and the other is clockwise.
Then $z_{C'} \circ z_{C''}$ maps that to zero, as required.
If the circle $C$ is clockwise then
$\delta$ maps $(a\:\bt[\bla]\:b)$
to a basis vector in which both $C'$ and $C''$ are clockwise.
Then $z_{C'} \circ z_{C''}$
converts both $C'$ and $C''$
to anti-clockwise circles. Finally
$m$
returns us to our original basis vector but with $C$ re-oriented so that
it is anti-clockwise.
This checks the claim.

Finally suppose that $C$ has a right-shift at level $s$.
Then we argue in a similar way to the previous paragraph,
this time raising the right-shift using Lemma~\ref{raisingshifts} instead of
lowering the left-shift with Lemma~\ref{loweringshifts}.
\end{proof}

\begin{Lemma}\label{anti}
For any internal circle $C$ of $\bt$,
the negative circle move $z_C$
can be written as a linear combination of
compositions of local moves of the form
$\overline y_r$ and $\overline\psi_r$ for various $r$.
\end{Lemma}

\begin{proof}
Proceed by induction on
the number of nested circles contained in $C$.
The base of the induction is Lemma~\ref{nonest}.
For the induction step suppose that $C$ contains at least one nested
circle.
Let $q$ be minimal such that some circle $C'$ contained in $C$
has a cap at level $q$.
As $C'$ is contained in $C$,
$C'$ and $q$ do not satisfy the hypotheses of Lemma~\ref{raisingcaps}(i),
hence as in Lemma~\ref{raisingcaps}(ii) there exists $1 \leq p < q$
such that $$
m := \overline\psi_{p}\circ \overline\psi_{p+1} \circ \cdots \circ \overline
\psi_{q-1}$$
is a sequence of height moves raising the cap
from level $q$ to level $p+1$ followed by one crossing move.
By the minimality of $q$, this composition $m$ must join the circles
$C$ and $C'$ into one circle $C''$ containing fewer nested
internal circles than $C$.
In the other direction,
$$
\delta := \overline\psi_{q-1} \circ \cdots \circ \overline\psi_{p+1} \circ \overline\psi_p$$
cuts $C''$ to recover the two circles $C$ and $C'$ back again.

By induction, we already have $z_{C'}$ and $z_{C''}$
available. Take a basis vector in
$K^\bt_{\bLa}\langle-\caps(\bt)\rangle$.
Let $1 \otimes 1, 1 \otimes x, x \otimes 1$
and $x \otimes x$ denote this vector re-oriented so that
$C$ (resp.\ $C'$) is anti-clockwise or clockwise according to whether
the first (resp.\ second) tensor is $1$ or $x$.
Similarly, represent the basis vector obtained from one of these
by applying $m$
by $1$ or $x$ according to whether $C''$ is anti-clockwise or clockwise.
With this notation $\delta \circ z_{C''}\circ m$ is the map
$$
\delta \circ z_{C''} \circ m:
1 \otimes 1 \mapsto 0,
1 \otimes x \mapsto 1 \otimes x + x \otimes 1,
x \otimes 1 \mapsto 1 \otimes x + x \otimes 1,
x \otimes x \mapsto 0.
$$
Setting $y_{C'} := \overline y_q$ (remembering
it is the circle $C'$ that is non-trivial at level $q$), we also have
\begin{align*}
y_{C'} :1 \otimes 1 \mapsto 1 \otimes x,
x \otimes 1 \mapsto x \otimes x,1 \otimes x \mapsto 0,
x \otimes x \mapsto 0,\\
z_{C'}:1 \otimes 1 \mapsto 0,
x \otimes 1 \mapsto 0,1 \otimes x \mapsto 1 \otimes 1,
x \otimes x \mapsto  x \otimes 1.
\end{align*}
Putting these things together, we get that
\begin{align*}
z_C =
z_{C'} \circ (\delta \circ z_{C''} \circ m) \circ
z_{C'}
\circ y_{C'} +
y_{C'}\circ z_{C'} \circ (\delta \circ z_{C''}
\circ m) \circ z_{C'}
\end{align*}
and we are done.
\end{proof}

\phantomsubsection{Proof of surjectivity}
Now we are ready to prove the main result of the section.

\begin{Theorem}
\label{surjectivity}
The algebra homomorphism $\omega:R^\La_\alpha \rightarrow E^\La_\alpha$
from Theorem~\ref{phi}
is surjective.
\end{Theorem}

\begin{proof}
We proceed by induction on $m+n+d$ where
$d := \height(\alpha)$,
the case $d = 0$ being trivial.
For the induction step, we assume that we are given
$\alpha \in Q_+$ of height $d > 0$ and that the theorem
has been proved for all smaller $m+n+d$.
If $\bu^* \bt$ is a proper stretched circle diagram of type $\alpha$,
there is a unique oriented stretched circle diagram
of the form $\bu^*[\bde^*]\wr\bt[\bga]$ of minimal degree;
all of its  internal
circles are anti-clockwise.
Denote the corresponding basis vector of $E^\La_\alpha$ by
$e(\bu^*\bt)$.
Clearly every basis vector of $E^\La_\alpha$ from (\ref{Hbasis}) can be obtained
from $e(\bu^*\bt)$
by applying positive circle moves,
i.e. by multiplying on the left and/or right by $\pm \omega(y_r)$
for various $r$.
Hence it suffices to show that
$e(\bu^*\bt) \in \im \omega$
for every proper stretched circle diagram
$\bu^*\bt$ of type $\alpha$.
We divide the proof of this statement into nine different
cases, the last of which is the general case.

\vspace{2mm}

\noindent
{\em Case one:
$\bu^*\bt$ has a boundary circle $C$ containing no nested circles,
such that for some $i \in I$
the $i$th and $(i+1)$th vertices of the boundary line
are in the interior of $C$
and are labelled
$\circ\times$ in the block diagram
for $\La-\alpha$.}

\vspace{1mm}

\noindent
Let $\bi \in I^\alpha$ denote the admissible sequence
underlying the matching $\bt$.
By inspecting (\ref{CKLR}), the assumption that the $i$th and $(i+1)$th
vertices of $\La-\alpha$ are labelled $\circ \times$
implies that
there exists $1 \leq p \leq d$ such that
$i_p = i$,
$\bt$ has a cup at level $p$, and
$|i_p-i_q| > 1$ for all $p < q \leq d$.
Hence we are in the situation of
Lemma~\ref{loweringcups}(i) and $\overline\psi_{d-1}\circ\cdots\circ
\overline\psi_p$ is a sequence of height moves lowering the cup to level
$d$. Let $\hat\bt$ be the proper stretched cap diagram obtained from $\bt$
by applying this sequence of height moves. Then
$$
e(\bu^* \bt) = \overline \psi_p \circ \cdots \circ \overline \psi_{d-1}
(e(\bu^* \hat \bt)).
$$
Since each of the maps $\overline \psi_r$ here is given by multiplying on the
right by
$\pm \omega(\psi_r)$, we can deduce from this that
$e(\bu^* \bt) \in \im \omega$
if we show that $e(\bu^* \hat \bt) \in \im \omega$.
Replacing $\bt$  by $\hat \bt$, this reduces to the situation that $\bt$
has a cup at level $d$ in the strip between the $i$th and $(i+1)$th vertices.
A similar argument
involving multiplying on the left by various $\pm\omega(\psi_r)$
reduces further
to the situation that $\bu$ also has a cup at level $d$ in the same strip.

Now let $\bar\bt := t_{d-1} \cdots t_0$ and
$\bar \bu := u_{d-1}\cdots u_0$.
Then $\bar\bu^* \bar\bt$ is a proper stretched circle diagram
of type $\alpha-\alpha_i$.
By the main
induction hypothesis, we know that
$e(\bar\bu^*\bar\bt) \in \im \omega$.
The map $\theta_i:E^\La_{\alpha-\alpha_i}
\rightarrow E^\La_\alpha$
from (\ref{kwon}) in this situation joins two boundary circles
from $\bar\bu^*\bar\bt$
together to form the given boundary circle $C$, so that
$$
\theta_i(e(\bar\bu^*\bar\bt)) = e(\bu^* \bt).
$$
Applying Corollary~\ref{cddcd} we deduce that
$e(\bu^*\bt) \in \im \omega$.
\iffalse
$$
\begin{picture}(100,112)
\put(-20,90){$\bar\bt$}
\put(-20,62){$s$}
\put(-20,42){$s^*$}
\put(-20,14){$\bar\bu^*$}
\put(-3,75){\line(1,0){106}}
\thicklines
\put(-3,55){\line(1,0){106}}\thinlines
\put(-3,35){\line(1,0){106}}
\put(-2,53){$\scriptstyle\bullet$}
\put(98,53){$\scriptstyle\bullet$}
\put(-2,73){$\scriptstyle\bullet$}
\put(98,73){$\scriptstyle\bullet$}
\put(-2,33){$\scriptstyle\bullet$}
\put(98,33){$\scriptstyle\bullet$}
\put(38,33){$\scriptstyle\bullet$}
\put(58,33){$\scriptstyle\bullet$}
\put(38,73){$\scriptstyle\bullet$}
\put(58,73){$\scriptstyle\bullet$}
\put(20,15){\oval(40,30)[b]}
\put(80,15){\oval(40,30)[b]}
\put(20,95){\oval(40,30)[t]}
\put(80,95){\oval(40,30)[t]}
\put(50,35){\oval(20,20)[t]}
\put(50,75){\oval(20,20)[b]}
\put(0,35){\line(0,1){40}}
\put(100,35){\line(0,1){40}}
\put(37.5,52.3){{$\circ$}}
\put(56.7,52.9){{$\scriptstyle\times$}}
\dashline{3}(0,75)(0,95)
\dashline{3}(100,75)(100,95)
\dashline{3}(0,15)(0,35)
\dashline{3}(100,15)(100,35)
\dashline{3}(40,75)(40,95)
\dashline{3}(60,75)(60,95)
\dashline{3}(40,15)(40,35)
\dashline{3}(60,15)(60,35)
\end{picture}
$$
\fi

\vspace{2mm}

\noindent
{\em Case two: $\bu^*\bt$ has a boundary circle $C$
crossing the boundary line exactly twice at vertices $j < k$,
there are no nested circles inside $C$,
and the height of $C$ is
equal to $2(k-j)$.}

\vspace{1mm}

\noindent
Each of the
vertices $j+1,\dots,k-1$ of $\La-\alpha$ must be labelled either
$\times$ or $\circ$. In view of case one we may assume further that
the vertices $j+1,\dots,i$ are labelled
$\times$ and the vertices $i+1,\dots,k-1$ are labelled $\circ$,
for some $j \leq i < k$.
Of course $C$ consists of a generalised cap $T$ from the matching $\bt$
on top of a generalised cup $B$ from $\bu^*$.
The smallest possible height of a generalised cap
passing through
the $j$th and $k$th vertices of the boundary line is $(k-j)$.
Hence both $T$ and $B$ are of height exactly $(k-j)$.
Considering (\ref{CKLR}), $T$
must involve $(i-j)$ right shifts, $(k-1-i)$ left shifts, and one cap at the
top. For example,
here are the possibilities for $T$ in case $k-j=3$
(omitting trivial levels):
$$
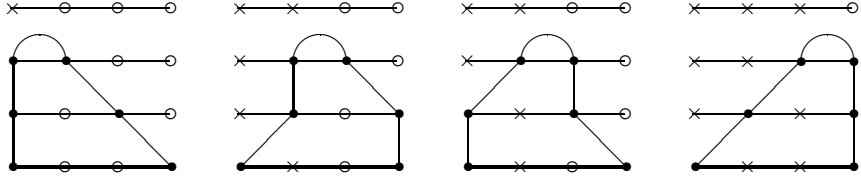
\begin{picture}(150,63)
\put(-88,59){\line(1,0){60}}
\put(-88,39){\line(1,0){60}}
\put(-88,19){\line(1,0){60}}
\thicklines
\put(-88,-1){\line(1,0){60}}\thinlines
\put(-90.6,57){{$\scriptstyle\times$}}
\put(-70.2,56.3){{$\circ$}}
\put(-50.2,56.3){{$\circ$}}
\put(-30.2,56.3){{$\circ$}}
\put(-77,39){\oval(20,20)[t]}
\put(-87,39){\line(0,-1){40}}
\put(-89,37){{$\scriptstyle\bullet$}}
\put(-69,37){{$\scriptstyle\bullet$}}
\put(-50.2,36.3){{$\circ$}}
\put(-30.2,36.3){{$\circ$}}
\put(-67,39){\line(1,-1){20}}
\put(-89,17){$\scriptstyle\bullet$}
\put(-49,17){$\scriptstyle\bullet$}
\put(-70.2,16.3){{$\circ$}}
\put(-30.2,16.3){{$\circ$}}
\put(-89,-3){$\scriptstyle\bullet$}
\put(-70.2,-3.7){$\circ$}
\put(-50.2,-3.7){$\circ$}
\put(-29,-3){{$\scriptstyle\bullet$}}
\put(-47,19){\line(1,-1){20}}
%%%%%%%%%%%%
\put(-2,59){\line(1,0){60}}
\put(-2,39){\line(1,0){60}}
\put(-2,19){\line(1,0){60}}
\thicklines
\put(-2,-1){\line(1,0){60}}\thinlines
\put(-4.6,57){{$\scriptstyle\times$}}
\put(15.4,57){{$\scriptstyle\times$}}
\put(35.8,56.3){{$\circ$}}
\put(55.8,56.3){{$\circ$}}
\put(29,39){\oval(20,20)[t]}
\put(19,39){\line(0,-1){20}}
\put(19,19){\line(-1,-1){20}}
\put(39,39){\line(1,-1){20}}
\put(59,19){\line(0,-1){20}}
\put(-4.6,37){{$\scriptstyle\times$}}
\put(17,37){{$\scriptstyle\bullet$}}
\put(37,37){{$\scriptstyle\bullet$}}
\put(55.8,36.3){{$\circ$}}
\put(-4.6,17){$\scriptstyle\times$}
\put(17,17){$\scriptstyle\bullet$}
\put(35.8,16.3){{$\circ$}}
\put(57,17){{$\scriptstyle\bullet$}}
\put(-3,-3){$\scriptstyle\bullet$}
\put(15.4,-3){$\scriptstyle\times$}
\put(35.8,-3.7){$\circ$}
\put(57,-3){{$\scriptstyle\bullet$}}
%%%%%%%%%%%%%%%%%%%%%%%%%%%%%
\put(84,59){\line(1,0){60}}
\put(84,39){\line(1,0){60}}
\put(84,19){\line(1,0){60}}
\thicklines
\put(84,-1){\line(1,0){60}}\thinlines
\put(81.4,57){{$\scriptstyle\times$}}
\put(101.4,57){{$\scriptstyle\times$}}
\put(121.8,56.3){{$\circ$}}
\put(141.8,56.3){{$\circ$}}
\put(115,39){\oval(20,20)[t]}
\put(105,39){\line(-1,-1){20}}
\put(85,19){\line(0,-1){20}}
\put(125,39){\line(0,-1){20}}
\put(125,19){\line(1,-1){20}}
\put(81.4,37){{$\scriptstyle\times$}}
\put(103,37){{$\scriptstyle\bullet$}}
\put(123,37){{$\scriptstyle\bullet$}}
\put(141.8,36.3){{$\circ$}}
\put(83,17){$\scriptstyle\bullet$}
\put(123,17){$\scriptstyle\bullet$}
\put(101.4,17){{$\scriptstyle\times$}}
\put(141.8,16.3){{$\circ$}}
\put(83,-3){$\scriptstyle\bullet$}
\put(101.4,-3){$\scriptstyle\times$}
\put(121.8,-3.7){$\circ$}
\put(143,-3){{$\scriptstyle\bullet$}}
%%%%%%%%%%%%%%%%%%%%%%%%%%%%%%%%
\put(170,59){\line(1,0){60}}
\put(170,39){\line(1,0){60}}
\put(170,19){\line(1,0){60}}
\thicklines
\put(170,-1){\line(1,0){60}}\thinlines
\put(167.4,57){{$\scriptstyle\times$}}
\put(207.4,57){{$\scriptstyle\times$}}
\put(187.4,57){{$\scriptstyle\times$}}
\put(227.8,56.3){{$\circ$}}
\put(221,39){\oval(20,20)[t]}
\put(231,39){\line(0,-1){40}}
\put(211,39){\line(-1,-1){40}}
\put(167.4,37){{$\scriptstyle\times$}}
\put(187.4,37){{$\scriptstyle\times$}}
\put(209,37){{$\scriptstyle\bullet$}}
\put(229,37){{$\scriptstyle\bullet$}}
\put(167.4,17){$\scriptstyle\times$}
\put(207.4,17){$\scriptstyle\times$}
\put(189,17){{$\scriptstyle\bullet$}}
\put(229,17){{$\scriptstyle\bullet$}}
\put(169,-3){$\scriptstyle\bullet$}
\put(187.4,-3){$\scriptstyle\times$}
\put(207.4,-3){$\scriptstyle\times$}
\put(229,-3){{$\scriptstyle\bullet$}}
\end{picture}
$$
By applying height moves  like we did in case one, this time
using Lemma~\ref{loweringshifts}, we reduce
to the case that
the $(i-j)$ right shifts
occur in
the bottom $(i-j)$ levels of $\bt$,
then the $(k-1-i)$ left shifts appear in the next $(k-1-i)$ levels up,
and finally the cap at the top of $T$ occurs
at the  $(d+1-k+j)$th level of $\bt$.
A similar argument applied to $B$ reduces further to the situation that
$B$ is the mirror image of $T$.
Finally in this special situation, we have that
$$
e(\bu^*\bt) = \theta_i(e(\bar \bu^* \bar \bt))
$$
where
$\bar\bu^* \bar \bt$ is the proper stretched circle diagram
of type $(\alpha-\alpha_i)$ with
$\bar\bt := t_{d-1}\cdots t_0$ and $\bar \bu := u_{d-1} \cdots u_0$.
By induction $e(\bar\bu^*\bar \bt) \in \im \omega$.
Hence applying Corollary~\ref{cddcd} as in case one,
we deduce that $e(\bu^* \bt) \in \im \omega$ too.

\vspace{2mm}

\noindent
{\em Case three: $\bu^*\bt$ has a boundary circle $C$
containing no other nested circles.}

\vspace{1mm}

\noindent
Proceed by induction on the height of $C$.
The base case when $C$ is of height $2$ follows by
case two. Now assume that $C$ is of height greater than $2$.
Suppose first that $C$ has a {\em concave cap}, i.e. a cap
such that the region immediately
above the cap is the interior of $C$.
Applying Lemma~\ref{raisingcaps} to $\bt$
(resp.\ Lemma~\ref{loweringcups} to $\bu$)
if the cap is above (resp.\ below) the boundary line,
we get a sequence of moves which
either raises the cap until it reaches the boundary line from below
or until it collides with another part of the same circle $C$.
In the former case, we deduce that there exists $i \in I$
such that the $i$th and $(i+1)$th vertices of the boundary line
are in the interior of $C$ and are labelled $\circ\times$ in
the block diagram for $\La-\alpha$, so we are done by case one.
In the latter case the given sequence of moves cuts the
boundary circle $C$ into another boundary circle $C'$ of
strictly smaller height than $C$, together with another circle $C''$
which is definitely not nested inside $C'$.
Let $\hat\bu^*\hat\bt$ be the proper stretched circle diagram
obtained from $\bu^*\bt$ by applying
these moves.
By the induction hypothesis, we get that
$e(\hat\bu^*\hat\bt) \in \im \omega$.
As $e(\bu^*\bt)$ is obtained from this by multiplying
on the left or right by the same moves taken in reverse order, we deduce that
$e(\bu^*\bt) \in \im \omega$ too, as required.
A similar argument applies if $C$ has a {\em concave cup}, i.e. a cup
such that the region
immediately below the cup is the interior of $C$.

We have now reduced to the situation that $C$ has no
concave cap or cup on its circumference.
It follows that $C$ actually has just one cap and just one
cup. In particular, $C$
crosses the boundary line exactly twice,
say at vertices $j < k$.
If the height of $C$ is equal to $2(k-j)$
then we are done by case two. Hence either the top half $T$ of $C$
(which is a single
generalised cap) or the bottom half $B$ of $C$
(which is a single generalised cup) must
have height strictly greater than $(k-j)$.
Let us assume it is $T$ that is of height greater than $(k-j)$, a similar
argument applying if it is $B$.
We know already that $T$ has exactly one cap and no cups.
This means that $T$ consists just of left-shifts,
right-shifts and one cap at the top.
Since the height of $T$ is strictly greater than $(k-j)$,
there is either a right-shift on the part of the
$T$ that is
between the $j$th vertex of the boundary line and the cap at the top,
or there is a left-shift on the part of $T$
that is between the
cap at the top and the $k$th
vertex of the boundary line.
We just explain the argument now in the former case, since the latter
case is similar.
Let $q$ be minimal such that $T$ has a right-shift at level $q$
between the $j$th vertex of the boundary line and the cap at the top.
Applying Lemma~\ref{raisingshifts} to the component of
$t_1 \cdots t_q$ containing this right-shift, we get a sequence
$\overline{\psi}_{p+1}\circ\cdots\circ \overline\psi_{q-1}$
of moves raising the right-shift. The last of these moves is a crossing move
cutting
$C$ into a smaller boundary circle $C'$ plus one new
circle $C''$ that is not nested inside $C'$.
Now we complete the proof as before by using the induction hypothesis.

\vspace{2mm}

\noindent
{\em Case four: $\bu^*\bt$ has a boundary circle.}

\vspace{1mm}

\noindent
Choose a boundary circle $C$ in $\bu^* \bt$
that is minimal in the sense that it contains no nested boundary circles.
Proceed by induction on the number of nested circles contained in $C$.
The base case, when $C$ contains no nested circles,
follows by case three.
For the induction step,
assume $C$ contains at least one nested circle.
Since it contains no nested boundary circles by assumption,
it must contain a nested
upper or lower circle. We explain the argument in the upper case,
since the lower case is entirely similar.
Let $q$ be minimal such that
the component $C'$ of $\bu^*\bt$ that is non-trivial at level
$q$ is an upper circle contained in $C$.
Applying Lemma~\ref{raisingcaps}(ii),
we get a sequence $\overline{\psi}_p\circ\cdots\circ \overline\psi_{q-1}$
of local moves
raising the cap at level $q$.
Let $\bu^*\hat\bt$ be the proper stretched circle diagram
obtained by applying this sequence of moves to $\bu^*\bt$.
By the minimality of $q$,
the last move
$\overline{\psi}_p$ in this sequence is a crossing move
joining $C$ and $C'$ into one circle $C''$.
The circle $C''$
is a boundary circle
in $\bu^* \hat\bt$
containing no nested boundary circles
and one less nested circle than $\bu^* \bt$.
Hence by induction $e(\bu^*\hat\bt) \in \im \omega$.
Moreover,
$$
e(\bu^* \bt) = z_{C'} \circ \overline{\psi}_{q-1} \circ \cdots \circ
\overline{\psi}_{p} (e(\bu^* \hat \bt)).
$$
In view of Lemma~\ref{anti}, the map
$z_{C'} \circ \overline{\psi}_{q-1}\circ\cdots\circ \overline{\psi}_p$
can be implemented by right multiplication by an element of $R^\La_\alpha$,
so we deduce that $e(\bu^*\bt) \in \im \omega$.

\vspace{2mm}

\noindent
{\em Case five: $\bu^*\bt$ has a propagating line $L$ with no
other circles or propagating lines to its left,
such that for some $i \in I$ the $i$th and $(i+1)$th vertices of the
boundary line lie strictly to the left of $L$ and are labelled
$\circ\times$ in the block diagram for $\La-\alpha$.}

\vspace{1mm}
\noindent
This follows by a very similar argument to case one.
\iffalse
$$
\begin{picture}(100,112)
\put(20,90){$\hat\bt$}
\put(20,62){$s$}
\put(20,42){$s^*$}
\put(20,14){$\hat\bu^*$}
\put(37,75){\line(1,0){66}}
\thicklines
\put(37,55){\line(1,0){66}}\thinlines
\put(37,35){\line(1,0){66}}
\put(98,53){$\scriptstyle\bullet$}
\put(98,73){$\scriptstyle\bullet$}
\put(98,33){$\scriptstyle\bullet$}
\put(38,33){$\scriptstyle\bullet$}
\put(58,33){$\scriptstyle\bullet$}
\put(38,73){$\scriptstyle\bullet$}
\put(58,73){$\scriptstyle\bullet$}
\put(80,15){\oval(40,30)[b]}
\put(80,95){\oval(40,30)[t]}
\put(50,35){\oval(20,20)[t]}
\put(50,75){\oval(20,20)[b]}
\put(100,35){\line(0,1){40}}
\put(40,95){\line(0,1){14}}
\put(40,15){\line(0,-1){14}}
\put(37.5,52.3){{$\circ$}}
\put(56.7,52.9){{$\scriptstyle\times$}}
\dashline{3}(100,75)(100,95)
\dashline{3}(100,15)(100,35)
\dashline{3}(40,75)(40,95)
\dashline{3}(60,75)(60,95)
\dashline{3}(40,15)(40,35)
\dashline{3}(60,15)(60,35)
\end{picture}
$$
\fi

\vspace{2mm}

\noindent
{\em Case six: $\bu^*\bt$ has a propagating line $L$
crossing the top and bottom number lines at vertex $k$
and
crossing the boundary line exactly once at vertex $j \leq k$,
such that
there are no other circles or propagating lines to the left of $L$
and the height of $L$ is equal to $2(k-j)$.}

\vspace{1mm}

\noindent
The propagating line $L$ consists of a line segment $T$ from $\bt$ on top
of another line segment $B$ from $\bu^*$.
The assumptions on $L$ imply that
$T$ has $(k-j)$
left-shifts
and $B$ has $(k-j)$ right-shifts.
By applying height moves, we can move the left-shifts up so that
they occur in the top $(k-j)$ levels of $T$, and similarly
we move the right-shifts down so that they occur in the
bottom $(k-j)$ levels of $B$.
Then we simply erase the leftmost propagating line and
the top and bottom $(k-j)$ levels
from the diagram $\bu^* \bt$ altogether,
removing one vertex from each number line in the process, to obtain a new
proper stretched circle diagram
$\hat\bu^*\hat \bt$ parametrizing a basis
vector in a smaller case in the sense that for the
new picture the number
$m+n+d$ is strictly smaller than before.
By the induction hypothesis,
$e(\hat\bu^*\hat\bt) \in \im\omega$. Restoring the parts of the diagram that
were erased gives that
$e(\bu^*\bt) \in \im \omega$ too.

\vspace{2mm}

\noindent
{\em Case seven: $\bu^*\bt$ has a propagating line $L$
with no other circles or propagating lines to its left.}

\vspace{1mm}

\noindent
Proceed by induction on the height of $L$. The base case when $L$ is of
height zero follows by case six.
Now suppose $L$ is of strictly positive
height.
If $L$ has a cap such that the region immediately
above the cap is the region to the left of $L$, then
we raise this cap as usual, either until it
reaches the boundary line from below or until it collides with another part
of the line $L$.
In the former case we are done by case five. In the latter
case the given sequence of moves cuts $L$ into a new propagating line $L'$ of
strictly smaller height together with a circle that necessarily
lies to the right of
$L$. Hence we are done by induction.
An entirely similar argument treats the situation that
$L$ has a cup such that the region immediately below the cup is the region to the left of $L$.

This reduces to the situation that $L$
involves only left-shifts and right-shifts. In particular, it crosses the
boundary line exactly once, say at vertex $k$.
Let $T$ (resp.\ $B$) be the top (resp.\ bottom) half of $L$.
If $T$ involves no right-shifts and $B$ involves no left-shifts,
then we are done by case six.
Hence either $T$ involves right-shifts or $B$ involves left-shifts.
In fact, each time there is a right-shift in $T$, a vertex
labelled $\circ$ gets added to the left hand side of $L$, so
the number of right-shifts in $T$ is equal to the number of vertices of
$\La-\alpha$ to the left of vertex $k$ that are labelled $\circ$.
By similar considerations this is the same as the number of left-shifts in $B$.
Finally, considering (\ref{CKLR}), all the right-shifts in $T$ must
below all its left-shifts, hence we can use Lemma~\ref{loweringshifts}(i)
to move all the right-shifts in $T$ to the bottom.
Similarly we can move all the left-shifts in $B$ to the top.
We have now reduced to the situation that both $\bt$ and $\bu$ have
right-shifts
at level $d$ in the strip between the $(k-1)$th and $k$th vertices.
It is then the case that
$$
e(\bu^*\bt) = \theta_{k-1}(e(\bar\bu^*\bar\bt))
$$
where $\bar\bu:=u_{d-1}\cdots u_0$ and $\bar\bt := t_{d-1}\cdots t_0$,
and we are done by the main induction hypothesis as usual.

\vspace{2mm}

\noindent
{\em Case eight: $\bu^*\bt$ has a propagating line but no boundary circles.}

\vspace{1mm}

\noindent
Let $L$ be the leftmost
propagating line and proceed by induction on the number of internal
circles to the left of $L$.
The base of the induction follows from case seven.
For the induction step, let $C$ be an internal circle
to the left of $L$.
We just explain the argument if $C$ is an upper circle, since
 lower circles are treated similarly.
Let $q$ be minimal such that $C$ has a cap at level $q$.
Apply Lemma~\ref{raisingcaps},
we get a sequence $\overline\psi_{p+1}\circ\cdots\circ\overline\psi_{q-1}$
of moves raising the cap until either it
collides with another upper circle or it collides with the line $L$.
Let $\hat\bt$ be the matching obtained from $\bt$ by applying this
sequence of moves. By the induction hypothesis,
we have that $e(\bu^*\hat\bt) \in \im\omega$. As
$$
e(\bu^*\bt) = z_C \circ \overline\psi_{q-1}\circ\cdots\circ\overline\psi_{p+1} (e(\bu^*\hat\bt)),
$$
we deduce that $e(\bu^*\bt) \in \im \omega$.

\vspace{2mm}

\noindent
{\em Case nine: the general case.}

\vspace{1mm}

\noindent
If $\bu^*\bt$ has a boundary circle or a propagating line
we are done by cases five and eight. Hence $\bu^*\bt$
contains only upper and lower circles. Moreover as $d > 0$
it contains at least one upper circle and one lower circle.
This means that $\bt$ must have a cup at level $d$ which is
part of some upper circle $C$.
Say this cup is in the strip between the $i$th and $(i+1)$th vertices.
The $i$th and $(i+1)$th vertices of $\La-\alpha$ are labelled
$\circ\times$, hence like in case one $\bu^*$ must have
a cap in the same strip which can be raised by a sequence of height moves
so that it is
immediately below the boundary line.
In this way we have reduced to the situation that both
$\bt$ and $\bu$ have a cup at level $d$ in the strip between
the $i$th and $(i+1)$th vertices.
Let $\bar\bt := t_{d-1} \circ\cdots t_1$ and
$\bar \bu := u_{d-1}\circ\cdots u_1$.
By induction we have that $e(\bar\bu^*\bar \bt) \in \im \omega$.
It remains to observe that
$$
e(\bu^*\bt) = z_C \circ \theta_i(e(\bar\bu^*\bar\bt))
$$
as $\theta_i$ here cuts one circle into two.
This completes the proof of case nine, hence the theorem.
\end{proof}

\section{Equivalence of categorifications}\label{sE}

In this section we prove the main results of the article
by comparing the homogeneous Schur-Weyl duality
from section~\ref{sD}
with the
known Schur-Weyl duality for level two
on the category $\cO$ side from \cite{BKschur}.

\phantomsubsection{\boldmath The prinjective generator $\cT^\La_\alpha$
and its endomorphism algebra}
For $\bi \in \Z^d$, let $\cF_\bi$ denote the composition
$\cF_{i_d} \circ \cdots \circ \cF_{i_1}$ of the special projective
functors from (\ref{donely}).
Let $\cL(\iota)$ denote the irreducible $\mathfrak{g}$-module of
highest weight $\iota$, where
\begin{equation}
\iota = (o+m)(\eps_1+\cdots+\eps_m) + (o+m+n)(\eps_{m+1}+\cdots+\eps_{m+n})
\end{equation}
is the element of $\mathfrak{h}^*$
corresponding to the ground-state (\ref{groundstate})
under the weight dictionary.
In view of Corollary~\ref{all}, we have that
\begin{equation}\label{morita}
\cL(\iota)\otimes \cV^{\otimes d} = \bigoplus_{\bi \in \Z^d} \cF_\bi \cL(\iota).
\end{equation}
The summands $\cF_\bi \cL(\iota)$ and $\cF_\bj \cL(\iota)$ here
belong to the same block according
the decomposition (\ref{allblocks})
if and only if $\bi$ and $\bj$ lie in the same $S_d$-orbit.

Now fix $\alpha \in Q_+$ of height $d$ and set
\begin{equation*}\label{alto}
\cO^\La_\alpha := \left\{\begin{array}{ll}
\cO_{\La-\alpha}&\text{if $\La-\alpha \in \PImn$,}\\
{\mathbf 0}\:\text{(the zero category)}&\text{if $\La - \alpha \notin \PImn$,}
\end{array}\right.\end{equation*}
for short.
Paralleling (\ref{dc}), we define
\begin{equation}\label{cT}
\cT^\La_\alpha :=
\bigoplus_{\bi \in I^\alpha} \cF_\bi \cL(\iota).
\end{equation}
This space is zero unless
$\La-\alpha \in \PImn$,
in which case
\begin{equation}\label{itsablock}
\cT^\La_\al
= \pr_{\La-\alpha} (\cL(\iota) \otimes \cV^{\otimes d}),
\end{equation}
i.e. in all cases it is the largest submodule of $\cL(\iota)\otimes \cV^{\otimes d}$ that belongs to the subcategory $\cO^\La_\alpha$.
By a {\em prinjective generator} for $\cO^\La_\alpha$, we mean
a prinjective object of $\cO^\La_\alpha$ that involves each of the
prinjective indecomposable modules from Lemma~\ref{princ2} as a summand.

\begin{Lemma}\label{pring2}
The module $\cT^\La_\alpha$ is a prinjective generator for $\cO^\La_\alpha$.
It is non-zero if and only if
$\La-\alpha\in\PImn$.
\end{Lemma}

\begin{proof}
This is proved in exactly the same way as Lemma~\ref{pring1},
using Lemma~\ref{a} in place of Lemma~\ref{ga}.
\end{proof}

Next let $H_d$ denote the degenerate affine Hecke algebra from \cite{D2}.
So $H_d$ is the associative algebra
equal as a vector space to
the tensor product
$\C[x_1,\dots,x_d] \otimes \C S_d$
of a polynomial algebra and the group algebra of the symmetric group $S_d$.
Multiplication
is defined so that
 $\C[x_1,\dots,x_d] \equiv \C[x_1,\dots,x_d] \otimes 1$ and $\C S_d
\equiv 1 \otimes \C S_d$ are subalgebras of $H_d$,
and also
\begin{equation*}\label{ahgens}
s_r x_{r'} = x_{r'} s_r \:\:\text{if $r' \neq r,r+1$},
\qquad
s_r x_{r+1} = x_{r} s_r + 1.
\end{equation*}
By \cite[$\S$2.2]{AS},
there is a natural right action of $H_d$ on $\cL(\iota)\otimes \cV^{\otimes d}$
commuting with the left action of
$\g$, such that the
elements of $S_d$ act by permuting tensors in $\cV^{\otimes d}$
like in classical Schur-Weyl duality, and
the remaining
generator $x_1$ acts as
$\Omega \otimes 1^{\otimes (d-1)}$,
where $\Omega$ is as in (\ref{traceform}).

\begin{Lemma}\label{easyz}
The element $(x_1-o-m)(x_1-o-n) \in H_d$ acts as zero on $\cL(\iota)\otimes \cV^{\otimes d}$.
\end{Lemma}

\begin{proof}
It suffices to check this in the special case $d=1$.
By Lemma~\ref{tid}, $\cL(\iota) \otimes \cV$
has a two step filtration whose sections are highest weight
modules of highest weights $\iota+\eps_1$ and $\iota+\eps_{m+1}$,
respectively.
As explained in the proof of Lemma~\ref{crd}, the action of
$x_1 =
\Omega$ respects this filtration, and the
induced action
on the two sections are by multiplication by the
scalars $(\iota+\rho,\eps_1) = o+m$ and $(\iota+\rho,\eps_{m+1}) = o+n$,
respectively.
Hence $(x_1-o-m)(x_1-o-n)$ acts as zero.
\end{proof}

In view of Lemma~\ref{easyz},
the right action of $H_d$ on $\cL(\iota)\otimes\cV^{\otimes d}$
factors to give an action of the quotient algebra
\begin{equation}\label{tqa}
H_d^{o+m,o+n} := H_d / \langle (x_1-o-m)(x_1-o-n)\rangle,
\end{equation}
which is a {\em degenerate cyclotomic Hecke algebra} of level two.
As in \cite[$\S$3.1]{BKyoung},
this finite dimensional algebra possesses a natural system of mutually
orthogonal
{\em weight idempotents} $\{e(\bi)\:|\:\bi \in \Z^d\}$
summing to $1$, which are
characterised by the property that
$$\label{wis}
e(\bi) (x_r-i_r)^N = (x_i-i_r)^N e(\bi) = 0
$$
for each $r=1,\dots,d$ and $N \gg 0$.
Multiplication by $e(\bi)$ projects any $H^{o+m,o+n}_d$-module
onto its {\em $\bi$-weight space}, that is, the simultaneous
generalised eigenspace for the commuting operators
$x_1,\dots,x_r$ with respective eigenvalues
$i_1,\dots,i_r$.

\begin{Lemma}\label{i}
The idempotent $e(\bi) \in H_d^{o+m,o+n}$ acts on
$\cL(\iota)\otimes \cV^{\otimes d}$ as the projection
onto the summand $\cF_\bi \cL(\iota)$ along the decomposition
(\ref{morita}).
\end{Lemma}

\begin{proof}
By Lemma~\ref{crd}, $\cF_\bi \cL(\iota)$ is $\bi$-weight space
of $\cL(\iota)\otimes \cV^{\otimes d}$ with respect to the right
action of $H_d^{o+m,o+n}$.
\end{proof}

It is clear from (\ref{itsablock}) that $\cT^\La_\al$
is invariant under the right action of $H_d^{o+m,o+n}$.
By Lemma~\ref{i}, the idempotent
\begin{equation}\label{wis2}
e_\alpha := \sum_{\bi \in I^\alpha} e(\bi) \in H_d^{o+m,o+n}
\end{equation}
acts on $\cL(\iota)\otimes \cV^{\otimes d}$ as
the projection onto the summand $\cT^\La_\alpha$.
Since this projection clearly commutes with
the action of $H_d^{o+m,o+n}$, we deduce that
$e_\alpha$ is actually a {\em central}
idempotent in $H_d^{o+m,o+n}$.
Hence $e_\alpha H_d^{o+m,o+n}$ is a subalgebra of $H_d^{o+m,o+n}$
with identity $e_\alpha$,
and $\cT^\La_\alpha$ is a unital right
$e_\alpha H_d^{o+m,o+n}$-module.

\begin{Theorem}\label{id1}
The right action of $H_d$ on $\cT^\La_\alpha$ induces an
algebra isomorphism
$$
e_\alpha H_d^{o+m,o+n}
\stackrel{\sim}{\rightarrow}
\End_{\mathfrak{g}}(\cT^\La_\alpha)^{\op}.
$$
\end{Theorem}

\begin{proof}
This is a
consequence of \cite[Theorem 5.13]{BKschur}
and \cite[Corollary 6.7]{BKschur}.
It is formulated in exactly this way in
\cite[Theorem 3.6]{BKariki} in the case $m \geq n$
or \cite[Theorem 4.13]{BKariki} in the case $m \leq n$.
\end{proof}

\phantomsubsection{\boldmath The isomorphism theorem}
Continue working with a fixed $\alpha \in Q_+$ of height $d$.
The following theorem is the key to
our approach to the equivalence between $\cO^\La_\alpha$
and $\rep{K^\La_\alpha}$: combined with Theorems~\ref{id1} and \ref{surjectivity}
it implies at long last that
the prinjective generators $\cT^\La_\alpha$ and $T^\La_\alpha$
on the two sides have {the same endomorphism algebras}.

\begin{Theorem}\label{id2}
There is an algebra isomorphism
$$
\sigma:e_\alpha H_d^{o+m,o+n} \stackrel{\sim}{\rightarrow} R^\La_\alpha
$$
such that $e(\bi) \mapsto e(\bi)$,
$x_r e(\bi) \mapsto (y_r+i_r) e(\bi)$
and
$s_r e(\bi) \mapsto (\psi_r q_r(\bi) -
p_r(\bi))e(\bi)$ for each $r$ and $\bi \in I^\alpha$, where
$p_r(\bi), q_r(\bi) \in R^\La_\alpha$ are chosen as
in \cite[$\S$3.3]{BKyoung}, e.g. one could take
\begin{align*}
p_r(\bi) &:= \left\{
\begin{array}{ll}
1&\text{if $i_r = i_{r+1}$,}\\
-(i_{r+1}-i_{r}+y_{r+1}-y_{r})^{-1}\hspace{10mm}&\text{if $i_r \neq i_{r+1}$;}
\end{array}
\right.\\
q_r(\bi) &:= \left\{
\begin{array}{ll}
1+y_{r+1}-y_r&\text{if $i_r = i_{r+1}$,}\\
(2+y_{r+1}-y_r)(1+y_{r+1}-y_{r})^{-1}&\text{if $i_{r+1} = i_r+1$,}\\
1&\text{if $i_{r+1} = i_r-1$,}\\
1+(i_{r+1}-i_{r}+y_{r+1}-y_{r})^{-1}&\text{if $|i_r-i_{r+1}| > 1$.}
\end{array}
\right.
\end{align*}
(The inverses on the right hand side of these formulae make sense because each
$y_{r+1}-y_r$ is nilpotent; see \cite[Lemma 2.1]{BKyoung}.)
\end{Theorem}

\begin{proof}
This is a special case of the main theorem of \cite{BKyoung}.
\end{proof}

\begin{Corollary}\label{id3}
The homomorphism $\omega:R^\La_\alpha \rightarrow E^\La_\alpha$
from Theorem~\ref{phi} is an isomorphism.
\end{Corollary}

\begin{proof}
By Theorem~\ref{surjectivity}, $\omega$ is surjective, so it suffices to show
that $R^\La_\alpha$ and $E^\La_\alpha$ have the same dimensions.
According to Theorems~\ref{id1} and \ref{id2} and the definition (\ref{cT}),
we have that
$$
\dim R^\La_\alpha =
\sum_{\bi,\bj \in I^\alpha}
\dim \hom_{\mathfrak{g}}(\cF_\bi \cL(\iota), \cF_\bj \cL(\iota)).
$$
On the other hand by the definition (\ref{dc}),
$$
\dim E^\La_\alpha =
\sum_{\bi,\bj \in I^\alpha}
\dim \hom_{K^\La_\alpha}(F_\bi L(\iota), F_\bj L(\iota)).
$$
Therefore we are done if we can show that
\begin{equation}\label{nts}
\dim \hom_{\mathfrak{g}}(\cF_\bi \cL(\iota), \cF_\bj \cL(\iota))
=
\dim \hom_{K^\La_\alpha}(F_\bi L(\iota), F_\bj L(\iota))
\end{equation}
for each $\bi,\bj \in I^\alpha$.
As $\iota$ is the only weight in its block,
$\cL(\iota)$  is projective,
hence so is $\cF_\bi \cL(\iota)$.
Applying Theorem~\ref{cat2} we deduce that
$$
\dim \hom_{\mathfrak{g}}(\cF_\bi \cL(\iota), \cF_\bj \cL(\iota))
= \langle [\cF_\bi \cL(\iota)], [\cF_\bj \cL(\iota)] \rangle
=
 \langle \cF_\bi \cL_\iota, \cF_\bj \cL_\iota\rangle.
$$
A similar application of Theorem~\ref{cat1}, setting $q=1$ at the end,
gives that
$$
 \langle \cF_\bi \cL_\iota, \cF_\bj \cL_\iota\rangle=
\dim \hom_{K^\La_\alpha}(F_\bi L(\iota), F_\bj L(\iota)).
$$
This establishes (\ref{nts}).
\end{proof}

\begin{Remark}\rm\label{newr3}
The equality $\dim E^\La_\alpha = \dim R^\La_\alpha$ in the above proof of
Corollary~\ref{id3}
was deduced using Theorem~\ref{id1}, hence the argument relies ultimately on properties
of parabolic category $\mathcal{O}$.
Alternatively, this equality of dimensions
can be proved by using the observation made in the last sentence of Remark~\ref{newr1}
together with \cite[Theorem 4.20]{BKllt}, which gives a combinatorial formula for $\dim R^\La_\alpha$ in terms
of standard tableaux.
One advantage of this alternate argument is that it is valid over an arbitrary ground field (including positive characteristic),
not just over $\C$.
\end{Remark}

\phantomsubsection{Comparison of embeddings}
There is one more important identification to be made.
Consider the obvious embedding
$H_d \hookrightarrow H_{d+1}$.
This factors through the quotients to induce an embedding
$\theta:H_d^{o+m,o+n} \hookrightarrow H_{d+1}^{o+m,o+n}$.
Composing $\theta$ on one side with the embedding
$e_\alpha H_d^{o+m,o+n} \hookrightarrow H_d^{o+m,o+n}$
and on the other side with the projection
$H_{d+1}^{o+m,o+n} \twoheadrightarrow e_{\alpha+\alpha_i}
H_{d+1}^{o+m,o+n}$ defined by multiplication by the
central idempotent $e_{\alpha+\alpha_i}$, we obtain
a non-unital algebra homomorphism
\begin{equation}\label{thetaii}
\theta_i:e_{\alpha} H_d^{o+m,o+n} \rightarrow e_{\alpha+\alpha_i}
H_{d+1}^{o+m,o+n}.
\end{equation}
This maps the identity element $e_{\alpha}$
to the idempotent
\begin{equation}\label{baa2}
e_{\alpha,\alpha_i} :=
\sum_{\substack{\bi \in I^{\alpha+\alpha_i} \\ i_{d+1} = i}} e(\bi),
\end{equation}
just like in (\ref{baa}).
Recall also the homomorphism
$\theta_i$
from (\ref{thetai}).

\begin{Lemma}\label{af}
The following diagram commutes:
$$
\begin{CD}
e_{\alpha}H^{o+m,o+n}_d&@>\theta_i>>&e_{\alpha+\alpha_i}H^{o+m,o+n}_{d+1}\\
@V\sigma VV&&@VV\sigma V\\
R^\La_\alpha&@>>\theta_i>&R^\La_{\alpha+\alpha_i}.
\end{CD}
$$
\end{Lemma}

\begin{proof}
This is clear from the explicit form of the isomorphism
in Theorem~\ref{id2} together with the definitions
of (\ref{thetai}) and (\ref{thetaii}).
\end{proof}

For the remainder of the article
we will simply {\em identify} the following four algebras
according to the isomorphisms from  Theorems~\ref{id1} and \ref{id2} and
Corollary~\ref{id3}:
\begin{equation}\label{4id}
E^\La_\alpha \equiv R^\La_\alpha \equiv e_{\alpha} H^{o+m,o+n}_d \equiv \End_{\mathfrak{g}}(\cT^\La_\alpha)^{\op}.
\end{equation}
Under the first of these identifications,
the non-unital algebra homomorphism $\theta_i$ from
Theorem~\ref{thetaithm}
coincides with the map $\theta_i$
from (\ref{thetai}),
as follows from the commutative diagram in
Corollary~\ref{cddcd}.
Under the second of these identifications,
the map $\theta_i$ from (\ref{thetai}) coincides
with the map $\theta_i$ from (\ref{thetaii}),
as follows from the commutative diagram in Lemma~\ref{af}.
This justifies our use of the same notation $\theta_i$
in all the settings.

\phantomsubsection{\boldmath Schur functors: the category $\mathcal O$ side}
For $\alpha \in Q_+$,
consider the category $\rep{R^\La_\alpha}$
of (ungraded) finite dimensional left $R^\La_\alpha$-modules.
As we noted in (\ref{baa2}), the homomorphism
$\theta_i:R^\La_\alpha \rightarrow R^\La_{\alpha+\alpha_i}$
maps the identity element $e_{\alpha}$ of $R^\La_\alpha$ to the
idempotent $e_{\alpha,\alpha_i}$.
Via this homomorphism, if $M$ is any left (resp.\ right)
$R^\La_{\alpha+\alpha_i}$-module then $e_{\alpha,\alpha_i} M$
(resp.\ $M e_{\alpha,\alpha_i}$) is naturally an
$R^\La_\alpha$-module.
Following \cite[$\S$3.3]{BKariki}, we introduce the
{\em $i$-induction and $i$-restriction functors}
\begin{equation*}
\cF_i:\rep{R^\La_\alpha} \rightarrow \rep{R^\La_{\alpha+\alpha_i}},\qquad
\cE_i:\rep{R^\La_{\alpha+\alpha_i}} \rightarrow \rep{R^\La_\alpha},
\end{equation*}
by defining
$\cF_i$ to be the right exact functor
$R^\La_{\alpha+\alpha_i} e_{\alpha,\alpha_i} \otimes_{R^\La_\alpha} ?$
and $\cE_i$ to be the exact functor
defined by left multiplication by the idempotent $e_{\alpha,\alpha_i}$.
By adjointness of tensor and hom,
$(\cF_i, \cE_i)$ is an adjoint pair of functors. In fact, it is known that
$\cE_i$ and $\cF_i$ are biadjoint, hence in particular $\cF_i$ is actually
exact; see \cite[Lemma 8.2.2]{Kbook} (the functors
$\cE_i$ and $\cF_i$ are the same as the induction and restriction
functors there restricted to particular blocks).

Recalling that we have identified
$R^\La_\alpha$ with $\End_{\mathfrak{g}}(\cT^\La_\alpha)^{\op}$
in (\ref{4id}),
we introduce the {\em Schur functor} (or {\em quotient functor}
in the sense of \cite[$\S$III.1]{Gab}):
\begin{equation}\label{Sf}
\pi
:= \hom_{\mathfrak{g}}(\cT^\La_\alpha, ?):\cO^\La_\alpha \rightarrow \rep{R^\La_\alpha}.
\end{equation}
In view of Lemma~\ref{pring2}
we are in a well-studied situation.
Assuming that $\La-\alpha \in \PImn$,
define
\begin{equation}\label{mods1}
\cD(\la) := \pi (\cL(\la)),
\qquad
\cY(\la) := \pi(\cP(\la)).
\end{equation}
for any $\la \in \La-\alpha$.
The following lemma collects some basic facts about this
situation; see \cite[$\S$3.5]{BKariki} for
a more detailed account.

\begin{Lemma}\label{sf1}
Assume that $\Ga := \La-\alpha \in \PImn$.
For $\la \in \Ga$, the module $\cD(\la)$ is non-zero if and only if
$\la \in \Ga^\circ$. The modules
$\{\cD(\la)\:|\:\la \in \Ga^\circ\}$
give a complete set of representatives for the isomorphism classes
of irreducible $R^\La_\alpha$-modules.
Moreover:
\begin{itemize}
\item[(i)] For $\la \in \Ga^\circ$, $\cY(\la)$ is the projective cover
of $\cD(\la)$.
\item[(ii)] The functor $\pi$ is fully faithful on projectives.
\item[(iii)] There is an isomorphism
$\eta:\cF_i \circ \pi\stackrel{\sim}{\rightarrow}
\pi \circ \cF_i$
of functors from $\rep{K^\La_\alpha}$ to $\rep{R^\La_{\alpha+\alpha_i}}$
for each $i \in I$.
\end{itemize}
\end{Lemma}

\begin{proof}
If $m \geq n$ this follows by
a special case of \cite[Theorem 3.7]{BKariki}.
The case $m < n$ can be treated similarly; see also
\cite[Theorem 4.15]{BKariki}.
\iffalse
If $m \geq n$,
the category $\mathcal O^\La_\alpha$ here is the same as the category
denoted $\mathcal O^\La_\alpha$ in \cite[$\S$3.1]{BKariki}.
Moreover,
for $\la \in \Ga$,
our modules $\cP(\la)$ and $\cL\la)$ are the same as the modules
$\cP(A)$ and $\cL(A)$ in \cite{BKariki},
where $A$ is the column strict $\La$-tableau with
entries
$(\la+\rho,\eps_1)>\cdots>(\la+\rho,\eps_m)$ in the left hand column
and
$(\la+\rho,\eps_{m+1})>\cdots>(\la+\rho,\eps_{m+n})$ in the right hand column.
Finally we observe that the condition that $\la$ is of maximal defect is
equivalent to the condition that
$A$ is a standard tableau in the sense of \cite[$\S$2.1]{BKariki}.

Suppose instead that $m \leq n$.
Then the category $\mathcal O^\La_\alpha$ here is the same as the category
$\widetilde \cO^\La_\alpha$ from \cite[$\S$3.1]{BKariki},
and for $\la \in \Ga$ our $\cP(\la)$ is the module denoted
$\widetilde\cP(A)$ there, where $A$ is the column strict $\La$-tableau
with entries $(\la+\rho,\eps_1),\dots,(\la+\rho,\eps_m)$ in the rightmost
column and $(\la+\rho,\eps_{m+1}),\dots,(\la+\rho,\eps_{m+n})$ in
the leftmost column. Also, the condition that $\la$ belongs to $(\La-\alpha)^\circ$
is equivalent to the condition that $A$ is a reverse-standard $\La$-tableau
in the sense of \cite[$\S$2.5]{BKariki}.
\fi
\end{proof}

For any $\alpha \in Q_+$, we now define
$\cP^\La_\alpha \in \mathcal O^\La_\alpha$
and $\cY^\La_\alpha \in \rep{R^\La_\alpha}$ by
\begin{equation}\label{gut}
\cP^\La_\alpha :=
\bigoplus_{\la \in \Ga} \cP(\la),\qquad\qquad
\cY^\La_\alpha :=
\bigoplus_{\la \in \Ga} \cY(\la)
\end{equation}
if $\Ga := \La-\alpha \in \PImn$,
or $\cP^\La_\alpha = \cY^\La_\alpha := \{0\}$
if $\La-\alpha \notin \PImn$.
So $\cP^\La_\alpha$ is a {\em minimal projective generator}
for $\cO^\La_\alpha$, and $\cY^\La_\alpha$ is its image under the
Schur functor $\pi$.
Theorem~\ref{sf1}(ii) immediately implies the following corollary.

\begin{Corollary}\label{mapi}
The functor $\pi$ defines an algebra isomorphism
$$
i:\End_{\mathfrak{g}}\big(\cP^\La_\alpha)^{\op}
\stackrel{\sim}{\rightarrow}
\End_{R^\La_\alpha}\big(\cY^\La_\alpha)^{\op}.
$$
\end{Corollary}

\phantomsubsection{\boldmath Schur functors: the
diagram algebra side}
Now we go back to the diagram algebra side of the story.
Let $\Rep{R^\La_\alpha}$ denote the category of finite
dimensional {\em graded} left $R^\La_\alpha$-modules,
recalling the grading on $R^\La_\alpha$ introduced immediately
after (\ref{R7}).
Let
\begin{equation}\label{forget2}
\forget:\Rep{R^\La_\alpha} \rightarrow \rep{R^\La_\alpha}
\end{equation}
be the forgetful functor here, just like in (\ref{forget1}).
There are graded versions of the $i$-induction and
$i$-restriction functors
\begin{equation*}
F_i:\Rep{R^\La_\alpha} \rightarrow \Rep{R^\La_{\alpha+\alpha_i}},\qquad
E_i:\Rep{R^\La_{\alpha+\alpha}} \rightarrow \Rep{R^\La_\alpha},
\end{equation*}
where $F_i$ is
the functor
$R^\La_{\alpha+\alpha_i}e_{\alpha,\alpha_i}
\otimes_{R^\La_\alpha} ?\langle 1-(\La-\alpha,\delt_i-\delt_{i+1})\rangle$
and $E_i$ is
defined by left multiplication by the idempotent $e_{\alpha,\alpha_i}$ as
before.
It is obvious from these definitions that
\begin{equation}\label{whenyouforget}
\forget \circ F_i \cong \cF_i \circ \forget,\qquad
\forget \circ E_i \cong \cE_i \circ \forget.
\end{equation}
Hence $F_i$ and $E_i$ are exact,
since we already know that about $\cF_i$ and $\cE_i$.
For the next lemma, we let
\begin{equation}
D_i^{\pm 1}: \Rep{R^\La_\alpha} \rightarrow \Rep{R^\La_\alpha}
\end{equation}
be the degree shift functor mapping a module
$M$ to $M\langle \pm (\La-\alpha,\delt_i)\rangle$,
as in (\ref{Ki}).

\begin{Lemma}\label{adjpair}
There exists a canonical adjunction of degree zero
which makes
$(F_i D_iD_{i+1}^{-1} \langle -1\rangle, E_i)$
into an adjoint pair of functors.
\end{Lemma}

\begin{proof}
By the preceeding definition of $F_i$,
we have that
$$
F_i D_i D_{i+1}^{-1} \langle -1 \rangle
=
R^\La_{\alpha+\alpha_i} e_{\alpha,\alpha_i} \otimes_{R^\La_\alpha} ?.
$$
Given this, the lemma reduces to the usual
adjointness of tensor and hom.
\end{proof}

\begin{Remark}\rm
Although not needed here, there is another degree zero adjunction
making
$(E_i, F_i D_i^{-1} D_{i+1} \langle 1 \rangle)$
into
an adjoint pair of functors too.
This can be derived from Lemma~\ref{adjunctions} using
Lemma~\ref{sf2}(iii) and Remark~\ref{come} below.
\end{Remark}

Recalling that we have identified
$R^\La_\alpha$ with $\End_{K^\La_\alpha}(T^\La_\alpha)^{\op}$
in (\ref{4id}),
we also have the {\em graded Schur functor}
\begin{equation}\label{gsf}
\pi
:= \hom_{K^\La_\alpha}(T^\La_\alpha, ?):
\Rep{K^\La_\alpha} \rightarrow \Rep{R^\La_\alpha}.
\end{equation}
We have used the same notation for this as in (\ref{Sf}),
hoping that it is always clear from context which we mean.
Assuming that $\La-\alpha \in \PImn$,
define
\begin{equation}\label{mods2}
D(\la) := \pi (L(\la)),
\qquad
Y(\la) := \pi(P(\la))
\end{equation}
for any $\la \in \La-\alpha$.
The following gives a graded version of Lemma~\ref{sf1}.

\begin{Lemma}\label{sf2}
Assume that $\Ga := \La-\alpha \in \PImn$.
For $\la \in \Ga$, the module $D(\la)$ is non-zero if and only if
$\la \in \Ga^\circ$. The modules
$\{D(\la)\langle j \rangle\:|\:\la \in \Ga^\circ, j \in \Z\}$
give a complete set of representatives for the isomorphism classes
of irreducible graded $R^\La_\alpha$-modules.
Moreover:
\begin{itemize}
\item[(i)] For $\la \in \Ga^\circ$, $Y(\la)$ is the projective cover
of $D(\la)$.
\item[(ii)] The functor $\pi$ is fully faithful on projectives.
\item[(iii)] There is an isomorphism
$\eta:F_i\circ \pi \stackrel{\sim}{\rightarrow} \pi\circ F_i$
of functors from $\Rep{K^\La_\alpha}$ to $\Rep{R^\La_{\alpha+\alpha_i}}$
for each $i \in I$.
\end{itemize}
\end{Lemma}

\begin{proof}
The first statement and (i) follow from Lemma~\ref{pring1}
by the same general argument as in the proof of Lemma~\ref{sf1}(i).
For (ii), see \cite[Corollary 6.3]{BS2}.
Finally to prove (iii),
we first construct a natural
transformation
$$
\eta:F_i \circ \pi \rightarrow \pi \circ F_i.
$$
Take a graded $R^\La_\alpha$-module $M$.
The functor $F_i$ defines a natural degree-preserving linear map
\begin{multline*}
\bar\eta_M:
\hom_{K^\La_\alpha}(T^\La_\alpha,
M)\langle 1-(\La-\alpha,\delt_i-\delt_{i+1})\rangle\\
\longrightarrow
\hom_{K^\La_{\alpha+\alpha_i}}(F_i T^\La_\alpha
\langle (\La-\alpha,\delt_i-\delt_{i+1})-1\rangle, F_i M).
\end{multline*}
In view of Theorem~\ref{thetaithm}, this map is an $R^\La_\alpha$-module
homomorphism.
Note next by (\ref{defede}) that
$(\La-\alpha,\delt_i-\delt_{i+1})-1=\defect(\La-\alpha-\alpha_i)-\defect(\La-\alpha)$.
Recalling also the definition (\ref{dc}), it follows that
$F_i T^\La_\alpha \langle (\La-\alpha,\delt_i-\delt_{i+1})-1\rangle =
T^\La_{\alpha+\alpha_i} e_{\alpha,\alpha_i}$.
So we have that
\begin{align*}
\hom_{K^\La_\alpha}(T^\La_\alpha,M)\langle 1-(\La-\alpha,\delt_i-\delt_{i+1})\rangle&=
D_i^{-1} D_{i+1}\langle 1 \rangle \circ \pi (M),\\
\hom_{K^\La_{\alpha+\alpha_i}}(F_i T^\La_\alpha
\langle (\La-\alpha,\delt_i-\delt_{i+1})-1\rangle, F_i M) &=
\hom_{K^\La_{\alpha+\alpha_i}}(T^{\La}_{\alpha+\alpha_i} e_{\alpha,\alpha_i}, F_i M)\\
&=
e_{\alpha,\alpha_i} \hom_{K^\La_{\alpha+\alpha_i}}(T^{\La}_{\alpha+\alpha_i}, F_i M)\\
&=
 E_i \circ \pi \circ F_i(M).
\end{align*}
This means that the maps $\bar\eta_M$ actually define a natural transformation
$$
\bar\eta:D_i^{-1} D_{i+1}\langle 1 \rangle\circ \pi
\rightarrow E_i\circ \pi\circ F_i.
$$
Now we use the adjunction from
Lemma~\ref{adjpair}, to get from this
the natural transformation $\eta:F_i\circ \pi \rightarrow \pi\circ F_i$
that we wanted.

Now we need to show that $\eta$ is an isomorphism.
We first check that it gives an isomorphism
on the module $M = T^\La_\alpha$. In that case
by Corollary~\ref{id3} and Theorem~\ref{thetaithm},
$\bar\eta_M$ can be identified with the map
$$
\theta_i:R^\La_\alpha \rightarrow e_{\alpha,\alpha_i}
R^\La_{\alpha+\alpha_i} e_{\alpha,\alpha_i}.
$$
Pushing through the adjunction from Lemma~\ref{adjpair},
this induces the identity map
$R^\La_{\alpha+\alpha_i}
e_{\alpha,\alpha_i} \rightarrow
R^\La_{\alpha+\alpha_i} e_{\alpha,\alpha_i}$, hence $\eta_M$
is an isomorphism.
From this, Lemma~\ref{pring1} and naturality, it follows that
$\eta$ gives an isomorphism on every prinjective indecomposable
$K^\La_\alpha$-module.
By \cite[Theorem 6.2]{BS2},
 every projective indecomposable $K^\La_\alpha$-module $P$
has a two step resolution $0 \rightarrow P \rightarrow P^0 \rightarrow P^1$
where $P^0$ and $P^1$ are prinjective.
From this we deduce by the five lemma (and the exactness of the functors involved)
that $\eta$ gives an isomorphism on
every projective indecomposable
module. Finally we take an arbitrary $M$ and
consider a two step projective resolution $P_1 \rightarrow P_0 \rightarrow M
\rightarrow 0$ and use the five lemma again to complete the proof.
\end{proof}

\begin{Remark}\rm\label{come}
There is also
an isomorphism $E_i \circ  \pi \stackrel{\sim}{\rightarrow}
\pi \circ E_i$ of functors from $\Rep{K^\La_{\alpha+\alpha_i}}$
to $\Rep{R^\La_{\alpha}}$. Since this does not play a central role
in the rest of the article, we leave the details of its construction to
the reader.
\end{Remark}

\begin{Corollary}\label{mapj}
Identifying
$\widehat T^\La_\alpha$ with $\pi(K^\La_\alpha)$
according to the isomorphism from
Lemma~\ref{precd}, the functor $\pi$
induces a graded algebra isomorphism
$$
j:K^\La_\alpha \equiv \End_{K^\La_\alpha}(K^\La_\alpha)^{\op}
\stackrel{\sim}{\rightarrow}
\End_{R^\La_\alpha}(\widehat T^\La_\alpha)^{\op}.
$$
\end{Corollary}

\phantomsubsection{The main equivalence of categories}
In this subsection, we are going to ignore all gradings.
We still write $P(\la), L(\la), Y(\la), T^\La_\alpha, \widehat{T}^\La_\alpha$,
\dots for the various
modules introduced earlier but view them as ordinary ungraded modules, i.e. we
implicitly apply the functor $\forget$ everywhere.
Recall also the endofunctors $E_i$ and $F_i$ of $\Rep{\KImn}$
from (\ref{gFi}), which are defined by tensoring by certain graded bimodules.
Tensoring with the same underlying bimodules but forgetting the grading
gives endofunctors $E_i$ and $F_i$ of $\rep{\KImn}$.
The next important lemma identifies the
modules
from (\ref{mods1}) and (\ref{mods2}).

\begin{Lemma}\label{ided}
Assume that $\Ga := \La-\alpha \in \PImn$.
\begin{itemize}
\item[(i)] We have that
$\cD(\la) \cong D(\la)$ for any $\la \in \Ga^\circ$.
\item[(ii)]
We have that
$\cY(\la) \cong Y(\la)$ for any $\la\in \Ga$.
\end{itemize}
\end{Lemma}

\begin{proof}
(i)
It is obvious that $\cD(\iota) \cong D(\iota)$
where $\iota$ is the ground-state, because the
algebra $R^\La_0$ is just a copy of the ground field $\C$.
Now take any $\la \in \Ga^\circ$ with $\la \neq \iota$.
By Lemma~\ref{stupidcomb}, we can write
$\la = \tilde f_{i_d} \cdots \tilde f_{i_1}(\iota)$
for $d > 0$.
Set $\mu := \tilde f_{i_{d-1}} \cdots \tilde f_{i_1}(\iota)$.
Proceeding by induction on $d$, we may assume we have shown already that
$\cD(\mu) \cong D(\mu)$.
By Lemma~\ref{a},
$\cL(\la)$ is the irreducible head of $\cF_{i_d}(\cL(\mu))$.
Applying the Schur functor exactly as in the proof of \cite[Theorem 3.10(vi)]{BKariki},
 we deduce that
$\cD(\la)$ is the irreducible head of $\pi\circ \cF_{i_d}(\cL(\mu))$.
Equivalenty by Lemma~\ref{sf1}(iii), $\cD(\la)$
is isomorphic to the irreducible head of $\cF_{i_d} \circ \pi(\cL(\mu))
= \cF_{i_d}(\cD(\mu))$.
A similar argument on the diagram algebra side
using Lemma~\ref{ga}, Lemma~\ref{sf2}(iii) and (\ref{whenyouforget})
gives that $D(\la)$ is isomorphic to
the irreducible head of
$\cF_{i_d}(D(\mu))
\cong \cF_{i_d} (\cD(\mu))$.
Hence $\cD(\la) \cong D(\la)$.

(ii)
To start with suppose that $\la \in \Ga$ is maximal in the
Bruhat order.
Then we have that $P(\la) \cong V(\la)$ by \cite[Theorem 5.1]{BS1}.
Consider the possible $\mu \in \Ga^\circ$ such that
$L(\mu)$ appears as a composition factor of $V(\la)$ (possibly
shifted in degree). As $\la$ is Bruhat maximal, its diagram involves
$p$ $\up$'s followed by $q$ $\down$'s.
By \cite[Theorem 5.2]{BS1}, it must be the case that
$\underline{\mu} \la$ is an oriented cup diagram.
Since $\mu$ is of maximal defect it follows that there is
only one possibility for $\mu$: it is the weight obtained from
$\la$ by switching the rightmost
$|p-q|$ vertices labelled $\up$ with the
leftmost $|p-q|$ vertices labelled $\down$.
Applying the graded Schur functor using Lemma~\ref{sf2},
we deduce that $Y(\la) =
\pi(P(\la)) \cong \pi(V(\la)) \cong
\pi(L(\mu)) \cong D(\mu)$ (ignoring the grading).
Since the combinatorics is exactly the same on the category $\mathcal O$
side (compare Theorems~\ref{cat1} and \ref{cat2}),
the same argument
gives that
$\cY(\la) \cong \cD(\mu)$. In view of (i),
we deduce that $\cY(\la) \cong Y(\la)$.

Now suppose that $\la \in \Ga$ is not maximal in the Bruhat order.
Applying Lemma~\ref{cgl}, we get $\mu \in \Ga$ that is maximal
and $i_1,\dots,i_k \in I$ such that $\la =
\tilde f_{i_k} \cdots \tilde f_{i_1} (\mu)$.
Applying Theorem~\ref{cgt} combined with Theorems~\ref{cat1}
and \ref{cat2}, we deduce for some $N > 0$ that
\begin{align*}
\cF_{i_k} \circ\cdots\circ \cF_{i_1}(\cP(\mu))
&\cong \cP(\la)^{\oplus N},\\
F_{i_k} \circ\cdots\circ F_{i_1}(P(\mu))
&\cong P(\la)^{\oplus N}.
\end{align*}
Now apply the Schur functors on both sides and use Lemmas~\ref{sf1}(iii),
\ref{sf2}(iii) and (\ref{whenyouforget}) to deduce that
\begin{align*}
\cF_{i_k} \circ\cdots\circ \cF_{i_1}(\cY(\mu))
&\cong \cY(\la)^{\oplus N},\\
\cF_{i_k} \circ\cdots\circ \cF_{i_1}(Y(\mu))
&\cong Y(\la)^{\oplus N}.
\end{align*}
The modules on the left hand side here are isomorphic by the
previous paragraph. Hence applying the Krull-Schmidt theorem we get that
$\cY(\la) \cong Y(\la)$ as required.
\end{proof}

Recall for the next lemma that
$K_\Ga$ possesses a system
$\{e_\la\:|\:\la \in \Ga\}$ of mutually orthogonal idempotents
such that
the projective indecomposable $K_\Ga$-module $P(\la)$ is just the left ideal
$K_\Ga e_\la$; see \cite[$\S$5]{BS1}.

\begin{Lemma}\label{mapf}
Recalling the definition (\ref{gut}), there exists an isomorphism of $R^\La_\alpha$-modules
$h:\cY^\La_\alpha \stackrel{\sim}{\rightarrow} \widehat T^\La_\alpha$
mapping each summand $\cY(\la)$ of $\cY^\La_\alpha$
to the summand $\widehat T^\La_\alpha
e_\la \cong Y(\la)$ of $\widehat T^\La_\alpha$.
\end{Lemma}

\begin{proof}
We may assume that $\Ga := \La-\alpha \in \PImn$.
Recalling Lemma~\ref{precd}, we have that
$\widehat T^\La_\alpha = \bigoplus_{\la \in \Ga} \widehat T^\La_\alpha e_\la$,
and each $\widehat T^\La_\alpha e_\la$ is isomorphic to $Y(\la)$.
Now apply Lemma~\ref{ided}(ii).
\end{proof}

Fix once and for all a choice of isomorphism
$h$ as in Lemma~\ref{mapf} for each $\alpha \in Q_+$.
Now we come to what is really the main theorem of the article.

\begin{Theorem}\label{mainiso}
For each $\alpha \in Q_+$, there is an algebra isomorphism
$$
\End_{\mathfrak{g}}\big(\cP^\La_\alpha\big)^{\op}
\stackrel{\sim}{\rightarrow}
K^\La_\alpha, \qquad
\theta \mapsto
j^{-1}(h \circ i(\theta) \circ h^{-1}),
$$
where $i, j$ and $h$ are the maps from
Corollaries~\ref{mapi} and \ref{mapj} and Lemma~\ref{mapf}.
\end{Theorem}

\begin{proof}
This follows because all of $i,j$ and $h$
are already known to be isomorphisms.
\end{proof}

\begin{Corollary}\label{mainequiv}
Viewing $\cP^\La_\alpha$ as a right $K^\La_\alpha$-module
via the isomorphism from Theorem~\ref{mainiso}, the functor
$$
\mathbb{E}_\alpha := \hom_{\mathfrak{g}}\big(\cP^\La_\alpha, ?\big):
\mathcal O^\La_\alpha \rightarrow \rep{K^\La_\alpha}
$$
is an equivalence of categories.
Moreover, if $\Ga := \La-\alpha \in \PImn$ then
$\mathbb{E}_\alpha(\cL(\la)) \cong L(\la),
\mathbb{E}_\alpha(\cV(\la)) \cong V(\la)$
and
$\mathbb{E}_\alpha(\cP(\la)) \cong P(\la)$ for each $\la \in \Ga$.
\end{Corollary}

\begin{proof}
For $\la \in \Ga$ let
$p_\la \in
\End_{\mathfrak{g}}\big(\textstyle\bigoplus_{\la \in \Ga} \cP(\la)\big)^{\op}$
be the projection onto the summand $\cP(\la)$. Then by the definitions of
$i, j$ and $h$ we have that
$c(p_\la) = e_\la$.
It follows easily that $\mathbb{E}_\alpha (\cP(\la)) \cong P(\la)$,
hence also $\mathbb{E}_\alpha (\cL(\la)) \cong L(\la)$.
Finally we get that $\mathbb{E}_\alpha (\cV(\la)) \cong V(\la)$ because
these are the standard modules of the highest weight categories on the two sides.
\end{proof}

\begin{Corollary}\label{maincommuting}
Let $\mathbb{E}_\alpha$ be as in
Corollary~\ref{mainequiv} and $\pi$ denote the respective Schur functors from (\ref{Sf}) and (\ref{gsf}).
Then the following diagram
$$
\begin{CD}
\mathcal{O}^\La_\alpha @>\phantom{hell}\mathbb{E}_\alpha\phantom{hell}>> &
\rep{K^\La_\alpha}
\\\\
&\:\rep{R^\La_\alpha} &\\
\end{CD}
\begin{picture}(0,0)
\put(-86,-6){\makebox(0,0){$\searrow$}}
\put(-102,-3){\makebox(0,0){$\scriptstyle\pi$}}
\put(-90,-2){\line(-1,1){13}}
\put(-47,-6){\makebox(0,0){$\swarrow$}}
\put(-43,-2){\line(1,1){13}}
\put(-32,-3){\makebox(0,0){$\scriptstyle\pi$}}
\end{picture}
$$
commutes in the sense that there is an isomorphism of functors $\pi \circ \mathbb{E}_\alpha
\cong \pi$.
\end{Corollary}

\begin{proof}
Regard $\cY^\La_\alpha$ as an $(R^\La_\alpha, K^\La_\alpha)$-bimodule by defining the right $K^\La_\alpha$-action
to be the one obtained by lifting its natural action on $\widehat{T}^\La_\alpha$ through
the isomorphism $h$ from Lemma~\ref{mapf}.
Then the map
$$
\hom_{R^\La_\alpha}(\widehat{T}^\La_\alpha, R^\La_\alpha)
\rightarrow
\hom_{R^\La_\alpha}(\cY^\La_\alpha, R^\La_\alpha),\qquad
f \mapsto f \circ h
$$
becomes an isomorphism of $(K^\La_\alpha, R^\La_\alpha)$-bimodules.
Recall also from Lemma~\ref{precd} that $\widehat{T}^\La_\alpha \cong \hom_{K^\La_\alpha}(T^\La_\alpha, K^\La_\alpha)$
as $(R^\La_\alpha,K^\La_\alpha)$-bimodules.
Combined with Lemmas~\ref{sf1}(ii) and \ref{sf2}(ii) (``Schur functors are fully faithful on projectives'') and the definitions (\ref{mods1})--(\ref{gut}), we get
the following sequence of $(K^\La_\alpha, R^\La_\alpha)$-bimodule isomorphisms:
\begin{align*}
T^\La_\alpha\equiv \hom_{K^\La_\alpha}(K^\La_\alpha, T^\La_\alpha)
&\cong \hom_{R^\La_\alpha}(\hom_{K^\La_\alpha}(T^\La_\alpha, K^\La_\alpha), \hom_{K^\La_\alpha}(T^\La_\alpha, T^\La_\alpha))\\
&\cong \hom_{R^\La_\alpha}(\widehat{T}^\La_\alpha, R^\La_\alpha)\cong \hom_{R^\La_\alpha}(\cY^\La_\alpha, R^\La_\alpha)\\
&\cong \hom_{R^\La_\alpha}(\hom_{\mathfrak{g}}(\cT^\La_\alpha,\cP^\La_\alpha), \hom_{\mathfrak{g}}(\cT^\La_\alpha, \cT^\La_\alpha))\\
&\cong \hom_{\mathfrak{g}}(\cP^\La_\alpha, \cT^\La_\alpha).
\end{align*}
Now to prove the lemma, take $M \in \cO^\La_\alpha$.
Applying the equivalence of categories $\mathbb{E}_\alpha = \hom_{\mathfrak{g}}(\cP^\La_\alpha, ?)$
and using the bimodule isomorphism just constructed, we get
natural $R^\La_\alpha$-module isomorphisms
\begin{align*}
\hom_{\mathfrak{g}}(\cT^\La_\alpha, M)
&\cong
\hom_{K^\La_\alpha}(\hom_{\mathfrak{g}}(\cP^\La_\alpha, \cT^\La_\alpha), \hom_{\mathfrak{g}}(\cP^\La_\alpha, M))\\
&\cong \hom_{K^\La_\alpha}(T^\La_\alpha, \mathbb{E}_\alpha(M)).
\end{align*}
This shows that $\pi \cong \pi \circ \mathbb{E}_\alpha$.
\end{proof}

\begin{Corollary}\label{maineq}
The functor $\mathbb{E} := \bigoplus_{\alpha \in Q_+} \mathbb{E}_\alpha$
is an equivalence of categories
$$
\mathbb{E}:\OImn \rightarrow \rep{\KImn}
$$
such that
$\mathbb{E}(\cL(\la)) \cong L(\la)$,
$\mathbb{E}(\cV(\la)) \cong V(\la)$ and
$\mathbb{E}(\cP(\la)) \cong P(\la)$ for each $\la \in \LaImn$.
\end{Corollary}

\begin{Corollary}
The parabolic Verma modules $\cV(\la)$ in $\mathcal O(m,n)$ are rigid
for all $\la \in \LaImn$.
\end{Corollary}

\begin{proof}
This is immediate from Corollary~\ref{maineq} and \cite[Corollary 6.7]{BS2}.
\end{proof}

For further discussion of rigidity in the context of parabolic category $\mathcal O$, we refer the reader to \cite[$\S$9.17]{Hbook}.

\phantomsubsection{Identification of special projective functors}
We continue to ignore all gradings throughout the subsection.
Let $\mathbb{E}$ be as in Corollary~\ref{maineq}.

\begin{Theorem}\label{idspf}
There are isomorphisms
$\mathbb{E} \circ \cE_i \cong E_i \circ \mathbb{E}$
and
$\mathbb{E} \circ \cF_i \cong F_i \circ \mathbb{E}$.
\end{Theorem}

\begin{proof}
It suffices to prove the second isomorphism; the first then follows
as $\cE_i$ and $\cF_i$ are biadjoint, as are $E_i$ and $F_i$.
For the second isomorphism,
it is enough to prove for each $\alpha \in Q_+$ that
$\mathbb{E}_{\alpha+\alpha_i} \circ \cF_i \circ \mathbb{E}_\alpha^*
\cong F_i$
as functors from $\mathcal O^\La_\alpha$ to $\mathcal O^\La_{\alpha+\alpha_i}$,
where we write $\mathbb{E}_\alpha^*$ for the functor
$\cP^\La_\alpha \otimes_{K^\La_\alpha} ?$
that is a
quasi-inverse equivalence to $\mathbb{E}_\alpha$.

Observe to start with that
$F_i:\Rep{K^\La_\alpha} \rightarrow \Rep{K^\La_{\alpha+\alpha_i}}$
is the functor defined by tensoring with the $(K^\La_{\alpha+\alpha_i},
K^\La_\alpha)$-bimodule
$\hom_{K^\La_{\alpha+\alpha_i}}(K^\La_{\alpha+\alpha_i},
F_i K^\La_\alpha)$.
Let us identify $\pi(K^\La_\alpha)$
(resp.\ $\pi(K^\La_{\alpha+\alpha_i})$)
with $\widehat T^\La_\alpha$ (resp.\ $\widehat T^\La_{\alpha+\alpha_i}$)
according to
Lemma~\ref{precd},
and
identify $K^\La_{\alpha+\alpha_i}$
with
$\End_{R^\La_{\alpha+\alpha_i}}(\widehat{T}^\La_{\alpha+\alpha_i})^{\op}$
via the isomorphism from Corollary~\ref{mapj}.
 Then
Lemma~\ref{sf2}(ii)--(iii) gives us a
$(K^\La_{\alpha+\alpha_i}, K^\La_\alpha)$-bimodule isomorphism
$$
\hom_{K^\La_{\alpha+\alpha_i}}(K^\La_{\alpha+\alpha_i},
F_i K^\La_\alpha)
\stackrel{\sim}{\rightarrow}
\hom_{R^\La_{\alpha+\alpha_i}}(\widehat{T}^\La_{\alpha+\alpha_i},
F_i \widehat{T}^\La_{\alpha}),
\quad\theta\mapsto\eta_{K^\La_\alpha}^{-1} \circ \pi(\theta).
$$
Similarly, $\mathbb{E}_{\alpha+\alpha_i}
\circ \mathcal F_i \circ \mathbb{E}_\alpha^*$ is defined by
tensoring with the $(K_{\alpha+\alpha_i}^\La, K_\alpha^\La)$-bimodule
$\hom_{\mathfrak{g}}(\cP^\La_{\alpha+\alpha_i}, \cF_i \cP^\La_{\alpha})$
(where $\cP^\La_{\alpha+\alpha_i}$ and $\cP^\La_\alpha$ are viewed as right modules
via the isomorphism from Theorem~\ref{mainiso}).
Lemma~\ref{sf1}(ii)--(iii) gives us a bimodule isomorphism
$$
\hom_{\mathfrak{g}}(\cP^\La_{\alpha+\alpha_i}, \cF_i \cP^\La_{\alpha})
\stackrel{\sim}{\rightarrow}
\hom_{R^\La_{\alpha+\alpha_i}}(\cY^\La_{\alpha+\alpha_i},
\cF_i \cY^\La_\alpha),
\quad\theta \mapsto \eta_{\cP^\La_\alpha}^{-1} \circ \pi(\theta).
$$
So now we are reduced to checking that
$$
\hom_{R^\La_{\alpha+\alpha_i}}(\cY^\La_{\alpha+\alpha_i},
\cF_i \cY^\La_\alpha)
\cong
\hom_{R^\La_{\alpha+\alpha_i}}(\widehat{T}^\La_{\alpha+\alpha_i},
F_i \widehat{T}^\La_{\alpha})
$$
as bimodules. This follows from Lemma~\ref{mapf} and (\ref{whenyouforget}).
\end{proof}

\phantomsubsection{Proof of Theorems~\ref{thm1} and \ref{thm2}}
With notation as in Theorem~\ref{thm1} from the introduction,
let $\Ga \in P(m,n)$.
Choose $o \in \Z$ so that the $o$th vertex and all vertices
to the left of that are labelled $\circ$ in all the weights from
$\Ga$. Then clearly the algebra $K_\Ga$
from the introduction can be identified with the algebra $K_\Ga$
from (\ref{KI}) for the index set
$I := \{o+1,o+2,\dots\}$. Letting $\La := \La_{o+m}+\La_{o+n}$
as usual and $\alpha := \La-\Ga$,
the equivalence $\mathbb{E}_\alpha$ from Corollary~\ref{mainequiv}
gives us an equivalence between
$\mathcal O_\Ga \rightarrow \rep{K_\Ga}$.
Taking the direct sum of these equivalences over all $\Ga \in P(m,n)$
gives the equivalence $\mathbb{E}$ required to prove Theorem~\ref{thm1}.

Now consider Theorem~\ref{thm2}. The equivalence $\mathbb{E}$
from Theorem~\ref{thm1} sets up a bijection between the isomorphism
classes of endofunctors of $\mathcal O(m,n)$ and of
$\rep{K(m,n)}$.
In particular every indecomposable projective functor
$G^t_{\De\Ga}$ on $\rep{K(m,n)}$ lifts to an endofunctor of $\mathcal O(m,n)$
and vice versa. Therefore it is enough to identify
the (at first sight quite different)
notions of projective
functors on the two sides. This follows from
Lemma~\ref{gen}, Lemma~\ref{gen2} and Theorem~\ref{idspf}.

\section{Applications}

In this section we give a couple of applications.
First,
we give a self-contained algebraic proof of a recent conjecture of Khovanov and Lauda from \cite[$\S$3.4]{KLa}
about the cyclotomic algebra $R^\La_\alpha$ for level two weights in finite type $A$.
Then we study further the graded cellular basis
for $R^\La_\alpha$ from Theorem~\ref{cell}, constructing a special basis for level two Specht
modules. This basis has the remarkable property that it also induces a basis in the
irreducible quotients of Specht modules.
In particular we deduce from this a dimension formula for irreducible $R^\La_\alpha$-modules.

\phantomsubsection{The Khovanov-Lauda categorification
conjecture for level two}
Here we briefly discuss another application of the machinery we have developed,
Recall from $\S$\ref{sB} that $U$ is the quantised enveloping algebra
associated to the Lie algebra of $I^+ \times I^+$ matrices.
For $\La$ as in (\ref{gsw}), let $V(\La)$ denote the irreducible
$U$-module of highest weight $\La$.
The vector
$$
v_+ := V_\iota = (v_{o+m}\wedge\cdots\wedge v_{o+1}) \otimes (v_{o+n}\wedge\cdots\wedge v_{o+1}) \in {\textstyle\bigwedge}^m V \otimes {\textstyle\bigwedge}^n V
$$
is a non-zero highest weight vector of weight $\La$, and we can
identify $V(\La)$ with the submodule $U v_+$ of
$\bigwedge^m V \otimes \bigwedge^n V$.
Recalling that $U_{\!\Laurent}$ denotes Lusztig's $\Laurent$-form for
$U$, let
$V(\La)_\Laurent := U_{\!\Laurent} v_+$,
which is the standard $\Laurent$-form for $V(\La)$.
Recall the quasi-canonical basis $\{P_\la\:|\:\la \in \LaImn\}$
from (\ref{boo}).
The following lemma connects this to Lusztig's canonical basis for $V(\La)$.

\begin{Lemma}\label{itscan}
The vectors $\{P_\la\:|\:\la \in \LaImn^\circ\}$
give a basis for $V(\La)_\Laurent$ as a free $\Laurent$-module which
up to rescaling each vector by a power of $q$ coincides with
Lusztig's canonical basis.
More precisely, for $\Ga \in \PImn$,
Lusztig's canonical basis for the $\Ga$-weight space of $V(\La)_\Laurent$
is $\{q^{-\defect(\Ga)} P_\la\:|\:\la \in \Ga^\circ\}$ (cf. the last statement
from Lemma~\ref{princ1}).
\end{Lemma}

\begin{proof}
See \cite[Theorem 2.7]{BKariki}, a special case of which treats the case
$m \geq n$, together with \cite[(2.48)]{BKariki} which
explains how to deduce the case $m < n$.
\end{proof}

Consider the following functor which is left adjoint to
the graded Schur functor $\pi$ from (\ref{gsf}):
\begin{equation}
\pi^*:= T^\La_\alpha\otimes_{R^\La_\alpha} ?
: \Rep{R^\La_\alpha} \rightarrow \Rep{K^\La_\alpha}.
\end{equation}
Let $\Proj{R^\La_\alpha}$ denote the category
of finitely generated projective graded left $R^\La_\alpha$-modules,
with Grothendieck group $[\Proj{R^\La_\alpha}]$.

\begin{Theorem}\label{klt}
Identify the $\Laurent$-modules $[\Rep{\KImn}]$ and
$\bigwedge^m V_\Laurent \otimes \bigwedge^n V_{\Laurent}$
as in Theorem~\ref{cat1}.
Then the functor $\pi^*$ induces an $\Laurent$-module
isomorphism
$$
\pi^*:\bigoplus_{\alpha \in Q_+} [\Proj{R^\La_\alpha}]
\stackrel{\sim}{\rightarrow} V(\La)_\Laurent.
$$
Moreover:
\begin{itemize}
\item[(i)]
Up to a degree shift, $\pi^*$ maps the
isomorphism classes of projective
indecomposable modules to
the canonical basis of $V(\La)_\Laurent$.
\item[(ii)] The endomorphisms of $\bigoplus_{\alpha \in Q_+}
[\Proj{R^\La_\alpha}]$ induced by the $i$-restriction and $i$-induction
functors
$E_i$ and $F_i$ correspond to the action of the Chevalley generators
$E_i$ and $F_i$ of $U_{\!\Laurent}$.
\end{itemize}
\end{Theorem}

\begin{proof}
Suppose that $\Ga := \La-\alpha \in \PImn$.
By Lemma~\ref{sf2}(i), the free $\Laurent$-module
$[\Proj{R^\La_\alpha}]$ has basis
$\{[Y(\la)]\:|\:\la \in \Ga^\circ\}$.
By a standard fact about Schur functors, see e.g. \cite[Theorem 3.7(ii)]{BKariki}, we have that $\pi^*(Y(\la)) \cong P(\la)$ for each $\la \in \Ga^\circ$.
Hence, using also
Theorem~\ref{cat1}(i),
the map $\pi^*$ maps the basis
$\{[Y(\la)]\:|\:\la \in \Ga^\circ\}$ for $[\Proj{R^\La_\alpha}]$
to
$\{P_\la\:|\:\la \in \Ga^\circ\}$.
By  Lemma~\ref{itscan}, the latter collection of vectors
is a basis for the $\Ga$-weight
space of $V(\La)_\Laurent$ that coincides with Lusztig's canonical basis
up to rescaling. This establishes the first statement of
the theorem and (i).
For (ii), note by Lemma~\ref{sf2}(iii) and Remark~\ref{come}
that $\pi^*$ intertwines the $i$-induction and $i$-restriction functors
$E_i$ and $F_i$ with the special projective functors $E_i$ and $F_i$.
So we are done by Theorem~\ref{cat1}(ii).
\end{proof}

Theorem~\ref{klt}
proves the
conjecture formulated by Khovanov and Lauda in
\cite[$\S$3.4]{KLa} for level two weights in finite type $A$.

\phantomsubsection{A special basis for level two Specht modules}
Fix $\alpha \in Q_+$ of height $d$ such that $\La - \alpha \in \PImn$ and set $\Ga := \La-\alpha$.
Recall the graded cellular basis for $R^\La_\alpha$ from Theorem~\ref{cell}. \label{speccie}
For $\la \in \Ga$ we denote the corresponding cell module by $S(\la)$ constructed following the general
procedure of Graham and Lehrer \cite{GL}. Note in particular that $S(\la)$ is automatically a graded module because
our cellular basis is graded. So, as a graded vector space, $S(\la)$ has homogeneous basis
$$
\left\{|\bt^*[\bga^*]|\:\:\bigg|\:\begin{array}{l}\text{for all oriented stretched cup diagrams $\bt^*[\bga^*]$}\\\text{such that $\bga = \ga_0\cdots\ga_d$ with $\ga_d = \la$}\end{array}
\right\},
$$
with $\Z$-grading defined according to (\ref{degzz}).
The left action of a basis vector $|\bs^*[\btau^*]\wr\br[\bsigma]| \in R^\La_\alpha$ on
$|\bt^*[\bga^*]| \in S(\la)$ can be computed as follows. First compute the left action of
$|\bs^*[\btau^*] \wr \br[\bsigma]|$ on the basis vector
$|\bt^*[\bga^*] \, \overline{\la}) \in \widehat{T}^\La_\alpha$ using the usual procedure; in particular
we get zero unless $\br = \bt$ and all mirror image pairs of internal circles in
$\br[\sigma]$ and $\bt^*[\bga^*]$ are oppositely oriented. Then replace all the diagram basis vectors in the
resulting expansion by zero if they do not have the weight $\la$ decorating their top number line,
and drop the cap diagram $\overline{\la}$ from the very top of all the remaining basis vectors
to get back to an element of $S(\la)$.

\begin{Lemma}\label{durham1}
For $\la \in \Ga$ we have that $S(\la) \cong \pi (V(\la))$, where $\pi$ is the graded Schur functor from (\ref{gsf}).
\end{Lemma}

\begin{proof}
Recall from \cite[Theorem 5.1]{BS1} that $V(\la)$ is isomorphic to
the quotient of $P(\la) = K^\La_\alpha e_\la$ by the submodule
$P'(\la)$ spanned by all basis vectors of the form $(a \mu \overline{\la})$ with $\mu > \la$ in the Bruhat order.
Note that
$$
\pi (P(\la))
= \hom_{K^\La_\alpha}(T^\La_\alpha, P(\la))
= \hom_{K^\La_\alpha}(T^\La_\alpha, K^\La_\alpha e_\la)
= \hom_{K^\La_\alpha}(T^\La_\alpha, K^\La_\alpha) e_\la.
$$
Using the isomorphism from Lemma~\ref{precd}, we deduce that $\pi(P(\la)) \cong \widehat{T}^\La_\alpha e_\la$.
Because $\widehat{T}^\La_\alpha e_\la$ is realised explicitly in terms of diagrams, the same is true via
this isomorphism for the Young module $Y(\la) = \pi(P(\la))$
from (\ref{mods2}). In other words,
we can identify $Y(\la)$ with
the left $R^\La_\alpha$-module with
basis given by all diagrams of the form (\ref{thes2}) such that $b = \overline{\la}$.
Let $Y'(\la)$ denote the submodule of $Y(\la)$ spanned by all such diagrams in with $\delta_d > \la$.
Then it is clear from the explicit description of $S(\la)$ from the paragraph before the lemma that
$S(\la) \cong Y(\la) / Y'(\la)$.
Now we
claim that $Y'(\la) \subseteq \pi (P'(\la))$. Given the claim, we get a surjective homorphism
$$
S(\la) \cong Y(\la) / Y'(\la)
\twoheadrightarrow \pi(P(\la)) / \pi(P'(\la)) = \pi (P(\la) / P'(\la)) \cong \pi (V(\la))
$$
and then deduce that $S(\la) \cong \pi(V(\la))$ by comparing dimensions:
forgetting gradings, we have using (\ref{dc}) and adjointness that
$$
\dim \pi(V(\la)) = \dim \hom_{K^\La_\alpha}(T^\La_\alpha, V(\la))
=
\sum_{\bi \in I^\alpha}
\dim \hom_{K^\La_0}(L(\iota), E_{\bi^*} V(\la)),
$$
recalling $\bi^* = (i_d,\dots,i_1)$.
Using \cite[Theorems 3.5--3.6]{BS2} and then
\cite[Theorem 4.5]{BS2}, this is equal to the number of oriented stretched cup diagrams
$\bt^*[\bde^*]$ with $\de_d = \la$, i.e. it is the same as the dimension of the
cell module $S(\la)$.

It remains to prove the claim.
Take an element $y\in Y'(\la)$ represented under the identification $Y(\la) \equiv  \widehat{T}^\La_\alpha e_\la$
by a basis vector of the form
$|\bu^*[\bde^*]\, \overline{\la})$ with $\delta_d > \la$.
We need to show that the map $\phi(? \otimes y):T^\La_\alpha \rightarrow P(\la)$ has image contained in $P'(\la)$.
This follows from the explicit diagrammatic description of the map $\phi$ from (\ref{varphi}) together with
\cite[Corollary 4.5]{BS1}.
\end{proof}

For the next corollary, recall the classification of the irreducible $R^\La_\alpha$-modules
$\{D(\la)\:|\:\la \in \Ga^\circ\}$
from Lemma~\ref{sf2}.

\begin{Corollary}\label{indec} For each $\la \in \Ga$, the cell module $S(\la)$ is indecomposable with irreducible
socle isomorphic to $D(\la^\circ)\langle \deg(\underline{\la^\circ} \la) \rangle$, where $\la^\circ \in \Ga^\circ$
is defined as in \cite[Theorem 6.6]{BS2}.
Moreover if $\la \in \Ga^\circ$ then $S(\la)$ has irreducible head isomorphic to $D(\la)$.
\end{Corollary}

\begin{proof}
By \cite[Theorem 6.6]{BS2}, the cell module $V(\la)$ has irreducible socle isomorphic to $L(\la^\circ)\langle \deg(\underline{\la^\circ}\la)\rangle$. Given this and Lemma~\ref{sf2},
a standard argument involving the Schur functor $\pi$
shows that
$\pi(V(\la))$ has irreducible socle isomorphic to $D(\la^\circ)\langle \deg(\underline{\la^\circ}\la) \rangle$.
In view of Lemma~\ref{durham1}, this proves the statement about the socle of $S(\la)$, hence $S(\la)$ is indecomposable
as its socle is irreducible.
Finally if $\la \in \Ga^\circ$ then $V(\la)$ has irreducible head $L(\la)$ and a similar argument shows that
$S(\la) \cong \pi(V(\la))$ has irreducible head $D(\la) = \pi(L(\la))$.
\end{proof}

\begin{Corollary}\label{durham2}
On forgetting the grading, we have that $S(\la) \cong \pi(\cV(\la))$, where $\pi$ is the ungraded Schur functor from
(\ref{Sf}).
\end{Corollary}

\begin{proof}
This follows from Lemma~\ref{durham1} using Corollary~\ref{maincommuting} and the fact that
$\mathbb{E}_\alpha (\cV(\la)) \cong V(\la)$ by Corollary~\ref{mainequiv}.
\end{proof}

\begin{Corollary}
For any $\la \in \Ga$,
the cell module $S(\la)$ is isomorphic to the graded Specht module from \cite{BKW}
parametrised by the bipartition obtained
from $\la$ by applying the map from Remark~\ref{dinnertime}
(taking $e:=0$, $l := 2$ and $(k_1,k_2) := (o+m,o+n)$ in \cite{BKW}).
\end{Corollary}

\begin{proof}
By a special case of \cite[Theorem 3.7]{BKariki} if $m \geq n$, or \cite[Theorem 4.15]{BKariki} if $m \leq n$,
it is known that $\pi(\cV(\la))$ is isomorphic to the Specht module from \cite{BKW} as an ungraded module.
Hence by Corollary~\ref{durham2} we get that $S(\la)$ is isomorphic to the Specht module on forgetting gradings.
Since it is an indecomposable module by Corollary~\ref{indec},
it follows from this and the unicity of gradings from \cite[Lemma 2.5.3]{BGS}
that $S(\la)$ is isomorphic to the graded Specht module
up to a shift in grading. Finally to see that no shift in grading is required, we observe that $S(\la)$
has the same graded dimension as the graded Specht module. This follows because the two modules have homogeneous
bases indexed by certain sets of oriented stretched cap diagrams and of standard bitableaux, respectively, and these two sets
are in bijection in a way that respects the degrees of the two bases thanks to Remark~\ref{newr1}.
\end{proof}

By the general theory of cellular algebras, the cell module $S(\la)$ is equipped with a symmetric bilinear form $(.,.)$
which is associative in the sense that $(xv,w) = (v,x^*w)$ for all $x \in R^\La_\alpha$ and $v,w \in S(\la)$,
where $*$ is the anti-automorphism from (\ref{Stars}); see e.g. \cite[Definition 2.3]{GL}.
Using the map $\phi$ from (\ref{varphi}), we can reformulate the definition of this form as follows.
For basis vectors $|\bt^*[\bga^*]|, |\bu^*[\bde^*]| \in S(\la)$, their inner product
$(|\bt^*[\bga^*]|, |\bu^*[\bde^*]|)$ is the coefficient of $e_\la = (\underline{\la} \la \overline{\la})$
when $\phi((\underline{\la}\, \bu[\bde]| \otimes |\bt^*[\bga^*]\, \overline\la))$ is expanded in terms of the diagram
basis of $K^\La_\alpha$. The following lemma gives a more concrete description.

\begin{Lemma}\label{easy}
For $|\bt^*[\bga^*]|, |\bu^*[\bde^*]| \in S(\la)$, the inner product
$(|\bt^*[\bga^*]|, |\bu^*[\bde^*]|)$ is equal to $1$ if
$\bt^* = \bu^*$, all
matching pairs of internal circles in $\bt^*[\bga^*]$ and $\bt^*[\bde^*]$ are oppositely oriented, and all boundary circles in both diagrams are anti-clockwise; otherwise it is zero.
\end{Lemma}

\begin{proof}
This follows from the diagrammatic description of
the map $\phi$.
\end{proof}

Note in particular that the bilinear form $(.,.)$ on $S(\la)$ is homogeneous of degree zero, and
the $\bi$-weight spaces $e(\bi) S(\la)$ for different $\bi\in I^\alpha$ are orthogonal.
Let $\rad S(\la)$ denote the radical of the form $(.,.)$, which is a graded $R^\La_\alpha$-submodule of $S(\la)$.
By general theory again, the non-zero $S(\la) / \rad S(\la)$'s give a complete (up to grading shift)
set of non-isomorphic irreducible $S(\la)$-modules; see \cite[Theorem 3.4]{GL}.

\begin{Theorem}
For $\la \in \Ga$, we have that $S(\la) / \rad S(\la) \neq \{0\}$ if and only if $\la \in \Ga^\circ$.
Moreover $S(\la) / \rad S(\la) \cong D(\la)$ for each
$\la \in \Ga^\circ$.
\end{Theorem}

\begin{proof}
Suppose that $S(\la) / \rad S(\la) \neq \{0\}$, i.e. the form $(.,.)$ on $S(\la)$ is non-zero.
Then by Lemma~\ref{easy} there exists
at least one oriented stretched cup diagram $\bt^*[\bga^*]$ with $\ga_d = \la$
in which every generalised cup is anti-clockwise. As $\defect(\Ga) = \cups(\bt^*) - \caps(\bt^*)$,
the lower reduction $a$ of $\bt^*$ has exactly $\defect(\Ga)$ cups. As every generalised cup in $a \la$
is anti-clockwise,
we must have that $a = \underline{\la}$. Hence $\defect(\la) = \defect(\Ga)$, so $\la \in \Ga^\circ$.
To complete the proof, it remains to observe that the number of isomorphism classes of irreducible
$R^\La_\alpha$-modules up to grading shift is equal to the cardinality of the set $\Ga^\circ$.
This is a consequence of Corollary~\ref{id3}, since $T^\La_\alpha$ has exactly
$|\Ga^\circ|$ non-isomorphic indecomposable summands
up to grading shift by Lemma~\ref{pring1}.
The final statement now follows using also Corollary~\ref{indec}.
\end{proof}

Observe finally from Lemma~\ref{easy} that the diagram basis for $S(\la)$ is {\em special} in the sense that it
contains a basis for $\rad S(\la)$. In other words, the non-zero vectors obtained by considering the canonical
images of our basis vectors in the quotient $D(\la) \equiv S(\la) / \rad S(\la)$
give a basis for the irreducible module $D(\la)$.

\begin{Theorem}\label{dimform}
For $\bi \in I^\alpha$, $j \in \Z$ and $\la \in \Ga^\circ$,
the dimension of the homogeneous component of $e(\bi) D(\la)$ of degree $j$
is equal to the number of
oriented stretched cap diagrams of the form $\bt[\bga]$ with the following properties:
\begin{itemize}
\item[(i)] the admissible sequence underlying $\bt$ is equal to $\bi$;
\item[(ii)] the degree of $\bt[\bga]$ in the sense of (\ref{degz}) is equal to $j$.
\item[(iii)] $\ga_d = \la$;
\item[(iv)] all boundary caps of $\bt[\bga]$ are anti-clockwise.
\end{itemize}
\end{Theorem}

\begin{proof}
By Lemma~\ref{easy}, the vectors $|\bt^*[\bga^*]|$ for which $\bt[\bga]$ does {\em not}
satisfy the conditions (i)--(iv) give a basis for $\rad S(\la)$.
\end{proof}

We finish with an example to illustrate Theorem~\ref{dimform}.

\begin{Example}\rm
Take $o=0$, $m=n=2$,
$\alpha = 2\alpha_1+3\alpha_2+2\alpha_3+\alpha_4$
and $\bi = (2,3,4,1,2,2,3,1)$ as in Remark~\ref{newr1}.
Let $\la$ be the weight corresponding to the bipartition $((1,1),(3,3))$ under the bijection from
Remark~\ref{dinnertime}; this corresponds to the weight at the bottom of the diagram (\ref{itmightbewrong}).
Then the graded dimension of $e(\bi) D(\la)$ is $q+q^{-1}$, and its basis is given by the images of the
diagram basis vectors of $e(\bi) S(\la)$ parametrised by the oriented stretched cap diagram
from (\ref{itmightbewrong}) and the one obtained from that by reversing the orientation of the internal circle.
Under the bijection from Remark~\ref{newr1} these diagrams map to the following two standard bitableaux,
which are of degrees $1$ and $-1$, respectively:
$$
\left(\:\diagram{5\cr 8\cr}\:,\:\:\:
\diagram{1&2&3\cr 4&6&7\cr}\:\right),
\qquad
\left(\:\diagram{6\cr 8\cr}\:,\:\:\:
\diagram{1&2&3\cr 4&5&7\cr}\:\right).
$$
Similar considerations show $\rad e(\bi) S(\la)$ has graded dimension
$2+q+2q^2+q^3$.
\end{Example}

\small

\section*{Index of notation}

\begin{tabular}{llr}
\hspace{-5mm}$I^+$&Index set
for vertices on number lines&\hspace{6.5mm}\pageref{II}\\
\hspace{-5mm}$P$&Weight lattice $\bigoplus_{i \in I^+} \Z \delta_i
= \bigoplus_{i \in I^+} \Z \Lambda_i$ of $\mathfrak{gl}(I^+)$&\pageref{srp}\\
\hspace{-5mm}$Q$&Root lattice $\bigoplus_{i \in I} \Z \alpha_i$&\pageref{srp}\\
\hspace{-5mm}$U$&Quantized enveloping algebra associated to $\mathfrak{gl}(I^+)$&\pageref{qea}\\
\hspace{-5mm}$P(m,n;I)$&Lie theoretic weights
of $U$-module $\bigwedge^m V \otimes \bigwedge^n V$&\pageref{pimn}\\
\hspace{-5mm}$\Gamma,\Delta,...$&
Elements of $P(m,n;I)$ which are
{blocks} in the diagram calculus&\\
\hspace{-5mm}$\La(m,n;I)$&Diagrammatic weights, i.e. number lines
decorated by $\down,\up,\circ,\times$&\pageref{wtsdef}\\
\hspace{-5mm}$\lambda,\mu,...$& Elements of $\La(m,n;I)$&\\
\end{tabular}

%This cut in the table needs to be moved around by hand to get the page break
%right.

\begin{tabular}{llr}
\hspace{-5mm}$\sim$&Linkage relation on $\La(m,n;I)$ so that $P(m,n;I) = \La(m,n;I) / \sim$&\pageref{wtsdef}\\
\hspace{-5mm}$\leq$&Bruhat order on $\La(m,n;I)$&\pageref{wtsdef}\\
\hspace{-5mm}$\Lambda=\{\iota\}$&The ground-state block
containing the ground-state weight&\pageref{groundstate}\\
\hspace{-5mm}$\defect(\Ga)$&Defect of a block&\pageref{defectdef}\\
\hspace{-5mm}$V_\la,L_\la,P_\la$&Monomial, dual-canonical,
quasi-canonical basis of $\bigwedge^m V \otimes \bigwedge^n V$&\!\!\!\!\!\!\!\!\pageref{mon},\pageref{ib}\\
\hspace{-5mm}$\mathcal V_\la, \mathcal L_\la, \mathcal P_\la$&Specialisations of
$V_\la,L_\la,P_\la$ at $q=1$&\pageref{spec}\\
\hspace{-5mm}$\cups(t)$&Number of cups in crossingless matching $t$&\pageref{ncup}\\
\hspace{-5mm}$\caps(t)$&Number of caps in crossingless matching $t$&\pageref{ncup}\\
\hspace{-5mm}$\circles(a b)$\!&Number of circles
in $ab$ (cup diagram $a$ under cap diagram $b$)&\pageref{ncirc}\\
\hspace{-5mm}$K(m,n;I)$&The diagram algebra
$\bigoplus_\Ga K_\Ga$ summing over blocks $\Ga \in P(m,n;I)$&\pageref{KI}\\
\hspace{-5mm}$(a \lambda b)$&Basis for $K(m,n;d)$ indexed by
oriented circle diagrams&\\
\hspace{-5mm}$V(\la)$&Graded cell modules of $K(m,n;I)$
for $\la\in\La(m,n;I)$&\pageref{cellmod}\\
\hspace{-5mm}$L(\la)$&Irreducible head of $V(\la)$
&\\
\hspace{-5mm}$P(\la)$&Projective cover of $L(\la)$&\\
\hspace{-5mm}$K^t_{\Delta\Gamma}, K^\bt_{\bGa}$&Geometric bimodules
associated to crossingless matchings&\pageref{sgbpf}\\
\hspace{-5mm}$(a\:\bt[\bga]\:b)$&Basis for $K^\bt_{\bGa}$
indexed by oriented circle diagrams&\\
\hspace{-5mm}$G^{t}_{\Ga\De}, G^{\bt}_{\bGa}$&Projective functors
defined up to shift by tensoring with $K^t_{\Ga\De}, K^{\bt}_{\bGa}$\!&\pageref{sgbpf}\\
\hspace{-5mm}$E_i, F_i$&Special projective functors on $\Rep{K(m,n;I)}$&\pageref{gFi}\\
\hspace{-5mm}$\mathfrak{g}, \mathfrak{b}, \mathfrak{h}$&
Lie algebra $\mathfrak{gl}_{m+n}(\C)$, its usual Borel, Cartan subalgebras&\pageref{rhodef}\\
\hspace{-5mm}$\eps_1,\dots,\eps_{m+n}$\!\!\!\!\!\!\!\!\!\!&Standard basis for $\mathfrak{h}^*$&\pageref{rhodef}\\
\hspace{-5mm}$\rho$&Origin used to define the
embedding of
$\Lambda(m,n;I)$ into $\mathfrak{h}^*$&\!\!\!\!\!\!\!\!\pageref{rhodef},\pageref{Id}\\
\hspace{-5mm}$\mathcal O(m,n;I)$&Sum $\bigoplus_{\Ga \in P(m,n;I)} \mathcal O_\Ga$ of blocks of parabolic category $\mathcal O$ for $\mathfrak{g}$&\pageref{cato}\\
\hspace{-5mm}$\mathcal V(\la)$&Standard modules in $\mathcal O(m,n;I)$
for $\la \in \La(m,n;I)$&\pageref{feet}\\
\hspace{-5mm}$\mathcal L(\la)$&Irreducible head of $\mathcal V(\la)$
&\\
\hspace{-5mm}$\mathcal P(\la)$&Projective cover of $\mathcal L(\la)$&\\
\hspace{-5mm}$\mathcal E_i, \mathcal F_i$&Special projective functors
on $\mathcal O(m,n;I)$&\pageref{donely}\\
\hspace{-5mm}$I^\alpha$&Multi-indices $\bi = (i_1,\dots,i_d) \in I^d$
with $\alpha_{i_1}+\cdots+\alpha_{i_d} = \alpha$&\pageref{mind}\\
\hspace{-5mm}$E_{\bi}, F_{\bi}$&Compositions of special projective functors&\pageref{cf}\\
\hspace{-5mm}$y(i), y_r(\bi)$&Endomorphisms of $F_i, F_{\bi}$&\!\!\!\!\!\!\!\!\pageref{cold},\pageref{ci}\\
\hspace{-5mm}$\psi(ij), \psi_r(\bi)$&Natural transformations
$F_j \circ F_i \rightarrow F_i \circ F_j, F_{\bi} \rightarrow F_{s_r\cdot \bi}$&\!\!\!\!\!\!\!\!\pageref{bold},\pageref{bi}\\
\hspace{-5mm}$K^\La_\alpha, \mathcal O^\La_\alpha$&Alternative notations
for $K_{\La-\alpha}, \mathcal O_{\La-\alpha}$&\!\!\!\!\!\!\!\!\pageref{altk},\pageref{alto}\\
\hspace{-5mm}$T^\La_\alpha, \mathcal T^\La_\alpha$&Prinjective generators
for $K^\La_\alpha, \mathcal O^\La_\alpha$ (``tensor space'')
&\!\!\!\!\!\!\!\!\pageref{dc},\pageref{cT}\\
\hspace{-5mm}$E^\La_\alpha$&Endomorphism algebra of $T^\La_\alpha$&\pageref{Thend}\\
\hspace{-5mm}$\bt[\bga]$&Oriented stretched cap diagram&\pageref{oscd}\\
\hspace{-5mm}$(\bu^*[\bde^*]\wr\bt[\bga])$\!\!\!\!&Basis for $E^\La_\alpha$
indexed by oriented stretched circle diagrams&\pageref{Hbasis}\\
\hspace{-5mm}$H_d^{o+m,o+n}$&Cyclotomic quotient of degenerate
affine Hecke algebra&\pageref{tqa}\\
\hspace{-5mm}$x_r, s_r$&Generators for $H_d^{o+m,o+n}$&\\
\hspace{-5mm}$e(\bi), e_\alpha$&Weight idempotents,
central idempotents in $H_d^{o+m,o+n}$&\pageref{wis2}\\
\hspace{-5mm}$e_\alpha H^{o+m,o+n}_d$\!\!&Endomorphism algebra of $\cT^\La_\alpha$&\pageref{id1}\\
\hspace{-5mm}$R^\La_\alpha$&Cyclotomic quotient of
Khovanov-Lauda-Rouquier algebra&\pageref{EKLGens}\\
\hspace{-5mm}$e(\bi), y_r, \psi_r$&Generators for $R^\La_\alpha$&\\
\hspace{-5mm}$D(\la), \mathcal D(\la)$&Irreducible
modules for $E^\La_\alpha \cong R^\La_\alpha \cong e_\alpha H^{o+m,o+n}_d$&\!\!\!\!\!\!\!\!\pageref{mods1},\pageref{mods2}\\
\hspace{-5mm}$Y(\la), \mathcal Y(\la)$&
Young modules&\!\!\!\!\!\!\!\!\pageref{mods1},\pageref{mods2}\\
\hspace{-5mm}$S(\la)$&
Specht modules&\pageref{speccie}\\
\hspace{-5mm}$(\la^{(1)},\la^{(2)})$&Bipartitions (a combinatorial alternative to
weights)&\pageref{dinnertime}\\
\hspace{-5mm}$(\mathtt{T}^{(1)},\mathtt{T}^{(2)})$&Bitableaux
(an alternative to oriented stretched cap diagrams)\!\!\!\!\!\!\!\!\!\!\!\!&\pageref{newr1}\\
\end{tabular}

\end{document}